\documentclass[a4paper, 12pt, oneside, notitlepage]{book}
\usepackage{amsmath,amssymb,amsthm,graphicx,mathrsfs,bbm,url}
\usepackage[margin=3.2cm]{geometry}
\usepackage{amsthm}
\usepackage{stmaryrd}
\usepackage{wrapfig}
\usepackage{enumitem}
\usepackage{mathtools}
\usepackage{minitoc}
\usepackage[all]{xy}
\usepackage[utf8]{inputenc}
\usepackage[usenames,dvipsnames]{color}
\usepackage[colorlinks=true,linkcolor=Red,citecolor=Green]{hyperref}
\usepackage[super]{nth}
\usepackage[open, openlevel=2, depth=3, atend]{bookmark}
\hypersetup{pdfstartview=XYZ}
\usepackage[font=footnotesize]{caption}
\usepackage{titlesec}
\usepackage{subcaption}
\usepackage{caption}
\usepackage{tikz-cd}
\usepackage{setspace}
\usepackage[normalem]{ ulem }
\usepackage{soul}
\usetikzlibrary{decorations.pathreplacing}
\usetikzlibrary{arrows}
\usetikzlibrary{shapes.misc}
\usetikzlibrary{shapes.symbols}
\usetikzlibrary{patterns}
\titleformat{\subsubsection}[runin]
       {\normalfont\bfseries}
       {\thesubsection}
       {0.5em}
       {}
       [.]

\theoremstyle{plain}
\newtheorem{theorem}{Theorem}[section]
\newtheorem*{theorem*}{Theorem}
\newtheorem*{question*}{Question}
\newtheorem*{example*}{Example}
\newtheorem*{remark*}{Remark}

\newtheorem{lemma}[theorem]{Lemma}
\newtheorem{proposition}[theorem]{Proposition}
\newtheorem{corollary}[theorem]{Corollary}

\newtheorem{claim}[theorem]{Claim}

\theoremstyle{definition}
\newtheorem{definition}[theorem]{Definition}
\newtheorem{example}[theorem]{Example}

\theoremstyle{remark}
\newtheorem{remark}[theorem]{Remark}

\numberwithin{equation}{section}

\newcommand{\B}{\mathbb{B}}
\newcommand{\C}{\mathbb{C}}
\newcommand{\R}{\mathbb{R}}
\newcommand{\Q}{\mathbb{Q}}
\newcommand{\Z}{\mathbb{Z}}

\newcommand{\M}{\mathcal{M}}

\newcommand{\T}{\mathbb{T}}
\newcommand{\V}{\mathbb{V}}
\newcommand{\X}{\mathbf{X}}
\newcommand{\Y}{\mathbf{Y}}

\newcommand{\PP}{\mathbf{P}}
\newcommand{\HH}{\mathbb{H}}
\newcommand{\Ss}{\mathbb{S}}
\newcommand{\eps}{\varepsilon}

\newcommand{\mc}{\mathcal}

\newcommand{\dd}{\mathrm{d}}
\newcommand{\rank}{\mathrm{rank}}
\newcommand{\e}{\mathbf{e}}
\newcommand{\f}{\mathbf{f}}
\newcommand{\Lk}{\mathbf{L}^{\otimes \mathbf{k}}}
\newcommand{\pr}{\mathrm{pr}}
\newcommand{\spec}{\mathrm{spec}}

\DeclareMathOperator{\vol}{vol}
\DeclareMathOperator{\Ell}{ell}
\DeclareMathOperator{\Tr}{Tr}
\DeclareMathOperator{\Op}{Op}
\DeclareMathOperator{\WF}{WF}

\DeclareMathOperator{\id}{id}

\DeclareMathOperator{\Div}{div}

\DeclareMathOperator{\supp}{supp}

\DeclareMathOperator{\comp}{comp}
\DeclareMathOperator{\End}{End}
\DeclareMathOperator{\Lie}{\mc{L}}

\DeclareMathOperator{\ad}{ad}
\DeclareMathOperator{\Ad}{Ad}

\title
{Semiclassical analysis on principal bundles \\
\large Applications to rapid mixing for isometric extensions of Anosov flows, and spectral theory of horizontal Laplacians}

\date{\today}

\author{Mihajlo Ceki\'{c}, Thibault Lefeuvre}
\begin{document}

\dominitoc

\maketitle

\hspace{1cm}

\newpage

\section*{Abstract}

Let $G$ be a compact Lie group. We introduce a semiclassical framework, called \emph{Borel-Weil calculus}, to investigate $G$-equivariant (pseudo)differential operators acting on $G$-principal bundles over closed manifolds. In this calculus, the semiclassical parameters correspond to the highest roots in the Weyl chamber of the group $G$ that parametrize irreducible representations, and operators are pseudodifferential in the base variable, with values in Toeplitz operators on the flag manifold associated to the group. This monograph unfolds two main applications of our calculus.

Firstly, in the realm of dynamical systems, we obtain explicit sufficient conditions for rapid mixing of volume-preserving partially hyperbolic flows obtained as extensions of an Anosov flow to a $G$-principal bundle (for an arbitrary $G$). In particular, when $G = \mathrm{U}(1)$, we prove that the flow on the extension is rapid mixing whenever the Anosov flow is not jointly integrable, and the circle bundle is not torsion. When $G$ is semisimple, we prove that ergodicity of the extension is equivalent to rapid mixing. Secondly, we study the spectral theory of sub-elliptic Laplacians obtained as horizontal Laplacians of a $G$-equivariant connection on a principal bundle. When $G$ is semisimple, we prove that the horizontal Laplacian is globally hypoelliptic as soon as the connection has a dense holonomy group in $G$. Notably, this result encompasses all flat bundles with a dense monodromy group in $G$. We also prove a quantum ergodicity result for flat (and in some situations non-flat) principal bundles under a suitable ergodicity assumption.

We believe that this monograph will serve as a cornerstone for future investigations applying the Borel-Weil calculus across different fields, including tensor tomography and geometric inverse problems, Quantum Ergodicity on non-flat principal bundles, as well as wave decay for sub-elliptic Laplacians.

\newpage

\hspace{1cm}

\newpage

\tableofcontents

\newpage

\hspace{1cm}

\newpage

\chapter{Introduction}

We review the content of this monograph and indicate further possible applications of the Borel-Weil calculus.

\minitoc

\newpage

The purpose of this monograph is to study the class of (pseudo)differential operators defined on a $G$-principal bundle $P \to M$, where $M$ is a closed manifold and $G$ is a compact Lie group, that commute with the right-action of the group. When $P$ is equipped with a $G$-invariant connection, natural examples of such (differential) operators are provided by horizontal lifts of vector fields on $M$, or by the horizontal Laplacian associated to the connection. 

To study these operators, we shall introduce an algebra of semiclassical pseudodifferential operators $\Psi^{\bullet}_{h,\mathrm{BW}}(P)$ defined on the flag bundle $F := P/T$, where $T \leqslant G$ is a maximal torus. The equivariant (pseudo)differential operators under consideration in this monograph will naturally induce operators in the Borel-Weil calculus. The Borel-Weil calculus is developed in Chapters \ref{chapter:analysis} and \ref{chapter:analytic}. This formalism is applied to two distinct problems:
\begin{enumerate}[label=(\roman*)]
\item \textbf{Dynamical systems.} We prove that the frame flow over a negatively curved Riemannian manifold is rapid mixing whenever it is ergodic, that is the correlation of functions, computed with respect to the natural smooth invariant volume, decays faster than any polynomial power of time. More generally, our study encompasses all isometric extensions of volume-preserving Anosov flows, see \S\ref{ssection:dynamical-systems} and Chapter \ref{chapter:flow}.

\item \textbf{Sub-elliptic Laplacians.} We study horizontal Laplacians obtained from a connection on a $G$-principal bundle; these operators are sub-elliptic since their principal symbol vanishes on a non-compact submanifold in phase space. We prove that they are (globally) hypoelliptic whenever the connection admits a dense holonomy group and $G$ is semisimple. This turns out to be false without the semisimplicity assumption (e.g. $G=\mathrm{U}(1)$). In particular, on a closed manifold $M$, all representations $\rho : \pi_1(M) \to G$ with dense image $\rho(\pi_1(M)) \leqslant G$ give rise to a flat connection on a $G$-principal bundle with this property, and in this example H\"ormander's hypoelliptic theorem does not apply. Finally, we will also prove a quantum ergodicity statement for flat (and also in some situations non-flat) bundles under a certain ergodicity assumption. We refer to Chapter \ref{chapter:hypoelliptic} and \S\ref{ssection:hypoelliptic-laplacians} for further details.
\end{enumerate}

\section{Semiclassical analysis on principal bundles}

\label{section:sc-analysis}

In Chapters \ref{chapter:analysis} and \ref{chapter:analytic}, we analyze (pseudo)differential operators $\mathbf{Q} : C^\infty(P) \to C^\infty(P)$ that are invariant by the right-action of the group in the fibers. It is straightforward to show that such operators induce a natural family of operators
\[
\mathbf{Q}_{\lambda} : C^\infty(M,E^\lambda) \to C^\infty(M,E^\lambda)
\]
on every associated vector bundle $E^\lambda := P \times_\lambda V_\lambda$ over $M$, where $\lambda : G \to \mathrm{GL}(V_\lambda)$ is an irreducible representation of $G$. The purpose of the Borel-Weil calculus is to provide a suitable framework in which the isomorphism classes of irreducible unitary $G$-representations $\lambda \in \widehat{G}$ can be interpreted as a natural semiclassical parameter.

Given a maximal torus $T \leqslant G$, the flag manifold $G/T$ admits a natural complex structure. By the Borel-Weil theory, every irreducible representation $\lambda \in \widehat{G}$ of $G$ can be realized as the space of holomorphic sections of a certain line bundle $J^\lambda \to G/T$ over the flag manifold, that is $V_\lambda = H^0(G/T,J^\lambda)$. More precisely, the irreducible representations $\lambda \in \widehat{G}$ are parametrized by integer points $\mathbf{k}$ of a certain cone in the dual to the Lie algebra of the maximal torus (the so-called Weyl chamber). To each $\lambda \in \widehat{G}$ corresponds a $d$-tuple $\mathbf{k} := (k_1, \dotsc, k_d)$ of (possibly signed) integers, where $d$ is the rank of $G$. The line bundle $J^\lambda \to G/T$ is then given by $J^\lambda = J_1^{\otimes k_1} \otimes \dotsm \otimes J_d^{\otimes k_d} =: \mathbf{J}^{\otimes \mathbf{k}}$, where $J_1, \dotsc, J_d \to G/T$ are some specific holomorphic line bundles over $G/T$. The parameter $\mathbf{k}$ will play the role of a semiclassical parameter in the Borel-Weil calculus. Notice that, unlike standard semiclassical analysis, the semiclassical parameter $\mathbf{k}$ is now a $d$-tuple instead of a real number.

\begin{example*} In this example we assume that $M$ is a point. If $G = \mathrm{U}(1)= \mathbb{R}/2\pi\Z$, then $T=G$ and $G/T$ reduces to a point. The line bundle $J_1 =: L \to \{\bullet\}$ is thus isomorphic to $\C$ and corresponds to the span of the function $\theta \mapsto e^{i\theta}$ on $\mathrm{U}(1)$; the line bundles $L^{\otimes k}$ (with $k \in \Z$) correspond to the span of the function $\theta \mapsto e^{ik\theta}$. One therefore retrieves the usual Fourier series on the circle.

If $G= \mathrm{SO}(3)$, then $T = \mathrm{SO}(2)$ and the flag manifold is given by $\mathrm{SO}(3)/\mathrm{SO}(2) = \C \mathbb{P}^1$. The irreducible representations representations of $\mathrm{SO}(3)$ are all given by $H^0(\C \mathbb{P}^1, L^{\otimes k})$, where $L \to \C \mathbb{P}^1$ is the tautological line bundle, and $k \geq 0$. (Equivalently, $H^0(\C \mathbb{P}^1, L^{\otimes k}) \simeq \Omega_k := \{u \in C^\infty(\C \mathbb{P}^1) ~|~ \Delta_{\C \mathbb{P}^1}u=k(k+1)u\}$, the space of spherical harmonics of degree $k$.)
\end{example*}

The description of irreducible representations using the Borel-Weil correspondence has the following consequence. The space of sections $C^\infty(M,E^{\lambda})$ of an associated vector bundles $E^{\lambda} \to M$ can be realized as the space of fiberwise holomorphic sections $C^\infty_{\mathrm{hol}}(F,\mathbf{L}^{\otimes \mathbf{k}})$ of a corresponding line bundle $\mathbf{L}^{\otimes \mathbf{k}} \to F$, where $F := P/T$ is the flag bundle over $M$ all of whose fibers are biholomorphic to $G/T$. (Over a single fiber $F_x$ of $F$, $x \in M$, the line bundle $\mathbf{L}^{\otimes \mathbf{k}}|_{F_x} \to F_x$ is biholomorphic to $\mathbf{J}^{\otimes \mathbf{k}} \to G/T$.) By the Peter-Weyl theorem, any function $f \in C^\infty(P)$ can thus be naturally identified with a sequence of elements $(f_{\mathbf{k},i})_{\mathbf{k} \in \widehat{G},i =1,\dotsc ,d_{\mathbf{k}}}$, where $f_{\mathbf{k},i} \in C^\infty_{\mathrm{hol}}(F,\mathbf{L}^{\otimes \mathbf{k}})$ and $d_{\mathbf{k}} := \dim(H^0(G/T,\mathbf{J}^{\otimes \mathbf{k}}))$ is the (complex) dimension of the vector space of the corresponding irreducible representation.

A (pseudo)differential operator $\mathbf{Q} \in \Psi^{\bullet}_{h, \mathrm{BW}}(P)$ in the Borel-Weil calculus is then defined as a family of pseudodifferential operators $(\mathbf{Q}_{\mathbf{k}, i})_{\mathbf{k} \in \widehat{G},i =1,\dotsc ,d_{\mathbf{k}}}$, such that
\[
\mathbf{Q}_{\mathbf{k}, i} : C^\infty_{\mathrm{hol}}(F,\mathbf{L}^{\otimes \mathbf{k}}) \to C^\infty_{\mathrm{hol}}(F,\mathbf{L}^{\otimes \mathbf{k}}),
\]
i.e. $\mathbf{Q}_{\mathbf{k}, i}$ preserves the property of being fiberwise holomorphic for a section $f_{\mathbf{k}} \in C^\infty_{\mathrm{hol}}(F,\mathbf{L}^{\otimes \mathbf{k}})$. Furthermore, we require that the family of pseudodifferential operators $(\mathbf{Q}_{\mathbf{k}, i})_{\mathbf{k} \in \widehat{G},i =1,\dotsc ,d_{\mathbf{k}}}$ be semiclassical in the parameter $h := 1/|\mathbf{k}|$. This class of operators actually corresponds to pseudodifferential operators on $M$ with values in Toeplitz operators on the fibers of $F$. Proving the latter interpretation turns out to be involved and is the content of Theorems \ref{theorem:penible} and \ref{theorem:quantization-bw} in Chapter \ref{chapter:analytic}. All natural $G$-equivariant differential operators on principal bundles (such as connections, horizontal vector fields, horizontal Laplacians, etc.) induce operators in the Borel-Weil calculus.

This calculus should be thought of as a non-Abelian version of a calculus introduced by Charles in his thesis \cite{Charles-00} to study the magnetic Laplacian (this corresponds to the case $G = \mathrm{U}(1)$). His calculus deals with semiclassical operators on $C^\infty(M,L^{\otimes k})$, where $L \to M$ is a fixed Hermitian line bundle, and $k \in \Z_{\geq 0}$ is seen as a semiclassical parameter (see also \cite{Guillarmou-Kuster-21}). In the Borel-Weil calculus, the space $M$ is replaced by the flag bundle $F := P/T$ over $M$, and $L^{\otimes k}$ is replaced by a `multi' line bundle $\mathbf{L}^{\otimes \mathbf{k}} = L_1^{\otimes k_1} \otimes \dotsm \otimes L_d^{\otimes k_d}$ over $F$. In the process of writing this monograph, we also discovered recent preprints \cite{Ben-Ovadia-Ma-Rogdriguez-Hertz-24, Ma-Ma-23} where similar ideas seem to be developed in the specific case where the bundle $P \to M$ is flat, in view of studying Quantum Ergodicity (see \S\ref{ssection:weyl} for further comparison).

\section{Applications}

We provide two main applications of the Borel-Weil calculus. 

\subsection{Decay of correlations for isometric extensions}

\label{ssection:dynamical-systems}

In Chapter \ref{chapter:flow}, we study isometric extensions of volume-preserving Anosov flows and establish quantitative decay for their correlation functions. More precisely, let $M$ denote a closed manifold equipped with a volume-preserving Anosov flow $(\varphi_t)_{t \in \R}$, and let $\pi : P \to M$ be a $G$-principal bundle, where $G$ is a compact Lie group. An \emph{extension} of $(\varphi_t)_{t \in \R}$ is a flow $(\psi_t)_{t \in \R}$ defined over $P$ such that:
\[
\varphi_t \circ \pi = \pi \circ \psi_t, \qquad R_g \circ \psi_t = \psi_t \circ R_g, \qquad \forall t \in \R, g \in G.
\]
The flow $(\psi_t)_{t \in \R}$ is \emph{partially hyperbolic}, which makes its analysis harder than that of $(\varphi_t)_{t \in \R}$. The typical example of such a flow is provided by the \emph{frame flow} over a negatively curved Riemannian manifold $(N^n,g)$ of dimension $n \geq 3$. In this case, $P = FM$ is the frame bundle, $G = \mathrm{SO}(n-1)$, $M = SN \subset TN$ is the unit tangent bundle, $(\varphi_t)_{t \in \R}$ is the geodesic flow, and $(\psi_t)_{t \in \R}$ is the frame flow. We refer to \S\ref{ssection:frame-flows} for further details (see also \cite{Cekic-Lefeuvre-Moroianu-Semmelmann-22} for an introduction to frame flows).

Since $(\varphi_t)_{t \in \R}$ preserves a smooth probability measure $\mu_M$ by assumption, $(\psi_t)_{t \in \R}$ also preserves smooth probability measures $\mu$ obtained locally by taking the product of $\mu$ with any bi-invariant probability measure on $G$. The first problem is to characterize and understand when $(\psi_t)_{t \in \R}$ is ergodic with respect to $\mu$. This question is now fairly well understood, especially when the underlying flow $(\varphi_t)_{t \in \R}$ is the geodesic flow, see \cite{Brin-75-1,Brin-75-2,Lefeuvre-23,Cekic-Lefeuvre-Moroianu-Semmelmann-21,Cekic-Lefeuvre-22,Cekic-Lefeuvre-Moroianu-Semmelmann-23}. Our goal is to study quantitative rate of mixing for $(\psi_t)_{t \in \R}$. More precisely, given $f_1, f_2 \in C^\infty(P)$, define the \emph{correlation function} as
\begin{equation}
\label{equation:correlation-intro}
C_t(f_1,f_2) := \int_{P} f_1\circ\Phi_t \cdot f_2 ~\dd\mu - \int_{P} f_1~\dd\mu \int_P f_2~\dd\mu.
\end{equation}
The flow $(\Phi_t)_{t \in \R}$ is \emph{mixing} if $C_t(f_1,f_2) \to 0$ as $t \to \infty$; it is \emph{rapid mixing} if $C_t(f_1, f_2) = \mc{O}(t^{-N})$ as $t \to \infty$ for every $N > 0$ and \emph{exponentially mixing} if $C_t(f_1,f_2) = \mc{O}(e^{-\nu t})$ as $t \to \infty$ for some $\nu > 0$.

We obtain a general result characterizing when $(\psi_t)_{t \in \R}$ is rapid mixing in terms of the curvature of the \emph{dynamical connection} on the principal bundle $P$ (see Theorem \ref{theorem:main2}). To keep the discussion simple, here we only state as an illustration the main two consequences of this result when $G=\mathrm{U}(1)$, and when $G$ is semisimple.

When $G = \mathrm{U}(1)$, $P \to M$ is also called a \emph{circle} extension. Recall that a circle extension is \emph{torsion} if it admits a certain tensor power which is trivial over $M$. An Anosov flow is said to be \emph{non jointly integrable} if the sum of the stable and unstable bundles $E_s \oplus E_u$ is not integrable over $M$. We will prove that the following holds:

\begin{theorem*}
Suppose that $P \to M$ is a circle extension, and $(\varphi_t)_{t \in \R}$ is not jointly integrable on $M$. If $P$ is not torsion, then $(\psi_t)_{t \in \R}$ is rapid mixing.
\end{theorem*}

In the semisimple case, we shall prove the following:

\begin{theorem*}
Suppose that $P \to M$ is a $G$-principal bundle, $G$ is semisimple and $(\varphi_t)_{t \in \R}$ is not jointly integrable. Then $(\psi_t)_{t \in \R}$ is rapid mixing if and only if it is ergodic.
\end{theorem*}

In particular, the previous two theorems show that the frame flow over a negatively curved manifold is rapid mixing whenever it is ergodic, see Corollary \ref{corollary}. We also point out that Dolgopyat \cite{Dolgopyat-02} had established a similar result as the second theorem in the related case of Anosov diffeomorphisms. While writing this monograph, we learnt that Pollicott-Zhang \cite{Pollicott-Zhang-24} had also established at the same time the second of the above theorems (for semisimple Lie groups $G$), using different techniques.

The proof of the previous two results will build on the calculus developed in Chapters \ref{chapter:analysis} and \ref{chapter:analytic}. By $G$-equivariance, the generator $X_P \in C^\infty(P,TP)$ of the flow $(\psi_t)_{t \in \R}$ induces an operator
\[
\mathbf{X}_{\mathbf{k}} : C^\infty_{\mathrm{hol}}(F,\mathbf{L}^{\otimes \mathbf{k}}) \to C^\infty_{\mathrm{hol}}(F,\mathbf{L}^{\otimes \mathbf{k}})
\]
such that the rescaled family of operators $\mathbf{Q} := (\mathbf{Q}_{\mathbf{k}})_{\mathbf{k} \in \widehat{G}} \in \Psi^1_{|\mathbf{k}|^{-1},\mathrm{BW}}(P)$ defined by $\mathbf{Q}_{\mathbf{k}} := |\mathbf{k}|^{-1}\mathbf{X}_{\mathbf{k}}$ (for $\mathbf{k} \neq 0$) belongs to the Borel-Weil calculus. Rapid mixing will then follow from a \emph{uniform} (in $\mathbf{k} \in \widehat{G}$) estimate on the resolvent $(\mathbf{Q}_{\mathbf{k}}-z)^{-1}$ for $\Re(z)=0$ and $\Im(z) \to \pm \infty$, in adequate Sobolev spaces (so-called \emph{anisotropic Sobolev spaces}).

\subsection{Spectral theory of horizontal Laplacians}

\label{ssection:hypoelliptic-laplacians}

In Chapter \ref{chapter:hypoelliptic}, the second application of the Borel-Weil calculus is concerned with the study of a class of subelliptic Laplacians obtained as horizontal Laplacians associated to a connection on a principal bundle. More precisely, let $P \to M$ be a $G$-principal bundle over the closed Riemannian manifold $(M,g)$, and further assume that $P$ is equipped with a $G$-equivariant connection $\nabla$. This provides a decomposition of the tangent space
\[
TP = \HH \oplus \V,
\]
where $\V$ is the tangent space to the fibres of $P$. Let $\HH^*$ be the annihilator of $\V$. The restriction of the usual exterior derivative to $\HH$ defines a horizontal exterior derivative
\[
	d_{\HH} : C^\infty(P) \to C^\infty(P,\HH^*), \qquad d_{\HH}f := df|_{\HH}.
\]
The horizontal Laplacian is then defined as
\[
\Delta_{\HH} := (d_{\HH})^* d_{\HH} : C^\infty(P) \to C^\infty(P),
\]
where $(d_{\HH})^*$ is the formal $L^2$-adjoint of $d_{\HH}$ (defined with respect to the natural $L^2$ structures). As explained in \S\ref{section:sc-analysis}, this operator induces a family of operators in the Borel-Weil calculus
 \begin{equation}
 \label{equation:deltak-intro}
 \Delta_{\mathbf{k}} : C^\infty_{\mathrm{hol}}(F,\mathbf{L}^{\otimes \mathbf{k}}) \to C^\infty_{\mathrm{hol}}(F,\mathbf{L}^{\otimes \mathbf{k}}).
 \end{equation}
 
 \begin{example*}
If $G = \mathrm{U}(1)$, then $F = M$ and one retrieves in \eqref{equation:deltak-intro} the magnetic Laplacian. More precisely, taking a local patch of coordinates $U \subset M$, the operator \eqref{equation:deltak-intro} can be written locally as $\Delta_k = (d+ik \alpha)^*(d+ik\alpha)$, where $\alpha \in C^\infty(U,T^*U)$ is a real-valued $1$-form (locally defined on $U$). 

If $G = \mathrm{SO}(3)$, then $F$ can be trivialized locally over $U$ as $F|_{U} \simeq U \times \C\mathbb{P}^1$. As explained in \S\ref{section:sc-analysis}, the irreducible representations of $\mathrm{SO}(3)$ are all given by $H^0(\C\mathbb{P}^1,L^{\otimes k})$ for $k \geq 0$, where $L \to \C\mathbb{P}^1$ is the tautological line bundle. Taking $k=1$ in \eqref{equation:deltak-intro}, and identifying locally $H^0(\C\mathbb{P}^1,L) \simeq \C^3$, the operator $\Delta_1$ can be written locally as
\[
\Delta_1 = (d+A)^*(d+A),
\]
where $A$ is a locally defined $1$-form $A \in C^\infty(U,T^*U \otimes \mathfrak{so}(3))$ with values in the Lie algebra $\mathfrak{so}(3)$ (that is $3 \times 3$ antisymmetric matrices), and $\Delta_1$ is restricted to $C^\infty_{\mathrm{comp}}(U,\C^3)$. Here, the action of $A$ should be understood as $A : C^\infty_{\mathrm{comp}}(U,\C^3) \to C^\infty(U,T^*U \otimes \C^3)$. More precisely, given a vector field $Z \in C^\infty_{\mathrm{comp}}(U,TU)$, one has that $A(Z) \in C^\infty(U,\mathfrak{so}(3))$ is a field of antisymmetric $3 \times 3$ matrices; given $u \in C^\infty_{\mathrm{comp}}(U,\C^3)$, $A(Z)u \in C^\infty(U,\C^3)$ is the field of matrices applied to $u$. For higher values of $k \geq 2$ in \eqref{equation:deltak-intro}, the action of $A$ is slightly less easy to describe geometrically and is given by symmetrization, i.e. the representation action of $\mathrm{SO}(3)$ (or $\mathrm{SU}(2)$) on the space $\mathrm{Sym}^k_0 (\mathbb{R}^3)$ (homogeneous harmonic polynomials of degree $k$ on $\mathbb{R}^3$).
\end{example*}

\subsubsection{Hypoellipticity}

Recall that the operator $\Delta_{\HH}$ is (globally) \emph{hypoelliptic} if the equation $\Delta_{\HH} u = f$ with $u \in \mc{D}'(P)$ a distribution and $f \in C^\infty(P)$ only admits smooth solutions $u \in C^\infty(P)$. As the operator $\Delta_{\HH}$ can be written (locally) as a ``sum of squares'' of vector fields, i.e. $\Delta_{\HH} = \sum_{i = 1}^n X_i^*X_i + \mathrm{l.o.t.}$ (modulo lower order terms), where $X_i$ are horizontal vector fields, one can apply Hörmander's bracket condition \cite{Hormander-67} to study (local) hypoellipticity: if for all $x \in M$, the iterated brackets 
\begin{equation}
\label{equation:brabra}
\left\{[X_{i_1},[X_{i_2}, \dotsc]] \mid i_1, \dotsc,i_k \in \{1, \dotsc, d\}, k \in \Z_{\geq 0}\right\}
\end{equation}
span the tangent space $T_xM$, then $\Delta_{\HH}$ is (locally) hypoelliptic. In particular, this translates geometrically as follows: if the curvature $F \in C^\infty(M,\Lambda^2 T^*M \otimes \mathrm{Ad}(P))$ is everywhere nondegenerate (where $\Ad(P) \to M$ is associated to $P$ via the adjoint representation, see \S \ref{sssection:associated-vector-bundles}), that is
\[
\mathrm{Span}\{F_x(X,Y) ~|~ X,Y \in T_xM\} = \mathrm{Ad}_x(P)
\]
for all $x \in M$, then the operator is locally hypoelliptic.

However, in certain problems, the bracket condition is not always easily verified or can even fail. Our aim is to give a natural condition on the connection $\nabla$ that ensures global hypoellipticity of $\Delta_{\HH}$. Recall that the holonomy group (based at $x_0 \in M$) of the connection $\nabla$ is defined as
\[
 \mathrm{Hol}(P,\nabla) = \{\tau_{\gamma} ~|~ \gamma \text{ loop based at } x_0\} \leqslant G,
 \]
where $G$ is identified with $P_{x_0}$, and $\tau_\gamma : P_{x_0} \to P_{x_0}$ is the parallel transport with respect to the connection along $\gamma$. We will establish a general result (for arbitrary $G$) later in Theorem \ref{theorem:hypoellipticity}, and for simplicity we will only state its consequence (this is very much analogous to the statement we prove for rapid mixing of isometric extensions, Theorem \ref{theorem:main2}). We will prove the following:

 \begin{theorem*}
Assume that $G$ is semisimple and that the holonomy group of the connection is dense in $G$. Then $\Delta_{\HH}$ is hypoelliptic.
 \end{theorem*}
 
We shall see that this statement fails if $G$ is not semisimple -- see \S\ref{ssection:not-hypo} for a counterexample with $G=\mathrm{U}(1)$. A remarkable case where the condition of the theorem holds is that of a flat $G$-bundle over $M$ with dense monodromy. Recall that a flat $G$-bundle is equivalent to the data of a representation $\rho : \pi_1(M) \to G$ (see Example \ref{example:flat-connection} for further details); its holonomy (or monodromy) group is then given by $\rho(\pi_1(M))$. The theorem therefore applies if $G$ is semisimple and $\rho(\pi_1(M)) \leqslant G$ is dense and shows that $\Delta_{\HH}$ is hypoelliptic despite all the Lie brackets \eqref{equation:brabra} being horizontal everywhere. We will also establish a lower bound on the first eigenvalue of the induced Laplace operator in the Borel-Weil calculus
 \[
 \Delta_{\mathbf{k}} : C^\infty_{\mathrm{hol}}(F,\mathbf{L}^{\otimes \mathbf{k}}) \to C^\infty_{\mathrm{hol}}(F,\mathbf{L}^{\otimes \mathbf{k}}),
\]
 under a non-degenerate condition on the curvature of the connection.
 
 \subsubsection{Quantum ergodicity} \label{sssection:qe}
 
Let $(M,g)$ be a closed Riemannian manifold and denote by $0=\lambda_0 < \lambda_1 \leq \dotsb$ the eigenvalues of $\Delta_g$ (Laplacian on functions) counted with multiplicity, and $u_j \in C^\infty(M)$ the associated $L^2$-normalized eigenfunctions. Quantum ergodicity consists in understanding high-frequency limits of
\[
\Delta_g u_j = \lambda_j^2 u_j, \quad \lambda_j \to \infty.
\]
More precisely, it can be shown that, up to extraction in $j$, there exists a measure $\mu$ supported on the cosphere bundle $S^*M$ such that for all $a \in C^\infty_{\mathrm{comp}}(T^*M)$,
\begin{equation}
\label{equation:op-lim}
\lim_{j \to \infty} \langle\Op_{h_j}(a)u_j,u_j\rangle_{L^2(M)} = \int_{S^*M} a(x,\xi) ~\dd\mu(x,\xi),
\end{equation}
where $h_j := \lambda_j^{-1}$ and $\Op_{h}(\bullet)$ denotes a semiclassical quantisation propcedure on $M$. The measure $\mu$ is called a \emph{semiclassical defect measure}; it is invariant by the geodesic flow.

A notably difficult question is to understand the possible quantum limits $\mu$ of eigenfunctions. On the round sphere, it is known that \emph{any} measure invariant by the geodesic flow can be obtained as a quantum limit. On the other hand, in negative curvature, it is conjectured that $\mu = \mu_{\mathrm{Liouville}}$, that is the Liouville measure should be the only quantum limit. This is known as the Quantum Unique Ergodicity (QUE) conjecture \cite{Rudnick-Sarnak-94}. It is only known on arithmetic hyperbolic surfaces \cite{Lindenstrauss-06}; further limitations on $\mu$ were also established, see \cite{Anantharaman-08, Anantharaman-Nonnenmacher-07,Riviere-10, Dyatlov-Jin-18,Dyatlov-Jin-Nonnenmacher-22,Bourgain-Dyatlov-18,Dyatlov-22}. However, when the geodesic flow is ergodic on $S^*M$ (with respect to the Liouville measure), QUE is known along a density $1$ subsequence of eigenfunctions, see \cite{Shnirelman-74-1,Shnirelman-74-2,Colindeverdiere-85,Zelditch-87}.

In this monograph, we will prove a quantum ergodicity statement on a density $1$ subset of eigenstates for the operators $\Delta_{\mathbf{k}}$ defined in \eqref{equation:deltak-intro} under certain assumptions on the curvature of the connection $\nabla$ on $P$. Define
 \[
\Omega = \{(\mathbf{k},\lambda) \in \widehat{G} \times [0,\infty) ~|~\exists u_{\mathbf{k},\lambda} \in C^\infty_{\mathrm{hol}}(F,\mathbf{L}^{\otimes \mathbf{k}}),\, u_{\mathbf{k},\lambda} \neq 0,\, \Delta_{\mathbf{k}} u_{\mathbf{k},\lambda} = \lambda^2 u_{\mathbf{k},\lambda}\},
 \]
where we recall that $\Delta_{\mathbf{k}}$ is the operator on $C^\infty(F, \Lk)$ naturally associated to the horizontal Laplacian $\Delta_{\HH}$ of a connection on the $G$-principal bundle $P \to M$. The eigenstate corresponding to $(\mathbf{k},\lambda) \in \Omega$ is denoted by $u_{\mathbf{k},\lambda}$.

We note that a connection on $P$ and a choice of a bi-invariant Riemannian metric on $G$ induce a natural Riemannian metric on $P$ making the projection $\pi: P \to M$ a Riemannian submersion. 
We shall prove the following:
 
 \begin{theorem*}
 Let $(M,g)$ be a Riemannian manifold with Anosov geodesic flow, and let $P \to M$ be a flat $G$-principal bundle with dense holonomy group. Then there exists a density $1$ subset $\Lambda \subset \Omega$ such that for all sequences $(\mathbf{k}_j,\lambda_j)_{j \geq 0} \in \Lambda^{\Z_{\geq 0}}$, one has: for all $a \in C^\infty(F)$,
 \[
\langle a u_{\mathbf{k}_j,\lambda_j},u_{\mathbf{k}_j,\lambda_j}\rangle_{L^2(F,\mathbf{L}^{\otimes \mathbf{k}_j})} \to_{j \to \infty} \dfrac{1}{\vol(F)} \int_F a(w)\, \dd w,
 \]
 where $\dd w$ stands for the Riemannian measure on $F$.
 \end{theorem*}

 Here, $\Lambda \subset \Omega$ is of density $1$ if
 \[
 \dfrac{\sharp (\Lambda \cap B(0,R))}{\sharp(\Omega \cap B(0,R))} \to_{R \to \infty} 1,
 \]
 where $B(0, R)$ is a ball of radius $R$ of a natural metric coming from an identification $\Omega \subset \mathbb{R}^{d+1}$ (see Theorem \ref{theorem:quantum-ergodicity} below). We note that the just stated theorem actually holds under a more general assumption of ergodicity of a certain Hamiltonian flow (replacing the Anosov geodesic flow and dense holonomy group assumptions), but we decided to state the theorem this way for simplicity. We actually establish a more general result holding for \emph{nearly flat} connections but with some further limitations on the possible sequences $(\mathbf{k}_j,\lambda_j)$, see Theorem \ref{theorem:quantum-ergodicity2}. While writing this monograph, we discovered that a similar result had been recently established independently in \cite{Ma-Ma-23}, and also in the non-flat setting in \cite{Ben-Ovadia-Ma-Rogdriguez-Hertz-24}. Further perspectives are discussed below in \S\ref{ssection:weyl}. \\

 We conclude this section with an important remark:

\begin{remark*}[Diophantine property of semisimple Lie groups]
A common feature of the theorems stated in the introduction involving compact semisimple Lie groups is that they ultimately rely on a certain Diophantine property verified by semisimple groups and violated by Abelian groups. More precisely, in a semisimple compact Lie group, a dense set is automatically \emph{quantitatively dense}. For instance, if $a,b \in \mathrm{SU}(2)$ generate a dense free group in $\mathrm{SU}(2)$, it can be shown that the set of $ab$-words of word-length $\leq n$ is $\mc{O}(1/n^{1-\eps})$ dense in $\mathrm{SU}(2)$ as $n \to \infty$, for any $\eps > 0$, see \S\ref{ssection:diophantine}. Such a property does not hold on the torus.
\end{remark*}

\section{Potential applications}

Finally, we conclude with a discussion of some future problems to which the Borel-Weil calculus could potentially be applied.

\subsection{Tensor tomography. Geometric inverse problems}

Let $(M^n, g)$ be a closed Riemannian manifold of dimension $n \geq 2$ and further assume that its geodesic flow on its unit tangent bundle $SM$ is Anosov. In this case, it is well-known that the set of free homotopy classes $\mc{C}$ on $M$ is in $1$-to-$1$ correspondence with closed geodesics on $M$. Given $c \in \mc{C}$, we let $\ell_g(c)$ denotes the length of this unique closed geodesic; below, $\gamma : [0,\ell_g(c)] \to M$ is an arc-length parametrization of this closed geodesic. We will denote by $\mathrm{Sym}^kT^*M \to M$ the vector bundle of symmetric tensors of order $k$, and by $\mathrm{Sym}_0^kT^*M$ its sub-bundle of tensors with zero trace.

A classical question in geometric inverse problem and integral geometry is to show that the \emph{X-ray transform} $I_k : C^\infty(M, \mathrm{Sym}^k T^*M) \to \ell^\infty(\mc{C})$, defined by
\begin{equation}
\label{equation:ik}
 I_kf(c) := \dfrac{1}{\ell_g(c)} \int_0^{\ell_g(c)} f_{\gamma(t)}(\dot{\gamma}(t),\dotsc ,\dot{\gamma}(t))\, \dd t,
\end{equation}
is \emph{solenoidal injective} on symmetric $k$-tensors, that is injective in restriction to divergence-free symmetric tensors. This is known for Anosov surfaces \cite{Guillemin-Kazhdan-80,Paternain-Salo-Uhlmann-14-1,Guillarmou-17-1}, for Anosov manifolds in higher dimensions for $k=0,1$ \cite{Dairbekov-Sharafutdinov-03}, for negatively curved metrics \cite{Croke-Sharafutdinov-98} in any dimension, and generically \cite{Cekic-Lefeuvre-21-2}. Solenoidal injectivity has various consequences. For instance, for $k=2$, this implies the local rigidity of the marked length spectrum in a neighborhood of $g$ \cite{Guillarmou-Lefeuvre-18}; using the Anosov embedding theorem \cite{Chen-Erchenko-Gogolev-23} and \cite{Erchenko-Lefeuvre-24}, this also proves tensor tomography for simple manifolds.

Denote by $\Omega_k$ the space of smooth functions on the unit tangent bundle $\pi : SM \to M$ that are fiberwise spherical harmonics of degree $k$. It is a standard fact
\[
\pi_k^* : C^\infty(M, \mathrm{Sym}^k_0 T^*M) \to \Omega_k, \qquad \pi_k^*f(v) := f_{\pi(v)}(v, \dotsc,v)
\]
is an isomorphism. By the classical Abelian Liv\v sic theorem, injectivity of \eqref{equation:ik} is equivalent to the following transport problem, known as \emph{tensor tomography}. Let $u \in C^\infty(SM)$ be a function such that $X u = f \in \Omega_k \oplus \Omega_{k-2} \oplus \dotsb$ (with the convention that $\Omega_j = \{0\}$ for $j < 0$). Then, show that $u \in \Omega_{k-1} \oplus \dotsb$.

The \emph{geodesic vector field} $X : C^\infty(SM) \to C^\infty(SM)$ acts as $X : \Omega_k \to \Omega_{k-1} \oplus \Omega_{k+1}$ and therefore splits as a sum $X = X^- + X^+$, where $X^\pm : \Omega_k \to \Omega_{k \pm 1}$. The solutions to $X^+ f_k= 0$ with $f_k \in \Omega_k$ and $k \neq 0$ are called \emph{Conformal Killing Tensors} (CKTs).

Solving the transport problem $X u = f$ is usually done in two steps:
\begin{enumerate}[label=(\roman*)]
\item Show that $u$ has \emph{finite Fourier degree}, that is $u \in \Omega_N \oplus \Omega_{N-2} \oplus \dotsb$ for some finite $N \geq 0$.
\item Show that $(M,g)$ has no Conformal Killing Tensors.
\end{enumerate}
The combination of (i-ii) easily implies solenoidal injectivity of \eqref{equation:ik}. It turns out that the operators $X^\pm$ actually \emph{belong} to the Borel-Weil calculus. We believe that this should help understand why (i) or (ii) should hold in the Anosov case. In particular, studying high-frequency limits (as $k \to \infty$) of CKTs $f_k \in \Omega_k$ (i.e. $X^+f_k = 0$) using the tools of semiclassical analysis might help understand (ii).

Indeed, let us quickly explain why these operators fit into the calculus. Denote by $FM \to M$ the frame bundle. Let $F := FM/T$ be the flag bundle where $T \leqslant \mathrm{SO}(n)$ is a maximal torus. Note that for $n=3$, $F = SM$ is the unit tangent bundle over $M$. It can be shown that there exists a fiberwise holomorphic line bundle $L \to F$ such that there exists a natural isomorphism
\begin{equation}
\label{equation:theta-k}
i_k : \Omega_k \to C^\infty_{\mathrm{hol}}(F,L^{\otimes k}),
\end{equation}
see \cite[Exercise 7.2, page 154]{Sepanski-07} (for $n=3$, $L$ is the tautological line bundle over $\mathbb{C}P^1$). We set for $k \neq 0$
\[
\mathbf{A}^\pm_k := k^{-1} i_{k\pm 1} \circ X^\pm \circ i_k^{-1}.
\]
It can then be shown that
\[
\mathbf{A}^\pm := (\mathbf{A}^\pm_k)_{k \geq 1} \in \Psi^1_{k^{-1},\mathrm{BW}}(FM)
\]
belongs to the Borel-Weil calculus.

\begin{question*}
Can one prove the absence of CKTs on Anosov Riemannian manifolds $(M,g)$ using high-frequency estimates on $\mathbf{A}^\pm$? Can one prove finite Fourier degree for solutions of the transport equation $X u = f \in \Omega_k \oplus \Omega_{k-2} \oplus \dotsb$ using the Borel-Weil calculus?
\end{question*}

We expect this new point of view on tensor tomography to shed new light on this problem.

\subsection{Wave decay for sub-elliptic Laplacians}

Let $(M,g)$ be a non-compact complete Riemannian manifold and denote by $\Delta_g$ the (non-negative) metric Laplacian. Let $U \Subset M$ be a compact subset. A standard question in scattering theory consists in proving that, under some mild geometric assumptions on $(M,g)$ (e.g. hyperbolic trapped set), the solutions to the wave equation
\begin{equation}
\label{equation:wave-equation}
(\partial_t^2 + \Delta_g) u = 0, \qquad u|_{t=0}\in H^1_{\comp}(U), \partial_tu|_{t=0} \in L^2(U)
\end{equation}
admit exponential decay in any (other) fixed compact subset $U' \Subset M$, namely
\begin{equation}
\label{equation:decay-wave}
\|u(t)\|^2_{L^2(U')} \leq C (\|u|_{t=0}\|^2_{H^1(U)} + \|\partial_t u|_{t=0}\|^2_{L^2(U)}) e^{-\nu t},
\end{equation}
for some $C, \nu > 0$.

Standard geometric assumptions on $(M,g)$ usually involve supposing that the trapped set, that is the set of geodesics staying in a bounded region of $M$ for all times, admits a hyperbolic structure (Axiom A flow). Typical examples are provided by asymptotically hyperbolic manifolds of negative sectional curvature. We refer to \cite{Dyatlov-Zworski-19} for further details.

The proof of \eqref{equation:decay-wave} usually amounts to proving that $(\Delta_g-P(\lambda))^{-1}$ (where $P : \C \to \C$ is a polynomial of degree $2$) admits a meromorphic extension from $\Im(\lambda) \gg 0$ to $\C$ on adequate weighted $L^2$-spaces, and establishing the existence of a \emph{resonance free strip} $\{-\delta \leq \Im(\lambda)\leq 0\}$, that is a band below the real axis without resonances.

A related question is to understand the analogue of \eqref{equation:wave-equation} when $\Delta_g$ is replaced by a subelliptic operator such as a horizontal Laplacian on a principal bundle, as in \S\ref{ssection:hypoelliptic-laplacians}. Physically, subellipticity corresponds to wave propagation in an anisotropic medium structured in layers, where propagation only occurs parallel to the layers.

More precisely, we let $FM \to M$ be the frame bundle, $\Delta_{\HH}$ be the horizontal Laplacian on $FM$ provided by the Levi-Civita connection. Consider the subelliptic wave equation
\begin{equation}
\label{equation:wave-equation2}
(\partial_t^2 + \Delta_{\HH}) u = 0, \qquad u|_{t=0}\in H^1_{\mathrm{comp}}(U), \partial_tu|_{t=0} \in L^2(U),
\end{equation}
where $U \Subset FM$ is a compact subset.

\begin{question*}
Assume that $(M,g)$ is asymptotically hyperbolic with negative sectional curvature. Can one prove exponential decay for solutions to \eqref{equation:wave-equation2}, namely
\[
\|u(t)\|^2_{L^2(U')} \leq C (\|u|_{t=0}\|^2_{H^1(U)} + \|\partial_t u|_{t=0}\|^2_{L^2(U)}) e^{-\nu t},
\]
for all compact sets $U' \Subset FM$?
\end{question*}

This would amount to proving meromorphic extension of $(\Delta_{\mathbf{k}}-P(\lambda))^{-1}$ (where $\Delta_{\mathbf{k}}$ is the horizontal Laplacian induced on $\Lk$) on weighted $L^2_{\mathrm{hol}}(F,\mathbf{L}^{\otimes \mathbf{k}})$ spaces from $\Im(\lambda) \gg 0$ to $\C$, and then that there exists a resonance free strip $\{-\delta \leq \Im(\lambda) \leq 0\}$ \emph{uniform} in $\mathbf{k} \in \widehat{G}$. The Borel-Weil calculus seems to be the right tool to study this high-frequency problem involving simultaneous semiclassical parameters $h := 1/\Re(\lambda)$ and $1/|\mathbf{k}|$ with $\mathbf{k} \in \widehat{G}$.

\subsection{Quantum Ergodicity}

\label{ssection:weyl}

The quantum ergodicity problem explained in \S\ref{sssection:qe} consists in understanding the high-frequency behaviour in phase space of the solutions to the equation
\begin{equation}
\label{equation:qe-limits}
\Delta_{\mathbf{k}} u_{\mathbf{k},\lambda}  = \lambda^2 u_{\mathbf{k},\lambda},
\end{equation}
where $\mathbf{k} \in \widehat{G}$, $\lambda \geq 0$, and $u_{\mathbf{k},\lambda} \in C^\infty_{\mathrm{hol}}(F,\mathbf{L}^{\otimes \mathbf{k}})$, with $F=P/T$, the flag bundle. The result stated in \S\ref{sssection:qe} is restrictive in two ways. First, it only deals with flat (or nearly flat) bundles, and the extension to the non-flat setting would be interesting to consider.

\begin{question*}
Can one extend the results of \S\ref{sssection:qe} beyond the flat or the nearly flat case?
\end{question*}

The main issue in this case lies in the underlying dynamics associated to high-frequencies. If $G = \mathrm{U}(1)$ (magnetic case), the Hamiltonian flow associated to high-frequency limits of \eqref{equation:qe-limits} is the magnetic flow on the cosphere bundle $S^*M$. In the absence of a magnetic field (flat bundle), this boils down to the geodesic flow; on the other hand, if the magnetic field is high enough, it is well-known that the magnetic flow may fail to be ergodic. For instance, when $M := \Sigma$ is a hyperbolic surface, and the magnetic field is proportional to the volume form of $\Sigma$ (by a constant), there is a threshold for the magnetic field above which the magnetic flow is conjugate to the rotation flow in the circle fibers of $S^*\Sigma$, see \cite{Charles-Lefeuvre-24} for further details. This is the reason why we have to limit ourselves to nearly flat connections in the general statement of quantum ergodicity (see Theorem \ref{theorem:quantum-ergodicity2}). Understanding the ergodic properties of the Hamiltonian flows appearing in this construction seems to be a challenging question. In addition, for a general Lie group $G$, another issue arises in the fact that the underlying Hamiltonian flow is less easy to describe geometrically than in the $\mathrm{U}(1)$ case.

Secondly, the results of \S\ref{sssection:qe} hold for functions $a \in C^\infty(F)$ whereas the natural phase space for this problem is $\HH^* \subset T^*F$, the horizontal bundle of the flag bundle $F \to M$. We point out that the parameters $\lambda$ and $\mathbf{k}$ play the role of two simultaneous (and somehow competing) semiclassical parameters. It should be possible to assign to \eqref{equation:qe-limits} a natural defect measure on $\HH^*$, similarly to \eqref{equation:op-lim}, but this is not yet clear. (For instance, what should be the value of the semiclassical parameter $h$, $h=1/\lambda$ or $h=1/|\mathbf{k}|$, or $h=(|\mathbf{k}|^2+\lambda^2)^{-1/2}$?)

\begin{question*}
Can one establish a quantum ergodicity statement similar to \S\ref{sssection:qe} for observables $a \in C^\infty_{\mathrm{comp}}(\HH^*)$?
\end{question*}

These two questions are left for future investigation. 

\subsection{Other perspectives}

The theory presented here deals with compact Lie groups; it is very likely that it could be generalized to non-compact Lie groups. However, the classification of irreducible representations then becomes more subtle. The Borel-Weil calculus is an intriguing combination of the standard semiclassical calculus with Toeplitz operators; we believe that new phenomena at the crossroads of these two quantizations could be discovered. \\

\noindent \textbf{Acknowledgement:} We thank Laurent Charles for several fruitful discussions. In particular, we are grateful for his explanation of the Borel-Weil correspondence which is used crucially in this work. We also thank Semyon Dyatlov, Yannick Guedes Bonthonneau, and Gabriel Paternain for their feedback on several aspects of this manuscript.

\newpage

\hspace{1cm}

\newpage

\chapter[Preliminaries]{Preliminaries on representation theory and principal bundles}

\label{chapter:analysis}

The aim of this chapter is two-fold: to present the Borel-Weil theory of a compact Lie group, as well as to interpret its analytical and differential geometric implications in the context of principal bundles.

 \minitoc
 
 \newpage

\section{Representation theory preliminaries}

\label{section:representation}

In the following, $G$ is a connected compact Lie group of rank $d \geq 1$ equipped with a bi-invariant metric normalized so that $\dd g$, the induced Haar measure, is a probability measure. Let $\widehat{G}$ be the set of irreducible complex unitary representations of $G$ on finite dimensional vector spaces (modulo isomorphisms).

\subsection{Associated bundles} 

\subsubsection{General construction} \label{sssection:associated-bundle} Let $M$ be a smooth closed manifold, and $\pi : P \to M$ be a $G$-principal bundle. We first recall the general construction of an associated bundle to the $G$-principal bundle $P$.

\begin{definition}[Associated bundle]
\label{definition:associated-bundle}
Let $H \leqslant G$ be a closed subgroup and $\lambda : H \to \mathrm{Aut}(F)$ be a representation of $H$, where $F$ is a closed manifold or a vector space. The bundle $P \times_{\lambda} F$ is defined as the set of equivalence classes in $P \times F$ modulo the relation given by the right $H$-action
\begin{equation}
\label{equation:relation-0}
(w,\xi) \sim (w, \xi) \cdot h := (w \cdot h, \lambda(h^{-1})\xi),
\end{equation}
 for all $w \in P$, $h \in H$ and $\xi \in F$.
\end{definition}

It can be verified that $P \times_{\lambda} F$ is a double fibration over $M$ in the sense that there is a tower of projections
\begin{equation}
\label{equation:tower}
P \times_{\lambda} F \to P/H \to M, \quad [w, \xi] \mapsto wH \mapsto \pi(w),
\end{equation}
where $[w, \xi] \in P \times F$ represents an equivalence class in $P \times_{\lambda} F$ and $wH$ denotes the orbit of $w$ in the quotient bundle $P/H$. Note that $P \times_{\lambda} F \to P/H$ has fibers isomorphic to $F$ and $P/H \to M$ has fibers isomorphic to $G/H$.

\subsubsection{Associated vector bundles}

\label{sssection:associated-vector-bundles}

We apply the associated bundle construction of \S\ref{sssection:associated-bundle} in the specific case where $H = G$, $F$ is a finite-dimensional vector space and $\lambda$ is a linear representation of $G$, we obtain an \emph{associated vector bundle} over $M$. Given $\lambda \in \widehat{G}$, a unitary irreducible representation $\lambda : G \to \mathrm{GL}(V^\lambda)$ on the complex finite dimensional vector space $V^\lambda$, we can construct the corresponding vector bundle
\begin{equation}
\label{equation:e-lambda}
E^\lambda := P \times_\lambda V^\lambda,
\end{equation}
defined over $M$ by taking $P \times V^\lambda/\sim$, where
\begin{equation}
\label{equation:right-product}
(w,\xi) \sim (w,\xi)\cdot g := (w \cdot g, \lambda(g^{-1})\xi),
\end{equation}
for any $g\in G$, $w \in P$, $\xi \in V^\lambda$.

The vector bundle $E^\lambda \to M$ is a bundle all of whose fibers are naturally isomorphic to $V^\lambda$. More precisely, the quotient map $P \times V^\lambda \to E^\lambda$ is a $G$-principal bundle with right-action given by \eqref{equation:right-product}. Freezing the basepoint $x \in M$, we get a fiber bundle $\pi_x : P_x \times V^\lambda \to E^\lambda_x$ such that $\pi_x(w,\xi) := [w,\xi]$ (the equivalence class modulo \eqref{equation:right-product}). Choosing an arbitrary point $w \in P_x$ then allows to identify $V^\lambda \to E^\lambda_x$ via the linear isomorphism $V^\lambda \ni \xi \mapsto \pi_x(w,\xi)$. In order to keep notation simple, we use the letter $w : V^\lambda \to E^\lambda_x$ for this isomorphism.

In particular, if $\lambda \in \widehat{G}$, $g \in G$ and $x \in M$, $w \in P_x$, we denote by $w \cdot \lambda(g) \in \mathrm{Hom}(V^\lambda,E^\lambda_x)$ the composition of the map $\lambda(g) \in \mathrm{GL}(V^\lambda)$ with the identification $w : V^\lambda \to E^\lambda_x$. Note that, for $w$ and $w \cdot g$ (for some $g \in G$), we can compare these two identifications using \eqref{equation:right-product} and we obtain for all $\xi \in V^\lambda$:
\begin{equation}
\label{equation:inv}
	w \cdot g(\lambda(h)\xi)= \pi_x(w \cdot g,\lambda(h)\xi) = \pi_x(w, \lambda(gh) \xi) = w(\lambda(gh) \xi), \quad h \in G,
\end{equation}
and so in particular we obtain
\begin{equation}
\label{equation:id-inv}
w\cdot g(\lambda(g^{-1})\xi)= \pi_x(w \cdot g,\lambda(g^{-1})\xi) = \pi_x(w,\xi) = w(\xi).
\end{equation}

Finally, we will write $\mathrm{Hom}(V^\lambda,E^\lambda) \to M$ for the associated bundle constructed via \eqref{equation:e-lambda} by taking $P \times_{\lambda \otimes \mathbf{1}} (V^\lambda \otimes {V^\lambda}^*)$, where $\lambda \otimes \mathbf{1} : G \to \mathrm{GL}(V^\lambda \otimes {V^\lambda}^*)$ is the representation obtained by the tensor product of $\lambda : G \to \mathrm{GL}(V^\lambda)$ and the trivial representation $\mathbf{1} : G \to \mathrm{GL}({V^\lambda}^*)$, $\mathbf{1}(g) = \mathbf{1}_{{V^\lambda}^*}$; here, ${V^\lambda}^*$ denotes the dual representation. Note that $\mathrm{Hom}(V^\lambda,E^\lambda)$ is naturally isomorphic to the tensor product of the trivial vector bundle ${V^\lambda}^* \times M$ and $E^\lambda$.

\subsection{Borel-Weil Theorem}

\label{ssection:borel-weil}

Let $\rho : G \to \mathrm{GL}(V^\lambda)$ be an irreducible representation. The Borel-Weil correspondence gives a concrete geometric realization of the spaces $V^\lambda$ as holomorphic sections of certain complex line bundles $J^\lambda \to G/T$, where $T$ is a maximal torus in $G$. We follow closely the exposition in \cite[Chapter 7]{Sepanski-07} (see also \cite[Chapter VI]{Brocker-Dieck-85} or \cite[Theorem 4.12.5]{Duistermaat-Kolk-00}).

\subsubsection{Roots. Irreducible representations} A compact Lie group $G$ is \emph{semisimple} if its Lie algebra $\mathfrak{g}$ is semisimple. In turn, $\mathfrak{g}$ is called \emph{semisimple} if it is a direct sum of simple Lie algebras; a Lie algebra is called \emph{simple} if it is non-Abelian and has no non-trivial ideals.

We let $T$ be a maximal torus in $G$, $\mathfrak{t} \subset \mathfrak{g}$ its Lie algebra. Recall by structure theory of Lie groups \cite[Theorem 8.1]{Brocker-Dieck-85} that there is a finite cover $p: \widetilde{G} \to G$ such that $\widetilde{G} \cong (\mathbb{S}^1)^a \times G_0$, where $G_0$ is simply connected (and hence semisimple) and $\widetilde{G}/\Gamma \cong G$ where $\Gamma \leqslant Z(\widetilde{G})$. The Lie algebra $\mathfrak{g}$ then splits as 
\begin{equation}
\label{equation:lie-algebra-g}
\mathfrak{g} = \mathfrak{z}(\mathfrak{g}) \oplus \mathfrak{g}',
\end{equation}
where $\mathfrak{g}' := [\mathfrak{g},\mathfrak{g}]$ is the Lie algebra of $G_0$, and $\mathfrak{z}(\mathfrak{g})$ is the (Abelian) Lie algebra of the center $Z(G)$ of $G$, which also corresponds to the Lie algebra of $(\mathbb{S}^1)^a$. The Lie algebra of $T$ then splits as
\begin{equation}
\label{equation:lie-algebra-t}
\mathfrak{t} =  \mathfrak{z}(\mathfrak{g}) \oplus \mathfrak{t}',
\end{equation}
where $\mathfrak{t}' := \mathfrak{t} \cap \mathfrak{g}'$.

A \emph{weight} (or \emph{infinitesimal weight}) is a purely imaginary valued linear form $\alpha: \mathfrak{t} \to i\R$ (we shall write $\alpha \in (i\mathfrak{t})^*$). A \emph{global weight} is a group homomorphism $\gamma : T \to \mathbb{S}^1$. If $\rho : G \to \mathrm{GL}(V)$ is a finite-dimensional complex unitary representation, one can look at the induced (Abelian) Lie algebra representation $d\rho : \mathfrak{t} \to \mathrm{End}(V)$. Since $\mathfrak{t}$ is Abelian, one can diagonalize simultaneously elements of $\mathfrak{t}$ acting on $V$, that is
\begin{equation}
\label{equation:v-decomp}
V = \oplus_{\alpha \in \Delta(V)} V_\alpha, \qquad V_\alpha:=\{v \in V ~|~ d\rho(H)v=\alpha(H)v, \forall H \in \mathfrak{t}\}.
\end{equation}
Note that $\alpha$ is purely imaginary as $\rho$ is unitary. The above decomposition is known as the weight space decomposition of $V$. When $V = \mathfrak{g}_{\C}$ (complexified adjoint representation), $\mathfrak{g}_{\C} = \mathfrak{t}_{\C} \oplus_{\alpha \in \Delta(\mathfrak{g}_{\C})} \mathfrak{g}_\alpha$ and the weights $\Delta(\mathfrak{g}_{\C})$ are called the \emph{roots} of $\mathfrak{g}_{\C}$. Note that $\mathfrak{z}(\mathfrak{g}) \subset \ker \alpha$ for all $\alpha \in \Delta(\mathfrak{g}_{\C})$, so the roots can be seen equivalently as elements of $(i\mathfrak{t}')^*$.

Define $A \subset \mathfrak{t}$, the set of \emph{analytically integral weights} as
\[
A := \{ \alpha\in (i\mathfrak{t})^* ~|~ \alpha(H) \in 2\pi i \Z, \forall H \in \mathfrak{t} \text{ s.t. } \exp(H)=1\}.
\]
A system of simple roots $\Pi$ is a subset of $\Delta(\mathfrak{g}_{\C})$ that is a basis of $(i\mathfrak{t}')^*$ and satisfies the property that any $\beta \in \Delta(\mathfrak{g}_{\C})$ can be written as $\beta = \sum_{\alpha \in \Pi} k_\alpha \alpha$, with $k_\alpha$ being either all nonnegative or nonpositive integers. Note that such a system always exists, see \cite[Lemma 6.42]{Sepanski-07} for instance. A system of simple roots therefore decomposes the space
\[
	\Delta(\mathfrak{g}_{\C}) = \Delta_+(\mathfrak{g}_{\C}) \sqcup \Delta_-(\mathfrak{g}_{\C})
\]
into \emph{positive} and \emph{negative} roots, respectively. They give rise to $\ad(\mathfrak{t})$- and $\Ad(T)$-invariant (i.e. invariant under the action of the adjoint representation) splitting
\begin{equation}\label{eq:positive-negative-splitting}
	\mathfrak{n}^\pm := \oplus_{\alpha \in \Delta_\pm(\mathfrak{g}_{\mathbb{C}})} (\mathfrak{g}_{\mathbb{C}})_{\alpha}.
\end{equation}
Taking the real part $\mathfrak{m}$ of $\mathfrak{n}^+ \oplus \mathfrak{n}^-$ we obtain a sum of real, two dimensional irreducible representations of $\ad(\mathfrak{t})$.

Define $\mathfrak{a}_+^{\mathrm{ss}} \subset (i\mathfrak{t}')^*$ as the polyhedral convex positive cone spanned by $\Delta_+(\mathfrak{g}_{\C})$ (the superscript $\mathrm{ss}$ stands for semisimple) and set $\mathfrak{a}_+ := (i\mathfrak{z}(\mathfrak{g}))^* \oplus \mathfrak{a}_+^{\mathrm{ss}} \subset (i\mathfrak{t})^*$, the \emph{positive Weyl chamber}. The Weyl group $W:=N/T$ (where $N := \{g \in G ~|~ gTg^{-1} = T\}$) is finite and acts faithfully on $(i\mathfrak{t})^*$ by the adjoint representation; the set of all possible positive Weyl chambers is given by $w\cdot\mathfrak{a}_+$ for $w\in W$ (the non-empty intersections of the different Weyl chambers are called the \emph{walls}). Finally, for $\gamma,\lambda \in (i\mathfrak{t})^*$, introduce the relation $\lambda \leq \gamma$ if $\lambda \in \mathrm{Conv}(W\gamma)$, where $\mathrm{Conv}(B)$ denotes the convex closure (barycenters of points in $B$) of a subset $B \subset (i\mathfrak{t})^*$.

Equivalence classes of irreducible representations $\widehat{G}$ are classified by their corresponding \emph{highest weight root}. More precisely, if $\rho : G \to \operatorname{GL}(V)$ is an irreducible complex unitary representation, then in the weight space decomposition \eqref{equation:v-decomp}, there is a unique weight $\alpha_0 \in A \cap \mathfrak{a}_+$, called the highest weight, such that for all $\alpha \in \Delta(V)$, $\alpha \leq \alpha_0$ with equality if and only if $\alpha=w\alpha_0$ for some $w\in W$. Conversely, given $\alpha_0 \in A \cap \mathfrak{a}_+$, there is a unique (modulo isomorphism) irreducible representation $\rho : G \to \operatorname{GL}(V)$ with highest weight $\alpha_0$, see \cite[Theorem 7.34]{Sepanski-07}.

Let 
\[
a := \dim \mathfrak{z}(\mathfrak{g}), \qquad b := \dim \mathfrak{t}', \qquad d:=a+b=\mathrm{rk}(G).
\]
Let $\{\lambda_1, \dotsc, \lambda_a\}$ be a system of generators of $\mathfrak{z}(\mathfrak{g})$, that is such that any $\alpha \in (i\mathfrak{z}(\mathfrak{g}))^* \cap A$ can be written as $\alpha=\sum_{i=1}^a k_i \lambda_i$ with $k_i \in \Z$. Let $\{\lambda_{a+1}, \dotsc,\lambda_{d}\}$ be a system of generators of $\mathfrak{a}_+^{\mathrm{ss}} \cap A$, that is such that any $\alpha \in \mathfrak{a}_+^{\mathrm{ss}} \cap A$ can be written as $\alpha = \sum_{i=a+1}^d k_i \lambda_i$ with $k_i \in \Z_{\geq 0}$. We can then form the following surjective map 
\begin{equation}
\label{equation:ecriture}
\phi : \Z^a \times \Z_{\geq 0}^{b} \to \mathfrak{a}_+ \cap A, \qquad \mathbf{k} \mapsto \sum_{i=1}^d k_i \lambda_i.
\end{equation}
Note that $\mathfrak{a}_+$ is given by the set of all linear combinations $\sum_i \ell_i \lambda_i$ where $\ell_1, \dotsc, \ell_a \in \R, \ell_{a+1}, \dotsc, \ell_d \in [0,\infty)$. When $G = (\mathbb{S}^1)^a\times G_0$ with $G_0$ semi-simple and simply connected, \eqref{equation:ecriture} is an isomorphism (as monoids); however, this might not be the case if $G$ is not simply connected. As a consequence, $\widehat{G}$ is indexed by $\Z^a \times \Z_{\geq 0}^b/\sim$, where the equivalence relation is defined by $\mathbf{k}\sim\mathbf{k}'$ if and only if $\phi(\mathbf{k})=\phi(\mathbf{k}')$. In the following, we shall mostly drop the notation $\lambda \in \widehat{G}$ for irreducible representations and index them by $\mathbf{k}$.

\subsubsection{Compactification of the positive Weyl chamber} \label{sssection:compactification}

It will be convenient for later purposes to compactify the Weyl chamber by adding a boundary at infinity. More precisely, define $\partial_\infty \mathfrak{a}_+$, the boundary at infinity of the Weyl chamber, as the set of all possible limits $\mathbf{k}/|\mathbf{k}|$ as $|\mathbf{k}| \to \infty$ and $\mathbf{k} \in \widehat{G}$. Define $\partial_\infty \mathfrak{a}_+^{\mathrm{ss}} \subset \partial_\infty  \mathfrak{a}_+$ by further requiring $k_1 = \dotsb = k_a = 0$ (this is the boundary at infinity corresponding to the semisimple part of the group). Typically, if $G = \mathrm{U}(1)^a$, then $\partial_\infty \mathfrak{a}_+ = \mathbb{S}^{a-1}$, while if $G = \mathrm{SU}(2)$, $\partial_\infty  \mathfrak{a}_+ = \partial_\infty  \mathfrak{a}_+^{\mathrm{ss}} = \{\bullet\}$ is reduced to a point. More generally, we have the following diffeomorphisms:
\begin{equation}
\label{equation:boundary-infinity}
\partial_\infty  \mathfrak{a}_+^{\mathrm{ss}} \simeq \mathbb{S}^{b-1} \cap \R^{b}_+, \qquad \partial_\infty \mathfrak{a}_+ \simeq \mathbb{S}^{d-1} \cap (\R^a \times \R^b_+).
\end{equation}

We then set:
\[
\overline{\mathfrak{a}_+} := \mathfrak{a}_+ \sqcup \partial_\infty \mathfrak{a}_+, \qquad \overline{\mathfrak{a}_+^{\mathrm{ss}}} := \mathfrak{a}_+^{\mathrm{ss}} \sqcup \partial_\infty \mathfrak{a}_+^{\mathrm{ss}}.
\]
This space is equipped with the following topology: a neighborhood of $\mathbf{l}_0 \in \mathfrak{a}_+$ is given by a neighborhood for the standard topology on $\mathfrak{a}_+$, while a neighborhood of $\mathbf{l}_0 \in \partial_\infty \mathfrak{a}_+$ is given for $R, \eps > 0$ by
\[
U_{R,\eps} := \{\mathbf{l} \in \mathfrak{a}_+ ~|~ |\mathbf{l}| > R, |\mathbf{l}/|\mathbf{l}|-\mathbf{l}_0| < \eps\}.
\]
Finally, observe that $\overline{\mathfrak{a}_+}$ and $\overline{\mathfrak{a}_+^{\mathrm{ss}}}$ are compact topological spaces. (These are the radial compactifications of the Weyl chambers.)

\subsubsection{Flag manifold}\label{sssection:holomorphic-line-bundle}

Let $G/T$ be the flag manifold of $G$.  It is a standard fact that $G/T$ admits a complex structure such that the left-action of $G$ on $G/T$ is holomorphic (and transitive), see \cite[Theorem 7.50]{Sepanski-07}.

More precisely, the spaces $\mathfrak{n}^\pm \subset \mathfrak{g}_{\mathbb{C}}$ in the root decomposition \eqref{eq:positive-negative-splitting} are $\Ad(T)$-invariant Lie sub-algebras. Letting $\pi : G \to G/T$ be the projection, the real part $\mathfrak{m}$ of $\mathfrak{n}^+ \oplus \mathfrak{n}^-$ projects via $d\pi$ onto $T(G/T)$.
The associated vector bundles to $\mathfrak{n}^\pm$ may be identified with the holomorphic and anti-holomorphic tangent bundles of $G/T$
\[
	T^{1, 0} (G/T) = G \times_{\Ad(T)} \mathfrak{n}^+, \quad T^{0, 1} (G/T) \otimes \mathbb{C}= G \times_{\Ad(T)} \mathfrak{n}^-,
\]
where the complexified tangent space splits as
\[
	T(G/T) \otimes \mathbb{C}= T^{1, 0} (G/T) \oplus T^{0, 1} (G/T),
\]
and we write $\pi^{1, 0}$ and $\pi^{0, 1}$ for the respective projections. (Explicitly, the identification is given by $G \times_{\Ad(T)} \mathfrak{n}^\pm \ni [p, v] \mapsto d\pi(g) dL_g v$.)

\subsubsection{Holomorphic line bundles}

Given a global weight $\gamma : T \to \mathbb{S}^1$, one can form the complex line bundle 
\[
J^\gamma := G \times_\gamma \C
\]
over $G/T$, using the construction of \S\ref{sssection:associated-bundle} with $M = \{\bullet\}$. It is holomorphic and homogeneous (see \cite[Chapter 7, Section 7.4.3]{Sepanski-07}), that is it is endowed with a natural holomorphic left $G$-action given by $g\cdot [h,\xi] = [gh, \xi]$ for all $g,h \in G$, $\xi \in \C$, where $[h,\xi]$ denotes a class in $J^\gamma$.

The line bundle $J^\gamma$ is equipped with a natural inner product $g^{J^\gamma}$ (inherited from the Hermitian product on $\C$). Since the bi-invariant Riemannian metric on $G$ descends to a $G$-invariant Riemannian metric on $G/T$, there is a volume form on $G/T$ denoted at $hT \in G/T$ by $\mathrm{d}(hT)$, and we obtain a natural $L^2$-scalar product on $L^2(G/T,J^\gamma)$ given by
\begin{equation}
\label{equation:metric-j-gamma}
\langle u,v\rangle_{L^2(G/T,J^\gamma)} := \int_{G/T} g_{hT}^{J^\gamma}(u(hT),v(hT))\, \mathrm{d}(hT).
\end{equation}

We now proceed to make the holomorphic structure on $J^\gamma$ more explicit. Notice that, by construction, a section $s \in C^\infty(G/T,J^\gamma)$ is equivalent to the data of a function $\overline{s} \in C^\infty(G)$ with a $T$-equivariant property, namely
\begin{equation}
\label{equation:t-equivariance}
\overline{s}(gt) = \gamma(t^{-1})\overline{s}(g), \qquad \forall g \in G, t \in T.
\end{equation}

Given a vector $X \in T(G/T)$ at $gT \in G/T$, it can be lifted to a $T$-invariant section $\overline{X}$ on the fiber above $gT$ in $G$ with values in the real part $\mathfrak{m}$ of $\mathfrak{n}^+ \oplus \mathfrak{n}^-$. Then, define a connection on $J^\gamma$ as follows: 
\begin{equation}
\label{equation:fiberwise-chern}
(\nabla_{J^\gamma})_X s(gT) := d\overline{s}(\overline{X}(g)),
\end{equation}
where $d$ is the exterior derivative. It is straightforward to check that the right hand side descends to the quotient. 

The $\overline{\partial}$ operator on $J^\gamma$ is simply given by the $(0,1)$ part of this connection. More precisely, given $X \in T^{0, 1} (G/T)$ at $gT \in G/T$, it can be lifted to a $T$-invariant section $\overline{X}$ on the fiber above $gT$ in $G$ with values in $\mathfrak{n}^-$. Then
\begin{equation}
\label{equation:fiberwise-dbar}
	(\overline{\partial}_{J^\gamma})_X s(gT) := d\overline{s}(\overline{X}(g)).
\end{equation}
We will denote by $H^0(G/T,J^\gamma)$ the space of holomorphic sections of $J^\gamma$. It is immediate to check that $\nabla_{J^\gamma}$ is the (unique) Chern connection associated to $\overline{\partial}_{J^\gamma}$ and the Hermitian structure on $J^\gamma$.

\subsubsection{Borel-Weil Theorem}

The $G$-action on $J^\gamma$ defined above (linearly) extends the $G$-action on $G/T$ so the space $L^2(G/T,J^\gamma)$ is a unitary $G$-representation (it preserves \eqref{equation:metric-j-gamma}):
\begin{equation}\label{eq:G-action-holomorphic}
	(g \cdot u)(hT) := g \cdot [u(g^{-1}hT)], \quad g \in G,\, hT \in G/T,\, u \in L^2(G/T, J^\gamma).
\end{equation}
This space is `very large'; however, since the $G$-action on $J^\gamma$ is holomorphic, $G$ preserves the finite-dimensional space $H^0(G/T, J^\gamma)$. It can be easily verified that $H^0(G/T, J^\gamma)$ is always an irreducible $G$-representation (it can be the trivial space though), see \cite[Theorem 3.1.1]{Charles-23} for instance. 

Given a global weight $\gamma$, write $\lambda(\gamma) \in A$ for the corresponding infinitesimal weight (conversely, any infinitesimal weight in $A$ lifts to define a global weight). The Borel-Weil Theorem then asserts that $H^0(G/T,J^\gamma) = 0$ if $-\lambda(\gamma) \notin \mathfrak{a}_+ \cap A$ and otherwise
\begin{equation}
\label{equation:bw}
	H^0(G/T,J^\gamma) \cong  (V^{-\lambda(\gamma)})^*,
\end{equation}
where $V^{-\lambda(\gamma)} \in \widehat{G}$ is the unique irreducible representation with highest weight $-\lambda(\gamma)$, and ${}^*$ denotes the dual representation, see \cite[Theorem 7.58 and Lemma 7.5]{Sepanski-07}.



Combining \eqref{equation:ecriture} and \eqref{equation:bw}, we therefore see that the space of all irreducible representations can be realized geometrically as
\begin{equation}
\label{equation:irreps}
\bigoplus_{\mathbf{k} \in \Z^a \times \Z^b_{\geq 0}/\sim} H^0(G/T,\mathbf{J}^{\otimes \mathbf{k}}),
\end{equation}
where we use the compact notation
\begin{equation}
\label{equation:compress}
\mathbf{J}^{\otimes \mathbf{k}} := J^{\otimes k_1}_1 \otimes \dotsm \otimes J^{\otimes k_d}_d,
\end{equation}
and $J_i \to G/T$ is the homogeneous holomorphic line bundle obtained with the weight $\lambda_i$. The space $H^0(G/T,\mathbf{J}^{\otimes \mathbf{k}})$ is the unique (modulo isomorphism) irreducible representation with highest weight $\sum_{j=1}^d k_j \lambda_j$.

\subsubsection{One-dimensional representations}\label{sssection:abelian-rep-0} In this section we classify the one-dimensional representations of $G$ in terms of Borel-Weil theory and express the dependence on $\mathbf{k}$ introduced in \eqref{equation:ecriture}. We start with the latter. Denote by $[G, G]$ the commutator subgroup of $G$, i.e. the subgroup generated by commutators $xyx^{-1}y^{-1}$ for all $x, y \in G$; it is a Lie subgroup of $G$.

	\begin{lemma}\label{lemma:extension}
		The global weight $\gamma: T \to \mathbb{S}^1$ admits a lift $\beta$ to $G$, i.e. a homomorphism $\beta: G \to \mathbb{S}^1$ such that $\beta|_{T} = \gamma$, if and only if $k_{a + 1} = \dotsb = k_{a + b} = 0$. If such $\beta$ exists, then it is the representation corresponding to the global weight $\gamma$. 
	\end{lemma}
	\begin{proof}
	Notice that $k_{a + 1} = \dotsb = k_{a + b} = 0$ is equivalent to $d\gamma(e)|_{\mathfrak{t} \cap [\mathfrak{g}, \mathfrak{g}]} \equiv 0$. Denote by $Z(G)$ and $[G, G]$ the centre and the commutator subgroup of $G$, respectively.
	
	 Assume firstly that such an extension $\beta$ exists. Since $\mathbb{S}^1$ is Abelian, we get that $\beta|_{[G, G]}$ is trivial, and hence so is $\gamma|_{[G, G] \cap T}$. Therefore $d\gamma(e)|_{\mathfrak{t} \cap [\mathfrak{g}, \mathfrak{g}]} \equiv 0$.
	
	For the other direction, assume $d\gamma(e)|_{\mathfrak{t} \cap [\mathfrak{g}, \mathfrak{g}]} \equiv 0$. Since homomorphisms respect the exponential map, and as $[G, G] \cap T$ is connected, this implies that $\gamma|_{[G, G] \cap T}$ is trivial. By the structure theory of Lie groups, any $g \in G$ can be written as $g = zc$, where $z \in Z(G)$ and $c \in [G, G]$. If $g = z' c'$ for some other $z' \in Z(G)$ and $c' \in [G, G]$, then $z' \in zF$, where $F := [G, G] \cap Z(G)$. (Note that $F \leqslant [G, G] \cap T$, as $Z(G) \leqslant T$, and $F$ is finite since $\mathfrak{z}$ and $[\mathfrak{g}, \mathfrak{g}]$ are transverse to each other.) Define $\beta(g) := \gamma(z)$ for any such choice. Then $\beta$ is well-defined since $\gamma|_F$ is trivial, and by a similar argument $\beta|_T = \gamma$. Also, for any $g_i \in G$ satisfying $g_i = z_i c_i$ for some $z_i \in Z(G)$ and $c_i \in [G, G]$, for $i = 1, 2$, we have
	\[
		\beta(g_1 g_2) = \gamma(z_1 z_2) = \gamma(z_1)\gamma(z_2) = \beta(g_1) \beta(g_2),
	\]
	and thus $\beta$ is a homomorphism. The final statement is clear by definition of being a global weight, which completes the proof.
	\end{proof}

Next, we will make a standing assumption that the global weight $\gamma: T \to \mathbb{S}^1$ admits an extension to a homomorphism
\[
	\beta: G \to \mathbb{S}^1, \quad \beta|_{T} = \gamma.
\]
(Conversely, the global weight corresponding to a one-dimensional unitary representation $\rho: G \to \mathbb{S}^1$ is simply $\gamma = \rho|_T$.) Then, there are two natural line bundles over $G/T$; the trivial one $B^\gamma := G/T \times G \times_\beta \mathbb{C}$ and $J^\gamma = G \times_\gamma \mathbb{C}$, which we will show are canonically isomorphic. Indeed, write $[\bullet, \bullet]_{\beta/\gamma}$ to denote equivalence classes of elements in $G\times_{\beta/\gamma} \mathbb{C}$. As in \eqref{equation:fiberwise-dbar}, we equip $B^\gamma$ with a structure of a holomorphic vector bundle. Note that $G \times_\beta \mathbb{C} \cong \mathbb{C}$ via $[g, z]_{\beta} \mapsto \beta(g)z$; we will see in what follows that the holomorphic structure on $B^\gamma$ respects this trivialisation. Then:

\begin{lemma}\label{lemma:1d-rep}
	There is an equivariant unitary biholomorphism of holomorphic line bundles over $G/T$ covering the identity
	\begin{equation}\label{eq:map-Y}
		\Upsilon: J^\gamma \to B^\gamma, \quad [g, z]_\gamma \mapsto (gT, [g, z]_\beta), \quad g \in G, z \in \mathbb{C}.
	\end{equation}
	Moreover, $B^\gamma$ is biholomorphic to $G/T \times \mathbb{C}$ and so the space $H^0(G/T, J^\gamma)$ of holomorphic sections can be be identified with the space of constant functions on $G/T$. 
\end{lemma}
\begin{proof}
	The map $\Upsilon$ is well-defined as 
	\[
		\Upsilon(g t, \gamma(t^{-1})z) = (gt T, [g t, \gamma(t^{-1}) z]_{\beta}) = (g T, [g , z]_{\beta})= \Upsilon(g, z), \quad g \in G, t \in T, z \in \mathbb{C}.
	\]
	(More precisely, this defines a map $G \times \mathbb{C} \to B^\gamma$ which respects the equivalence relation and so descends to a map $\Upsilon: J^\gamma \to B^\gamma$ denoted by the same letter. We will make similar abuse of notation below.) It is $G$-equivariant in the sense that
	\[
		\Upsilon(h[g, z]_\gamma) = \Upsilon([hg, z]_\gamma) = (hgT, [hg, z]_\beta) = h (gT, [g, z]_{\beta}) = h \Upsilon([g, z]_{\gamma}), \quad h \in G.
	\]
	
	The map $\Upsilon$ is clearly surjective. Assume $\Upsilon([g_1, z_1]_\gamma) = \Upsilon([g_2, z_2]_\gamma)$ for some $g_i \in G, z_i \in \mathbb{C}$ for $i = 1, 2$. This implies $g_2 = g_1 t$ for some $t \in T$ and $[g_1, z_1]_\beta = [g_1 t, z_2]_\beta$ then implies $z_1 = \gamma(t) z_2$. Therefore also $[g_1, z_1]_\gamma = [g_2, z_2]_\gamma$. Since $\Upsilon^{-1}$ is smooth it is thus a diffeomorphism, which also respects the linear fibre structure and covers the identity. That $\Upsilon$ respects the holomorphic structure follows from the fact that $\Upsilon$ descends from the identity map on $G \times \mathbb{C}$. (Alternatively, from the definition \eqref{equation:fiberwise-dbar} it is possible to show $\Upsilon(\overline{\partial}^{J^\gamma}s) = \overline{\partial}^{B^\gamma} \Upsilon(s)$ for any section $s$ of $J^\gamma$, where $\overline{\partial}^{B^\gamma}$ is the natural $\overline{\partial}$-operator on $B^\gamma$ descending from $G \times \mathbb{C}$.)
	
	For the second claim, we note that the holomorphic structure on $B^\gamma$ really is the product holomorphic structure. Indeed, the constant section $s(gT) := [e, 1]_{\beta}$ is holomorphic, since it lifts to the function
	\[
	\overline{s}(g) := g^{-1} s(gT) = g^{-1} [g, \beta(g^{-1})]_{\beta} = \beta(g^{-1})
	\]
	on $G$. Since $\mathbb{S}^1$ is Abelian, $\beta$ is constant along $[G, G]$, and so in particular its derivative vanishes on $\mathfrak{n}^{\pm}$; thus $s$ is holomorphic, which completes the proof.
%


\end{proof}

\subsubsection{Laplacian eigenvalues} \label{sssection:laplace-eigenvalue} Recall that we equipped $G$ with a bi-invariant Riemannian metric $g$. Using the bi-invariance, the choice of $g$ bijectively corresponds to the choice of an $\Ad(G)$-invariant inner product on $\mathfrak{g} = \mathfrak{z} \oplus [\mathfrak{g}, \mathfrak{g}]$. For this, consider the Killing form $B(X, Y) := \Tr(\ad_X \circ \ad_Y)$ on $[\mathfrak{g}, \mathfrak{g}]$ ($-B$ is positive definite on $[\mathfrak{g}, \mathfrak{g}]$ and $\Ad(G)$-invariant by \cite[Theorem 6.16]{Sepanski-07}), and an arbitrary negative definite $A$ on $\mathfrak{z}$; then we define $g$ to be induced from the negative of $A \oplus B$. The negative definite symmetric bilinear form $A \oplus B$ extends complex bilinearly to $\mathfrak{g}_{\mathbb{C}} \times \mathfrak{g}_{\mathbb{C}}$, and in particular restricts to an inner product on $i\mathfrak{g} \times i\mathfrak{g}$. Write $\|\bullet\|$ for the norm defined by duality on $(i\mathfrak{t})^*$, and $\Delta_G$ for the (non-negative) Laplace operator induced by $g$.

\begin{lemma}\label{lemma:laplace-eigenvalue}
	Let $\rho: G \to \operatorname{U}(V)$ be an irreducible unitary representation with highest weight $\lambda = \phi(\mathbf{k})$ for some $\mathbf{k} \in \mathbb{Z}^a \times \mathbb{Z}^b_{\geq 0}$. Then for any $C \in \End(V)$, the function
	\[
		f_C(g) := \Tr(C \rho(g)^{-1}), \quad g \in G,
	\]
	is an eigenfunction of the Laplace operator with eigenvalue
	\begin{equation}
	\label{equation:definition-ck}
		c(\mathbf{k}) = \|\phi(\mathbf{k}) + \delta\|^2 - \|\delta\|^2,
	\end{equation}
	where $\delta$ is half of the sum of the positive roots, i.e. $\delta = \tfrac{1}{2} \sum_{\alpha \in \Delta^+} \alpha$. We have $c(\mathbf{k}) \geq 0$ with equality if and only if $\rho$ is the trivial representation.
\end{lemma}
\begin{proof}
	Let $(X_i)_{i = 1}^{\dim G}$ be an orthonormal basis of $\mathfrak{g}$, and use the same notation for their extension as left-invariant vector fields on $G$. It is straightforward to compute that $\Delta_G = -\sum_i X_i^2$. Then we compute
	\[
		\Delta_G f_C(g) = \sum_i \partial_t^2|_{t = 0} f_C(e^{tX_i}g) = \sum_i \partial_t^2|_{t = 0} f_{\rho(e^{-tX_i}) C}(g) = f_{d\rho(\Omega) C}(g), 
	\]
	where $\Omega := \sum_i X_i \otimes X_i \in \mathfrak{g} \otimes \mathfrak{g}$. Next, it can be shown that $d\rho(\Omega)$ commutes with the representation, see \cite[Proposition 5.24]{Knapp-02} (in fact, $\Omega$ is known as the \emph{Casimir element} and belongs to the centre of the \emph{universal enveloping algebra}) and so by Schur's lemma it is a constant multiple of identity. Write $c(\mathbf{k})$ for this constant; we are left to compute the value $c(\mathbf{k})$.
	
	In the case when $\mathfrak{g}$ is semisimple, this is done in \cite[Proposition 5.28]{Knapp-02}. In the general case, we may simply take $X_1, \dotsc, X_a$ to be an orthonormal basis of $\mathfrak{z}$, and complete this basis as in \cite[Proposition 5.28]{Knapp-02} (where now $\mathfrak{t}$ is replaced by $\mathfrak{t}'$); the same argument using the highest weight vector of $\rho$ then applies and gives the claimed formula. The final claim follows by the same argument as in \cite[Proposition 5.28]{Knapp-02} (from the \emph{dominant property} of the highest weight vector).
	\end{proof}

By Lemma \ref{lemma:laplace-eigenvalue}, we note that by equivalence of norms there is a constant $C > 1$ such that
\begin{equation}
\label{equation:ck-asymptotic}
	C^{-1} (1 + |\mathbf{k}|^2) \leq c(\mathbf{k}) \leq C(1 + |\mathbf{k}|^2), \quad \mathbf{k} \neq 0.
\end{equation}

\subsection{Topology of the line bundles}
\label{sssection:topology}

In this section we investigate the topology of the holomorphic line bundles constructed previously. We first study the cohomology of $G/T$.

\begin{proposition}\label{prop:G/T-topology}
	Let $G$ be a compact connected Lie group and $T \leqslant G$ a maximal torus. Then 
	\begin{enumerate}[itemsep=5pt]
		\item[1.] The cohomology ring $H^{\bullet}(G/T, \mathbb{Z})$ has no torsion and is zero in odd degrees. 
		
		\item[2.] Moreover, $H^2(G/T, \mathbb{Z}) \cong \Z^{b}$, where $b = \dim T - \dim Z(G)$.
	\end{enumerate}
\end{proposition}
\begin{proof}
	\emph{Item 1.} The first point is due to \cite[Section 26]{Borel-53}.
	\medskip
	
	\emph{Item 2.} The second point is an exercise in algebraic topology, which is (more or less) known but is difficult to locate a precise reference, so we give a direct proof. The starting point is the long exact sequence of the pair $(G, T)$ in cohomology:
	\begin{align*}
\dotso &\longrightarrow \underbrace{H^1(G, T; \Z)}_{\cong \{0\}} \longrightarrow H^1(G,\Z) \longrightarrow H^1(T,\Z) \overset{\delta}{\longrightarrow} H^2(G, T; \Z)\\
 &\longrightarrow H^2(G,\Z) \longrightarrow H^2(T, \Z) \longrightarrow \dotso
	\end{align*}
	That the first space is zero, we will see at the end of Step 1. 
	\medskip
	
	\emph{Step 1.} We first claim that $H^2(G/T, \Z)$ is isomorphic to $H^2(G, T; \Z)$. Indeed, denoting the quotient map by $\pi: G \to G/T$, the pullback map $\pi^*: H^2(G/T, \mathbb{Z}) \to H^2(G, \Z)$ is well-defined. Moreover, denoting $j: T \to G$ the inclusion, the map $\pi \circ j$ is constant, and so trivial on cohomology. In particular, the map $\pi^*$ lifts as a map $\pi^*: H^2(G/T, \Z) \to H^2(G, T; \Z)$, which we claim is an isomorphism. Consider the commutative diagram, coming from the naturality of Hurewicz maps:
	\[
	\begin{tikzcd}
		H_2(G, T; \Z) \dar{\pi_*} & \lar[swap]{h_{\mathrm{rel}}} \pi_2(G, T)\dar{\cong}\\
		H_2(G/T, \Z) & \lar{h_{\mathrm{abs}}} \pi_2(G/T),
	\end{tikzcd}
	\]
	where $\pi_2$ denotes the second homotopy group, $h_{\mathrm{rel}}$ and $h_{\mathrm{abs}}$ denote relative and absolute Hurewicz maps, respectively, $\pi_*$ is defined analogously to $\pi^*$ above (using the long exact sequence in homology), and the map to the right is an isomorphism thanks to \cite[Theorem 4.41]{Hatcher-02}. We claim that $\pi_*$ is an isomorphism; it suffices to show that $h_{\mathrm{rel}}$ and $h_{\mathrm{abs}}$ are both isomorphisms. 
	
	By the long exact sequence in homotopy groups for the pair $(G, T)$ we get 
	\[\small
	\begin{tikzcd}
		\dotso \rar & \underbrace{\pi_2(G)}_{\cong \{0\}}\rar & \pi_2(G, T) \rar & \pi_1(T) \rar[two heads] & \pi_1(G) \rar & \pi_1(G, T) \rar & 0, 
	\end{tikzcd}
	\]
	where the fourth arrow is surjective and $\pi_2(G)$ is trivial, both thanks to \cite[Theorem 7.1]{Brocker-Dieck-85}, implying that $\pi_1(G, T) \cong \{0\}$ and $\pi_2(G, T)$ is Abelian (as $\pi_1(T)$ is Abelian). It follows that also $\pi_1(G/T) \cong \{0\}$. Next, we claim that the action of $\pi_1(T)$ on $\pi_2(G, T)$ is trivial. Then, by the absolute and relative Hurewicz theorems (see \cite[Theorem 4.32]{Hatcher-02}), we deduce that $\pi_*$ is an isomorphism.
	
	To see the latter claim, let $f: (D^2, \mathbb{S}^1, z_0) \to (G, T, e)$ define a class $[f] \in \pi_2(G, T)$ (where $D^2 \subset \mathbb{R}^2$ is the unit disk) and take a loop $\gamma: [0, 1] \to T$ based at $e$. Then, by definition, $[\gamma] [f] = [f_1]$, where $(f_s)_{s \in [0, 1]}$ is a homotopy satisfying $f_s(z_0) = \gamma(1 - s)$ (see \cite[Section 4.A]{Hatcher-02}). Using the group structure, observe we may take $f_s(z) = \gamma(1 - s) f(z)$ which satisfies all the conditions; thus $[\gamma][f] = [f]$ in this case. We conclude that the action of $\pi_1(T)$ on $\pi_2(G, T)$ is trivial, as claimed.
	
	Next, by the universal coefficients theorem (see \cite[Corollary 3.3]{Hatcher-02}), we get that 
	\[
		H^2(G, T; \Z) \cong \operatorname{Hom}(H_2(G, T; \Z), \mathbb{Z}) \oplus \operatorname{Tor}(H_1(G, T; \Z)),
	\]
	where $\mathrm{Tor}$ denotes the torsion operator. As noted above, the inclusion $\pi_1(T) \to \pi_1(G)$ is surjective, and hence is its Abelianisation, $H_1(T) \to H_1(G)$; thus, by the long exact sequence in homology for the pair $(G,T)$ we conclude that $H_1(G, T; \Z) \cong \{0\}$. It follows that $H^2(G, T; \Z)$ has no torsion and that $\pi^*$ is also an isomorphism.
	
	Finally, the space $H^1(G, T; \Z)$ is zero by the universal coefficients theorem and the fact that $H_1(G, T; \Z)$ is zero.
	\medskip

	\emph{Step 2.} First we claim that $H^1(G, \Z) \cong \mathbb{Z}^a$, where $a := \dim Z(G)$. By universal coefficients theorem $H^1(G, \Z) \cong \operatorname{Hom}(H_1(G, \Z), \Z)$, and since $\pi_1(G)$ is Abelian (for instance, as $\pi_1(T)$ surjects onto it), we know that $H_1(G, \Z) \cong \pi_1(G)$. By structure theory for compact Lie groups, there is a finite cover $p: \widetilde{G} := (\mathbb{S}^1)^a \times G_0 \to G$, which is also a homomorphism, and $G_0$ is simply connected (see \cite[Theorem 8.1]{Brocker-Dieck-85}). It follows that the rank of $\pi_1(G)$ is equal to $a$, proving the claim.
	
%
	 Again, by structure theory of Lie groups, $\widetilde{G}/\Gamma \cong G$ where $\Gamma \leqslant Z(\widetilde{G})$. A maximal torus of $\widetilde{G}$ is $\widetilde{T} = p^{-1}(T) = (\mathbb{S}^1)^a \times T_0$, where $T_0 \leqslant G_0$ is a maximal torus. Thus $p$ descends to $p: \widetilde{G}/\widetilde{T} \to G/T$, and this map is surjective. It is in fact also injective, since $Z(\widetilde{G}) \leqslant \widetilde{T}$ by Cartan's theorem \cite[Chapter IV, Theorem 1.6]{Brocker-Dieck-85} (every element of a Lie group is conjugate to an element of $T$). Since $p$ is also an immersion, it follows that $p$ is a diffeomorphism. Since $\widetilde{G}/\widetilde{T} \cong G_0/T_0$, we have reduced the main claim to a claim about simply connected Lie groups $G$.
	
	In this case $H^1(G, \Z) \cong \{0\}$, so $\delta$ is injective, and $H^2(G, \Z) \cong \{0\}$ by Hurewicz's theorem (and the fact that $\pi_1(G) \cong \pi_2(G) \cong \{0\}$), so $\delta: H^1(T, \Z) \to H^2(G, T; \Z)$ is an isomorphism, which together with Step 1 completes the proof.
	\end{proof}
	
	\begin{remark}
		In fact, the \emph{transgression map} appearing on the second page of the Leray-Serre spectral sequence is then equal to $\tau = \delta \circ (\pi^*)^{-1}$, see \cite[page 540]{Hatcher-notes}, and using the proof of Proposition \ref{prop:G/T-topology} it follows that
	\[
		H^2(G/T, \mathbb{Z})/ \mathrm{im}(\tau) \cong H_1(G, \Z).
	\]
	Also, an alternative argument that circumvents the reduction to simply connected Lie groups in Step 2 is possible: it starts by noting $H^2(G, \Z) \cong \mathrm{Hom}(H_2(G, \Z), \Z) \oplus \mathrm{Tor}(H_1(G, \Z))$, and that the non-torsion part is determined by Hopf's theorem \cite[Theorem 3C.4]{Hatcher-02}, isomorphic to $\mathbb{Z}^{{a \choose 2}}$ and generated by cup products of elements of $H^1(G, \Z)$, and so by K\"unneth's formula, the map $H^2(G, \Z) \to H^2(T, \Z)$ is injective on the non-torsion part and zero on the torsion part. It follows that the quotient space $H^2(G, T; \Z)/ \mathrm{im}(\delta)$ is isomorphic to $\mathrm{Tor}(H_1(G, \Z))$, which concludes the argument.
	\end{remark}

Let us now introduce a suitable basis of de Rham cohomology classes on $H^2_{\mathrm{dR}}(G/T)$ and compute the Chern classes of the bundles $J^\gamma$ in terms this basis and the weight $\gamma: T \to \mathbb{S}^1$, where $T \leqslant G$ is a maximal torus with Lie algebra $\mathfrak{t}$. Recall that $\mathfrak{z}$ denotes the Lie algebra of the centre of $G$, and $[\mathfrak{g}, \mathfrak{g}]$ denotes the commutator Lie algebra.


The \emph{Maurer-Cartan} $1$-form $\omega_{\mathrm{MC}}$ on $G$, with values in $\mathfrak{g}$, is defined as
\[
	\omega_{\mathrm{MC}}(g)(v) := dL_{g^{-1}} v, \quad v \in T_gG,
\]
where $L_{\bullet}$ denotes left multiplication by $\bullet \in G$. It is known (see \cite[page 41]{Kobayashi-Nomizu-63}) that $\omega_{\mathrm{MC}}$ satisfies
\[
	R_g^* \omega_{\mathrm{MC}} = \Ad(g^{-1}) \omega_{\mathrm{MC}},\,\, g \in G, \quad \omega_{\mathrm{MC}}(X_\xi) = \xi, \quad \xi \in \mathfrak{g},
\]
where $\Ad(\bullet)$ denotes the adjoint representation of $G$ on $\mathfrak{g}$ and $X_\xi$ is the fundamental vector field generated by $\xi$. The curvature of $\omega_{\mathrm{MC}}$ is zero (when thinking of $\omega_{\mathrm{MC}}$ as the connection $1$-form on the trivial bundle principal bundle $G \to \{e\}$, where $e \in G$ is the neutral element), i.e.
\begin{equation}\label{eq:MC-flat}
	d\omega_{\mathrm{MC}}(X, Y) + [\omega_{\mathrm{MC}}(X), \omega_{\mathrm{MC}}(Y)] = 0,\,\, \forall X, Y \in C^\infty(G, TG).
\end{equation}
There is a natural $\Ad(T)$-invariant splitting
\[
	\mathfrak{g} = \mathfrak{t} \oplus \mathfrak{m} = \mathfrak{z} \oplus \mathfrak{t} \cap [\mathfrak{g}, \mathfrak{g}] \oplus \mathfrak{m},
\]
where $\mathfrak{m}$ is the real part of $\mathfrak{n}^+ \oplus \mathfrak{n}^-$ (see \eqref{eq:positive-negative-splitting}). Note that $\mathfrak{m} \leqslant [\mathfrak{g}, \mathfrak{g}]$ by definition of $\mathfrak{n}^\pm$. Denoting the projection to $\mathfrak{t}$ by $\Pi_{\mathfrak{t}}$, it is straightforward to check that $\omega := \Pi_{\mathfrak{t}} \omega_{\mathrm{MC}}$ is a connection $1$-form on the principal $T$-bundle $G \to G/T$. Take a basis $(\f_j)_{j = 1}^{a + b}$ of $\mathfrak{g}$ such that $(\f_j)_{j = 1}^a$ is a basis of $\mathfrak{z}$ dual to $(-i\lambda_j)_{j = 1}^{a}$, and $(\f_j)_{j = a + 1}^{a + b}$ is a basis of $\mathfrak{t} \cap [\mathfrak{g}, \mathfrak{g}]$ dual to $(-i\lambda_j)_{j = a + 1}^{a + b}$. In this basis write
\[
	\omega = (\omega_1, \dotsc, \omega_a, \omega_{a + 1}, \dotsc, \omega_{a + b}),
\]
where $(\omega_j)_{j = 1}^{a + b}$ are real-valued $1$-forms. Notice that $d\omega_1 = \dotsb = d\omega_a = 0$ thanks to \eqref{eq:MC-flat} and the fact that $[\mathfrak{m}, \mathfrak{m}] \leqslant \mathfrak{t} \cap [\mathfrak{g}, \mathfrak{g}]$, as well as $[\mathfrak{m}, \mathfrak{t}] \leqslant \mathfrak{m}$. Moreover, write
\[
	d\omega_{a + 1} = \pi^* \eta_1, \dotsc, d\omega_{a + b} = \pi^*\eta_{b},
\]
for some closed real-valued $2$-forms $\eta_1, \dotsc, \eta_b$ on $G/T$ (this is possible as $R_g^*\omega_i = \omega_i$ and $d\omega_i$ vanishes on vertical vectors for all $i$). Then

\begin{proposition}\label{prop:cohomology-basis}
	We have that $([\omega_i])_{i = 1}^a$ and $(\eta_i)_{i = 1}^b$ are bases of $H^1_{\mathrm{dR}}(G)$ and $H^2_{\mathrm{dR}}(G/T)$, respectively. 
\end{proposition}
\begin{proof}
	For the first claim, note that $H^1(G, \Z) \cong \Z^a$, as follows from the structure theory of Lie groups: indeed, $\pi_1(G)$ is Abelian and of rank $a$ (by the proof of Proposition \ref{prop:G/T-topology}), so the claim follows using $H^1(G, \Z) \cong \mathrm{Hom}(H_1(G, \Z), \Z)$ and Hurewicz's theorem. Thus, it suffices to show that $([\omega_i])_{i = 1}^a$ are linearly independent. Assume $\sum_i c_i [\omega_i] = 0$ for some $c_i \in \mathbb{R}$, i.e. 
	\[
		\sum_{i = 1}^a c_i \omega_i = df, \quad f \in C^\infty(G).
	\]   
	For fixed $j$, plugging in the fundamental vector fields $X_{\f_j}$, we get that $c_j = X_{\f_j} f$; integrating along an arbitrary fibre and using that $X_{\f_j}$ preserves the Haar measure, we get that $c_j = 0$. As $j$ was arbitrary, this proves the claim.
	
	For the latter claim, by the preceding proposition, we know that the dimension of de Rham cohomology is $\dim H_{\mathrm{dR}}^2(G/T) = b$, so it suffices to show that $([\eta_i])_{i = 1}^b$ are linearly independent. Assume that $\sum_i c_i [\eta_i] = 0$ for some $c_i \in \mathbb{R}$. Then $\sum_i c_i \eta_i = d\beta$ for some $1$-form $\beta$ on $G/T$, and so
	\[
		d\big(- \pi^*\beta + \sum_i c_i \omega_{a + i}\big) = 0,
	\]
	which, combined with the first part of the proposition, implies there is $f \in C^\infty(G)$ and $C_1, \dotsc, C_a \in \mathbb{R}$ such that
	\[
		\sum_{i = 1}^b c_i \omega_{a + i} = \pi^*\beta + \sum_{i = 1}^a C_i \omega_i + df.
	\]
	Plugging in fundamental vector fields $X_{\f_j}$ for some fixed $j$ and integrating in the fibres (as above) gives $c_j = 0$ for $j = 1, \dotsc, b$, which completes the proof.
\end{proof}

We turn to the first Chern class of $J^\gamma = G \times_\gamma \mathbb{C}$. By \eqref{equation:ecriture}, there is a $\mathbf{k} = (k_1, \dotsc, k_{a + b}) \in \mathbb{Z}^a \times \mathbb{Z}^b_{\geq 0}$ such that $d\gamma(e) = \phi(\mathbf{k}) \in (i \mathfrak{t})^*$. Then 

\begin{proposition}\label{prop:line-bundle-topology}
		The first Chern class of $J^\gamma$ satisfies
		\[
			2\pi c_1(J^\gamma) = \sum_{j = 1}^b k_{a + j} [\eta_j] \in H^2_{\mathrm{dR}}(G/T).
		\]
		In particular, $J^\gamma$ is topologically trivial if and only if $k_{a + 1} = \dotsb = k_{a + b} = 0$.
	\end{proposition}
	\begin{proof}
Write $F_\gamma$ for the (purely imaginary) curvature $2$-form on $G/T$ of $J^\gamma$. We will make use of the following formula for the curvature (see \cite[Chapter III, Theorem 5.1]{Kobayashi-Nomizu-63})
		 \[
		 	F_\gamma(x)(X, Y) = p \big(d\gamma(e) d\omega(p)(X^{\HH}, Y^{\HH})\big) p^{-1},
		 \]
		 where $x \in G/T$, $X, Y \in T_x(G/T)$, $p \in \pi^{-1}x$ is arbitrary, $X^{\HH}, Y^{\HH}$ denote horizontal lifts of $X, Y$ at $p$, and $p: \mathbb{C} \to J_\gamma(x)$, $\xi \mapsto [p, \xi]$, is the identification of the fibre with $\mathbb{C}$. Expanding we get
		 \[
		 	F_{\gamma}(x)(X, Y) = \sum_{j = 1}^b d\omega_{a + j}(p)(X^{\HH}, Y^{\HH}) d\gamma(e)(\f_{a + j}) = i \sum_{j = 1}^b k_{a + j} \eta_{j}(x)(X, Y).
		 \]
		 As by definition $c_1(J^\gamma) = \tfrac{1}{2\pi i} [F_\gamma]$, this completes the proof.
	\end{proof}

\subsection{Diophantine actions}

\label{ssection:diophantine}

When the fibre of the principal extension is an arbitrary Lie group, we will need to know how fast does the transitivity group generate $G$. We need:

\begin{definition}\label{def:diophantine-action}
	Let $G$ be a compact Lie group equipped with a bi-invariant Riemannian metric, let $H \leqslant G$ be a Lie subgroup, and let $W \subset G$ be a subset. Set $X := G/H$ and equip it with the quotient Riemannian metric, and for $n \in \mathbb{Z}_{\geq 1}$ write
	\begin{equation}\label{eq:diophantine-w_n-def}
		\mc{W}_n := \{w_1^{\pm 1} \dotsm w_k^{\pm 1} \mid w_i \in W,\, i = 1, \dotsc, k,\, 1\leq k \leq n\}.
	\end{equation}
	 We say that $W$ has a \emph{Diophantine} action (by left multiplication) on $X$ if there exist $C > 0$ and $\alpha > 0$ such that for all $x \in X$ and all $n \in \mathbb{Z}_{\geq 1}$ the orbit space $\mc{W}_n \cdot x \subset X$ is $C n^{-\alpha}$-dense in $X$.   
\end{definition}
By invariance of the metric under left multiplication and transitivity of the action, note that in the previous definition it suffices to consider the orbit of an arbitrary point $x_0 \in X$. 

We recall that for a semisimple Lie algebra, we have $[\mathfrak{g}, \mathfrak{g}] = \mathfrak{g}$. The following statements were essentially proved in \cite[Theorem A.3 and Corollary A.4]{Dolgopyat-02}. We slightly modify the proofs and improve the statements.

\begin{lemma}\label{lemma:diophantine-semisimple}
	Assume that $G$ is a compact connected semisimple Lie group. Then there exists $\varepsilon = \varepsilon(G) > 0$ such that any $W \subset G$ which is $\varepsilon$-dense, has a Diophantine action on $G$ (by left multiplication). Moreover, we may take any $\alpha \in (0, 1)$ in Definition \ref{def:diophantine-action}. 
\end{lemma}
\begin{proof}
		We divide the proof into steps. We will show that for any $\delta \in (0, \tfrac{1}{2})$, there is a constant $C > 0$, such that for $\varepsilon$ small enough, the sequence defined by $n_0 := 1$ and $\varepsilon_0 := \varepsilon$ and inductively by
		\[
			\varepsilon_{j + 1} := C \varepsilon_j^{2(1 - \delta)}, \quad n_{j + 1} := Cn_j \varepsilon_j^{-1 + 2\delta}, \quad j \geq 0,
		\]
		satisfies that $\mc{W}_{n_j}$ is $\varepsilon_j$-dense. This completes the proof: indeed, setting $\beta := 2(1 -\delta)$ (which is $> 1$), a straightforward computation gives for any $j \in \mathbb{Z}_{\geq 0}$
		\[
			\varepsilon_j = C^{\frac{1}{1 - \beta}} \Big(C^{\frac{1}{\beta - 1}} \varepsilon\Big)^{\beta^j}, \quad n_j = C^{2k - 1 - \frac{\beta (\beta^{k - 1} -1)}{\beta - 1}} \varepsilon^{1 - \beta^k}.
		\]
		Then, given any $\eta$ small enough (depending on $\varepsilon$ and $C$), there is a unique $j \in \mathbb{Z}_{\geq 0}$ such that
		\[
			\varepsilon_{j + 1} \leq \eta < \varepsilon_j.
		\]
		By assumption, $\mc{W}_{n_{j + 1}}$ is then $\eta$-dense, and so to have $n_{j + 1} < \eta^{-\alpha}$ for some $\alpha > 0$, it suffices to show $n_{j + 1} < \varepsilon_{j}^{-\alpha}$. By the above formulas this is equivalent to
		\[
			C^{1 + \frac{\beta - \alpha}{\beta - 1}} \varepsilon \times \Big(C^{\frac{2j}{\beta^j(\alpha - \beta)} + \frac{1}{\beta - 1}} \varepsilon\Big)^{\beta^j(\alpha - \beta)} < 1,
		\]
		which is true for $j$ large enough (i.e. $\varepsilon_0$ small enough) under the condition $\alpha > \beta$. Since $\beta = 2(1 - \delta)$ and we may take $\delta \in (0, \tfrac{1}{2})$, we conclude that $\alpha$ can be taken to be arbitrarily close to $1$, completing the proof. (Note that as opposed to Definition \ref{def:diophantine-action} for the convenience of the proof we have taken here $\alpha$ to be the exponent in the time $n \sim \eta^{-\alpha}$ for which $\mc{W}_n$ is $\eta$-dense; the relation with $\alpha$ in the mentioned definition is given simply by taking $\tfrac{1}{\alpha}$.)
		
		We are left to prove the existence of such positive $C$ inductively in $j$.
		\medskip
		
		\emph{Step 1.} Equip $\mathfrak{g}$ with a bi-invariant inner product $\langle{\bullet, \bullet}\rangle$ inducing a norm $\|\bullet\|$ on $\mathfrak{g}$ and take a basis $(X_k^{(j)})_{k = 1}^d$ of $\mathfrak{g}$ with
		\begin{equation}\label{eq:aligned}
			\frac{1}{2} \leq \|X_k^{(j)}\| \leq 2, \quad |\langle{X_k^{(j)}, X_{k}^{(j)}}\rangle| \leq \frac{1}{2} \|X_k^{(j)}\| \|X_{k}^{(j)}\|, \quad k \neq \ell 
		\end{equation}
		 such that $x_{k}^{(j)} := \exp(\varepsilon_j^{1 - \delta} X_k^{(j)}) \in W_{n_j}$ (it is possible to do so by $\varepsilon_j$-density of $\mc{W}_{n_j}$). By the Baker-Campbell-Hausdorff formula we have
		\begin{equation}\label{eq:BCH}
			x_{k \ell}^{(j)} := [x_{k}^{(j)}, x_{\ell}^{(j)}] = \exp\Big(\varepsilon_j^{2(1 - \delta)} [X_k^{(j)}, X_{\ell}^{(j)}] + \mc{O}(\varepsilon_j^{3(1 - \delta)})\Big) \in \mc{W}_{4n_j}.
		\end{equation}
		Since $G$ is semisimple we have $[\mathfrak{g}, \mathfrak{g}] = \mathfrak{g}$, and by linearity, continuity, and by compactness of the space of bases satisfying \eqref{eq:aligned}, we may extract a basis $(Y_k^{(j)})_{k = 1}^d$ from $([X_k^{(j)}, X_{\ell}^{(j)}])_{k, \ell = 1}^d$ satisfying for some $c \in (0, 1)$
		\begin{equation}\label{eq:aligned-commutator}
			c \leq \|Y_k^{(j)}\| \leq \frac{1}{c}, \quad |\langle{Y_k^{(j)}, Y_\ell^{(j)}}\rangle| \leq (1 - c) \|Y_k^{(j)}\| \|Y_\ell^{(j)}\|, \quad k \neq \ell.
		\end{equation}
		
		\medskip
		
		\emph{Step 2.} We claim that for each $\eta > 0$, there exists $C_1 > 0$, such that for each $C_2 \geq C_1$, there exists $C > 0$, such that for each $Y \in \mathfrak{g}$ with $C_1 \leq \|Y\| \leq C_2$, there exist integers $(p_k)_{k = 1}^d$, with $|p_k| \leq C$, satisfying
		\begin{equation}\label{eq:estimate}
			\Big\|Y - \sum_{k = 1}^d p_k Y_k^{(j)}\Big\| \leq \eta \|Y\|. 
		\end{equation}
		Indeed, write $\tfrac{Y}{\|Y\|} = \sum_{k = 1}^d a_k Y_k^{(j)}$ for some non-zero $Y$ and $(a_k)_{k = 1}^d \subset \mathbb{R}$, and choose integers $(p_k)_{k = 1}^d$ such that $|a_k - \tfrac{p_k}{\|Y\|}| \leq \frac{1}{\|Y\|}$. Then 
		\[
			\Big|\frac{Y}{\|Y\|} - \sum_{k = 1}^d \frac{p_k}{\|Y\|} Y_k^{(j)} \Big| \leq \sum_{k = 1}^d \Big|a_k - \frac{p_k}{\|Y\|}\Big| \|Y_k^{(j)}\| \leq \frac{\mc{O}(1)}{\|Y\|},
		\]
		where we used triangle inequality, and in the last inequality we used \eqref{eq:aligned-commutator} (by $\mc{O}(1)$ we denote a quantity bounded uniformly in $j$). Therefore, there exists $C_1$ such that \eqref{eq:estimate} holds; we just need to verify that $(p_k)_{k = 1}^d$ is uniformly bounded. Firstly, using \eqref{eq:aligned-commutator} again (this time, also its second inequality) and compactness of the unit sphere, we get that $|a_k| = \mc{O}(1)$ for $k = 1, \dotsc, d$. It then follows that
		\[
			|p_k| \leq 1 + |a_k| \|Y\| = \mc{O}(1), \quad k = 1, \dotsc, d,
		\]
		where we directly see how the choice of $C_2$ affects the value of $C$.
		\medskip
		
		\emph{Step 3.} Since the derivative of the exponential map at $0 \in \mathfrak{g}$ is equal to the identity, there is a neighbourhood $U$ of the identity in $G$ onto which $\exp$ maps diffeomorphically, and a constant $C \in (0, 1)$ such that
		\[	
			C\|X - Y\| \leq d(\exp(X), \exp(Y)) \leq \frac{1}{C} \|X - Y\|, \quad X, Y \in \exp^{-1}(U),
		\]
		where $d(\bullet, \bullet)$ denotes distance in $G$ induced by the Riemannian metric. Let $\eta > 0$; then, by Step 2, and the previous inequality, there exist $C_1 > 0$ such that for any $C_2 > C_1$, for any $y \in G$ satisfying $C_1\varepsilon_j^{2(1 - \delta)} \leq d(y, \id) \leq C_2 \varepsilon_j^{2(1 - \delta)}$, there exist integers $(a_{k \ell})_{k, \ell = 1}^d$ with $|a_{k \ell}| = \mc{O}(1)$, such that
		\[
			d\Big(y, \prod_{k, \ell = 1}^d (x_{k \ell}^{(j)})^{a_{k \ell}}\Big) \leq \eta \varepsilon_j^{2(1 - \delta)}.
		\]
		It is important to note that here we use \eqref{eq:BCH} to justify the scaling in the preceding equation.
		\medskip
		
		
		\emph{Step 4.} Partition $U$ into coordinate cubes $(\mc{C}_{i})_{i \in I}$ with side length $C \varepsilon^{2(1 - \delta)}$ for some $C > 0$ to be chosen later, satisfying $C = \mc{O}(1)$ and $C^{-1} = \mc{O}(1)$. By the induction hypothesis, for each $i_1 \in I$, there exists $i_2 \in I$ with $\mc{C}_{i_2} \cap \mc{W}_{n_j} \neq \emptyset$, such that $d(\mc{C}_{i_1}, \mc{C}_{i_2}) \leq \mc{O}(\varepsilon_j)$. Therefore, we can join $\mc{C}_{i_1}$ to $\mc{C}_{i_2}$ by a sequence $\mc{C}_{i_1} = \mc{C}^{(1)}, \mc{C}^{(2)}, \dotsc, \mc{C}^{(N)} = \mc{C}_{i_2}$ of $N$ adjacent cubes, such that $N = \mc{O}(\varepsilon_j^{-1 + 2\delta})$ (by equivalence of $\ell^1$ and $\ell^2$ norms in $\mathfrak{g}$, and as long as $2(1 - \delta) > 1$, that is, $\delta < \tfrac{1}{2}$). Now for some $n$ and $k$, assume $w \in \mc{W}_{n} \cap \mc{C}^{(k)} \neq \emptyset$; we claim that $\mc{W}_{n + \mc{O}(n_j)} \cap \mc{C}^{(k + 1)} \neq \emptyset$. Indeed, let $yw \in \mc{C}^{(k + 1)}$ be the midpoint of the cube, so 
		\[
			\tfrac{C}{2} \varepsilon^{2(1 - \delta)} \leq d(y, \id) \leq 2dC \varepsilon^{2(1 - \delta)}.
		\] 
By taking $\eta = 1$, $C_1 = \tfrac{C}{2}$, and $C_2 = 2dC$ in Step 3, with $C$ large enough, we conclude there exists $x = 	\prod_{k, \ell = 1}^d (x_{k \ell}^{(j)})^{a_{k \ell}}$ with $|a_k| = \mc{O}(1)$, so that $d(yw, xw) \leq \eta \varepsilon^{2(1 - \delta)}$ and $xw \in \mc{C}^{(k + 1)}$. This proves the claim.
		
		Therefore, for some $N_j := \mc{O}(n_j \varepsilon_j^{-1 + 2\delta})$ we have $\mc{W}_{N_j}$ is $\mc{O}(\varepsilon_{j}^{2(1 - \delta)})$-dense in $U$. It is left to observe that any element of $G$ can be written as a (uniformly) bounded product of elements in $U$ (for instance, as $\exp: \mathfrak{g} \to G$ is surjective by the fact that $G$ is connected and compact), and so for some $n_{j + 1} = \mc{O}(n_j \varepsilon_j^{-1 + 2\delta})$ we have shown that $\mc{W}_{n_{j + 1}}$ is $\varepsilon_{j + 1}$-dense in $G$, completing the induction step. It is left to observe that taking $\varepsilon$ small enough we achieve that $\varepsilon_j \to 0$ as $j \to \infty$, which completes the proof.
\end{proof}

\begin{corollary}\label{corollary:diophantine-flag-manifold}
	Let $T$ be a maximal torus in compact connected Lie group $G$. Then there exists $\varepsilon = \varepsilon(G) > 0$ such that any $W \subset G$ which is $\varepsilon$-dense, $W$ has a Diophantine action on $G/T$ (by left multiplication). Moreover, we may take any $\alpha \in (0, 1)$ in Definition \ref{def:diophantine-action}. 
\end{corollary}
\begin{proof}
	Denote by $Z(G)$ the centre of $G$ and by $Z_0(G)$ its identity component. Observe firstly that $G_0:= G/Z_0(G)$ is a compact semisimple Lie group: indeed, $\mathfrak{g} = \mathfrak{z} \oplus \mathfrak{g}'$ where $\mathfrak{g}'$ is semisimple and $\mathfrak{z}$ is the centre of $\mathfrak{g}$ (and so the Lie algebra of $Z_0(G)$); thus the Lie algebra of $G/Z_0(G)$ is $\mathfrak{g}'$ which proves the claim. Denote by $p: G \to G_0$ the quotient projection. By Lemma \ref{lemma:diophantine-semisimple}, there exists $\varepsilon > 0$ depending on $G_0$ such that any $\varepsilon$-dense subset of $G_0$ has a Diophantine action on $G_0$ and so any $\varepsilon$-dense subset of $G$ has a Diophantine action on $G_0$ (for any $\alpha \in (0, 1)$); here we use that since $W \subset G$ is $\varepsilon$-dense, then $p(W) \subset G_0$ is $\varepsilon$-dense. 
	
	For the main claim it suffices to observe that as $Z_0(G) \leqslant T$ there is a natural quotient map $G_0 \to G/T$. Let $gT \in G/T$ be an arbitrary point; by the previous paragraph and by Lemma \ref{lemma:diophantine-semisimple}, for any $\alpha \in (0, 1)$ there exists a $C > 0$ such that $\mc{W}_n \cdot gZ_0(G)$ is $Cn^{-\alpha}$-dense in $G_0$. This immediately implies that $\mc{W}_n \cdot gT$ is $Cn^{-\alpha}$-dense in $G/T$ and completes the proof.
\end{proof}

\section{Analysis on principal bundles}

We let $M$ be a closed manifold, $P \to M$ be a $G$-principal bundle. We use indistinctly $R_g w$ or $w \cdot g$ to denote the right-action of $g \in G$ on a point $w \in P$. The aim of this section is to describe in a convenient way the space $C^\infty(P)$ of smooth functions over $P$. We assume that the reader is familiar with the terminology of principal bundles (see \cite{Kobayashi-Nomizu-63,Kobayashi-Nomizu-69}).

\subsection{Fourier transform} The purpose of this paragraph is to introduce the Fourier transform of functions on $C^\infty(P)$. Fourier transforms are convenient ways of representing smooth functions on $C^\infty(P)$ via representation theory. Using the associated vector bundles of \S\ref{sssection:associated-vector-bundles}, we define the (fiberwise) Fourier transform on $P$ as follows:

\begin{definition}[Fourier transform]
\label{definition:fourier-transform}
The \emph{Fourier transform} on $P$ is the map
\[
\mc{F} : C^\infty(P) \to \bigoplus_{\lambda \in \widehat{G}} C^\infty(M,\mathrm{Hom}(V^\lambda,E^\lambda))
\]
such that for all $f \in C^\infty(P)$, $\lambda \in \widehat{G}$, and $x \in M$,
\begin{equation}
\label{equation:fourier}
\mc{F}f(\lambda,x) := \int_{G} f(x,w \cdot g) w \cdot \lambda(g) ~\dd g \in \mathrm{Hom}(V^\lambda, E^\lambda_x),
\end{equation}
where $\dd g$ is the Haar measure on $G$ and $w \in P_x$ is arbitrary.
\end{definition}

It can be checked that \eqref{equation:fourier} is independent of the choice of point $w \in P_x$; indeed, for any other $w' \in P_x$, there exists a unique $h \in G$ such that $w' = w \cdot h$ and thus using \eqref{equation:id-inv}:
\[
\int_{G} f(x,w' \cdot g) w' \cdot \lambda(g) ~\dd g = \int_{G} f(x,w \cdot (hg)) w \cdot \lambda(hg) ~\dd g = \int_{G} f(x,w \cdot g) w \cdot \lambda(g) ~\dd g,
\]
where in the first equality we used \eqref{equation:inv} and in the second one the invariance of the measure $\dd g$ under the group action. Note that, after the identification of $P_x$ with $G$ and of $E^\lambda_x$ with $V_\lambda$ (equivalently, taking $M$ to be a point), formula \eqref{equation:fourier} reduces to the Fourier transform on $G$ up to factor of $(\dim V^\lambda)^{1/2}$, see \cite[Definition 3.37]{Sepanski-07}. Also, it is immediate that $\mc{F}(\lambda, x)$ is smooth in $x$ for each fixed $\lambda$. 

Moreover, observe that $\mathrm{Hom}(V^\lambda,E^\lambda)$ can be identified with $(E^\lambda)^{\dim V^\lambda}$ by picking an arbitrary basis of $V^\lambda$; we will freely use this identification in what follows. We emphasise that somewhat counter-intuitively, this splitting is typically not stable by the group action, i.e. that the image inside $C^\infty(P_x)$ under the inversion map \eqref{equation:inversion} below of $E^\lambda_x$ is not invariant by the right multiplication. However, it is invariant under the parallel transport map (introduced below), see Lemma \ref{lemma:ft-intertwines-flow}. 

\begin{remark}We also emphasize here that an element in $\widehat{G}$ is an equivalence class of irreducible representations (i.e. representations modulo $G$-isomorphisms). As a consequence, when writing \eqref{equation:fourier} and in what follows, it is implicit that a certain representative $\lambda$ of the equivalence class $[\lambda] \in \widehat{G}$ was chosen. We will describe in \S\ref{ssection:borel-weil} a specific \emph{geometric realization} of classes in $\widehat{G}$ (Borel-Weil correspondence). 
\end{remark}

The inversion formula for the Fourier transform is given formally by
\begin{equation}
\label{equation:inversion}
f(x,w) = \sum_{\lambda \in \widehat{G}} \dim(V^\lambda) \Tr(w^{-1}\mc{F}f(\lambda,x)).
\end{equation}
Again, this agrees with the Fourier inversion formula on $G$ when $M$ is a point up to a factor of $(\dim V^\lambda)^{1/2}$, see \cite[Definition 3.37 and Theorem 3.38]{Sepanski-07}. For each fixed $x \in M$ and $\lambda \in \widehat{G}$, the image of $\Tr(w^{-1} S)$ for $S \in \mathrm{Hom}(V^\lambda, E^\lambda_x)$ inside $C^\infty(P_x)$ is a subspace invariant under the right action of $G$ and which splits as a direct sum of $\dim V^\lambda$ copies of the representation $\lambda$. Moreover, this subspace is uniquely determined by the following property: it is spanned by all $G$-invariant subspaces of $L^2(P_x)$ which are isomorphic as (right) representations to $\lambda$. All such subspaces for distinct $\lambda \in \widehat{G}$ are pairwise $L^2$-orthogonal. 

Next, if $f \in C^\infty(P)$ (resp. $L^2(P)$), then \eqref{equation:inversion} converges in $C^\infty$ (resp. $L^2$)-topology, see \cite{Taylor-68} (where the statement is formulated on $G$ in terms of \emph{matrix coefficients}) and \cite[Theorem 3.38]{Sepanski-07}. 
(Here, the Haar measure of the group $G$ is normalized to $1$ so it does not appear in the formula.) Note that for $G = \mathrm{U}(1)$, $\widehat{G} \simeq \Z$ and one recovers the usual Fourier transform. We refer to \S\ref{ssection:u1} where the $\mathrm{U}(1)$ case is further discussed. We point out that the summand $\dim(V^\lambda) \Tr(w^{-1}\mc{F}f(\lambda,\bullet))$ is an eigenfunction associated to the fiberwise Laplacian (Casimir) operator on $P$ (induced by the bi-invariant Riemannian metric on $G$) by Lemma \ref{lemma:laplace-eigenvalue}.



\subsection{Fiberwise holomorphic sections}


\subsubsection{Flag manifold bundle}

\label{sssection:flag-manifold-bundle}

Let $F := P/T$ be the associated homogeneous bundle over $M$, all of whose fibers are biholomorphic to the flag manifold $G/T$. Let $\pi_F : F \to M$ be the projection, $\V_F := \ker \dd \pi_F$ be the vertical bundle and $\V_F^*$ its dual. A crucial observation is that there is a fiberwise complex structure on $F$, that is every fiber of $F$ admits a complex structure and is biholomorphic to $G/T$. We denote by $\V_F^{0,1}$ (resp. $\V_F^{1,0}$) the $(0,1)$ part of the complexification
\[
\V_F \otimes \C =  \V_F^{1,0} \oplus \V_F^{0,1}.
\]

Notice that $\mathfrak{n}^\pm \subset \mathfrak{g}_{\mathbb{C}}$ are $\Ad(T)$-invariant Lie sub-algebras, and their associated vector bundles may be identified with the holomorphic and anti-holomorphic tangent bundle of the fibres of $P/T$
\[
	\V_F^{1,0} = P \times_{\Ad(T)} \mathfrak{n}^+, \quad \V_F^{0,1} \otimes \mathbb{C}= P \times_{\Ad(T)} \mathfrak{n}^-.
\]

Now, consider the representation $\rho_{\mathbf{k}} : G \to \mathrm{Aut}(\mathbf{J}^{\otimes \mathbf{k}})$, given by the $G$-action on $\mathbf{J}^{\otimes \mathbf{k}} = J_1^{\otimes k_1} \otimes \dotsm \otimes J_d^{\otimes k_d}$. We can then form the associated bundle $\mathbf{L}^{\otimes \mathbf{k}} := P \times_{\rho_{\mathbf{k}}} \mathbf{J}^{\otimes \mathbf{k}}$ over $F$. (Equivalently, it can be verified that this is the same as taking the highest weight $\phi(\mathbf{k}) := \sum_j k_j \lambda_j : T \to i\mathbb{R}$ and forming $\mathbf{L}^{\otimes \mathbf{k}} := P \times_{\phi(\mathbf{k})} \C$, where we identified $\phi(\mathbf{k})$ with the corresponding global weight.) As in \eqref{equation:tower}, there is a tower of fiber bundles $\mathbf{L}^{\otimes \mathbf{k}} \to F \to M$. Note that $\mathbf{L}^{\otimes \mathbf{k}}$ is a Hermitian line bundle (it is naturally equipped with a metric) which yields an $L^2$-scalar product on $C^\infty(F,\mathbf{L}^{\otimes \mathbf{k}})$ (this is the bundle version of \eqref{equation:metric-j-gamma} over $M$). As on $G/T$, a section $s \in C^\infty(F,\Lk)$ is equivalent to the data of a function $\overline{s} \in C^\infty(P)$ satisfying the equivariant property \eqref{equation:t-equivariance}.

\subsubsection{Sections on the flag bundle}

Similarly, there is a fiberwise holomorphic structure for the complex line bundle $\mathbf{L}^{\otimes \mathbf{k}} \to F$. More precisely, freezing $x \in M$, the fiber $F_x$ of $F$ over $x$ is (by construction) biholomorphic to $G/T$ and $\mathbf{L}^{\otimes \mathbf{k}}|_{F_x} \to F_x$ is fiberwise biholomorphic to $\mathbf{J}^{\otimes \mathbf{k}} \to G/T$. (This biholomorphism is not canonical but only depends on a choice of basepoint in the fiber $F_x$.)

Calling $\iota_x : \mathbf{L}^{\otimes \mathbf{k}}|_{F_x} \to \mathbf{J}^{\otimes \mathbf{k}}$ this biholomorphism, one can then pullback via $\iota_x$ the $\overline{\partial}$ operator on $\mathbf{J}^{\otimes \mathbf{k}}$ (defined in \eqref{equation:fiberwise-dbar}) in order to define a fiberwise operator $\overline{\partial}_{\mathbf{k}}$ acting on sections of $\mathbf{L}^{\otimes \mathbf{k}}$. (It is immediate to verify that this operator is indeed well-defined.) We then have:
\begin{equation}
\label{equation:del-bar-lambda}
\overline{\partial}_{\mathbf{k}} : C^\infty(F,\mathbf{L}^{\otimes \mathbf{k}}) \to C^\infty(F,\mathbf{L}^{\otimes \mathbf{k}} \otimes {\V_F^*}^{0,1}).
\end{equation}
This leads to the following definition:

\begin{definition}[Fiberwise holomorphic sections] A section $f \in C^\infty(F,\mathbf{L}^{\otimes \mathbf{k}})$ is \emph{fiberwise holomorphic} if it lies in $\ker \overline{\partial}_{\mathbf{k}}$. We shall use the notation
\[
C^\infty_{\mathrm{hol}}(F,\mathbf{L}^{\otimes \mathbf{k}}) := C^\infty(F,\mathbf{L}^{\otimes \mathbf{k}}) \cap \ker \overline{\partial}_{\mathbf{k}}.
\]
\end{definition}

Notice that, if $s \in C^\infty(F,\mathbf{L}^{\otimes \mathbf{k}})$ is identified with a function $\overline{s} \in C^\infty(P)$ satisfying the equivariant property \eqref{equation:t-equivariance}, then the equation $\overline{\partial}_{\mathbf{k}} s = 0$ translates into $Y \overline{s} = 0$ for all $Y \in \mathfrak{n}^+ \subset \mathfrak{g} \simeq \V_P$.

Since $\mathbf{L}^{\otimes \mathbf{k}} \to F$ is Hermitian and admits a (fiberwise) holomorphic structure, there is a unique (fiberwise) partial Chern connection 
\begin{equation}
\label{equation:partial-chern}
D^{\mathrm{Chern}}_{\mathbf{k}} : C^\infty(F,\mathbf{L}^{\otimes \mathbf{k}}) \to C^\infty(F,\mathbf{L}^{\otimes \mathbf{k}} \otimes \V^*_F)
\end{equation}
characterized by the fact that it is metric and its $(0,1)$-part is given by $\overline{\partial}_{\mathbf{k}}$. (In other words, $D^{\mathrm{Chern}}_{\mathbf{k}}$ allows to differentiate sections only in vertical directions, i.e. in directions that are tangent to the fibers of $F \to M$.) Notice that this connection is simply the fiberwise version of \eqref{equation:fiberwise-chern}.


Finally, using the Borel-Weil theorem, that is the realization of $V^\lambda$ as a space of holomorphic sections $H^0(G/T,\mathbf{J}^{\otimes \mathbf{k}})$, we now rewrite the Fourier transform introduced in Definition \ref{definition:fourier-transform} more explicitly as:
\begin{equation}\label{eq:ft-bw-isomorphism}
	\mc{F} : C^\infty(P) \to \bigoplus_{\mathbf{k} \in \widehat{G}} C^\infty_{\mathrm{hol}}(F, \mathbf{L}^{\otimes \mathbf{k}})^{\oplus d_{\mathbf{k}}},
\end{equation}
where $d_{\mathbf{k}}$ is the dimension of the irreducible representation with highest weight $\sum_i k_i \lambda_i$. Given $f \in C^\infty(P)$ we will write $f_{\mathbf{k}, i}$, $i = 1, \dotsc, d_{\mathbf{k}}$, to denote the components of the Fourier transform $\mc{F}f$. Since most of the arguments will be the same for $i = 1, \dotsc, d_{\mathbf{k}}$, we will sometimes suppress the index $i$ and simply write $f_{\mathbf{k}}$.

We conclude this paragraph with the following:

\begin{lemma}
\label{lemma:ft-isometry}
The Fourier transform $\mc{F}$ is an $L^2$ isometry. Namely, for all $f \in L^2(P)$:
\[
\|f\|^2_{L^2(P)} = \sum_{\mathbf{k} \in \widehat{G}} d_{\mathbf{k}}\sum_{i=1}^{d_{\mathbf{k}}} \|f_{\mathbf{k}, i}\|^2_{L^2(F,\mathbf{L}^{\otimes \mathbf{k}})}.
\]
\end{lemma}

\begin{proof}
This follows immediately from the Plancherel theorem for the Fourier transform, \cite[Corollary 3.40]{Sepanski-07}. 
\end{proof}

\subsubsection{$\mathrm{U}(1)$ case} 

\label{ssection:u1}

For $G = \mathrm{U}(1)$, $G=T$ and thus the flag manifold bundle $F$ is equal to the base manifold $M$ itself; the previous constructions can therefore be considerably simplified. In this case, $\widehat{G}=\Z$ and we will use the letter $k \in \Z$ (instead of $\lambda$) to index irreducible representations; note that they correspond to the representations $\rho_k : \mathrm{U}(1) \to \mathrm{U}(1)$ given by $\rho_k(z)=z^k$.

Taking the tautological representation $\rho : \mathrm{U}(1) \to \mathrm{U}(1)$ (identity), one can form the associated Hermitian complex line bundle $L := P \times_\rho \C$. By construction, $P$ is the unit circle bundle of $L$. The associated vector bundles $E^{\lambda}$ of \S\ref{sssection:associated-vector-bundles} then correspond to $L^{\otimes k}$.

Let $V \in C^\infty(P,TP)$ be the generator of the rotation in the circle fibers of $P$, that is $V = \partial_\theta R_\theta|_{\theta=0}$, where $\theta \in \mathrm{U}(1)\simeq \R/2\pi \Z$. Any smooth function $f \in C^\infty(P)$ can be decomposed in Fourier series in each circle of the fiber bundle $P \to M$, that is 
\begin{equation}
\label{equation:circle-fourier}
f = \sum_{k \in \Z} f_k, \qquad V f_k = i k f_k.
\end{equation}
The function $f_k \in C^\infty(P)$ corresponds to the $k$-th Fourier mode of $f$; alternatively, $f_k$ can be seen as a section $f_k \in C^\infty(M,L^{\otimes k})$.

Note that if $U \subset M$ is a contractible open set, then $P|_{U} \simeq U \times \mathrm{U}(1)$. The decomposition \eqref{equation:circle-fourier} then reads for $x \in U$, $\theta \in \mathrm{U}(1) \simeq \R/2\pi\Z$:
\[
f(x,\theta) = \sum_{k \in \Z} a_k(x)e^{ik\theta}, 
\]
that is $f_k(x,\theta) = a_k(x)e^{ik\theta}$. The (local) section $s(x,\theta) := e^{i \theta}$ gives a local trivialization of $L$, and $s^{\otimes k} = e^{i k \theta}$ is a (local) section of $L^{\otimes k}$.

\subsection{Connections on associated bundles}

\label{section:connection-principal-bundles}

We now further equip the $G$-principal bundle $\pi : P \to M$ with a $G$-equivariant connection $\nabla$ (see \cite[Chapter II]{Kobayashi-Nomizu-63} for a background on connections). This is equivalent to the choice of a horizontal subbundle $\HH_P \subset TP$ that is transverse to the vertical bundle $\V_P \subset TP$ and invariant by the right-action of $G$. We explain how the connection $\nabla$ induces functorially a connection on all associated bundles, that is it equips naturally all associated bundles with a horizontal space.

\subsubsection{Connection on associated vector bundles}

We first describe the parallel transport induced on each associated vector bundle $E^\lambda \to M$ constructed in \eqref{equation:e-lambda} for $\lambda \in \widehat{G}$. Given $x,y \in M$, and a path $\gamma : [0,1] \to M$ joining $x$ to $y$, we denote by $\tau_\gamma : P_x \to P_y$ the parallel transport along $\gamma$ with respect to $\nabla$. The parallel transport of vectors in $E^\lambda_x$ along $\gamma$ is then defined by the formula
\begin{equation}
\label{equation:transport}
	\tau_\gamma^{\lambda} : E^\lambda_x \to E^\lambda_{y}, \quad \tau_\gamma^{\lambda}(x, [w, \xi]) := (y, [u, \xi]),
\end{equation}
where $w \in P_x$ is arbitrary, $\xi \in V^\lambda$, $u \in P_{y}$ is defined by $\tau_\gamma(x, w) = (y, u)$ (parallel transport of $w$ along $\gamma$), and $[\bullet, \bullet]$ is the equivalence class defined in \eqref{equation:right-product}. This is well-defined: indeed, if $w' = w \cdot g \in P_x$ for some $g \in G$, since we have $\tau_\gamma(x, w') = (y, u \cdot g)$ (as $\tau_\gamma$ commutes with the right $G$-action), then

\[
	\tau_\gamma^{\lambda}(x, [w \cdot g, \lambda(g^{-1}) \xi]) = (y, [u \cdot g, \lambda(g^{-1}) \xi]) = \tau_\gamma^{\lambda}(x, [w, \xi]).
\]
(More precisely, $\tau_\gamma^\lambda$ is defined initially on $P \times V^{\lambda}$ and by this computation it descends to the quotient $E^{\lambda}$.) 
To give another point of view, we discuss what happens in a local trivialisation $P|_U \simeq U \times G$ for some open contractible set $U \subset M$; then $\tau_\gamma(x,g) = (y, b_\gamma \cdot g)$ for some $b_\gamma \in G$. The same trivialization also allows to write locally $E^\lambda|_U \simeq U \times V^\lambda$ and we then have, by definition
\begin{equation}
\label{equation:phi-lambda}
\tau_\gamma^\lambda(x,\xi) =\tau_\gamma^\lambda(x, [e, \xi]) = (y, [b_\gamma, \xi]) = (y, \lambda(b_\gamma)\xi),
\end{equation}
where we identified $[e, \xi]$ with $\xi$.

The collection of all parallel transports $\tau^\lambda$ defines a horizontal space $\HH^\lambda \subset TE^\lambda$ or, equivalently, a covariant derivative
\[
\nabla^\lambda : C^\infty(M,E^\lambda) \to C^\infty(M,E^\lambda \otimes T^*M).
\]
Given $x,y \in M$, a path $\gamma$ joining $x$ to $y$, and a function $f \in C^\infty(P_y)$, one can define the pullback $\tau_\gamma^* f \in C^\infty(P_x)$ by $\tau_\gamma^*f(x,w) := f(\tau_\gamma(x,w))$.

The following lemma asserts that the Fourier transform intertwines parallel transport on functions of $P$ and on sections of the associated bundles $E^\lambda \to M$: 

\begin{lemma}
\label{lemma:ft-intertwines-flow}
Let $x,y \in M$, $\gamma$ a path joining $x$ to $y$, $f \in C^\infty(P_y)$, and $\lambda \in \widehat{G}$. Then
\[
\mc{F}(\tau_\gamma^* f)(\lambda,x)= (\tau_\gamma^\lambda)^{-1}\left[\mc{F}f(\lambda,y)\right].
\]
\end{lemma}

\begin{proof}

Write $\tau_\gamma(x, w) = (y, u)$ for some $w \in P_x$ and $u \in P_{y}$. Then, we have for any $\xi \in V^\lambda$
\begin{align*}
	\mc{F}(\tau_\gamma^*f)(\lambda, x)\xi &= \int_G f(y, u \cdot g) w(\lambda (g)\xi) \dd g\\
	&= \Big[w, \int_G f(y, u \cdot g) \lambda (g)\xi \dd g \Big]\\
	&= (\tau_\gamma^\lambda)^{-1} \Big[u, \int_G f(y, u \cdot g) \lambda (g)\xi \dd g \Big]\\
	&= (\tau_\gamma^\lambda)^{-1} \int_G f(y, u \cdot g) u(\lambda (g)\xi) \dd g\\
	&= (\tau_\gamma^\lambda)^{-1} (\mc{F}f(\lambda,y) \xi),
\end{align*}
as required. 
%
\end{proof}

\subsubsection{Connection on the flag bundle}

\label{sssection:extension-flag-manifold}

Parallel transport on $P$ induces a parallel transport on the flag bundle $F = P/T$ (more generally, on any homogeneous $G$-space bundle) by
\[
\tau^F_\gamma(x,wT) := (y, u T),
\]
where $w \in P_x$ and $u \in P_{y}$ is defined by $\tau_\gamma(x,w)=(y,u)$; here $\bullet T$ denotes the orbit of $\bullet \in P$ in $P/T$; this is well-defined since $\tau_\gamma$ commutes with the right $G$-action. When the context is clear, we will simply write $\tau_\gamma$ instead of $\tau_\gamma^F$. This defines a horizontal subspace
\[
(\HH_F)_{x,wT} := \{\partial_t \tau^F_{\gamma(t)}(x,wT)|_{t=0} ~|~ \gamma(0)=x\} \subset TF
\]
which is transverse to the vertical bundle $\V_F := \ker d\pi_F$, where $\pi_F : F \to M$ is the footpoint projection.



\subsubsection{Connection on homogeneous line bundles}\label{sssection:homogeneous-line}

Parallel transport on $F$ can be lifted to a fiberwise (i.e. preserving the fibers of $\mathbf{L}^{\otimes \mathbf{k}}$) parallel transport $\mathbf{L}^{\otimes \mathbf{k}} \to F$ along horizontal directions, for all $\mathbf{k} \in \widehat{G}$. Indeed, by construction, $\mathbf{L}^{\otimes \mathbf{k}} = P \times_{\phi(\mathbf{k})} \C$, where $\phi(\mathbf{k}) : T \to \C$ is the weight corresponding to $\mathbf{k}$ (see \eqref{equation:ecriture}, and note that we identify $\phi(\mathbf{k})$ with the corresponding global weight). We can thus set for $x,y \in M$, $\gamma$ a path joining $x$ to $y$, $w \in P_x$ and $\xi \in \C$
\[
	\tau_\gamma^{\mathbf{k}}(x,[w,\xi]) = (y, [u,\xi]),
\]
where $u$ is defined by $\tau_\gamma(x,w)=(y,u)$ and $[w,\xi]$ denotes the equivalence modulo the relation \eqref{equation:relation-0}. (Equivalently, this may be seen as
\[
\tau_\gamma^{\mathbf{k}}(w'T, [w,\xi]) = (\tau_\gamma^F(w'T), [u, \xi]),
\]
where $w'T \in P/T$ is such that $w'T = wT$.) This is well-defined since any other representative is given by $[w \cdot t, \gamma(t^{-1})\xi]=[w,\xi]$ for $t \in T$, $\tau_\gamma(x, w \cdot t) = (y, u \cdot t)$, and thus 
\[
\tau_\gamma^{\mathbf{k}}(x, [w \cdot t, \gamma(t^{-1})\xi]) = (y, [u \cdot t,\gamma(t^{-1})\xi]) = \tau_\gamma^{\mathbf{k}}(x, [w,\xi]).
\]
Let $\pi : \mathbf{L}^{\otimes \mathbf{k}} \to F$ be the footpoint projection. By definition, for all $wT \in F$, $\xi \in \mathbf{L}^{\otimes \mathbf{k}}_{wT}$ and $t \in \R$, we have
\begin{equation}
\label{equation:proj}
\pi \circ \tau_\gamma^{\mathbf{k}}(wT,\xi) = \tau_\gamma^F(wT) = \tau_\gamma^F \circ \pi (wT,\xi).
\end{equation}
In other words, $\tau^{\mathbf{k}}$ is an extension of $\tau^F$ on the line bundle $\mathbf{L}^{\otimes \mathbf{k}} \to F$.

We next show that parallel transport respects the fibrewise holomorphic structures.
\begin{proposition}\label{proposition:parallel-transport-fibrewise-holomorphic}
	Fix $x, y \in M$ and a path $\gamma$ between $x$ and $y$. Then
	\[
		\tau_\gamma^F: F_x \to F_y, \quad \tau^{\mathbf{k}}_\gamma: \mathbf{L}^{\otimes \mathbf{k}}|_{F_x} \to \mathbf{L}^{\otimes \mathbf{k}}|_{F_y}
	\]
	are biholomorphisms, and $\tau^{\mathbf{k}}_\gamma$ is a unitary holomorphic isomorphism of vector bundles covering $\tau_{\gamma}^F$.
\end{proposition}
\begin{proof}
	Denote by $\pi_{x/y}$ the projection from $P_{x/y}$ to $F_{x/y}$. By definition, the parallel transport $\tau_\gamma^P$ on $P$ along $\gamma$ satisfies 
	\begin{equation}\label{eq:commutation-trivial}
		\tau_{\gamma}^F \circ \pi_x = \pi_y \circ \tau_{\gamma}^P. 
	\end{equation}
	Using fundamental vector fields we identify $TP_{x/y} = \mathbb{H}_{x/y} \oplus \mathfrak{g}$. Since $\tau_{\gamma}^P$ commutes with right $G$-action, $d\tau_{\gamma}(p)$ is equal to the identity on $\mathfrak{g}$. Recall that $TF_x \otimes \mathbb{C} = \mathbb{V}_F^{1, 0} \oplus \mathbb{V}_{F}^{0, 1}$, where $\mathbb{V}_{F}^{0, 1} = d\pi_x(\mathfrak{n}^-)$ and similarly for the $(1, 0)$-part. Collecting the two preceding facts and using \eqref{eq:commutation-trivial}, we conclude that $d\tau_{\gamma}^F$ commutes with the projection $\pi^{0, 1}$ to $\mathbb{V}_{F}^{0, 1}$ and so is a biholomorphism (it is clearly a diffeomorphism).
	
	 Denote by $\Pi_{x/y}$ the projection $P_{x/y} \times \mathbb{C} \to \mathbf{L}^{\otimes \mathbf{k}}|_{F_{x/y}}$. There is a commutation similar to above
	 \[
	 	\tau_{\gamma}^{\mathbf{k}} \circ \Pi_x = \Pi_y \circ \tau_{\gamma}^{P \times \mathbb{C}},
	 \]
	 where $\tau^{P \times \mathbb{C}}_\gamma(p, z) = (\tau^P_{\gamma}p, z)$ for $(p, z) \in P_x \times \C$. Moreover, there is a splitting $T(P_{x/y} \times \mathbb{C}) = TP_{x/y} \oplus T\mathbb{C}$. The differential $d\tau_{\gamma}^{\mathbf{k}}$ acts as the identity on $\mathfrak{g} \oplus T\mathbb{C}$, and therefore similarly to above, we have that $\tau^{\mathbf{k}}_{\gamma}$ is a biholomorphism, which is linear in the fibres. Notice also that $\tau^F_{\gamma} \circ \pi_{x, \mathbf{k}} = \pi_{y, \mathbf{k}} \circ \tau^{\mathbf{k}}_\gamma$, where $\pi_{x/y, \mathbf{k}}$ are the projections $\Lk|_{F_{x/y}} \to F_{x/y}$; this shows that the biholomorphism $\tau^{\mathbf{k}}_{\gamma}$ is compatible with the projections. Finally, $\tau_{\gamma}^{\mathbf{k}}$ also preserves the Hermitian structure since $\tau_{\gamma}^{P \times \mathbb{C}}$ does so, which concludes the proof.	 
\end{proof}

The previous discussion allows to define a partial horizontal connection
\begin{equation}
\label{equation:connection-horizontale-lk}
\nabla_{\mathbf{k}} : C^\infty(F,\mathbf{L}^{\otimes \mathbf{k}}) \to C^\infty(F,\mathbf{L}^{\otimes \mathbf{k}} \otimes \HH^*)
\end{equation}
on $\mathbf{L}^{\otimes \mathbf{k}} \to F$ as follows. Given a vector field $X \in C^\infty(M,TM)$, let $\varphi$ be the flow generated on $M$ and $\psi^F$ be the flow generated by its horizontal lift $X^{\HH_F}$ to $F$. Note that $\psi^F_t(x,wT) = (\varphi_t x, \tau_{\gamma(t)}^F(wT))$, where $\gamma(t)$ denotes the flow segment $(\varphi_sx)_{s \in [0,t]}$. We then set for $s \in C^\infty(F,\mathbf{L}^{\otimes \mathbf{k}})$,
\begin{equation}
\label{equation:def-conn}
(\nabla_{\mathbf{k}})_{X^{\HH_F}}s(wT) := \partial_{t}|_{t = 0} (\tau_{\gamma(t)}^{\mathbf{k}})^{-1}(s(\psi_t^F(wT))), \quad wT \in F.
\end{equation}
It is immediate to verify that \eqref{equation:def-conn} is linear in $X^{\HH_F}$.

The partial connection \eqref{equation:connection-horizontale-lk} satisfies the following Leibniz property
\[
	\nabla_{\mathbf{k}} (fs) = df|_{\HH_F} \otimes s + f \nabla_{\mathbf{k}}s, \quad s \in C^\infty(F,\mathbf{L}^{\otimes \mathbf{k}}), \quad f \in C^\infty(F).
\]
(In other words, \eqref{equation:connection-horizontale-lk} allows to differentiate sections of $\mathbf{L}^{\otimes \mathbf{k}} \to F$ in the horizontal directions $\HH_F \subset TF$.) Combining \eqref{equation:connection-horizontale-lk} and the vertical Chern connection defined in \eqref{equation:partial-chern}, one can form the global connection
\begin{equation}
\label{equation:dynamical-lambda}
\overline{\nabla}_{\mathbf{k}} := D^{\mathrm{Chern}}_{\mathbf{k}} + \nabla_{\mathbf{k}} : C^\infty(F,\mathbf{L}^{\otimes \mathbf{k}}) \to C^\infty(F,\mathbf{L}^{\otimes \mathbf{k}} \otimes T^*F).
\end{equation}

For later use, we also make the following useful observation. If $s \in C^\infty(F,\mathbf{L}^{\otimes \mathbf{k}})$ is identified with an equivariant function $\overline{s} \in C^\infty(P)$ satisfying $\overline{s}(wt) = \gamma_{\mathbf{k}}(t^{-1})\overline{s}(w)$ (where $\gamma_{\mathbf{k}}\circ \exp=\exp \circ \phi(\mathbf{k})$, see \eqref{equation:ecriture}) for all $w \in P, t \in T$, then
\begin{equation}
\label{equation:useful2}
\nabla_\mathbf{k} s = d\overline{s}|_{\HH_P}, \qquad D^{\mathrm{Chern}}_\mathbf{k} s = d\overline{s}|_{\mathfrak{m}},
\end{equation}
where $\mathfrak{m} \subset \mathfrak{g} \simeq \V_P$.

Finally, let us conclude this section with a remark about first order differential operators and the Fourier transform. Expanding on the discussion before \eqref{equation:def-conn}, write $X_P$ and $X_F$ for horizontal lifts of $X$ to $P$ and $F$, and $\X_{\mathbf{k}} := (\nabla_{\mathbf{k}})_{X_F} = (\overline{\nabla}_{\mathbf{k}})_{X_F}$. Thanks to \eqref{equation:def-conn} and Lemma \ref{lemma:ft-intertwines-flow} (by taking $\gamma$ to be a flowline of $X$), writing $\mathbf{k}$ instead of $\lambda$, we get
\begin{equation}\label{eq:ft-intertwines-infinitesimal}
	\mc{F}(X_P f)(\mathbf{k}, x) = \X_{\mathbf{k}} \mc{F}f(\mathbf{k}, x), \quad x\in M, f \in C^\infty(P).
\end{equation}

\subsubsection{One-dimensional representations revisited}\label{sssection:abelian-rep} In this section we revisit \S \ref{sssection:abelian-rep-0}, and show that the identification established in Lemma \ref{lemma:1d-rep} behaves well under the associated constructions. In many applications of our theory, one-dimensional representations will be treated separately from other representations.

We assume that the global weight $\gamma: T \to \mathbb{S}^1$ corresponding to $\mathbf{k}$ via \eqref{equation:ecriture}, extends to $\beta: G \to \mathbb{S}^1$. In this section we will denote the footpoint projection $F \to M$ by $\pi$. Then similarly to \S \ref{sssection:abelian-rep-0} we set $\Lk_M := P \times_\beta \mathbb{C}$ to be the complex line bundle over $M$ and by definition we have

\begin{align*}
	\pi^*\Lk_M = \{(pT, [p, z]_\beta) \mid p \in P,\,\, z \in \mathbb{C}\}.
\end{align*}
and there is a commutative diagram
	\begin{equation*}
	\begin{tikzcd}
		\Lk_M \dar & \lar[swap] \pi^*\Lk_M \dar\\
		M & \lar P/T,
	\end{tikzcd}
	\qquad
	\begin{tikzcd}[arrows={|->}]
		[p, z]_\beta \dar & \lar[swap] (pT, [p, z]_\beta) \dar\\
		\pi(pT) & \lar pT.
	\end{tikzcd}
	\end{equation*}
	Note that the definition \eqref{eq:map-Y} of the map $\Upsilon$ and the result of Lemma \ref{lemma:1d-rep} carry over to define a fibrewise holomorphic isomorphisms of fibrewise holomorphic line bundles
	\[
		\Upsilon: P\times_\gamma \mathbb{C} = \Lk \to \pi^*\Lk_M, \quad [p, z]_\gamma \mapsto (pT, [p, z]_\beta).
	\]
	There is a well-defined pullback action 
	\[
		\pi^*s (pT) = \big(pT, s(\pi(p))\big) \in \pi^*\Lk_M(pT), \quad p \in P, s \in C^\infty(M, \Lk_M).
	\]
	Dual to pullback, we have the \emph{pushforward}:
	\[
		\pi_*: C^\infty(F, \pi_F^*\Lk_M) \to C^\infty(M, \Lk_M), \quad \pi_* s (x) := \int_{pT \in F_x} s(pT)\, d\vol_{F_x},
	\]
	where $d\vol_{F_x}$ denotes the volume measure in the fibre $F_x$. The following identities are straightforward to check:
	\begin{equation}\label{eq:L^2-duality}
		\pi_* \pi^*s = \vol(G/T) s, \quad \langle{\pi^*s, u}\rangle_{L^2(F, \pi^*\Lk_M)} = \langle{s, \pi_*u}\rangle_{L^2(M, \Lk_M)},
	\end{equation}
	where $s\in C^\infty(M, \Lk_M)$ and $u \in C^\infty(F, \pi^*\Lk_M)$. The pushforward can also be defined on $1$-forms with values in $\pi^*\Lk_M$ in the following way:
	\begin{align*}
		&\pi_*: C^\infty(F, T^*F \otimes \pi^*\Lk_M) \to C^\infty(M, T^*M \otimes \Lk_M),\quad \pi_* \alpha (x)(v) := 	\int_{F_x} \alpha',
	\end{align*}
	where $x \in M$, $v \in T_x M$, and $\alpha'$ is the $\ell$-form on $F_x$ (where $\ell = \dim G/T$) with values in $\Lk_M(x)$ defined by the formula
	\[
		\alpha'(v_1, \dotsc, v_{\ell}) = d\vol_{F_x} \wedge \alpha (v_1, \dotsc, v_{\ell}, \widetilde{v}), \quad v_1, \dotsc, v_{\ell} \in T(F_x),
	\]
	where $\widetilde{v} \in T(F_x)$ is an arbitrary lift of $v$, i.e. $d\pi(\widetilde{v}) = v$. It is straightforward to show that (see \cite[Propositions 6.14.1 and 6.15]{Bott-Tu-82} where the non-twisted case is treated, but the proof carries over to include the setting of differential forms with values in vector bundles)
	\begin{equation}\label{eq:pushforward-identity}
		\pi_*(s\wedge \pi^*\alpha) =\alpha \wedge \pi_*s, \quad \alpha \in C^\infty(M, T^*M),\,\, s \in C^\infty(F, \pi^*\Lk_M).
	\end{equation}
	
	
	
	We now adopt the notation of \S \ref{sssection:homogeneous-line}, and in particular we fix a connection on $P$, a vector field $X$ on $M$ with horizontal lift to $P$ denoted by $X_P$. Write $\nabla_{\mathbf{k}, M}$ for the associated connection on $\Lk_M$ and $\X_{\mathbf{k}, M} := (\nabla_{\mathbf{k}, M})_X$. Write $\pi^*\nabla_{\mathbf{k}, M}$ for the pullback connection on $\pi^*\Lk_M$ and $\pi^*\X_{\mathbf{k}, M} := (\pi^*\nabla_{\mathbf{k}, M})_{X_F}$.
	
\begin{lemma}\label{lemma:pullback-equivalence}
	The following relations hold:
	\begin{enumerate}[itemsep=5pt]
	 \item We have the equivalence of connections: $\Upsilon^*(\pi^* \nabla_{\mathbf{k}, M}) = \overline{\nabla}_{\mathbf{k}}$.
%
	 \item Also, we have
	 \begin{equation}\label{eq:pushforward-identity-2}
	 	\pi_* (\Upsilon^{-*} \nabla_{\mathbf{k}}) s = \nabla_{\mathbf{k}, M} \pi_* s, \quad s \in C^\infty(F, \pi^*\Lk_M).
	 \end{equation}
	\end{enumerate}
\end{lemma}
\begin{proof}
\emph{Item 1.} We will prove that $\nabla_{\mathbf{k}}' := \Upsilon^*(\pi^*\nabla_{\mathbf{k}, M}|_{\mathbb{H}_F}) = \nabla_{\mathbf{k}}$ and $\Upsilon^*(\pi^*\nabla_{\mathbf{k}, M}|_{\mathbb{V}_F}) = D^{\mathrm{Chern}}_{\mathbf{k}}$. For the former, let $\eta$ be a path in $M$, which lifts horizontally to a path $\eta^{\mathbb{H}_{P}}$ in $P$ between $p$ and $q$; then $\eta^{\mathbb{H}_F} := \eta^{\mathbb{H}_{P}}T$ is the corresponding horizontal lift in $F$. We compute, for any $z \in \mathbb{C}$
\begin{multline*}
	\tau^{\nabla_{\mathbf{k}}'}_{\eta^{\mathbb{H}_F}} [p, z]_{\gamma} = \Upsilon^{-1} \tau^{\pi^*\nabla_{\mathbf{k}, M}}_{\eta^{\mathbb{H}_F}} \Upsilon [p, z]_\gamma\\ 
	=\Upsilon^{-1} \tau^{\pi^*\nabla_{\mathbf{k}, M}}_{\eta^{\mathbb{H}_F}} [p, z]_\beta = \Upsilon^{-1} [q, z]_\beta = [q, z]_\gamma = \tau_{\eta^{\mathbb{H}_F}}^{\mathbf{k}} [p, z]_\gamma,
\end{multline*}
where by $\tau^{\bullet}$ we denote the parallel transport with respect to $\bullet$. This shows $\nabla_{\mathbf{k}} = \nabla_{\mathbf{k}}'$.

For the latter claim, note that $\pi^*\nabla_{\mathbf{k}, M}|_{\mathbb{V}_F} = d|_{\mathbb{V}_F}$ (more precisely, the right hand side acts on $s \in C^\infty(F, \pi^*\Lk_M)$ as $d|_{\mathbb{V}_F}s(pT) = d(s|_{F_{\pi(pT)}})(pT)$). By Lemma \ref{lemma:1d-rep} we know that $\Upsilon$ is a fibrewise biholomorphic unitary isomorphism of line bundles, where $\pi^*\Lk_M$ is equipped with the fibrewise trivial holomorphic structure, i.e. it is an equivalence between $\overline{\partial}$-operators. By uniqueness of Chern connections, the result immediately follows.
\medskip

\emph{Item 2.} By Item 1, we know that $\Upsilon^{-*} \nabla_{\mathbf{k}} = \pi^*\nabla_{\mathbf{k}, M}|_{\mathbb{H}_F}$. By definition \eqref{eq:pushforward-identity}, the pushforward of a $1$-form with values in $\mathbb{V}_F^*$ is zero; it thus suffices to show $\pi_* \circ \pi^*\nabla_{\mathbf{k}, M} = \nabla_{\mathbf{k}, M} \circ \pi_*$. Again using \eqref{eq:pushforward-identity}, it suffices to show $\pi_* \circ (\pi^*\nabla_{\mathbf{k}, M})_{Y^{\HH}} = (\nabla_{\mathbf{k}, M})_Y \circ \pi_*$ for any vector field $Y$ on $M$. Observe that 
\[
	(\nabla_{\mathbf{k}, M})_Y^* = -(\nabla_{\mathbf{k}, M})_Y - \operatorname{div}(Y), \quad 
	(\pi^*\nabla_{\mathbf{k}, M})_{Y^{\HH}}^* = -(\pi^*\nabla_{\mathbf{k}, M})_{Y^{\HH}} - \pi^*\operatorname{div}(Y),
\]
where the divergence is taken with respect to the Riemannian measure on $M$, the adjoint on $F$ with respect to the product of this measure and the quotient measure on $G/T$, and we used that the connections are unitary; we also used that the divergence of $Y^{\HH}$ is equal to $\pi^*\operatorname{div}(Y)$ (see \eqref{eq:pullback-divergence} below). By taking $L^2$ duals and using \eqref{eq:L^2-duality}, $\pi_* \circ (\pi^*\nabla_{\mathbf{k}, M})_{Y^{\HH}} = (\nabla_{\mathbf{k}, M})_Y \circ \pi_*$ is equivalent to $\pi^* \circ (\nabla_{\mathbf{k}, M})_Y = (\pi^*\nabla_{\mathbf{k}, M})_{Y^{\HH}} \circ \pi^*$, which is the definition of being a pullback connection. This completes the proof.
\end{proof}

Using Lemma \ref{lemma:pullback-equivalence} and Lemma \ref{lemma:extension}, when $k_{a + 1} = \dotsb = k_{a + b} = 0$, from now using $\Upsilon$ we will freely identify sections of $\Lk$ with sections of $\pi^*\Lk_M$ and the operators $\X_{\mathbf{k}}$ with $(\pi^*\nabla_{\mathbf{k}, M})_{X_F}$, and connections $\overline{\nabla}_{\mathbf{k}}$ with $\pi^*\nabla_{\mathbf{k}, M}$.

\subsubsection{Fibrewise holomorphic structure on the horizontal bundle}\label{sssection:horizontal-holomorphic} We will consider the complexifications $(\HH_F^*)_{\C} := \HH_F^* \otimes_{\R} \C$ and $(\mathbb{H}_F)_{\C} := \mathbb{H}_F \otimes_{\mathbb{R}} \mathbb{C}$ of the dual horizontal and horizontal bundles, respectively. We equip these bundles with a fibrewise holomorphic structure as follows. When clear from context, we will drop the index $F$ and write simply $\mathbb{H}_{\C}^*$. For an arbitrary $1$-form $\alpha$ on $M$, its horizontal lift $\alpha^{\HH} \in C^\infty(F, \mathbb{H}_{\C}^*)$ to $F$ is defined as
\begin{equation}\label{eq:horizontal-lift-1-form}
	\alpha^{\HH}(w) := \alpha(\pi(w)) \circ d\pi(w), \quad w \in F.
\end{equation}
Given $U\subset M$ an open set and a local orthonormal frame $(\e_1, \dotsc, \e_n)$ of $TU$, let $(\e_1^*, \dotsc, \e_n^*)$ denote the dual orthonormal frame on $T^*U$. We see that $({\e_1^*}^{\HH}, \dotsc, {\e_n^*}^{\HH})$ is a local frame for $\HH_{\C}^*$ over $F|_U$. Moreover, by choosing another orthonormal frame $(\f_1, \dotsc, \f_n)$ over $U$, we see that the transition function between $({\e_1^*}^{\HH}, \dotsc, {\e_n^*}^{\HH})$ and $({\f_1^*}^{\HH}, \dotsc, {\f_n^*}^{\HH})$ is constant in the fibres of $F|_U \to U$, and hence fibrewise holomorphic. This shows that $\mathbb{H}_{\mathbb{C}}^*$ carries a fibrewise holomorphic structure with the fibrewise $\overline{\partial}$-operator defined locally as
\[
	\overline{\partial}^{\HH}(u_1 {\e_1^*}^{\HH} + \dotsb + u_n {\e_n^*}^{\HH}) := \overline{\partial}(u_1)  {\e_1^*}^{\HH} + \dotsb + \overline{\partial}(u_n) {\e_n^*}^{\HH}, \quad (u_i)_{i = 1}^n \subset C^\infty(F|_U).
\]

Observe that the bundle $\HH^*_{\C} \otimes \mathbf{L}^{\otimes \mathbf{k}}$ also carries a natural tensor product holomorphic structure. Then we have:

\begin{proposition}\label{prop:nabla-preserves-holomorphicity}
	The connection $\nabla^{\mathbf{k}}$ preserves fibrewise holomorphicity, i.e.
	\[
		\nabla_{\mathbf{k}} : C^\infty_{\mathrm{hol}}(F,\mathbf{L}^{\otimes \mathbf{k}}) \to C^\infty_{\mathrm{hol}}(F,\mathbf{L}^{\otimes \mathbf{k}} \otimes \HH^*_{\C}).
	\]
\end{proposition}
\begin{proof}	
	It suffices to prove claim locally. Let $s \in C^\infty(M, \mathbf{L}^{\otimes \mathbf{k}})$, and as above let $U \subset M$ be an open set and $(\e_1, \dotsc, \e_n)$ a local orthonormal frame. Then
	\begin{equation}\label{eq:formula-stupid}
		\nabla_{\mathbf{k}}s = \sum_{i=1}^n \iota_{\e_i^{\HH}} \nabla_{\mathbf{k}}s \otimes {\e_i^*}^{\HH}.
	\end{equation}
	Note that ${\e_i^*}^{\HH}$ is holomorphic by definition, so it suffices to show that $(\nabla_{\mathbf{k}})_{X^{\HH_F}} s$ is holomorphic for any vector field $X$ on $M$. Indeed, let $\overline{s}: P \to \C$ be the $T$-invariant and holomorphic lift of $s$ to $P$ (see \eqref{equation:t-equivariance} and \eqref{equation:fiberwise-dbar}). Then, by definition of the connection $\nabla_{\mathbf{k}}$ (see \eqref{equation:useful2})
	\[
		\overline{(\nabla_{\mathbf{k}})_{X^{\HH_F}} s}(p) = X^{\HH_P} \overline{s}(p) = \partial_t|_{t = 0} \overline{s}(\psi_t^Pp), \quad p \in P,
	\]
	where $X^{\HH_P}$ is the horizontal lift of $X$ to $P$ and $(\psi_t^P)_{t \in \mathbb{R}}$ is its flow. Since this flow is the horizontal lift of the flow of $X$ on $M$, it commutes with the right $G$-action on $P$. This implies that $\overline{(\nabla_{\mathbf{k}})_{X^{\HH_F}} s}$ is $T$-equivariant and holomorphic (recall that the fibrewise $\overline{\partial}$-operator is defined using only right multiplication, see \eqref{equation:fiberwise-dbar}). This completes the proof.
\end{proof}

\subsubsection{Fiberwise holomorphic projections}\label{sssection:fibrewise-holomorphic-projection}

For all $\mathbf{k} \in \widehat{G}$, we let
\[
\Pi_{\mathbf{k}} : L^2(F,\mathbf{L}^{\otimes \mathbf{k}}) \to L^2_{\mathrm{hol}}(F,\mathbf{L}^{\otimes \mathbf{k}})
\]
be the $L^2$-orthogonal fiberwise projection onto fiberwise holomorphic sections. Explicitly, at $(x, pT) \in F$ this is given by
\[
	\Pi_{\mathbf{k}}f (x, pT) = \sum_{i = 1}^r \langle{f(x, \bullet), f_i(\bullet)}\rangle_{L^2(F_x, \mathbf{L}^{\otimes \mathbf{k}}|_{F_x})} f_i(pT), \quad f \in C^\infty(F, \mathbf{L}^{\otimes k}), 
\]
where $(f_i)_{i = 1}^r$ is an $L^2$-orthonormal basis of the space of holomorphic sections of $\mathbf{L}^{\otimes k}|_{F_x} \to F_x$. It extends to $L^2$ spaces by continuity and satisfies $\Pi_{\mathbf{k}}^2 = \Pi_{\mathbf{k}}$. Moreover, it is straightforward to check that $L^2_{\mathrm{hol}}(F,\mathbf{L}^{\otimes \mathbf{k}}) \subset L^2(F,\mathbf{L}^{\otimes \mathbf{k}})$ is a closed subspace, and its $L^2$-orthogonal complement consists of sections which are fibrewise $L^2$-orthogonal to the space of holomorphic sections. Therefore, also $\Pi_{\mathbf{k}}^* = \Pi_{\mathbf{k}}$, as well as $\|\Pi_{\mathbf{k}}\|_{L^2 \to L^2} = 1$. 

Recall that $\HH_{\C}^*$ carries a fibrewise holomorphic structure, see \S \ref{sssection:horizontal-holomorphic}. Thus, we may also define the fiberwise orthogonal projection onto fiberwise holomorphic sections
\[
\Pi_{\mathbf{k}} : L^2(F,\HH^*_{\C} \otimes \mathbf{L}^{\otimes \mathbf{k}}) \to L^2_{\mathrm{hol}}(F,\HH^*_{\C} \otimes \mathbf{L}^{\otimes \mathbf{k}}).
\]
To avoid cumbersome notation, we still denote it by $\Pi_{\mathbf{k}}$; it will be clear from the context on sections of which bundle $\Pi_{\mathbf{k}}$ acts.

\begin{lemma}
The following holds: for all $\mathbf{k} \in \widehat{G}$, 
\begin{equation}
\label{equation:commutation-nablak-pik}
[\nabla_{\mathbf{k}},\Pi_{\mathbf{k}}]=0, \qquad [(\nabla_{\mathbf{k}})^*\nabla_{\mathbf{k}},\Pi_{\mathbf{k}}]=0.
\end{equation}
\end{lemma}

\begin{proof}
The second equality in \eqref{equation:commutation-nablak-pik} is a consequence of the first one (by taking adjoints) so it suffices to prove that $[\nabla_{\mathbf{k}},\Pi_{\mathbf{k}}]=0$. Let $U \subset M$ be an open subset, and let $(\e_1, \dotsc,\e_n)$ be a local orthonormal frame of $TU$, $(\e_1^*, \dotsc,\e_n^*)$ its dual frame, and let $({\e_1^*}^{\HH}, \dotsc,{\e_n^*}^{\HH})$ be its horizontal lift to $\HH_{\C}^* \to F$ (see \eqref{eq:horizontal-lift-1-form}). Since $\nabla_{\mathbf{k}}$ is a local operator and $\Pi_{\mathbf{k}}$ acts on each fibre, it suffices to show the claim on $F|_U$.

Since (by definition) ${\e_i^*}^{\HH}$ are holomorphic on $F$, the statement boils down to proving that $[\X_i,\Pi_{\mathbf{k}}] = 0$, where $\X_i = \iota_{\e_i^{\HH}} \nabla_{\mathbf{k}}$ (see \eqref{eq:formula-stupid}). By Proposition \ref{prop:nabla-preserves-holomorphicity}, $\X_i$ preserves fibrewise holomorphicity, and so
\begin{equation}
\label{equation:lala}
\X_i \Pi_{\mathbf{k}} = \Pi_{\mathbf{k}} \X_i \Pi_{\mathbf{k}}.
\end{equation}
Observe that $\X_i^* = -\X_i - \Div(\e_i^{\HH})$, where $\Div(\e_i^{\HH})$ is the divergence of the vector field $\e_i^{\HH}$ (computed with respect to the measure given locally by the product of the Riemannian volume on $M$ and the projection of the Haar measure on $G/T$). It is straightforward to check that 
\begin{equation}\label{eq:pullback-divergence}
	\Div(\e_i^{\HH}) = \pi^*\Div(\e_i),
\end{equation}
where $\Div(\e_i)$ is computed with respect to the Riemannian volume. In particular, it is constant along the fibers of $F$ and therefore commutes with $\Pi_{\mathbf{k}}$. Taking the adjoint in \eqref{equation:lala}, we then find that
\[
	\Pi_{\mathbf{k}} \X_i  = \Pi_{\mathbf{k}} \X_i \Pi_{\mathbf{k}} =\X_i \Pi_{\mathbf{k}}.
\]
This proves the claim.
\end{proof}

Finally, by Proposition \ref{proposition:parallel-transport-fibrewise-holomorphic} we observe that $\Pi_{\mathbf{k}}$ commutes with the parallel transport $\tau^{\mathbf{k}}_\gamma$, where $\gamma$ is any path between $x, y \in M$, i.e. we have
\begin{equation}\label{eq:parallel-transport-fibrewise-holomorphic}
	\Pi_{\mathbf{k}}(y) \tau^{\mathbf{k}}_{\gamma} s =  \tau^{\mathbf{k}}_{\gamma} \Pi_{\mathbf{k}}(x) s, \quad \forall s\in C^\infty(F_x, \mathbf{L}^{\otimes \mathbf{k}}|_{F_x}).
\end{equation}

\subsubsection{Fourier transform commutes with Laplacians} Finally, we show that the Fourier transform commutes with the horizontal Laplacian and behaves well under the action of the vertical Laplacian. To this end, we first extend the action of $\mc{F}$ 

\begin{equation}\label{eq:ft-vectors}
	\mc{F}: C^\infty(P, \mathbb{H}^*) \to \bigoplus_{\mathbf{k} \in \widehat{G}} C^\infty_{\mathrm{hol}}(F, \mathbf{L}^{\otimes \mathbf{k}} \otimes \HH_F^*)^{\oplus d_{\mathbf{k}}},
\end{equation}
as follows. Let $(\e_i)_{i = 1}^n$ be a basis of $T_xM$ and $\e_i^*$ its dual basis; denote by $\e_i^{*\HH_{F/P}}$ the horizontal lifts to $F_x/P_x$ as in \eqref{eq:horizontal-lift-1-form}. Then, for $p \in P_x$, we evaluate at $pT \in F$
\begin{equation*}
	(\mc{F}\alpha)_{\mathbf{k}, i} := \sum_{j = 1}^n \e_j^{*\HH_F} \otimes (\mc{F}\alpha_j)_{\mathbf{k}, i}.
\end{equation*}
where we wrote $\alpha(p) = \sum_{j = 1}^n \e_j^{*\HH_P}(p) \otimes \alpha_j(p)$ for some smooth functions $(\alpha_j)_{j = 1}^{n}$ in $C^\infty(P_x)$, and $\mc{F}\alpha_j$ denotes the Fourier transform on functions. It is straightforward to check that this is well-defined.

Write $d_{\HH} := d|_{\HH^*}$ acting as $C^\infty(P) \to C^\infty(P, \HH^*)$. There is natural inner product on fibres of $\HH_{P/F}^*$ induced from $(M, g)$; this equips the bundles on both sides of \eqref{eq:ft-vectors} with $L^2$ structures. We may then define the \emph{horizontal Laplacian} on $C^\infty(P)$ as $\Delta_{\HH} := (d_{\HH})^*d_{\HH}$, and on $C^\infty(F, \Lk)$ by $\Delta_{\mathbf{k}} := (\nabla_{\mathbf{k}})^* \nabla_{\mathbf{k}}$. 

Recall from \S\ref{sssection:laplace-eigenvalue} that $G$ was equipped with a bi-invariant Riemannian metric $g_{G}$; it gives rise to the Laplacian (Casimir) operator $\Delta_G$. In turn, since every fiber of $P$ can be identified with $G$, this allows to define a vertical Laplacian $\Delta_{\V}$ on $P$. More precisely, given $f \in C^\infty(P)$, and $(x,p) \in P$, define $\tilde{f}(g) := f(x,p\cdot g)$. Then:
\[
\Delta_{\V}f (x,p) := \Delta_{G}\tilde{f}|_{g=\mathbf{1}_G}.
\]

Finally, denote by $T_{\mathbf{d}_k}$ the diagonal operator that acts by mutiplication by $d_{\mathbf{k}}$ on the right hand side of \eqref{eq:ft-bw-isomorphism}. Then we have:

\begin{lemma}\label{lemma:ft-horizontal-laplacian}
	The Fourier transform satisfies $\mc{F}^*T_{d_\mathbf{k}} \mc{F} \equiv \id$ on both $C^\infty(P)$ and $C^\infty(P, \HH^*)$. Moreover, we have for any $f\in C^\infty(P)$ and any $\mathbf{k} \in \widehat{G}$
	\begin{enumerate}[label=\emph{(\roman*)}]
		\item $(d_{\HH} f)_{\mathbf{k}} = \nabla_{\mathbf{k}} f_{\mathbf{k}}$.	
		\item $(\Delta_{\HH}f)_{\mathbf{k}} = \Delta_{\mathbf{k}} f_{\mathbf{k}}$.
		\item $(\Delta_{\V}f)_{\mathbf{k}} = c(\mathbf{k})f_{\mathbf{k}}$,
	\end{enumerate}
	where $c(\mathbf{k})$ was defined in \eqref{equation:definition-ck}.
\end{lemma}
\begin{proof}
	The first claim on $C^\infty(P)$ follows directly from Lemma \ref{lemma:ft-isometry}. The claim on $C^\infty(P, \HH^*)$ is then a consequence of this, as well as of the fact that the projection identifies $\HH_P^*$ and $\HH_F^*$ isometrically.
	\medskip
	
	\emph{Item 1.} Let $(\e_i)_{i = 1}^n$ be an orthonormal frame over an open set $U \subset M$ containing $x \in M$. Then we evaluate at $pT$, $p \in P_x$:
	\[
		(d_{\HH}f)_{\mathbf{k}} = \sum_{i = 1}^n \e_i^{*\HH_F} \otimes (\e_i f)_{\mathbf{k}} = \sum_{i = 1}^n \e_i^{*\HH_F} \otimes (\nabla_{\mathbf{k}})_{\e_i^{\HH_F}} f_{\mathbf{k}} = \nabla_{\mathbf{k}} f_{\mathbf{k}},
	\]
	where we used \eqref{eq:ft-intertwines-infinitesimal} in the second equality. This completes the proof.
	\medskip
	
	\emph{Item 2.} Using Item 1, it suffices to prove that $(d_{\HH}^*\alpha)_{\mathbf{k}} = (\nabla_{\mathbf{k}})^*\alpha_{\mathbf{k}}$ for any $\alpha \in C^\infty(P, \HH^*)$, i.e. the Fourier transform commutes with the codifferentials. Taking $L^2$-adjoints, this is equivalent to $d_{\HH} \mc{F}^* = \mc{F}^* \nabla_{\mathbf{k}}$, which follows from the formula $\mc{F}^* = \mc{F}^{-1} T_{d_{\mathbf{k}}^{-1}}$ and by using Item 1. \medskip
	
	\emph{Item 3.} This is a mere rewriting of Lemma \ref{lemma:laplace-eigenvalue}.
\end{proof}

\subsection{Curvature} Let $\mathrm{Ad}(P) := P \times_{\mathrm{Ad}} \mathfrak{g}$ be the adjoint bundle associated over $M$ to the adjoint representation $\mathrm{Ad} : G \to \mathrm{End}(\mathfrak{g})$. Denote by $F_{\mathrm{Ad}(P)} \in C^\infty(M,\Lambda^2 T^*M \otimes \mathrm{Ad}(P))$ the curvature of the connection $\nabla$ on $P$ and recall that
\[
(F_{\mathrm{Ad}(P)})_x(X,Y) = d\Theta(w)(X^{\HH_P}, Y^{\HH_P}), \qquad x \in M, X,Y \in T_xM,
\]
where $w \in P_x$ is arbitrary, and $X^{\HH_P}$, $Y^{\HH_P}$ denote horizontal lifts of $X$, $Y$, respectively, and $\Theta$ denotes the connection $1$-form $\Theta$ on $P$.

\subsubsection{Curvature on the flag bundle}\label{sssection:curvature}

We now investigate the properties of the curvature of the connection $\overline{\nabla}_{\mathbf{k}}$. Denote by
\[
\mathbf{F}_{\overline{\nabla}} := (F_{\overline{\nabla}_1}, \dotsc , F_{\overline{\nabla}_d}) \in C^\infty(F, \Lambda^2 T^*F)^{\oplus d}
\]
the vector of curvatures of the respective line bundles $L_1, \dotsc, L_d \to F$. Observe that, by construction, the curvature $F_{\overline{\nabla}_{\mathbf{k}}}$ of $\overline{\nabla}_{\mathbf{k}}$ satisfies
\[
F_{\overline{\nabla}_{\mathbf{k}}} = \mathbf{k}\cdot\mathbf{F}_{\overline{\nabla}} = \sum_{i=1}^d k_i F_{\overline{\nabla}_i}.
\]
Notice that the curvature $\mathbf{F}_{\overline{\nabla}}$ admits a vertical part
\begin{equation}
\label{equation:curvature-vertical}
\mathbf{F}_{\nabla^{\V_F}} \in C^\infty(F, \Lambda^2 \V_F^*)^{\oplus d},
\end{equation}
obtained by restriction of $\mathbf{F}_{\overline{\nabla}} $ to $\V_F$. Equivalently, this is the vector of curvatures of the fiberwise Chern connections $\nabla^{\V_F}_j$ on the fibers of $F$. 

We next show that the curvature $\mathbf{F}_{\overline{\nabla}} $ does not have mixed components, i.e. it vanishes on $\V_F \times \HH_F$. 

\begin{lemma}\label{lemma:vanishing-curvature}
The following holds:
\begin{equation}
\label{equation:vanishing-curvature}
\mathbf{F}_{\overline{\nabla}} (\V_F,\HH_F) = 0.
\end{equation}
\end{lemma}

\begin{proof}
The bundle $\pi : P \to F$ is a principal $T$-bundle. Identifying $\V_P \simeq \mathfrak{g}$, the connection on $P$ yields a splitting
\[
TP = \mathfrak{t} \oplus \mathfrak{m} \oplus \HH_P,
\]
where $\mathfrak{m}$ is the real part of $\mathfrak{n}^+ \oplus \mathfrak{n}^-$ (see \S\ref{sssection:holomorphic-line-bundle}).  Observe that $d\pi(\mathfrak{m}) = \V_F, d\pi(\HH_P) = \HH_F$.

The space $\HH' := \mathfrak{m} \oplus \HH_P$ is $T$-invariant; it is therefore the horizontal space of a $T$-connection on $P \to P/T$. The connection $1$-form $\omega \in C^\infty(P,T^*P \otimes \mathfrak{t})$ is defined by $\omega(H) = 0$ for $H \in \HH'$, and $\omega(H)=H$ for $H \in \mathfrak{t}$.

Let $X \in \HH_P, Y \in \mathfrak{m}$; we claim that $d\omega(X,Y) = 0$. Indeed, we can always assume that $X$ is the horizontal lift of a vector on $M$ and $Y \in \mathfrak{m}$ is constant. Then:
\[
d\omega(X,Y) = X \omega(Y) - Y \omega(X) - \omega([X,Y]) = 0,
\]
as the first two terms are $0$ (because $X,Y \in \HH')$ and $[X,Y]=0$ (see \cite[Chapter II, Proposition 1.2]{Kobayashi-Nomizu-63}). This implies that the curvature of the connection on $P \to P/T$ verifies a similar vanishing property as \eqref{equation:vanishing-curvature}.

In turn, since the connection $\overline{\nabla}_j$ on $L_j \to F$ is associated to the connection on $P \to F$ (as $L_j = P \times_{\gamma} \C$, where $\gamma : T \to \C$ is the corresponding weight), it implies that \eqref{equation:vanishing-curvature} is verified.
\end{proof}

We now relate the connection $\mathbf{F}_{\overline{\nabla}}$ to the curvature $F_{\mathrm{Ad}(P)}$ of the connection on $P$. Recall the following notation: $\alpha_\mathbf{k} := \phi(\mathbf{k})$ is the highest weight corresponding to the representation $H^{0}(G/T,\mathbf{J}^{\otimes \mathbf{k}})$ (see \eqref{equation:ecriture}), $\Pi_{\mathfrak{t}} : \mathfrak{g} \to \mathfrak{t}$ is the orthogonal projection. The following holds:

\begin{lemma}
Let $wT \in F$, $X^{\HH_F}, Y^{\HH_F} \in \HH_F(wT)$. Then:
\begin{equation}
\label{equation:courbure-horizontale}
F_{\overline{\nabla}_{\mathbf{k}}}(X^{\HH_F},Y^{\HH_F}) = \alpha_{\mathbf{k}}\left(\Pi_{\mathfrak{t}}(d\Theta(w)(X^{\HH_P}, Y^{\HH_P}))\right),
\end{equation}
where $X^{\HH_P}, Y^{\HH_P} \in \HH_P(w)$ are the horizontal lifts of $X^{\HH_F},Y^{\HH_F}$.
\end{lemma}
\begin{proof}
	As in the proof of Lemma \ref{lemma:vanishing-curvature}, the connection $\overline{\nabla}_{\mathbf{k}}$ is associated to the connection on $P \to P/T$ with horizontal space $\mathbb{H} \oplus \mathfrak{m}$. Then, as in the proof of Proposition \ref{prop:line-bundle-topology}, the formula is an immediate consequence of the definition of curvature given in \cite[Chapter III, Section 5]{Kobayashi-Nomizu-63}.
\end{proof}

Finally, note that the adjoint representation $\Ad(G)$ of $G$ on the centre $\mathfrak{z}$ is trivial. Therefore, the associated bundle $P \times_{\Ad(G)} \mathfrak{z}$ can be identified with the trivial vector bundle $M \times \mathfrak{z}$ via the map 
\[
	M \times \mathfrak{z} \ni (x, v) \mapsto [p, v] \in P \times_{\Ad(G)} \mathfrak{z},  
\] 
where $p \in P_x$ is arbitrary. By definition, the curvature of $P$ is a $2$-form on $M$ with values in $\Ad(P)$; since the splitting $\mathfrak{z} \oplus [\mathfrak{g}, \mathfrak{g}]$ is $\Ad(G)$ invariant, projecting we get the component $F_{\mathfrak{z}}$ with values in $\mathfrak{z}$ defined by
\begin{equation}\label{eq:curvature-abelian-def}
	F_{\mathfrak{z}}(x)(X, Y) := d\Theta_{\mathfrak{z}}(p)(X^{\HH}, Y^{\HH}), \quad x \in M, X,Y \in T_xM,
\end{equation}
where $p \in P_x$ is arbitrary, and $X^{\HH}$, $Y^{\HH}$ denote horizontal lifts of $X$, $Y$, respectively, and $\Theta_{\mathfrak{z}}$ denotes the $\mathfrak{z}$ component of the connection $1$-form $\Theta$ on $P$. 

For $i = 1, \dotsc, a$, we note by Lemma \ref{lemma:extension} that there are global weights $\gamma_i: T \to \mathbb{S}^1$ and their extensions $\beta_i: T \to \mathbb{S}^1$ such that $\lambda_i = d\gamma_i(e)$. There are complex line bundles $L_{i, M} := P \times_{\beta_i} \mathbb{C}$ over $M$ equipped with associated connections $\nabla_{i, M}$ whose (purely imaginary) curvature $2$-form we denote by $F_{\nabla_i, M}$.

\begin{lemma}\label{lemma:abelian-curvature}
	For $i = 1, \dotsc, a$, we have $F_{\overline{\nabla}_i} = \pi^* F_{\nabla_i, M}$ and $F_{\nabla_i, M} = \lambda_i(F_{\mathfrak{z}})$.
\end{lemma}
\begin{proof}
	 The first claim follows from Lemma \ref{lemma:pullback-equivalence} which implies that $(\pi^*L_{i, M}, \pi^*\nabla_{i, M})$ is isomorphic to $(L_i, \nabla_i)$. Therefore, in particular, $F_{\overline{\nabla}_i} = \pi^* F_{\nabla_i, M}$. For the other claim, observe that by definition, for any $x \in M$, $X, Y \in T_xM$,
	 \begin{equation}
	 \label{equation:ehoui}
	 	F_{\nabla_i, M}(x)(X, Y) = d\beta_i(e) d\Theta(x)(X^{\HH}, Y^{\HH}) = \lambda_i (F_{\mathfrak{z}}(x)(X, Y)),
	 \end{equation}
	 where in the last equality we used that $d\beta_i(e)$ is equal to $\lambda_i$ on $\mathfrak{z}$ and to zero on $[\mathfrak{g}, \mathfrak{g}]$ (since $\beta_i$ is trivial on $[G, G]$), and in the second equality we used \eqref{eq:curvature-abelian-def}.
\end{proof}

\subsubsection{Non-degenerate curvature} \label{sssection:non-degenerate-curvature}
Given $x \in M$, the curvature $\nabla$ on $P$ is \emph{non-degenerate} at $x$ if the following holds:
\begin{equation}
\label{equation:non-degenerate-curvature}
\mathrm{Span}\left(F_{\mathrm{Ad}(P)}(x)(X,Y) \mid X,Y \in T_xM\right) = \mathrm{Ad}(P_x).
\end{equation}
The curvature is \emph{globally non-degenerate} if it is non-degenerate at every $x \in M$. By the Ambrose-Singer theorem, this is equivalent to the local holonomy group being equal to $G$ at every point $x \in M$.

Since the splitting $\mathfrak{g} = \mathfrak{z} \oplus [\mathfrak{g},\mathfrak{g}]$ is $\Ad$-invariant and the adjoint representation of $G$ on $\mathfrak{z} \cong \mathbb{R}^a$ is trivial, $\Ad(P) \cong \mathfrak{z} \oplus E$, where $\mathfrak{z} \cong M \times \R^a$ is trivial over $M$ and $E$ is associated to $P$ by restricting the $\Ad$-action to $[\mathfrak{g}, \mathfrak{g}]$. We shall denote by $\Pi_E : \mathrm{Ad}(P) \to E$ the fiberwise orthogonal projection onto $E$.
For $\mathbf{l} = (\ell_1, \dotsc,\ell_d) \in \R^a \times \R^b_+$, we set
\[
\mathbf{l}\cdot \mathbf{F}_{\overline{\nabla}} = \sum_{i=1}^d \ell_i F_{\overline{\nabla_i}}.
\]
Finally, recall by \eqref{equation:boundary-infinity} the identification $\partial_\infty \mathfrak{a}_+ \simeq \mathbb{S}^{d-1} \cap (\R^a \times \R^b_+)$.

The following holds:

\begin{lemma}
\label{lemma:enfin}
Fix $x \in M$. The following assertions are equivalent:
\begin{enumerate}[label=\emph{(\roman*)}]
\item The curvature is non-degenerate at $x$.
\item For all $\mathbf{l} \in \partial_\infty \mathfrak{a}_+$, for all $wT \in F_x$,  $\mathbf{l} \cdot \mathbf{F}_{\overline{\nabla}}(\bullet,\bullet)|_{\HH_F(x,wT) \times \HH_F(x,wT)} \neq 0$.
%
\end{enumerate}
\end{lemma}

\begin{proof}
\emph{Step 1: (i) $\implies$ (ii).} If the curvature is non-degenerate at $x \in M$, then for all $w \in P_x$:
\[
R(w) := \mathrm{Span}\left(d\Theta(w)(X^{\HH_P}, Y^{\HH_P}) \mid X, Y \in T_xM\right) = \mathfrak{g}.
\]
Define the weight $\alpha_\mathbf{l} := \phi(\mathbf{l}) = \sum_{i=1}^d \ell_i \lambda_i \neq 0$, seen as an element in $(i\mathfrak{g})^*$ (by extending it by $0$ on the orthogonal complement of $\mathfrak{t}$). By \eqref{equation:courbure-horizontale}, we obtain that for all $wT \in F_x$, there exist $X,Y \in T_xM$ such that
\[
\mathbf{l} \cdot \mathbf{F}_{\overline{\nabla}}(X^{\HH_F}(wT),Y^{\HH_F}(wT)) = \alpha_\mathbf{l}\left(\Pi_{\mathfrak{t}}d\Theta(w)(X^{\HH_P}, Y^{\HH_P})\right) \neq 0.
\]
This proves (ii). \\

\emph{Step 2: (ii) $\implies$ (i).} Conversely, assume that the curvature is degenerate at $x$, that is $R(w) \neq \mathfrak{g}$. Notice that $P_x \ni w \mapsto R(w) \subset \mathfrak{g}$ is a family of $G$-invariant subspaces, namely $R(wg) = \rho_{\mathrm{Ad}(G)}(g^{-1})R(w)$ for all $w \in P_x, g \in G$. We can further assume that $R(w)$ is embedded into a $G$-invariant family $H(w)$ of codimension one hyperplanes, and let $w \mapsto \xi(w) \in (i\mathfrak{g})^*$ be the corresponding family of $1$-forms of unit norm such that $\xi(w)(H(w)) = 0$. Indeed, take an arbitrary $w_0 \in P_x$ and $\xi \in (i\mathfrak{g})^*$ such that $\xi(R(w_0)) = 0$, and define for any $g \in G$, $\xi(w_0 g) := \rho_{\mathrm{Ad}(G)}(g^{-1})\xi := \xi \circ \rho_{\mathrm{Ad}(G)}(g)$; then $H(w) := \ker \xi(w)$. 
By \cite[Theorems 5.9, 6.43 (c)]{Sepanski-07}, the co-adjoint orbits intersect the Weyl chamber $\mathfrak{a}_+$ (in a unique point). Hence for arbitrary $w_0 \in P_x$, there exists (a unique) $g_0 \in G$ such that $\alpha := \xi(w_0g_0) = \rho_{\mathrm{Ad}(G)}(g_0^{-1})\xi(w_0) \in \mathfrak{a}_+$. Notice that $\xi(w_0)$ is a unit covector and so is $\alpha$; thus $\alpha \in \mathfrak{a}_+ \cap \mathbb{S}^{d-1} \simeq \partial_\infty \mathfrak{a}_+$, that is $\alpha = \alpha_{\mathbf{l}}$ for some $\mathbf{l}$ (seen as a $1$-form on $\mathfrak{g}$ by extending it by $0$ on $\mathfrak{t}^\perp$). But then
\[
\alpha_{\mathbf{l}}(\Pi_{\mathfrak{t}}R(w_0g_0)) = \alpha_{\mathbf{l}}(R(w_0g_0)) = 0 = \mathbf{l} \cdot \mathbf{F}_{\overline{\nabla}}(X^{\HH_F}(w_0g_0T),Y^{\HH_F}(w_0g_0T)),
\]
for all $X^{\HH_F},Y^{\HH_F}$ horizontal vector fields at $w_0g_0T \in F_x$. This contradicts our assumption and completes the proof.
\end{proof}

For $\mathbf{l} \in \partial_\infty \mathfrak{a}_+ \simeq \mathbb{S}^{d-1} \cap (\R^a \times \R^b_+)$, define the following number:
\[
F_{\mathrm{min}}(\mathbf{l}) := \min_{(x,w) \in F} \max_{\substack{X,Y \in T_xM \\ |X|=|Y|=1}} -i \times \mathbf{l} \cdot\mathbf{F}_{\overline{\nabla}}(w)(X^{\HH_F},Y^{\HH_F}) \geq 0,
\]
where $X^{\HH_F},Y^{\HH_F}$ are the horizontal lifts of $X,Y$ respectively to $F$. Finally, define
\[
F_{\mathrm{min}} := \min_{\mathbf{l} \in \partial_\infty \mathfrak{a}_+} F_{\mathrm{min}}(\mathbf{l}) \geq 0.
\]
Notice that $\partial_\infty \mathfrak{a}_+$ is compact so this is indeed a minimum. Finally, by Lemma \ref{lemma:enfin} the curvature $F_{\mathrm{Ad}(P)}$ is globally non-degenerate in the sense of \eqref{equation:non-degenerate-curvature} if and only if $F_{\mathrm{min}} > 0$.

 \subsection{Bergman kernel expansion}

The aim of this preliminary section is to describe Toeplitz operators on $G/T$. For simplicity, and \textbf{only in this paragraph}, we introduce $X := G/T$. Let
\[
m := \dim_{\C}(X) = (\dim(G) - \mathrm{rk}(G))/2.
\]
Recall that $d_\mathbf{k} := \dim H^0(X,\mathbf{L}^{\otimes \mathbf{k}})$. For a line bundle $L$, denote by $\overline{L}$ its dual line bundle, and given vector bundles $E, F \to X$, write $E \boxtimes F := \pi_r^*E \otimes \pi_{\ell}^*F \to X \times X$, where $\pi_{\ell}$ and $\pi_r$ denote left and right projections $X \times X \to X$, respectively. Let $\Delta \subset X \times X$ be the diagonal. In what follows, we will use the sub-indices $\ell$ and $r$ to denote the action in the left and right variables, respectively.

The Schwartz kernel of $\Pi_{\mathbf{k}}$ is a smooth section of $\mathbf{L}^{\otimes \mathbf{k}} \boxtimes \overline{\Lk} \to X \times X$. Since $\overline{\partial}_{\mathbf{k}} \Pi_{\mathbf{k}} = \Pi_{\mathbf{k}} (\overline{\partial}_{\mathbf{k}})^* \equiv 0$, it is holomorphic in the left variable and antiholomorphic in the right variable.

A \emph{multisection} $\mathbf{s} : \Z \to C^\infty(X,\mathbf{L})$ is the data of $\mathbf{s} := (s_1, \dotsc,s_d)$ such that $\mathbf{s}^{\otimes \mathbf{k}} := s_1^{\otimes k_1} \otimes \dotsm \otimes s_d^{\otimes k_d}$. We will prove the following:

 \begin{theorem}
 \label{theorem:bergman}
There exists a unique multisection
\[
\mathbf{E} : \Z^a \times \Z^b_{\geq 0} \to C^\infty(X \times X, \mathbf{L} \boxtimes \overline{\mathbf{L}})
\]
such that
\begin{equation}
\label{equation:partie-simple}
\Pi_{\mathbf{k}}(x_\ell,x_r) = \dfrac{d_\mathbf{k}}{\vol(X)} \mathbf{E}^{\otimes \mathbf{k}}, \qquad \mathbf{E}^{\otimes \mathbf{k}}(x,x) = \mathbbm{1}_{\mathbf{L}^{\otimes \mathbf{k}}}
\end{equation}
and $\mathbf{E}$ is holomorphic in the left variable, antiholomorphic in the right variable. Moreover, if $\mathbf{E} = (E_1, \dotsc, E_d)$, the multifunction 
\[
\boldsymbol{\varphi} := (\varphi_1, \dotsc, \varphi_d) = (-\log(|E_1|^2), \dotsc,-\log(|E_d|^2)) \in C^\infty(X \times X)
\]
satisfies
 \begin{equation}
 \label{equation:phase-important}
 \mathbf{\boldsymbol{\varphi}}|_{\Delta} = 0, \quad d\boldsymbol{\varphi}|_{\Delta} = 0, \quad \partial_{\ell} \overline{\partial}_\ell \boldsymbol{\varphi}|_{\Delta}((\xi_1,0),(\xi_2,0)) = - i \mathbf{F}_{D^{\mathrm{Chern}}}(\xi_1,J\xi_2).
 \end{equation}
 \end{theorem}
 
 In the last line, $J$ denotes the complex structure and $\mathbf{F}_{D^{\mathrm{Chern}}}$ denotes the (multi)curvature of the Chern connection on $\mathbf{L}$. The proof is rather standard so we only sketch it.
 
 \begin{proof}
For all $g \in G$, $g^* \Pi_{\mathbf{k}} = \Pi_{\mathbf{k}} g^*$, where the $G$-action is given by by \eqref{eq:G-action-holomorphic}. This yields at the level of Schwartz kernels the following equivariance property:
\begin{equation}
\label{equation:equivariance-schwartz}
\Pi_{\mathbf{k}}(gx_\ell,gx_r) = g_{\mathrm{left}} \cdot \Pi_{\mathbf{k}}(x_\ell,x_r) \cdot g^{-1}_{\mathrm{right}} \in \mathbf{L}^{\otimes \mathbf{k}}_{gx_\ell} \otimes (\mathbf{L}^{\otimes \mathbf{k}})^*_{gx_r},
\end{equation}
where the left (resp. right) action $g_{\mathrm{left}}$ is given by \eqref{eq:G-action-holomorphic} on the left (resp. right) factor. Now, $\Pi_{\mathbf{k}}(x,x) \in \mathbf{L}^{\otimes \mathbf{k}}_{x} \otimes (\mathbf{L}^{\otimes \mathbf{k}})^*_{x}$ is given by $\Pi_{\mathbf{k}}(x,x) = \lambda(x) \mathrm{id}_{\Lk_x}$ for some function $\lambda \in C^\infty(X)$. By \eqref{equation:equivariance-schwartz} and the transitivity of the $G$-action on the flag space $X$, we obtain that $\lambda(x) = \lambda_0 \in \C$ is a constant. Moreover,
\[
\Tr(\Pi_{\mathbf{k}}) = \int_M \Tr(\Pi_{\mathbf{k}}(x,x)) \dd x = \lambda_0 \vol(X) = \dim H^0(X,\mathbf{L}^{\otimes \mathbf{k}}) = d_\mathbf{k}.
\]
This proves that $\Pi_{\mathbf{k}} = \tfrac{d_{\mathbf{k}}}{\vol(X)} \mathbf{E}_{\mathbf{k}}$ for some section $\mathbf{E}_{\mathbf{k}}$ equal to the identity on the diagonal, (anti-)holomorphic in the left (resp. right) variable. It remains to show that $\mathbf{E}_{\mathbf{k}} = \mathbf{E}^{\otimes \mathbf{k}} = E_1^{\otimes k_1} \otimes \dotsm \otimes E_d^{\otimes k_d}$, where $E_i := \mathbf{E}_{(0, \dotsc,0,1,0, \dotsc, 0)} \in C^\infty(X\times X, L_i \boxtimes \overline{L}_i)$. Observe that, by construction, $\mathbf{G}_{\mathbf{k}}:=\mathbf{E}_{\mathbf{k}}-\mathbf{E}^{\otimes \mathbf{k}}$ vanishes on $\Delta$, and is also (anti-)holomorphic in the left (resp.) variable; it can then be shown that $\mathbf{G}_{\mathbf{k}}$ is identically $0$ on $X \times X$ (see the next paragraph for a justification using local coordinates). This proves \eqref{equation:partie-simple}.

We now go for the proof of \eqref{equation:phase-important}. Fix a point $x_0 \in X$ and consider $U \subset X \times X$, a neighborhood of the diagonal near $(x_0,x_0) \in X \times X$. It is a standard fact that this neighborhood can be biholomorphically identified with an open subset of $\C^{2m}$ via $\phi : U \to V \subset \C^{2m}$ such that $\phi(\Delta) \subset \R^{2m}$ and $\phi(x_0,x_0) = 0$. Moreover, a function $f \in C^\infty(U)$ near $(x_0,x_0)$ is holomorphic in the left variable, and antiholomorphic in the right-variable if and only if it can be written in these coordinates as $f(z,\overline{w})\footnote{The notation $f(z,\overline{w})$ is standard in complex analysis to insist on (anti-)holomorphic properties of the function in its variables.} = \sum_{\alpha \beta} a_{\alpha \beta} z^\alpha \overline{w}^\beta$. Observe that, if $f$ vanishes on the diagonal $z=w$, then $f \equiv 0$, which accounts for the claim made in the previous paragraph on the vanishing of $\mathbf{G}_{\mathbf{k}}$.

Fix $i=1, \dotsc,d$ and choose a holomorphic frame $s_i$ for $L_i \to X$ near $x_0$. The section $E_i \in C^\infty(X \times X, L_i \boxtimes \overline{L}_i)$ can be written locally as $E_i(x_\ell,x_r)=e^{\psi_i(x_\ell,x_r)} s_i(x_\ell) \otimes \overline{s}_i(x_r)$ for some function $\psi_i$. Let $\theta_i := - \log |s_i|^2$. Taking $x=y$ and using $E_i(x,x)=\mathbbm{1}$, we obtain $\psi_i(x,x)=\theta_i(x)$. Moreover, $\psi_i$ is (anti-)holomorphic in the left (resp. right) variable  so it is locally determined by its values on the diagonal. More precisely, writing near $x_0$ in the holomorphic coordinates $(z_i)_{1 \leq i \leq m}$, the function $\theta_i(z) = \sum_{\alpha \beta} a_{\alpha \beta} z^\alpha \overline{z}^\beta$ with $a_{\beta \alpha} = \overline{a}_{\alpha \beta}$, we get $\psi_i(z,\overline{w}) = \sum_{\alpha \beta} a_{\alpha \beta} z^\alpha \overline{w}^\beta$.

We then have:
\[
\begin{split}
\varphi_i(x_\ell,x_r) & = \theta_i(x_\ell) + \theta_i(x_r) - (\psi_i(x_\ell,x_r)+\overline{\psi}_i(x_\ell,x_r)) \\
&= \sum_{\alpha, \beta} a_{\alpha \beta}(z^\alpha - w^{\alpha})(\overline{z}^\beta - \overline{w}^\beta)\\
& = \sum_{jk} a_{jk} (z_j-w_j)(\overline{z}_k-\overline{w}_k) + \mc{O}(|z|^2|w|+|w|^2|z|),
\end{split}
\]
where $a_{jk} := \partial^2 \theta_i/\partial z_j \overline{\partial} z_k$. This proves the first two equations of \eqref{equation:phase-important}; it now remains to relate the complex Hessian to the curvature.

We freeze the point $y=x_0$ (identified with $w=0$ in the previous coordinates) and compute the complex Hessian of $z \mapsto \varphi_i(z,0)$ with respect to $z$ at $z=0$. Write $D^{\mathrm{Chern}}_i s_i = \beta_i \otimes s_i$ with $\beta_i$ a $(1,0)$-form; the curvature of $D^{\mathrm{Chern}}_i$ is given by $F_{D^{\mathrm{Chern}}_i} = \overline{\partial} \beta_i$. Differentiating $\theta_i(z)=-\log(|s_i|^2(z))$, we first see that
\[
	\partial \theta_i = \frac{\partial |s_i|^2}{|s_i|^2} =- \frac{\partial s_i \cdot \overline{s}_i}{|s_i|^2} = - \frac{\nabla_i s_i \cdot \overline{s}_i}{|s_i|^2} = - \beta_i, 
\]
and so we obtain, using also that $\psi_i$ is holomorphic in the $z$-variable:
\[
	\partial_{\ell} \overline{\partial}_\ell \varphi_i (0,0) = \partial \overline{\partial} \theta_i (0) = - \overline{\partial}\beta_i = - F_{D^{\mathrm{Chern}}_i}.
\]
Finally, one obtains the claim by using $F_{D^{\mathrm{Chern}}_i}(\xi_1,J\xi_2) = -iF_{D^{\mathrm{Chern}}_i}(\xi_1,\xi_2)$ since the curvature is of type $(1,1)$.
 \end{proof}

\chapter{Semiclassical analysis on principal bundles}

The purpose of this chapter is to introduce a semiclassical formalism to study (pseudo)differential operators acting on principal bundles over closed manifolds, and commuting with the right action of the group.

\label{chapter:analytic}

\minitoc

\newpage



We introduce the Borel-Weil calculus in three steps. We first introduce a notion of uniform semiclassical operators. This is then used to introduce the twisted semiclassical quantization on line bundles. Finally, we introduce the calculus on the flag bundle $F := P/T$ of a principal bundle and relate it to a fiberwise Toeplitz quantization.

\section{Uniform semiclassical quantization}

\label{ssection:uniform-scl}


\subsection{Definition}

Semiclassical analysis is the study of pseudodifferential operators, where the $\xi$ variable is rescaled by an asymptotic parameter $h > 0$. In what follows, we shall need a slight variant of it, where operators are also allowed to depend uniformly on an auxiliary parameter $\omega \in \Omega$, where $\Omega$ is an arbitrary set. 

\begin{definition} Given an open subset $U \subset \R^n$, define $S^m_{h,\omega}(T^*U)$ the space of symbols of order $m \in \R$ (depending on both $h > 0$ and $\omega \in \Omega$) as the space of smooth functions on $T^*U$ satisfying the symbolic estimates: for all open subsets $V \Subset U$, for all $\alpha,\beta \in \Z_{\geq 0}^{n}$, there exists $C_{V, \alpha, \beta} > 0$ such that for all $h \in (0,1]$, for all $\omega \in \Omega(h) \subset \Omega$,
\begin{equation}
\label{equation:local-symbol}
|\partial^\beta_x \partial^\alpha_\xi a_{h,\omega}(x,\xi)| \leq C_{V,\alpha,\beta} \langle \xi \rangle^{m-|\alpha|}, \qquad \forall (x,\xi) \in T^*V.
\end{equation}
\end{definition}
Note that $a$ may depend on both $h$ and $\omega$ but we require the estimate \eqref{equation:local-symbol} to be uniform in both parameters.

These symbols can be quantized in order to produce the class of pseudodifferential operators $\Psi^m_{h,\omega}(U)$. For $a_{h,\omega}\in S^m_{h,\omega}(T^*U)$ and $f \in C^\infty_{\comp}(U)$, define the operator $\Op^{\R^n}_{h}(a_{h,\omega}) \in \Psi^m_{h,\omega}(U)$ as
\begin{equation}
\label{equation:quantization-rn}
\Op^{\R^n}_h(a_{h,\omega})f(x) := \dfrac{1}{(2\pi h)^n} \int_{\R^n_\xi} \int_{\R^n_y} e^{\tfrac{i}{h}\xi\cdot(x-y)} a_{h,\omega}(x,\xi)f(y)\, \dd y \dd \xi,
\end{equation}
where $\cdot$ denotes Euclidean inner product.

If $M$ is a smooth closed manifold, the class $S^m_{h,\omega}(T^*M)$ is defined on $M$ by requiring \eqref{equation:local-symbol} to hold in any local patch of coordinates. Let $E_1, E_2 \to M$ be two Hermitian vector bundles on $M$. Similarly, one can define the class $S^m_{h,\omega}(T^*M, \End(E_1, E_2))$ as the space of $\End(E_1, E_2)$-valued symbols. This is achieved by taking matrix-valued symbols in local patches of coordinates and local trivialisations. (Note that $E_1, E_2$ are here seen as bundles over $T^*M$ by taking their pullback via the projection $\pi : T^*M \to M$; we refrain from writing $\pi^*E_1, \pi^*E_2$ in order to keep the notation simple.)


Denote by $\Op_{h} : S^m_{h,\omega}(T^*M) \to \Psi^m_{h,\omega}(M)$ an arbitrary quantization (this is defined by patching together local quantizations defined by \eqref{equation:quantization-rn}, see \cite[Proposition E.15]{Dyatlov-Zworski-19}). Let $\Psi^m_{h,\omega}(M)$ be the space of semiclassical pseudodifferential operators obtained by quantizing $S^m_{h,\omega}(T^*M)$, that is
\[
	\Psi_{h,\omega}^m(M) := \{\Op_h(a) + R ~|~ a \in S^m_{h,\omega}(T^*M),\, R \in h^\infty \Psi^{-\infty}_{h,\omega}(M)\},
\]
where $h^\infty \Psi^{-\infty}_{h,\omega}(M)$ denotes the space of smoothing operators $R$, such that each $C^\infty$ seminorm of the Schwartz kernel of $A$ is $\mc{O}(h^\infty)$ uniformly in $h$ and $\omega \in \Omega(h)$. Similarly, one can define the space $\Psi^{m}_{h,\omega}(M, E_1 \to E_2)$ by quantizing symbols in the class $S^m_{h,\omega}(T^*M, \End(E_1, E_2))$. Following the proofs from standard semiclassical analysis, it is straightforward to verify that $\Psi^{\bullet}_{h,\omega}(M)$ is an algebra (see \cite[Proposition E.17]{Dyatlov-Zworski-19}).

\begin{example}
\label{example:simple}
Consider $\Omega := \Z_{\geq 0}$; we will use the letter $k$ in place of $\omega$ to denote the auxiliary parameter. We will consider only the values $(h, k)$ for $0 \leq hk \leq 1$, i.e. $\Omega(h) = \{k \in \Z_{\geq 0} \mid k \leq h^{-1}\}$. Let $\alpha \in C^\infty(M,T^*M)$ be a real-valued $1$-form. The operator $P_{h,k} := hd + i hk \alpha \wedge \in \Psi^1_{h,k}(M, \C \to T^*M)$ is in the calculus (where $\C$ is a short notation for the trivial line bundle). In the usual semiclassical sense, principal symbol now depends on $k$: for $k$ fixed, $\sigma_{P_{h,k}}(x,\xi) = i \xi$ but we can also let $h$ depend on $k$ and for $h(k) := 1/k$, one obtains $\sigma_{P_{h(k), k}}(x,\xi) = i(\xi+\alpha(x))$.

An example of an operator that is \emph{not} in $\Psi^1_{h,k}(M,\C \to T^*M)$ is provided by $P'_{h,k} := hk d$. Indeed, in local coordinates, this is obtained by quantizing the symbol $a_{h,k}(x,\xi) := i k \xi$, which is not in $S_{h,k}(T^*\R^n, \End(\C,T^*\R^n))$ since it is not uniformly bounded in $k$.
\end{example}


\subsection{Principal symbol. Ellipticity}
\label{sssec:ellipticity}


Let $A \in \Psi_{h,\omega}^m(M)$ (from now on, we mostly drop the indices $(h,\omega)$ in the notation for the pseudodifferential operators). The principal symbol of $A$ is only well-defined for a given family $h \mapsto \omega(h) \in \Omega$ as pointed out in Example \ref{example:simple}.

\begin{definition}
Let $h \mapsto \omega(h) \in \Omega$ be an arbitrary function, and $A \in \Psi_{h,\omega}^m(M)$. The principal symbol $\sigma_{A,\omega(h)}$ of $A$ \emph{relative to the family $h \mapsto \omega(h)$} is defined as the (standard) semiclassical principal symbol $a_{h, \omega(h)}$ of $A_{h,\omega(h)} \in \Psi_h^m(M)$:
\begin{equation}
\label{equation:symbol-omegah}
	\sigma_{A,\omega(h)} := [a_{h,\omega(h)}] \in S_{h}^m(T^*M)/hS_{h}^{m-1}(T^*M).
\end{equation}
\end{definition}
In other words, in local coordinates the principal symbol is defined as the full symbol $a_{h,\omega(h)}$ in \eqref{equation:quantization-rn} modulo lower-order terms. Let us illustrate this with the same operator as in Example \ref{example:simple}.

\begin{example}
Consider $P:= hd + i hk \alpha \wedge \in \Psi^1(M,\C \to T^*M)$ with $hk \leq 1$. Let $h \mapsto k(h) \in \Z_{\geq 0}$ be a function and further assume that $hk(h) \to c \in [0,1]$ (note that, up to extraction, this always happens as $0 \leq hk \leq 1$). Then $\sigma_{P,k(h)}(x,\xi) = i\big(\xi + (c + o(1))\alpha(x)\big)$.
\end{example}

The principal symbol $\sigma_A = 0$ is said to vanish if $\sigma_{A,\omega(h)} = 0$ along any family $h \mapsto \omega(h)$. By definition (see \eqref{equation:symbol-omegah}), this means that $a_{h,\omega(h)} \in hS_{h}^{m-1}(T^*M)$ for any family $h \mapsto \omega(h)$.

\begin{lemma}
\label{lemma:kernel-symbol}
Let $A \in \Psi^m_{h,\omega}(M)$. Then $\sigma_A = 0$ if and only if $A \in h\Psi^{m-1}_{h,\omega}(M)$.
\end{lemma}

\begin{proof}
If $A \in h\Psi^{m-1}_{h,\omega}(M) \subset \Psi_{h, \omega}(M)$, it is clear that $\sigma_A = 0$. Conversely, by multiplying the operator by cutoff functions localized in a patch of coordinates $U \subset \mathbb{R}^n$, we can always assume that $A$ takes the form \eqref{equation:quantization-rn} (plus smoothing terms). It then suffices to verify that the full symbol $a_{h,\omega}$ belongs to $h S^{m-1}_{h,\omega}(T^*U)$. If not, then for some open subset $V \subset U$ and multi-indices $\alpha,\beta \in \Z^n_{\geq 0}$, the estimate \eqref{equation:local-symbol} does not hold uniformly in $h> 0$, $\omega \in \Omega(h)$, and $(x,\xi) \in T^*V$. But then, we could find in particular a subsequence $h \mapsto \omega(h)$ such that \eqref{equation:local-symbol} does not hold uniformly along this subsequence for all $(x,\xi) \in T^*V$. However, $\sigma_A = 0$ by assumption and thus $a_{h,\omega(h)} \in h S^{m-1}(T^*U)$ so \eqref{equation:local-symbol} must hold uniformly, which provides a contradiction.
\end{proof}

\begin{remark}
\label{remark:uniform}
When working with uniform semiclassical quantization, certain proofs (such as the previous one) boil down to inverting quantifiers. More precisely, one has a statement of the form: for all families $h \mapsto \omega(h)$, there exists $C_{\omega(h)} > 0$ such that for all $h > 0$, $a_{h,\omega(h)} > C_{\omega(h)}$ (where $a_{h,\omega}$ is, say, a certain function of $h > 0$ and $\omega \in \Omega$), and one wants to deduce that there exists a $C >0$ such that for all $h >0,\omega\in\Omega$, $a_{h,\omega} > C$. The argument by contradiction is straightforward: if it does not hold, then there exists a sequence $h_n > 0, \omega_n > 0$ such that $a_{h_n,\omega_n} < \tfrac{1}{n}$. But this defines (along a subsequence of $h$) a family $h \mapsto \omega(h)$ such that $a_{h,\omega(h)} = o(1)$ and this contradicts the assumption.
\end{remark}

Let $\overline{T^*M} := T^*M \sqcup \partial_\infty T^*M$ be the radial compactification of $T^*M$, see \cite[Section E.1.3]{Dyatlov-Zworski-19} for instance. (The compactification is made in such a way that $a \in S^m_{h}(T^*M)$ if and only if $a \langle{\xi}\rangle^{-m} \in C^\infty(\overline{T^*M})$.) Given $A \in \Psi^m_{h,\omega}(M)$, its (semiclassical) \emph{elliptic set relative to the family $h \mapsto \omega(h) \in \Omega(h)$} is defined as the (open) set of points 
\begin{equation}
\label{equation:elliptic-standard}
\begin{split}
\Ell_{\omega(h)}(A)&  := \{(x,\xi) \in \overline{T^*M} ~|~ \exists C > 0, |\sigma_{A,\omega(h)}(x,\xi)|/\langle\xi\rangle^m > C, \forall h > 0\}.
\end{split}
\end{equation}
For operators mapping sections of $E_1$ and to sections of $E_2$, \eqref{equation:elliptic-standard} should be replaced by: 
\[
\|\sigma_{A,\omega(h)}(x,\xi)\eta\|_{(E_2)_x} > C \langle\xi\rangle^m \|\eta\|_{(E_1)_x}, \qquad \forall \eta \in (E_1)_x,
\]
where $\|\bullet\|_{E_1/E_2}$ denote any norms on fibres of $E_1$ and $E_2$. The \emph{elliptic set} is defined as the intersection of all elliptic sets relative to arbitrary families:
\begin{equation}
\label{equation:ell}
\Ell(A) := \cap_{\{h \mapsto \omega(h)\}} \Ell_{\omega(h)}(A).
\end{equation}
By Remark \ref{remark:uniform}, it should be observed that
\[
\Ell(A) = \{(x,\xi) \in \overline{T^*M} ~|~ \exists C > 0, |a_{h,\omega}(x,\xi)|/\langle\xi\rangle^m > C, \forall h > 0, \forall \omega \in \Omega\},
\]
where $a_{h,\omega}$ denotes the full symbol (in local coordinates). (One has to use the full symbol here as the principal symbol is not well-defined unless we pick a family $h \mapsto \omega(h)$.)

\begin{example}
Once again, we illustrate this notion with $P := hd + ihk \alpha$. Given $h \mapsto k(h)$ such that $hk(h) \leq 1$ and $hk(h) \to c \in [0,1]$, the elliptic set relative to this family and the global elliptic set are given by
\[
\Ell_{\omega(h)}(A) = \overline{T^*M} \setminus \{\xi=-(c + o(1))\alpha\}, \qquad \Ell(A) =\overline{T^*M} \setminus \cup_{c \in [0,1]} \{\xi=-c\alpha\}.
\]
\end{example}

The \emph{wavefront set $\WF_{\omega(h)}(A) \subset \overline{T^*M}$ relative to $h \mapsto \omega(h)$} is the closed subset defined as the complement of the set of points $(x_0,\xi_0) \in \overline{T^*M}$ such that there exists a cutoff function $\chi \in S^0(T^*M)$ with support near $(x_0,\xi_0)$, such that $\chi(x_0,\xi_0)=1$ and 
\begin{equation}
\label{equation:wf-standard}
\Op_h(\chi)A_{h,\omega(h)},\,\,  A_{h,\omega(h)} \Op_h(\chi) \in h^\infty \Psi_{h}^{-\infty}(M).
\end{equation}
Finally, the global wavefront set is defined as the union of all wavefront sets relative to any family
\begin{equation}
\label{equation:wf}
\WF(A) := \cup_{\{h \mapsto \omega(h)\}} \WF_{\omega(h)}(A).
\end{equation}

Elliptic operators can be inverted in this calculus modulo negligible corrections. Along a family $h \mapsto \omega(h)$, this boils down to the usual parametrix for semiclassical parameters, but the parametrix construction can be done uniformly using the global elliptic and wavefront sets introduced in \eqref{equation:ell} and \eqref{equation:wf}.

\begin{proposition}
\label{proposition:ellipticity0}
Let $A \in \Psi^m_{h,\omega}(M)$, $B\in \Psi^{m'}_{h,\omega}(M)$. Assume that $\WF(B) \subset \Ell(A)$. Then, there exists $Q\in \Psi^{m'-m}_{h,\omega}(M)$ such that
\[
B = QA + \mc{O}_{\Psi^{-\infty}_{h,\omega}(M)}(h^\infty) = AQ  + \mc{O}_{\Psi^{-\infty}_{h,\omega}(M)}(h^\infty).
\]
\end{proposition}

We omit the proof as it follows the standard one, see \cite[Theorem 4.29]{Zworski-12} for instance.

\subsection{Technical result}
We end this section with the following technical result:

\begin{lemma}
\label{lemma:appendix}
Let $U \subset \R^n$ be an open subset. Let $f_{h,\omega} \in C^\infty(U)$ be a uniform family of functions on $U$, that is such that for all open subsets $V \Subset U$, for all $\beta \in \Z_{\geq 0}^{n}$, there exists $C_{V, \alpha} > 0$ such that for all $h>0,\omega \in \Omega$,
\[
|\partial^\beta_x f_{h,\omega}(x)| \leq C_{V,\alpha}.
\]
Then
\[
A \in \Psi^m_{h,\omega}(U) \Longleftrightarrow A' := e^{-i f_{h,\omega}/h}A(e^{+i f_{h,\omega}/h}\bullet) \in \Psi^m_h(U).
\]
Moreover,
\[
\sigma_{A'_{h,\omega(h)}}(x,\xi) := \sigma_{A_{h,\omega(h)}}(x,\xi+  df_{h,\omega(h)}(x)), \qquad \forall (x,\xi) \in T^*U.
\]
and
\[
	(x, \xi) \in \WF_{h, \omega(h)}(A') \iff (x, \xi + df_{h, \omega(h)}(x)) \in \WF_{h, \omega(h)}(A). 
\]
\end{lemma}

\begin{proof}
Let $u \in C^\infty_{\comp}(U)$, and $A := \Op_h(a)$ for some $a \in S^m_{h, \omega}(U)$. One has:
\begin{equation}
\label{equation:conjexp}
\begin{split}
A'u(x) & = \dfrac{1}{(2\pi h)^n} \int e^{\tfrac{i}{h}(\xi \cdot (x-y)  + f_{h,\omega}(y)-f_{h,\omega}(x))} a(x,\xi) u(y)\, \dd y \dd \xi \\
& = \dfrac{1}{(2\pi h)^n} \int e^{\tfrac{i}{h}(\xi-\theta_{h,\omega}(x,y))(x-y)} a(x,\xi) u(y)\, \dd y \dd \xi \\
& = \dfrac{1}{(2\pi h)^n} \int e^{\tfrac{i}{h}\xi \cdot (x-y)} a(x,\xi + \theta_{h,\omega}(x,y)) u(y)\, \dd y \dd \xi,
\end{split}
\end{equation}
where
\[
\theta_{h,\omega}(x,y) = \int_0^1 (df_{h,\omega})_{y+t(x-y)}(\bullet) \dd t, \qquad \theta(x,x) = (df_{h,\omega})_{x},
\]
and the integrals should be interpreted as oscillatory integrals. Observe now that $(x,y,\xi) \mapsto a(x,\xi + \theta_{h,\omega}(x,y))$ is a symbol in $S^m_{h,\omega}(U \times U \times \R^n)$ (i.e. a symbol which depends on an additional space variable, defined similarly to \eqref{equation:local-symbol}) which is uniform in $(h,\omega)$. A standard argument based on the stationary phase lemma allows to get rid of the $y$-dependence in the symbol on the last line of \eqref{equation:conjexp} up to smoothing terms, see \cite[Theorem 4.4.7]{Lefeuvre-book} for instance. This proves both claims.
\end{proof}

\section{Quantization on line bundles}

\label{ssection:quantization-line-bundles}

Let $M$ be a smooth closed manifold. Let $d \geq 1$ be an integer, $\Omega := \Z^d$, and $L_1, \dotsc, L_d$ be complex Hermitian line bundles over $M$. Given $\mathbf{k} := (k_1, \dotsc,k_d) \in \Omega$, we write $\mathbf{L}^{\otimes \mathbf{k}} := L_1^{\otimes k_1} \otimes \dotsm \otimes L_d^{\otimes k_d}$ with the convention that $L_j^{\otimes k_j} = (L_j^*)^{\otimes |k_j|}$ if $k_j \leq 0$, and set $|\mathbf{k}| := |k_1| + \dotsb + |k_d|$. Finally, let $E_1, E_2 \to M$ be two auxiliary Hermitian bundles over $M$.

\subsection{Definition. First properties} 

We introduce a family of pseudodifferential operators acting on high tensor products of sections of $\mathbf{L}^{\otimes \mathbf{k}}$:

\begin{definition}[$h$-semiclassical pseudodifferential operators on line bundles]
\label{definition:pdo-calcul-line}
The class of $h$-semiclassical pseudodifferential operators $\mathbf{A} \in \Psi^m_h(M, \mathbf{L}, E_1 \to E_2)$ is the set of families of operators defined for $h > 0$, $\mathbf{k} \in \Omega(h) := \{\mathbf{l} \in \mathbb{Z}^d \mid h |\mathbf{l}| \leq 1\}$,
\[
	\mathbf{A}_{h,\mathbf{k}} : C^\infty(M,\mathbf{L}^{\otimes \mathbf{k}} \otimes E_1) \to C^\infty(M, \mathbf{L}^{\otimes \mathbf{k}} \otimes E_2), \quad h > 0, \mathbf{k} \in \Omega(h)
\]
such that the following holds: for all contractible open subsets $U \subset M$, for all $s_1 \in C^\infty(U, L_1), \dotsc, s_d \in C^\infty(U,L_d)$ such that $|s_1|= \dotso =|s_d|=1$ (fiberwise) and $\chi,\psi \in C^\infty_{\comp}(U)$, there exists a family of pseudodifferential operators $A_{h,\mathbf{k}} \in \Psi^m_{h, \mathbf{k}} (M, E_1 \to E_2)$, uniform in $\mathbf{k} \in \Omega(h)$, such that for all $f \in C^\infty(M, E_1)$,
\begin{equation}
\label{equation:trivial}
\psi \mathbf{A}_{h,\mathbf{k}} (\chi f \otimes \mathbf{s}^{\mathbf{k}}) = A_{h,\mathbf{k}}(f) \otimes \mathbf{s}^{\mathbf{k}},
\end{equation}
where $\mathbf{s}^{\mathbf{k}} := s_1^{\otimes k_1} \otimes \dotsb \otimes s_d^{\otimes k_d}$.
\end{definition}

We used the notation $\mathbf{A}_{h,\mathbf{k}}$ in the previous definition to insist on the $(h,\mathbf{k})$ dependence of the family of operators. However, in the following, we will mostly drop the index and write $\mathbf{A}$ instead. The parameter $\mathbf{k}$ plays the role of the auxiliary parameter $\omega \in \Omega$ in \S\ref{ssection:uniform-scl}.

Examples of operators fitting in this calculus are provided below. When the auxiliary bundles $E$ and $F$ are both trivial line bundles, we shall simply write $\Psi^{m}_{h, \mathbf{k}}(M,\mathbf{L})$. To avoid cumbersome notation, we mostly discuss this case in this paragraph. As in the standard semiclassical calculus, all the results described here can be easily extended to twisted pseudodifferential operators acting on auxiliary vector bundles (principal symbols, $L^2$-boundedness, ellipticity, etc.). When $d=1$ (single line bundle $L \to M$), we write $\Psi^m_{h,k}(M,L)$; this calculus was introduced by Charles \cite{Charles-00} in his thesis, see also \cite{Guillarmou-Kuster-21} where it is further studied.

Finally, we emphasize that $h$ and $\mathbf{k}$ can be taken as two independent parameters in this calculus. This will be needed in Chapter \ref{chapter:flow} as $h$ will be (the inverse of) a spectral parameter (and $h \to 0$) while $\mathbf{k}$ will encode Fourier modes of functions on a principal bundle and these will be independent. However, in certain problems, it is natural to restrict to $h := |\mathbf{k}|^{-1}$.

\begin{remark}
Denote by $\mathrm{Vect}^1_{\C}(M)$ the group of line bundles over $M$ under the tensor product operation (with the inverse being given by taking the dual). It is well-known that the first Chern class $c_1 : \mathrm{Vect}^1_{\C}(M) \to H^2(M,\Z)$ is an isomorphism. Hence, if the set $\{c_1(L_1), \dotsc, c_1(L_d)\}$ generates additively $H^2(M,\Z)$, any line bundle $L$ over $M$ can be written as $L = \mathbf{L}^{\otimes k}$ for some $k \in \Z^d$ (up to isomorphism). The calculus defined above can thus be used to deal with \emph{all} line bundles over $M$ \emph{simultaneously}.
\end{remark}


We now proceed to study the properties of operators in $\Psi^m_{h, \mathbf{k}}(M,\mathbf{L})$. Given $\mathbf{A} \in \Psi^m_{h, \mathbf{k}}(M,\mathbf{L})$, its formal adjoint $\mathbf{A}^* \in \Psi^m_{h, \mathbf{k}}(M,\mathbf{L})$ is (uniquely) defined by the property that for all $h>0, \mathbf{k} \in \Omega$ and $f_1,f_2 \in C^\infty(M,\mathbf{L}^{\mathbf{k}})$,
\[
\langle \mathbf{A}f_1, f_2 \rangle_{L^2(M,\mathbf{L}^{\mathbf{k}})} = \langle f_1,\mathbf{A}^*f_2\rangle_{L^2(M,\mathbf{L}^{\mathbf{k}})}.
\]
In order to define a principal symbol on $\Psi^m_{h, \mathbf{k}}(M,\mathbf{L})$, we will need to equip each line bundle with a smooth unitary connection. Recall in the following proposition that $h|\mathbf{k}| \leq 1$.


\begin{proposition}
\label{proposition:proprietes}
Let $(\nabla_j)_{j = 1}^d$ be smooth unitary connections on $(L_j)_{j = 1}^d$. The following properties hold:
\begin{enumerate}[label=\emph{(\roman*)}]
\item \emph{\textbf{Principal symbol.}} Given $U \subset M$ a contractible open set, and $s_j \in C^\infty(U,L_j)$ (for $j \in \{1, \dotsc,d\}$), trivializing sections such that $|s_j|=1$ (fiberwise), write $\nabla_j s_j = i \beta_j \otimes s_j$ for some real-valued connection $1$-form $\beta_j \in C^\infty(U,T^*U)$ (that is, $\nabla_j = d + i\beta_j$ in these trivialisations). Let $h \mapsto \mathbf{k}(h) \in \Omega(h)$ be an arbitrary family. The principal symbol of $\mathbf{A} \in \Psi^m_{h, \mathbf{k}}(M,\mathbf{L})$ relative to the family $h \mapsto \mathbf{k}(h)$ is then defined as
\begin{equation}
\label{equation:symbole-principal}
\sigma_{\mathbf{A},\mathbf{k}(h)}(x,\xi) := \sigma_{A_{h,\mathbf{k}(h)}}(x,\xi-h\mathbf{k}(h)\cdot \mathbf{\beta}(x)) \in S^m_h(T^*M)/hS^{m - 1}_{h}(T^*M),
\end{equation}
for $x \in \{\chi=\psi=1\}$, where $\mathbf{\beta} := (\beta_1, \dotsc,\beta_d)$, $\mathbf{k}(h)\cdot \mathbf{\beta} := \sum_j k_j \beta_j$ and $A_{h,\mathbf{k}}$ is given by \eqref{equation:trivial}. Relative to any such $\mathbf{k}(h)$ the principal symbol is globally well-defined and depends only on $\nabla := (\nabla_1, \dotsc,\nabla_d)$.


\item \emph{\textbf{Kernel of the symbol map.}} If $\mathbf{A} \in \Psi^m_{h, \mathbf{k}}(M,\mathbf{L})$ and $\sigma_{\mathbf{A}} = 0$ (that is, relative to any family $h \mapsto \mathbf{k}(h)$), then $\mathbf{A} \in h\Psi^{m-1}_{h, \mathbf{k}}(M,\mathbf{L})$.

\item \emph{\textbf{Adjoint.}} If $\mathbf{A} \in \Psi^m_{h, \mathbf{k}}(M,\mathbf{L})$, then $\mathbf{A}^* \in \Psi^m_{h, \mathbf{k}}(M,\mathbf{L})$ and $\sigma_{\mathbf{A}^*} = \overline{\sigma_{\mathbf{A}^*}}$ (that is equality holds along any family $h \mapsto \mathbf{k}(h)$).

\item \emph{\textbf{Algebra property.}} If $\mathbf{A} \in \Psi^m_{h, \mathbf{k}}(M,\mathbf{L})$, $\mathbf{B} \in \Psi^{m'}_{h, \mathbf{k}}(M,\mathbf{L})$, then
\begin{equation}
\label{equation:mult-symbols}
\mathbf{A}\mathbf{B} \in \Psi^{m+m'}_{h,\mathbf{k}}(M,\mathbf{L}), \qquad \sigma_{\mathbf{A}\mathbf{B}} = \sigma_{\mathbf{A}}\sigma_{\mathbf{B}}.
\end{equation}

\item \emph{\textbf{$L^2$-boundedness, Calder\'{o}n-Vaillancourt Theorem.}} An operator $\mathbf{A} \in \Psi^0_{h, \mathbf{k}}(M,\mathbf{L})$ is bounded on the space $L^2(M,\mathbf{L}^{\otimes \mathbf{k}})$ as $h \to 0$ with bounds
\begin{equation}
\label{equation:cv}
\begin{split}
\|\mathbf{A}\|_{L^2(M,\mathbf{L}^{\otimes \mathbf{k}(h)}) \to L^2(M,\mathbf{L}^{\otimes \mathbf{k}(h)})} \leq \|\sigma_{\mathbf{A},\mathbf{k}(h)}\|_{L^\infty(T^*M)} + \mc{O}(h), 
\end{split}
\end{equation}
and
\begin{equation}
\label{equation:cv2}
\|\mathbf{A}\|_{L^2(M,\mathbf{L}^{\otimes \mathbf{k}}) \to L^2(M,\mathbf{L}^{\otimes \mathbf{k}})} \leq \sup_{h \mapsto \mathbf{k}(h)} \|\sigma_{\mathbf{A},\mathbf{k}(h)}\|_{L^\infty(T^*M)} + \mc{O}(h).
\end{equation}
The $\mc{O}(h)$ in \eqref{equation:cv2} only depends on $h > 0$ (it is independent of $\mathbf{k}$).
\end{enumerate}
\end{proposition}

The proof of the previous proposition is given below. It should be emphasized that the principal symbol $\sigma_{\mathbf{A}}$ \emph{does} depend on a choice of vector of connections $\nabla := (\nabla_1, \dotsc,\nabla_d)$ (we will mostly refrain from writing $\sigma^{\nabla}_{\mathbf{A}}$ in order to keep the notation simple). This choice of connection should be interpreted as fixing a '0-section' (i.e. an origin) in the phase space $T^*M$.

Changing $\nabla$ to $\nabla' = \nabla + i\beta'$ affects the principal symbol by
\begin{equation}
\label{equation:translation}
\sigma^{\nabla'}_{\mathbf{A}}(x,\xi) = \sigma^{\nabla}_{\mathbf{A}}(x,\xi-h\mathbf{k}(h)\cdot\beta'(x)) = \sigma^{\nabla}_{\mathbf{A}}(T_{\beta'}(x,\xi)),
\end{equation}
where 
\begin{equation}\label{eq:fibrewise-translation}
	T_{\beta'}(x,\xi) := (x,\xi-h\mathbf{k}(h)\cdot\beta'(x))
\end{equation} 
is the fiberwise translation. It should be observed however that most (if not \emph{all}) properties are actually independent of the choice of connection. For instance,
\eqref{equation:cv} is clearly independent of any choice. Also note that, as illustrated here, we will mostly manipulate operators by assuming that a family $h \mapsto \mathbf{k}(h)$ was chosen (we refrain from repeating it every time).

\begin{proof}
(i) Changing $s_j$ to $s'_j = e^{i \omega_j} s_j$ for some $\omega_j \in C^\infty(U)$ locally defined and real-valued, $A_{h,\mathbf{k}}$ is then changed to
\begin{equation}
\label{equation:ahprime}
A_{h,\mathbf{k}}'(\bullet) = e^{-i\mathbf{k}\cdot\omega}A_{h,\mathbf{k}}(e^{i \mathbf{k}\cdot\omega} \bullet),
\end{equation}
where $\omega = (\omega_1, \dotsc,\omega_d)$. Since $h|\mathbf{k}| \leq 1$, $A_{h,\mathbf{k}}' \in \Psi_{h,\mathbf{k}}^m(M)$ if and only if $A_{h,\mathbf{k}} \in \Psi_{h,\mathbf{k}}^m(M)$ and $\sigma_{A',\mathbf{k}(h)}(x,\xi) = \sigma_{A,\mathbf{k}(h)}(x,\xi+h \mathbf{k}(h)\cdot\dd \omega(x))$ by Lemma \ref{lemma:appendix}.

Now, we have $\nabla_j s_j = i \beta_j \otimes s_j$ and 
\[
	\nabla_j s'_j = i \dd \omega_j \otimes s'_j + i \beta_j \otimes s'_j = i \beta'_j \otimes s'_j, 
\] 
with $\beta'_j = \beta_j + \dd \omega_j$ and thus
\[
\begin{split}
\sigma_{A', \mathbf{k}(h)}(x,\xi-h\mathbf{k}(h)\cdot\beta'(x)) & = \sigma_{A,\mathbf{k}(h)}(x,\xi+ h\mathbf{k}(h)\cdot\dd \omega(x)-h\mathbf{k}(h)\cdot\beta'(x)) \\
& = \sigma_{A,\mathbf{k}(h)}(x,\xi-h\mathbf{k}(h)\cdot\beta(x)),
\end{split}
\]
so the principal symbol is intrinsically defined. On overlaps of charts, the defined symbols agree up to lower order terms in $hS_h^{m - 1}(T^*M)$, by the change of coordinates formula from standard semiclassical analysis (see \cite[Proposition E.10]{Dyatlov-Zworski-19}). Thus, the principal symbol is globally well-defined as an element of $S^m_h(T^*M)/hS^{m - 1}_{h}(T^*M)$.\\

(ii) Straightforward consequence of Lemma \ref{lemma:kernel-symbol}. \\

(iii-v) The algebra property and the computation for the adjoint are immediate and follow from similar results on $\Psi_h^m(M)$, see \cite[Chapter 14]{Zworski-12}. Similarly, the $L^2$-boundedness follows from the same result on $\Psi_h^0(M)$, see \cite[Theorem 13.13]{Zworski-12}.
\end{proof}

\begin{remark}[Principal symbols in $\Psi^m_{h,\mathbf{k}}(M,\mathbf{L})$]
It is convenient to see the principal symbol as follows. In the semiclassical parameter space $(0,1] \times \Omega \subset [0,1] \times \R^d$, define the \emph{blow-up} $\mathbf{b}$ of the point $0$ as the set of all possible limits $\lim_{h\to 0} h\mathbf{k}(h)$. Observe that $\mathbf{b}$ is compact. For instance if $\Omega = \Z$, then the blow-up is diffeomorphic to a half-circle, see Figure \ref{figure:blow-up}. One should then think of the principal symbol $\sigma_{\mathbf{A}}$ as an element of $S^m(T^*M \times \mathbf{b})$.
\end{remark}

\begin{figure}[htbp!]
\centering
\includegraphics{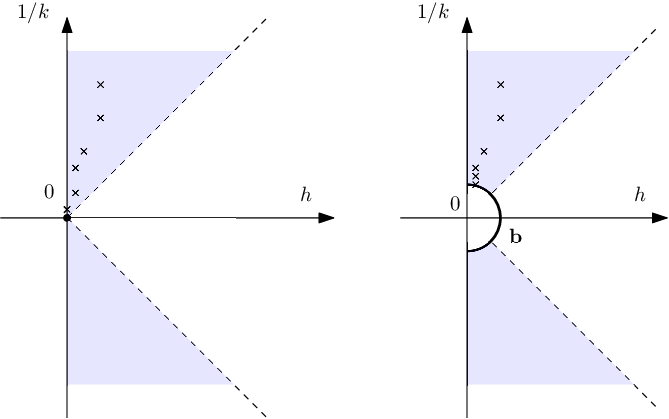}
\caption{The blowup of the point $0$ in the semiclassical parameter space. In shaded blue is the region $h |\mathbf{k}| \leq 1$.}
\label{figure:blow-up}
\end{figure}

\subsection{Examples}
\label{sssection:examples}
We now provide some examples of operators in this calculus and compute their principal symbol. Fix a vector of connections $\nabla := (\nabla_1, \dotsc,\nabla_d)$, where each $\nabla_j$ is a smooth unitary connection on the line bundle $L_j \to M$. For $\mathbf{k} \in \Omega(h)$, we set
\[
\nabla_{\mathbf{k}} : C^\infty(M,\mathbf{L}^{\otimes \mathbf{k}}) \to C^\infty(M,\mathbf{L}^{\otimes \mathbf{k}}\otimes T^*M),
\]
the connection on $\mathbf{L}^{\otimes k}$ induced (tensorially) by the connections $\nabla = (\nabla_1, \dotsc,\nabla_d)$.

\begin{enumerate}[label=(\roman*), itemsep=5pt]
\item Set $\mathbf{P} := h \nabla_{\mathbf{k}}$ such that $h|\mathbf{k}| \leq 1$. Then $\mathbf{P} \in \Psi^1_{h,\mathbf{k}}(M,\mathbf{L}, \C \to T^*M)$ (where $\C$ denotes the trivial line bundle) and:
\begin{equation}
\label{equation:symbol-nabla-example}
\sigma_{\mathbf{P}}^{\nabla}(x,\xi) = i \xi.
\end{equation}
By this, we mean that the principal symbol is actually \emph{independent} of the choice of family $h \mapsto \mathbf{k}(h)$, that is $\sigma_{\mathbf{P},\mathbf{k}(h)}^{\nabla}(x,\xi) = i \xi$ for all families $h \mapsto \mathbf{k}(h)$. 

Indeed, we let $\psi$, $\chi$, $(s_j)_{j = 1}^d$, $\mathbf{s}^{\otimes \mathbf{k}}$, and $\beta = (\beta_1, \dotsc, \beta_d)$ be as in Definition \ref{definition:pdo-calcul-line}. For $f \in C^\infty(M)$, we have
\[
	\psi\mathbf{P}\chi (f \mathbf{s}^{\otimes \mathbf{k}}) = \psi (hd + i h\mathbf{k} \cdot \beta \wedge )(\chi f) \mathbf{s}^{\mathbf{k}}.
\]
Taking an arbitrary family $h \mapsto \mathbf{k}(h)$, we then compute $\sigma_{\mathbf{P},\mathbf{k}(h)}^{\nabla}$ by \eqref{equation:symbole-principal} and conclude that it is given by \eqref{equation:symbol-nabla-example}.

This operator has the distinctive property that its principal symbol is independent of $h \mapsto \mathbf{k}(h)$. It holds because the operator is constructed naturally from the connection (it \emph{is} the connection in this case) and the principal symbol is computed with respect to the same connection. However, taking $P' := hd + i hk\alpha \wedge$ (where $d$ is the trivial connection on $M \times \C$ and $\alpha$ a smooth real-valued $1$-form), it is clear that $\sigma^d_{P',k(h)}(x, \xi) = i(\xi + ihk \alpha(x))$ \emph{does} depend on the family $h \mapsto k(h)$.
\item Given a vector field $X \in C^\infty(M,TM)$, set $\mathbf{P} := h \X_{\mathbf{k}} := \iota_X \nabla_{\mathbf{k}}$, where $\iota_X$ denotes contraction with $X$. Then $\mathbf{P} \in \Psi^1_{h,\mathbf{k}}(M,\mathbf{L})$, using (i) and the composition property of the calculus (or arguing directly as in (i)), we get
\[
	\sigma^{\nabla}_{\mathbf{P}}(x,\xi) = i\xi (X(x)).
\]
\item Assume $M$ is a complex manifold and $L_j \to M$ are holomorphic Hermitian line bundles, $j = 1, \dotsc, d$. Let $\overline{\partial}_j$ be the operator on $L_j$ induced by the holomorphic structure and denote by $\nabla_j$ the (unique) unitary Chern connection whose $(0,1)$-part is given by $\overline{\partial}_j$. The operators $(\overline{\partial}_j)_{j = 1}^d$ induce (tensorially) an operator 
\[
\overline{\partial}_{\mathbf{k}} : C^\infty(M,\mathbf{L}^{\otimes \mathbf{k}}) \to C^\infty(M,\mathbf{L}^{\otimes \mathbf{k}} \otimes T^*_{\C}M^{0,1}).
\]
on sections of $\mathbf{L}^{\otimes \mathbf{k}}$. Then $\mathbf{P} := h \overline{\partial}_{\mathbf{k}} \in \Psi^1_{h,\mathbf{k}}(M,\mathbf{L}, \C \to T^*_{\C}M^{0,1})$ and
\begin{equation}
\label{equation:01}
\sigma^{\nabla}_{\mathbf{P}}(x,\xi) = i \xi^{0,1},
\end{equation}
where $\xi^{0, 1} = \pi^{0, 1} \xi$, and $\pi^{0,1} : T^*_{\C}M \to T^*_{\C}M^{0,1}$ is the projection defined by the complex structure. This follows from \eqref{equation:symbol-nabla-example} above by observing that $\overline{\partial}_j = \pi^{0,1} \nabla_j$.

\item Assume that $F$ is the flag manifold bundle of \S\ref{sssection:flag-manifold-bundle}, and consider the operator $\overline{\partial}_{\mathbf{k}}$ introduced in \eqref{equation:del-bar-lambda}, acting on sections of $\mathbf{L}^{\otimes \mathbf{k}} \to F$. Then $\mathbf{P} := h \overline{\partial}_{\mathbf{k}} \in \Psi^1_{h,\mathbf{k}}(F,\mathbf{L}, \C \to {\V_F^*}^{0,1})$. Choose an arbitrary principal bundle connection $P \to M$, to which we can then associate partial connections $\nabla_j$ on $L_j$ for each $j = 1, \dotsc, d$. Let $\mathbb{H}_F \subset TF$ be the horizontal space of the associated connection $F \to M$; then $\mathbb{H}_F \oplus \mathbb{V}_F = TF$ with corresponding projections denoted by $\pi_{\mathbb{H}_F/\mathbb{V}_F}$. Let $\pi^{0, 1}$ be the projection from $\mathbb{V}_F^* \otimes \mathbb{C}$ to ${\mathbb{V}_F^*}^{0, 1}$ defined by the fibrewise complex structure. Also, let $\imath_{\mathbb{V}_F^*/\mathbb{H}_F^*}$ be the canonical embeddings of $\mathbb{V}_F^*/\mathbb{H}_F^*$ into $T^*F$, defined by $\xi \mapsto \xi \circ \pi_{\mathbb{V}_F/\mathbb{H}_F}$. Then, write $\pi_{\mathbb{V}_F^*/\mathbb{H}_F^*}$ for the projections from $T^*F$ onto $\imath_{\mathbb{V}_F^*}(\mathbb{V}_F^*)$ and $\imath_{\mathbb{H}_F^*}(\mathbb{H}_F^*)$, respectively. We claim that
\begin{equation}
\label{equation:symbol-bar-bundle}
	\sigma^{\nabla}_{\mathbf{P}}(x,\xi) = i  (\xi|_{\mathbb{V}_F})^{0, 1}.
\end{equation}
Indeed, we form the connections $\overline{\nabla}_{j} := \imath_{\mathbb{H}_F^*} {\nabla}_{j} +  \imath_{\mathbb{V}_F^*}D^{\mathrm{Chern}}_{j}$ (see \eqref{equation:dynamical-lambda}), and we observe that 
\[
	\overline{\partial}_{j} = \pi^{0, 1} (\imath_{\mathbb{V}_F^*})^{-1} \pi_{\mathbb{V}_F^*} \overline{\nabla}_{j},
\]
and notice that we simply have $(\imath_{\mathbb{V}_F^*})^{-1} \pi_{\mathbb{V}_F^*}(\xi) = \xi|_{\mathbb{V}_F}$. Then, a computation analogous to (i) (with $d$ replaced by $\overline{\partial}$ and $\beta$ with a suitable connection $1$-form) shows the validity of \eqref{equation:symbol-bar-bundle}.


\item Consider the case $d=1$ and a single line bundle $L \to M$. Let $g$ be a Riemannian metric over $M$, $\nabla$ be a unitary connection on $L$, and define $\mathbf{\Delta} := \mathbbm{1} + h^2(\nabla_k)^*\nabla_k : C^\infty(M,L^{\otimes k}) \to C^\infty(M,L^{\otimes k})$. Then
\[
\sigma^{\nabla}_{\mathbf{\Delta}}(x,\xi)  = 1+ |\xi|^2_g.
\] 
This follows from \eqref{equation:symbol-nabla-example}, as well as items (iii) and (iv) in Proposition \ref{proposition:proprietes}.
\end{enumerate}

\subsection{Ellipticity. Wavefront set. Sobolev spaces}



\subsubsection{Definitions} As in \S \ref{sssec:ellipticity}, one can define the elliptic set and the wavefront set of a pseudodifferential operator acting on sections (as introduced above). Once again, as for the principal symbol, these sets do depend on an arbitrary choice of connection $\mathbf{\nabla}$ but this choice is irrelevant for most purposes. They also depend on a choice of family $h \mapsto \mathbf{k}(h) \in \Omega(h) \subset \mathbb{Z}^d$.

\begin{definition}
Let $\mathbf{A} \in \Psi^m_{h,\mathbf{k}}(M,\mathbf{L})$ and $\nabla = (\nabla_1, \dotsc, \nabla_d)$. Then:
\begin{enumerate}[label=(\roman*), itemsep=5pt]
\item \textbf{Elliptic set.} The elliptic set $\Ell_{\mathbf{k}(h)}(\mathbf{A})$ relative to the family $h \mapsto \mathbf{k}(h) \subset \Omega(h)$ is defined as
\[
	\Ell^\nabla_{\mathbf{k}(h)}(\mathbf{A}) := \{(x,\xi) \in \overline{T^*M} ~|~ \exists c > 0,\,\, \forall h >0,\,\, |\sigma_{\mathbf{A},\mathbf{k}(h)}(x,\xi)| \langle\xi\rangle^{-m} \geq c\}.
\] 
The elliptic set is then defined similarly to \eqref{equation:ell} as
\[
	\Ell^\nabla(\mathbf{A}) := \cap_{\{h \mapsto \mathbf{k}(h)\}} \Ell_{\mathbf{k}(h)}(\mathbf{A}).
\]
\item \textbf{Wavefront set.} The wavefront set $\WF_h(\mathbf{A})$ relative to the family $h \mapsto \mathbf{k}(h)$ is defined as the union over all contractible open sets in Proposition \ref{proposition:proprietes}, cutoff functions $\psi$, and $\chi$
\[
	\WF^\nabla_{\mathbf{k}(h)}(\mathbf{A}) := \cup_{U, \chi, \psi} T_{\beta}^{-1}\big(\WF_{\mathbf{k}(h)}(A)|_{T^*U}\big),
\]
where $\WF_{\mathbf{k}(h)}(A)$ is the wavefront set of $A \in \Psi^m_{h,\mathbf{k}}(M)$ introduced in Proposition \ref{proposition:proprietes}, Item (i), relative to $h \mapsto \mathbf{k}(h)$ and defined in \eqref{equation:wf-standard}. (Note that this definition is independent of the choice of local trivialising sections $(s_j)_{j = 1}^d$ by Lemma \ref{lemma:appendix}, similarly as in the proof of Proposition \ref{proposition:proprietes}, Item (i).) The wavefront set is then defined as
\[
	\WF^\nabla(\mathbf{A}) := \cup_{\{h \mapsto \mathbf{k}(h)\}} \WF^\nabla_{\mathbf{k}(h)}(\mathbf{A}).
\]
We will say that $\mathbf{A}$ is \emph{compactly microlocalised} if $\WF^\nabla(\mathbf{A}) \subset T^*M$ is compact, and we will denoted the set of such operators by $\Psi_{h, \mathbf{k}}^{\comp}(M, \mathbf{L})$.
\end{enumerate}
\end{definition}

Upon changing the connection $\nabla = (\nabla_1, \dotsc, \nabla_d)$ to $\nabla' = \nabla + i \beta'$, using \eqref{equation:translation} we get that the elliptic and wavefront sets transform according to
\[
	\Ell^{\nabla'}(\mathbf{A}) = T_{\beta'}^{-1}\big(\Ell^\nabla(\mathbf{A})\big), \quad \WF^{\nabla'}(\mathbf{A}) = T_{\beta'}^{-1} \big(\WF^\nabla(\mathbf{A})\big).
\]
When clear from context, we will drop the index $\nabla$ from the notation.



Conversely, having fixed $\nabla$ (to define the principal symbol), it is often convenient to be able to produce a pseudodifferential $\mathbf{A} \in \Psi^m_{h,k}(M,\mathbf{L})$ with given principal symbol $\sigma_{\mathbf{A}} = a$.

\begin{lemma}
\label{lemma:surjectivity}
The principal symbol map 
\[
	\sigma : \Psi^m_{h,\mathbf{k}}(M,\mathbf{L}) \ni \mathbf{A} \mapsto \sigma_{\mathbf{A}} \in S_{h, \mathbf{k}}^m(T^*M)/hS_{h, \mathbf{k}}^{m - 1}(T^*M)
\] 
is surjective, that is, for all $a \in S^m_{h, \mathbf{k}}(T^*M)$, there exists $\mathbf{A} \in \Psi^m_{h,\mathbf{k}}(M,\mathbf{L})$ such that $\sigma_{\mathbf{A}, \mathbf{k}(h)} = [a_{h, \mathbf{k}(h)}]$ relative to every family $h \mapsto \mathbf{k}(h)$.
\end{lemma}

\begin{proof}
Up to taking a finite cover of $M$ and a partition of unity subordinated to that cover, we can always assume that $a$ has support in a contractible open subset $U \subset M$. Let $s_j \in C^\infty(U,L_j)$ be trivializing sections for $j=1, \dotsc,d$ (of fiberwise norm $|s_j|=1$), write $\nabla_j s_j = i\beta_j\otimes s_j$, and set $a_{\beta,h,\mathbf{k}}(x,\xi) := a(x,\xi+h\mathbf{k}\cdot\beta(x))$. Define 
\[
	\mathbf{A}_{h,\mathbf{k}}(u) := \chi \Op_h(a_{\beta,h,\mathbf{k}})(\chi u \mathbf{s}^{-\mathbf{k}(h)}) \otimes \mathbf{s}^{\mathbf{k}(h)}, \quad u \in C^\infty(M, \mathbf{L}^{\otimes \mathbf{k}}),
\]
where $\chi$ is a cutoff function in $U$ such that $\chi = 1$ on the projection to $U$ of $\supp(a)$, and $\Op_h : S^m_h(T^*M) \to \Psi_h^m(M)$ is an arbitrary quantization. Then $\sigma_{\mathbf{A}} = a$ by construction.

%

%
\end{proof}

With analogy to Proposition \ref{proposition:ellipticity0}, elliptic operators in this calculus admit parametrices:

\begin{proposition}
\label{proposition:ellipticity}
Let $\mathbf{A} \in \Psi^m_{h,\mathbf{k}}(M,\mathbf{L})$, $\mathbf{B}\in \Psi^{m'}_{h,\mathbf{k}}(M,\mathbf{L})$. Assume that $\WF(\mathbf{B}) \subset \Ell(\mathbf{A})$. Then, there exist $\mathbf{Q}, \mathbf{Q}' \in \Psi^{m'-m}_{h,\mathbf{k}}(M,L)$ such that
\[
\mathbf{B} = \mathbf{Q} \mathbf{A} + \mc{O}_{\Psi^{-\infty}_{h,\mathbf{k}}(M,\mathbf{L})}(h^\infty) = \mathbf{A} \mathbf{Q}'  + \mc{O}_{\Psi^{-\infty}_{h,\mathbf{k}}(M,\mathbf{L})}(h^\infty).
\]
\end{proposition}

Once again, we emphasize that the choice of connection is irrelevant. In the above proposition, the condition $\WF(\mathbf{B}) \subset \Ell(\mathbf{A})$ holds for a specific connection if and only if it holds for \emph{all} connections. Finally, following standard arguments, if $\mathbf{A} \in \Psi^m_{h,\mathbf{k}}(M,\mathbf{L})$ is such that $\Ell(\mathbf{A}) = \overline{T^*M}$, then $\mathbf{A}$ is invertible for $h$ small enough and $\mathbf{A}^{-1} \in \Psi^{-m}_{h,\mathbf{k}}(M,\mathbf{L})$. (Indeed, by the existence of parametrices $\exists \mathbf{Q} \in \Psi^{-m}_{h, \mathbf{k}}(M, \mathbf{L})$ such that $\mathbf{Q}\mathbf{A} = \mathbbm{1} + \mc{O}_{\Psi^{-\infty}_{h, \mathbf{k}}}(h^{\infty})$; the right hand side is invertible for $h$ small enough.)

\subsubsection{Sobolev spaces I}

\label{sssection:sobolev-spaces-1}


We now introduce Sobolev spaces. As usual, we fix a connection on $P \to M$ and set:
\[
	\Delta_{\mathbf{k}} := \nabla_{\mathbf{k}}^*\nabla_{\mathbf{k}}.
\]
Recall $\overline{\partial}_{\mathbf{k}}$ is the fibrewise holomorphic derivative of $\mathbf{L}^{\otimes \mathbf{k}} \to F$. Note that $h^2\Delta_{\mathbf{k}}$ and $h^2 \overline{\partial}_{\mathbf{k}}^*\overline{\partial}_{\mathbf{k}}$ commute with $\Pi_{\mathbf{k}}$ thanks to \eqref{equation:commutation-nablak-pik}, and the fact that
\begin{equation}\label{eq:commutation-vertical-derivative}
	\Pi_{\mathbf{k}} \overline{\partial}_{\mathbf{k}}^* = 0,\quad \overline{\partial}_{\mathbf{k}} \Pi_{\mathbf{k}} = 0, \quad [\Pi_{\mathbf{k}}, \overline{\partial}_{\mathbf{k}}^* \overline{\partial}_{\mathbf{k}}] = 0,
\end{equation}
where the second equality follows by definition, the first one by taking $L^2$ adjoints, and the last one then follows immediately. Moreover, $h^2(\Delta_{\mathbf{k}} + \overline{\partial}_{\mathbf{k}}^*\overline{\partial}_{\mathbf{k}})$ is elliptic in $\Psi^{\bullet}_{h,\mathbf{k}}(F,\mathbf{L})$ thanks to \S \ref{sssection:examples}, Items (iv) and (v). One can then form the operator, for $s \in \mathbb{R}$,
\[
	\mathbf{\Delta}(s) := (\mathbbm{1}+h^2(\Delta_{\mathbf{k}}  + \overline{\partial}_{\mathbf{k}}^*\overline{\partial}_{\mathbf{k}}))^{\frac{s}{2}} \in \Psi^s_{h, \mathbf{k}}(F,\mathbf{L}),
\]
where the operator is defined by the spectral theorem; that $(\mathbbm{1}+h^2(\Delta_{\mathbf{k}}  + \overline{\partial}_{\mathbf{k}}^*\overline{\partial}_{\mathbf{k}}))^{\frac{s}{2}} \in \Psi^s_{h, \mathbf{k}}(F,\mathbf{L})$ follows similarly as in the standard semiclassical calculus, see \cite[Theorem 4.9]{Zworski-12} and \cite[Chapter 8]{Dimassi-Sjostrand-99}. One can define $H^s_{h}(F,\mathbf{L}^{\otimes \mathbf{k}})$ as the completion of $C^\infty(F,\mathbf{L}^{\otimes \mathbf{k}})$ with respect to the norm
\begin{equation}\label{eq:sobolev-norm}
	\|f\|_{H^s_{h}(F,\mathbf{L}^{\otimes \mathbf{k}})} := \|\mathbf{\Delta}(s)f\|_{L^2(F,\mathbf{L}^{\otimes \mathbf{k}})}.
\end{equation}

Next we give a criterion for an operator to be negligible.

\begin{lemma}
\label{lemma:help}
An operator $\mathbf{R}$ belongs to $h^\infty\Psi^{-\infty}_{h,\mathbf{k}}(F,\mathbf{L})$ if and only if the following holds: for all $N > 0$, there exists a constant $C := C(N) > 0$ such that for all $h,\mathbf{k}$ such that $h|\mathbf{k}|\leq1$, 
\[
\mathbf{R} : H^{-N}_h(F,\mathbf{L}^{\otimes \mathbf{k}}) \to H^{+N}_h(F,\mathbf{L}^{\otimes \mathbf{k}}),
\]
is bounded with norm $\leq C h^N$.
\end{lemma}

The proof is standard and follows verbatim the one for standard semiclassical pseudodifferential operators (see e.g.\ \cite[Theorem 9.12]{Zworski-12}). We record the following fact which will be used later on:

\begin{lemma}
\label{lemma:eat-pik}
Let $\mathbf{R} \in h^\infty\Psi^{-\infty}_{h,\mathbf{k}}(F,\mathbf{L})$. Then
\[
\Pi_{\mathbf{k}} \mathbf{R}, \quad \mathbf{R} \Pi_{\mathbf{k}} \quad \in h^\infty\Psi^{-\infty}_{h,\mathbf{k}}(F,\mathbf{L}).
\]
\end{lemma}

\begin{proof}
By Lemma \ref{lemma:help}, it suffices to check that the projection $\Pi_{\mathbf{k}} : H^{s}_h(F,\mathbf{L}^{\otimes \mathbf{k}}) \to H^{s}_h(F,\mathbf{L}^{\otimes \mathbf{k}})$ is bounded with norm $\leq C(s)$ for all $h,\mathbf{k}$ such that $h|\mathbf{k}| \leq 1$. Now note that in fact $\Pi_{\mathbf{k}}$ commutes with $\mathbf{\Delta}(s)$: indeed, by \eqref{equation:commutation-nablak-pik} and \eqref{eq:commutation-vertical-derivative}, it commutes with $\Delta_\mathbf{k}$ and $\overline{\partial}_{\mathbf{k}}^* \overline{\partial}_{\mathbf{k}}$ (and so it commutes with the Cauchy integrals in the spectral theorem defining $\mathbf{\Delta}(s)$). This shows that we may take $C(s) \equiv 1$ and completes the proof. (Alternatively, by interpolation and duality, it suffices to prove the bound for integer values of $s \geq 0$; this can be checked directly.)
\end{proof}

\subsection{Commutators. Propagation of singularities} 

\label{sssection:propagation1}

We turn to the study of propagation of singularities in $T^*M$.  Let $\pi : T^*M \to M$ be the projection. Let $\mathbf{F}_{\nabla} := (F_{\nabla_1}, \dotsc, F_{\nabla_d}) \in C^\infty(M,\Lambda^2 T^*M)^{\oplus d}$ be the curvatures (which are purely imaginary $2$-forms) of the unitary connections $\nabla_j$ on the line bundles $L_j \to M$. Define
\begin{equation}
\label{equation:omega-twisted}
\omega_{h,\mathbf{k}} := \omega_0 + i h \mathbf{k} \cdot \pi^* \mathbf{F}_\nabla = \omega_0 + ih \sum_{j = 1}^d k_j \pi^*F_{\nabla_j},
\end{equation}
where $\omega_0$ is the Liouville $2$-form on $T^*M$. In the following, we will be taking as before arbitrary families $h \mapsto \mathbf{k}(h)$ and will sometimes use the shorthand notation $\omega_{\mathbf{k}(h)}$ for $\omega_{h,\mathbf{k}(h)}$.

\begin{lemma}
The $2$-form $\omega_{h,\mathbf{k}}$ is closed and non-degenerate.
\end{lemma}

\begin{proof}
That $\omega_{h,\mathbf{k}}$ is closed is immediate since $d \omega_0 = 0 = dF_{\nabla_j}$ (for any $j = 1, \dotsc, d$). To show non-degeneracy, it suffices to consider a local coordinate patch $U$ and to show that for any real $1$-form $\beta$, the $2$-form $\omega = \omega_0 + \pi^*d\beta$ is non-degenerate. In local coordinates $((x_i, \xi_i))_{i = 1}^n$ on $T^*U$ we have
\[
	\omega = \sum_{k=1}^n \left(d\xi_k + \sum_{\ell=1}^n \partial_{x_\ell}\beta_k dx_\ell\right) \wedge dx_k.
\]
For an arbitrary vector field $H = \sum (H_{x_k} \partial_{x_k} + H_{\xi_k} \partial_{\xi_k})$, we obtain
\begin{equation}
\label{equation:ihw}
\iota_H \omega = - \sum_{k=1}^n H_{x_k} d\xi_k + \sum_{k=1}^n \left(H_{\xi_k} + \sum_{\ell=1}^n H_{x_{\ell}}(\partial_{x_\ell} \beta_k-\partial_{x_k} \beta_\ell)\right)dx_k.
\end{equation}
This vanishes if and only if $H \equiv 0$.
\end{proof}

Given a function $p \in C^\infty(T^*M)$, we can define using \eqref{equation:omega-twisted} its Hamiltonian vector field $H_p^{\omega_{h,\mathbf{k}}}$ with respect to $\omega_{h,\mathbf{k}}$ by requiring the following identity
\begin{equation}
\label{equation:hp-vfield}
\omega_{h,\mathbf{k}}(H_p^{\omega_{h,\mathbf{k}}},\bullet) = dp(\bullet)
\end{equation}
to be satisfied. In a patch of local coordinates, using \eqref{equation:ihw}, we obtain:
\begin{equation}
\label{equation:hvfield}
	H_p^{\omega_{h,\mathbf{k}}} = \sum_{k=1}^n -\partial_{\xi_k}p~ \partial_{x_k} + \sum_{k=1}^n \left(\partial_{x_k}p + h\mathbf{k} \cdot \sum_{\ell=1}^n \partial_{\xi_\ell}p(\partial_{x_\ell}\boldsymbol{\beta}_k-\partial_{x_k}\boldsymbol{\beta}_\ell) \right)\partial_{\xi_k} 	
\end{equation}
where $\boldsymbol{\beta}$ is the vector of $1$-forms such that $d\boldsymbol{\beta} = i\mathbf{F}_{\nabla}= (iF_{\nabla_1}, \dotsc, iF_{\nabla_d})$. 

Given a vector field $Y \in C^\infty(M,TM)$ generating a flow $\phi := (\phi_t)_{t \in \R}$, its standard \emph{symplectic lift} $\Phi^{\omega_0} := (\Phi_t^{\omega_0})_{t \in \R}$ is the flow on $T^*M$ given by
\[
\Phi_t^{\omega_0}(x,\xi) = (\phi_t(x), \dd {\phi_t}^{-\top}(x)\xi),
\]
where ${}^{-\top}$ denotes the inverse transpose. The flow $\Phi^{\omega_0}$ preserves $\omega_0$ and is generated by the Hamiltonian vector field $H_{p_Y}^{\omega_0}$, where $p_Y(x,\xi) := \langle\xi,Y(x)\rangle$. More generally, we will denote by $\Phi^{\omega_{h,\mathbf{k}}}$ the flow generated by $H_{p_Y}^{\omega_{h,\mathbf{k}}}$. Using \eqref{equation:hvfield}, it is straightforward to check that 
\begin{equation}
\label{equation:salut}
H_{p_Y}^{\omega_{h,\mathbf{k}}} p_Z = p_{[Z, Y]} + h\mathbf{k}\cdot\mathbf{F}_{\nabla}(Y,Z).
\end{equation}

Let $\Y_{\mathbf{k}} := \iota_Y \nabla_{\mathbf{k}} : C^\infty(M,\Lk) \to C^\infty(M,\Lk)$. Given $\mathbf{A} \in \Psi^m_{h,\mathbf{k}}(M,\mathbf{L})$, we denote by $e^{t \Y} \mathbf{A} e^{-t \Y}$ the family of operators $e^{t \Y_{\mathbf{k}}} \mathbf{A}_{h,\mathbf{k}} e^{-t \Y_{\mathbf{k}}}$ induced on each $C^\infty(M,\Lk)$. In the following proposition, all symbols are computed with respect to $\nabla$.

\begin{proposition}[Egorov's Theorem in $\Psi^m_{h,\mathbf{k}}(M,\mathbf{L})$]
\label{proposition:egorov}
Let $Y$ be an arbitrary vector field and $\Y_{\mathbf{k}} := \iota_Y \nabla_{\mathbf{k}}$. The following holds:
\begin{enumerate}[label=\emph{(\roman*)}, itemsep=5pt]
\item For all $\mathbf{A} \in \Psi^m_{h,\mathbf{k}}(M,\mathbf{L})$, $[\Y,\mathbf{A}] \in \Psi^m_{h,\mathbf{k}}(M,\mathbf{L})$. For all $t \in \R$,
\[
\mathbf{A}(t) := e^{t \Y} \mathbf{A} e^{-t \Y} \in \Psi^m_{h,\mathbf{k}}(M,\mathbf{L}),
\]
with smooth dependence in the parameter $t \in \R$ (see the comment after the proposition for the topology on $\Psi^m_{h,\mathbf{k}}(M,\mathbf{L})$).

\item The principal symbol of $\mathbf{A}(t)$ is given by
\begin{equation}
\label{equation:integral2}
\sigma_{\mathbf{A}(t),\mathbf{k}(h)}(x,\xi) = \sigma_{\mathbf{A}_h}(\Phi^{\omega_{\mathbf{k}(h)}}_t(x,\xi)), \qquad \forall (x,\xi) \in T^*M,
\end{equation}
and
\begin{equation}
\label{equation:egorov-differentiel2}
\sigma_{[\Y,\mathbf{A}],\mathbf{k}(h)} = H_{p_Y}^{\omega_{\mathbf{k}(h)}} \sigma_{A}.
\end{equation}

\item If $\iota_Y \mathbf{F}_{\nabla} = 0$, then
\begin{equation}
\label{equation:integral}
\sigma_{\mathbf{A}(t), \mathbf{k}(h)}(x,\xi) = \sigma_{\mathbf{A}_h}(\Phi^{\omega_0}_t(x,\xi)), \qquad \forall (x,\xi) \in T^*M,
\end{equation}
and
\begin{equation}
\label{equation:egorov-differentiel}
\sigma_{[\Y,\mathbf{A}], \mathbf{k}(h)} = H_{p_Y}^{\omega_0} \sigma_{A}.
\end{equation}

\end{enumerate}
\end{proposition}

%

The main consequences of this Egorov-type result are propagation estimates, see Proposition \ref{proposition:propagation} below and \S\ref{ssection:radial-estimates} for radial estimates. In Item (i), the Fréchet topology on $\Psi^m_{h,\mathbf{k}}(M,\mathbf{L})$ is analogous to the usual one for standard pseudodifferential operators and defined in terms of local symbols, see \cite[Section 2.1]{Guillarmou-Knieper-Lefeuvre-22} for instance.

\begin{proof}
For simplicity, we assume that $d=1$. The general case is treated similarly. \\

(i) The first conclusion follows from Example (ii) in \S \ref{sssection:examples} and Proposition \ref{proposition:proprietes}. Let $s \in C^\infty(U, L)$ be a local trivializing section such that $|s|=1$ (fiberwise) on some open subset $U \subset M$, and let $\chi$ and $\psi$ be some cutoffs supported in $U$. By definition, there exists a uniform pseudodifferential operator $A \in \Psi^m_{h,\mathbf{k}}(M)$, such that $\chi \mathbf{A}_{h,k}( \psi f s^{k}) = A_{h,k}(f) s^{k}$ for all $f \in C^\infty(M)$. Let $h \mapsto k(h)$ be a family such that $h|k(h)| \leq 1$.

Let us write $\nabla s = i\beta \otimes s$. We first claim that (on the domain where the left hand side is defined)
\begin{equation}\label{eq:propagate-s^k}
	e^{t \mathbf{Y}} s^k = e^{ik \int_0^t \varphi_q^* \iota_Y \beta\, dq} s^k.
\end{equation}
Indeed, it suffices to consider $k = 1$ and we have
\begin{align*}
	\nabla_Y(e^{i \int_0^t \varphi_q^* \iota_Y \beta\, dq} s) &= i e^{i \int_0^t \varphi_q^* \iota_Y \beta\, dq} \left(\partial_{q'}|_{q' = 0} \int_0^t \varphi_{q + q'}^* \iota_Y\beta\, dq + \iota_Y \beta\right)\\ 
	&= i e^{i \int_0^t \varphi_q^* \iota_Y \beta\, dq} \varphi_t^* \iota_Y\beta = \partial_t \left(e^{i \int_0^t \varphi_q^* \iota_Y \beta\, dq} s \right).
\end{align*}
We will write $\chi_t = e^{tY} \chi$, $\psi_t = e^{tY} \psi$, and $s_t = e^{t\mathbf{Y}} s$ (since $\nabla$ is unitary, $s_t$ has fibrewise unit norm). Then for any $f \in C^\infty(M)$:
\begin{align}\label{eq:intertwined-t-easy}
	\chi_t e^{t \mathbf{Y}} \mathbf{A}_{h, k} e^{-t \mathbf{Y}} (\psi_t f s_t^k) = e^{tY} A_{h, k} (e^{-tY} f) s_t^k.
\end{align}
Since $U$, $\chi$, and $\psi$ were arbitrary, using the standard version of Egorov's theorem shows that $\mathbf{A}(t) \in \Psi^m_{h, \mathbf{k}}(M, \mathbf{L})$. The smooth dependence in the $t$-variable is straightforward. \\

(ii) We will use the same notation as in (i); write $\nabla s_t = i \beta(t) \otimes s_t$ and $A'_{h, k(h)}(t) := e^{tY} A_{h, k} e^{-tY}$. Using \eqref{eq:intertwined-t-easy}, we compute the principal symbol of $\mathbf{A}(t)$:
\begin{equation}\label{eq:general-symbol-computation}
	\sigma_{\mathbf{A}(t), k(h)} = \sigma_{A'(t)} \circ T_{\beta(t)} = \sigma_{A} \circ \Phi_t^{\omega_0} \circ T_{\beta(t)} = \sigma_{\mathbf{A}, k(h)} \circ T_{-\beta} \circ \Phi_t^{\omega_0} \circ T_{\beta(t)},
\end{equation}
where in the first and third equalities we used the definition of the principal symbol, while in the second one we used the standard Egorov's theorem. Therefore, it suffices to prove
\begin{equation}\label{eq:hamiltonian-flow-conjugacy}
	T_{-\beta} \circ \Phi_t^{\omega_0} \circ T_{\beta(t)} = \Phi_t^{\omega_{\mathbf{k}(h)}}, \quad \forall t \in \mathbb{R}.
\end{equation}

To compute $\beta(t)$, we use \eqref{eq:propagate-s^k}:
\begin{equation}\label{eq:beta(t)}
\begin{split}
	&\nabla e^{t \mathbf{Y}} s = i s_t \otimes \left(d \left(\int_0^t \varphi_q^* \iota_Y \beta\, dq\right) + \beta\right)\\ 
	&= is_t \otimes \left(-\int_0^t \varphi_q^* \iota_Y d\beta\, dq + \int_0^t \partial_q \varphi_q^*\beta\,dq + \beta\right) = is_t \otimes \underbrace{\left(\varphi_t^*\beta - \int_0^t \varphi_q^* \iota_Yd\beta\, dq\right)}_{=\beta(t)},
\end{split}
\end{equation}
where in the second equality we used Cartan's magic formula $\Lie_Y = \iota_Y d + d \iota_Y$.

Next, we observe that
\begin{equation}\label{eq:hamiltonian-flow-conjugacy-intermediate}
	T_{-\beta} \circ \Phi_t^{\omega_0} \circ T_{\varphi_t^*\beta} (x, \xi) = T_{-\beta}(\phi_t x, \xi \circ d\phi_{-t}(\phi_tx) - hk \beta(\phi_t x)) = \Phi_t^{\omega_0}(x, \xi).
\end{equation}
Therefore, to show \eqref{eq:hamiltonian-flow-conjugacy}, it suffices to prove
\begin{equation}\label{eq:hamiltonian-flow-conjugacy'}
		T_{-\beta} \circ \Phi_t^{\omega_0} \circ T_{\beta(t)} = \Phi_t^{\omega_0} \circ T_{-\int_0^t \varphi_q^* \iota_Yd\beta\, dq} = \Phi_t^{\omega_{\mathbf{k}(h)}}, \quad \forall t \in \mathbb{R},
\end{equation}
where we used the formula for $\beta(t)$ in \eqref{eq:beta(t)} and \eqref{eq:hamiltonian-flow-conjugacy-intermediate}. It is a straightforward exercise to show that the left hand side of the preceding equation is a one-parameter group in $t$, and therefore it suffices to prove that the generators of both sides agree. We compute the generator of the left hand side, in a patch of local coordinates:
\begin{align*}
	&\partial_t|_{t = 0} \Phi_t^{\omega_0} \circ T_{-\int_0^t \phi_q^* \iota_Yd\beta\, dq}(x, \xi)\\ 
	&= \partial_t|_{t = 0} \left(\phi_t x, \xi \circ d\phi_{-t}(\phi_t x) + hk \int_0^t (\phi_q^* \iota_Yd\beta)(x) \circ d\phi_{-t}(\phi_t x) \, dq\right)\\
	&= \partial_t|_{t = 0} \left(\phi_t x, \sum_i \xi_i dx_i \circ d\phi_{-t}(\phi_t x) + hk \sum_{j, \ell}\int_0^t \phi_q^*\big(Y_{\ell} (\partial_{x_\ell} \beta_j - \partial_j \beta_{\ell}) dx_j\big) (\phi_tx)\, dq\right)\\
	&= H_{p_Y}^{\omega_0}(x, \xi) + hk \sum_{j, \ell} Y_{\ell} (\partial_{x_\ell} \beta_j - \partial_j \beta_{\ell}) \partial_{\xi_j} = H^{\omega_{\mathbf{k}(h)}}_{p_Y}(x, \xi),
\end{align*}
where in the third line we change the variable inside the integral and expanded the terms under the integral sign using $Y = \sum_{\ell} Y_{\ell} \partial_{x_{\ell}}$ and $\beta = \sum_j \beta_j dx_j$, in the second to last equality we used that $(\phi_t x, \xi \circ d\phi_{-t}(\phi_t x))$ is generated by $H_{p_Y}^{\omega_0}$, as well as the formula for differentiation under the integral sign, and finally in the last equality we used \eqref{equation:hvfield}.

Lastly, \eqref{equation:egorov-differentiel} follows by differentiating \eqref{equation:integral} in the $t$-variable and evaluating at $t=0$. \\

(iii) If $\iota_Y d\mathbf{F} = 0$, the formula for $\beta(t)$ in \eqref{eq:beta(t)} gives $\beta(t) = \varphi_t^*\beta$. Using the formula \eqref{eq:hamiltonian-flow-conjugacy-intermediate}, this implies that $\Phi_t^{\omega_{\mathbf{k}(h)}} = \Phi_t^{\omega_0}$, i.e. the twisted and non-twisted Hamiltonian flows agree. Thus the claim follows from (ii).
\end{proof}

Finally, we conclude with a remark on connections with Hölder regularity.

\begin{remark}
\label{remark:regularity}
Proposition \ref{proposition:egorov} will be used later on in the next chapter with the dynamical connection $\nabla := \nabla^{\mathrm{dyn}}$ and $Y := X$, the Anosov vector field. This connection is only Hölder continuous. We emphasize that, despite its lack of regularity, \eqref{equation:integral} still holds. As we shall see from the proof, the crucial fact is that, trivializing locally $\nabla^{\mathrm{dyn}} = d + i \beta$, one has $\dd \iota_X \beta = \mc{L}_X \beta \in C^\infty(U,T^*U)$ (since $\iota_X d\beta = \iota_X F_{\nabla^{\mathrm{dyn}}}=0$, see Lemma \ref{lemma:computations}).

Equivalently, one could also work with an arbitrary background smooth unitary connection $\nabla' = \nabla^{\mathrm{dyn}} + i\beta'$ on $L$ in order to compute principal symbols. In this case, \eqref{equation:integral} is replaced by (according to the transformation law \eqref{equation:translation})
\begin{equation*}
\sigma^{\nabla'}_{\mathbf{A}(t),k(h)}(x,\xi) = \sigma^{\nabla'}_{\mathbf{A},k(h)}(T_{\beta'}^{-1}\Phi_t^{\omega_0} T_{\beta'}(x,\xi)), \qquad \forall (x,\xi) \in T^*M,
\end{equation*}
where we recall $T_{\beta'}(x,\xi) = (x,\xi-hk(h)\beta'(x))$.

Since $\beta'$ is only Hölder continuous (but $\beta + \beta'$ is smooth), $T_{\beta'}$ is a homeomorphism of $T^*M$ but the flow $(T_{\beta'}^{-1}\Phi_t^{\omega_0}T_{\beta'})_{t \in \R}$ is smooth (since it is equal to the Hamiltonian flow computed with respect to the connection $\nabla'$), topologically conjugate to $(\Phi_t^{\omega_0})_{t \in \R}$, and generated by $H_{p_X} + hkV_{\mc{L}_X\beta'}$, where $V_\alpha \in C^\infty(TM,T(TM))$ is defined for a smooth $1$-form $\alpha \in C^\infty(M,T^*M)$ as the vertical vector field generating the flow $\psi_t(x,\xi) := (x,\xi+t\alpha(x))$. (Notice that $H_{p_X} + hkV_{\mc{L}_X\beta}= \dd T_{\beta'}^{-1}(H_{p_X} \circ T_\beta)$; although $T_{\beta'}^{-1}$ is not $C^1$, its differential is well-defined in the direction of $H_{p_X} \circ T_\beta$.)

 We also point out that, for the dynamical connection, $H_{p_X} \sigma^{\nabla^{\mathrm{dyn}}}_{A}$ is always well-defined. Indeed, writing locally $\nabla^{\mathrm{dyn}} = d + i\beta$, and $\mathbf{A}$ as in \eqref{equation:trivial}, we have $\sigma^{\nabla^{\mathrm{dyn}}}_{\mathbf{A},k(h)}(x,\xi) = \sigma_{A,k(h)}(T_\beta(x,\xi))$ and thus
 \[
 H_{p_X} \sigma_{\mathbf{A},k(h)}^{\nabla^{\mathrm{dyn}}} = \dd \sigma_{A,k(h)}(\dd T_\beta(H_{p_X})) = \dd \sigma_{A, k(h)}((H_{p_X}+ hkV_{\mc{L}_X\beta}) \circ T_\beta)
 \]
 is well-defined (and Hölder-continuous).
\end{remark}

More generally, the following holds regarding commutation of operators:

\begin{lemma}
Let $\mathbf{A} \in \Psi^m_{h,\mathbf{k}}(M,\mathbf{L}), \mathbf{B} \in \Psi^{m'}_{h,\mathbf{k}}(M,\mathbf{L})$. Then
\[
[\mathbf{A},\mathbf{B}] \in h\Psi^{m+m'-1}_{h,\mathbf{k}}(M,\mathbf{L}).
\]
Moreover, given a family $h \mapsto \mathbf{k}(h)$, one has:
\[
\sigma_{h^{-1}[\mathbf{A},\mathbf{B}],\mathbf{k}(h)} := -\dfrac{1}{i} H^{\omega_{\mathbf{k}(h)}}_{\sigma_{\mathbf{A},\mathbf{k}(h)}} \sigma_{\mathbf{B},\mathbf{k}(h)} = -\dfrac{1}{i}\{\sigma_{\mathbf{A},\mathbf{k}(h)},\sigma_{\mathbf{B},\mathbf{k}(h)} \}^{\omega_{\mathbf{k}(h)}},
\]
where $\{\bullet, \bullet\}^{\omega_{\mathbf{k}(h)}}$ denotes the Poisson bracket with respect to $\omega_{\mathbf{k}(h)}$.
\end{lemma}

\begin{remark}
\label{remark:egorov-general} The previous commutator relation could be integrated to yield a general version of Egorov's theorem, namely to prove that $e^{it\mathbf{B}/h}\mathbf{A}e^{-it\mathbf{B}/h} \in \Psi^m_{h,\mathbf{k}}(M,\mathbf{L})$. The proof of that fact follows \emph{verbatim} the proof for the standard semiclassical calculus, see \cite[Theorem 15.2]{Zworski-12} for instance.
\end{remark}

\begin{proof}
That $[\mathbf{A},\mathbf{B}] \in h\Psi^{m+m'-1}_{h,\mathbf{k}}(M,\mathbf{L})$ follows from Proposition \ref{proposition:proprietes}, item (iv). To compute the principal symbol, for simplicity we assume $d = 1$. Let $U$ be an open set and $s$ a trivialising section of $(L, \nabla)$ over $U$, and $\nabla s = i\beta \otimes s$ for some $1$-form $\beta$ on $U$. Then let $\chi, \psi \in C^\infty_{\comp}(U)$, and let $\theta, \theta' \in C^\infty_{\comp}(U)$ such that $\theta = 1$ on $\supp(\psi)$ and $\theta' = 1$ on $\supp(\theta)$. By definition, there are operators $A, B, A', B' \in \Psi^\bullet_h(M)$ such that 
\begin{align*}
	s^{-k}\chi \mathbf{A} \theta'(f s^k) &= A(f),\quad \chi s^{-k}\mathbf{B} \theta'(f s^k) = B'(f)\\
	s^{-k}\theta \mathbf{B} \psi(fs^k) &= B(f),\quad s^{-k}\theta \mathbf{A} \psi(f s^k) = A'(f).
\end{align*}
Then we have
\begin{align*}
	s^{-k} \chi [\mathbf{A}, \mathbf{B}] \psi (f s^k) = (AB' - B'A)(f) + \mc{O}_{\Psi_h^{-\infty}}(h^\infty) (f).
\end{align*}
Therefore the principal symbol on $\{\theta = \theta' = \chi = \psi = 1\}$ is given by
\begin{align*}
	\sigma_{h^{-1}[\mathbf{A}, \mathbf{B}]} = -\frac{1}{i} H_{\sigma_A}^{\omega_0} \sigma_B \circ T_\beta = -\frac{1}{i} \partial_t|_{t = 0} \sigma_B \circ T_\beta \circ T_{-\beta} \circ \Phi_{t, \sigma_A}^{\omega_0} \circ T_\beta,
\end{align*}
where $\Phi_{t, \sigma_A}^{\omega_0}$ is the Hamiltonian flow of $\sigma_A$ with respect to $\omega_0$. Thus, it suffices to prove that 
\[
	T_{-\beta} \circ \Phi_{t, \sigma_A}^{\omega_0} \circ T_\beta = \Phi_{t, \sigma_{\mathbf{A}}}^{\omega_{\mathbf{k}(h)}},
\]
where on the right hand side is the Hamiltonian flow of $\sigma_{\mathbf{A}} = \sigma_A \circ T_\beta$ with respect to $\omega_{\mathbf{k}(h)}$. Differentiating at zero, it suffices to show
\[
	dT_{-\beta} (H^{\omega_0}_{\sigma_A} \circ T_\beta) = H^{\omega_{\mathbf{k}(h)}}_{\sigma_{\mathbf{A}}}.
\]
In local coordinates, the vector field on the right hand side is given by, using \eqref{equation:hvfield},
\begin{align*}
	&H^{\omega_{\mathbf{k}(h)}}_{\sigma_{\mathbf{A}}} = \sum_{j = 1}^n\big(-(\partial_{\xi_j} \sigma_A \circ T_\beta) \cdot \partial_{x_j} + (\partial_{x_j} \sigma_A \circ T_\beta) \cdot \partial_{\xi_j}\big)\\ 
	&- hk \sum_{i, j = 1}^n \partial_{x_j} \beta_i \cdot (\partial_{\xi_i} \sigma_A \circ T_\beta) \cdot \partial_{\xi_j} + hk \sum_{i, j = 1}^n (\partial_{\xi_i} \sigma_A \circ T_\beta) \cdot (\partial_{x_j} \beta_{\ell} - \partial_{x_i} \beta_j) \cdot \partial_{\xi_j}\\
	&= -\sum_{j = 1}^n\big((\partial_{\xi_j} \sigma_A \circ T_\beta) \cdot \partial_{x_j} + (\partial_{x_j} \sigma_A \circ T_\beta) \cdot \partial_{\xi_j}\big) - hk \sum_{i, j = 1}^n(\partial_{\xi_i} \sigma_A \circ T_\beta) \cdot \partial_{x_i} \beta_j \cdot \partial_{\xi_j}\\
	&= dT_{-\beta} (H^{\omega_0}_{\sigma_A} \circ T_\beta),
\end{align*}
where in the first equality we used that $iF_{\nabla} = -d\beta$. This proves the preceding claim and completes the proof.
\end{proof}

A consequence of Proposition \ref{proposition:egorov} is the standard lemma of propagation of singularities. We emphasize that the elliptic and wavefront sets in the next proposition are all computed with respect to the same connection $\nabla$. 

\begin{proposition}
\label{proposition:propagation}
Let $Y \in C^\infty(M,TM)$ be a vector field, $\nabla$ a unitary connection on $L$ such that $\iota_Y \mathbf{F}_{\nabla} = 0$, and set $\Y := \nabla_Y$. Let $\mathbf{A}, \mathbf{A}' \in \Psi^{\comp}_{h,\mathbf{k}}(F,\mathbf{L})$ be compactly microlocalised such that the following holds: for all $(x,\xi) \in \WF(\mathbf{A})$, there exists $t \in \R$ such that $\Phi_t^{\omega_0}(x,\xi) \in \Ell(\mathbf{A}')$. Then, there exists $\mathbf{B} \in \Psi^{\comp}_{h,\mathbf{k}}(F,\mathbf{L})$, compactly microlocalised such that for all $s \in \R, N > 0$, there exists $C >0$ such that for all $f_{h,\mathbf{k}} \in C^\infty(M,\Lk)$, the following holds:
\begin{equation}
\label{equation:propagation-classique}
\|\mathbf{A} f_{h,\mathbf{k}}\|_{H^s_h} \leq C\left(\|\mathbf{B} \mathbf{Y}f_{h,\mathbf{k}}\|_{H^s_h} + \|\mathbf{A}' f_{h,\mathbf{k}}\|_{H^s_h} + h^N \|f_{h,\mathbf{k}}\|_{H^{-N}_h} \right).
\end{equation}
Moreover, if $f_{h,\mathbf{k}}$ is merely a distribution and the right-hand side is finite, then $\mathbf{A}f_{h,\mathbf{k}} \in H^s_h$ and \eqref{equation:propagation-classique} holds.
\end{proposition}
\begin{proof}
	The proof is \emph{verbatim} the same as in the non-twisted case. We refer to \cite[Chapter 6, \S6.2]{Lefeuvre-book} for details.
\end{proof}
In the general case $\iota_Y \mathbf{F}_{\nabla} \neq 0$, a similar result holds but the Hamiltonian flows $\Phi^{\omega_{\mathbf{k}(h)}}$ vary according to the family $h \mapsto \mathbf{k}(h)$, and the statement is more complicated to formulate. In particular, the assumption ``for all $(x,\xi) \in \WF(\mathbf{A})$, there exists $t \in \R$ such that $\Phi_t^{\omega_{\mathbf{k}(h)}}(x,\xi) \in \Ell(\mathbf{A}')$'' can be made only along a single family $h \mapsto \mathbf{k}(h)$ and in this case, the conclusions of the previous proposition hold along this subsequence only. The assumption $\iota_Y \mathbf{F}_{\nabla} = 0$ allows to reduce to a single flow $\Phi_t^{\omega_0}$ on $T^*M$ (see Proposition \ref{proposition:egorov}, proof of Item (iii)).



\section{Quantization on principal bundles}

\label{ssection:twiste-line-bundles}

We now introduce the algebra of Borel-Weil operators. These operators are defined as families of operators acting diagonally on the Fourier decomposition of a function on $P$ into a sum of fiberwise holomorphic sections on $F$ (with values in $\mathbf{L}^{\otimes \mathbf{k}}$). Furthermore, it is required that this action on $C^\infty_{\mathrm{hol}}(F,\mathbf{L}^{\otimes \mathbf{k}})$ is the restriction of a larger action on \emph{all} sections $C^\infty(F,\mathbf{L}^{\otimes \mathbf{k}})$ by an operator that fits in the previous calculus introduced in \S\ref{ssection:quantization-line-bundles}. Of course, this calculus is designed in such a way that any natural right-invariant differential operator on $P$ (e.g. horizontal lift of a vector field, connection, horizontal or vertical Laplacian, etc.) will belong to it.

Let $M$ be a closed manifold, $P \to M$ be a $G$-principal bundle, where $G$ is a compact Lie group. Denote by $F\to M$, the flag bundle over $M$. Recall that the irreducible $G$-representations are classified by $\widehat{G} \simeq (\Z^a \times \Z^b_{\geq 0})/\sim$ (see \eqref{equation:ecriture}). Let $d=a+b$ and let $L_1, \dotsc, L_d \to F$ be the fiberwise holomorphic complex Hermitian line bundles over $F$ appearing in the Borel-Weil theorem (see \S \ref{sssection:flag-manifold-bundle}). 


\subsection{Basic notions} We will use the calculus introduced in \S\ref{ssection:quantization-line-bundles}. The base manifold will be $F$, the flag bundle, and the line bundles $L_1, \dotsc, L_a, L_{a+1}, \dotsc, L_d \to F$ are the ones that appear in the Borel-Weil theorem. Recall that $\Pi_{\mathbf{k}}$ is the fibrewise holomorphic projection introduced in \S \ref{sssection:fibrewise-holomorphic-projection}.

\subsubsection{Definition} We proceed to define the Borel-Weil calculus. We start with the following definition:

\begin{definition}[Admissible operators]
An \emph{admissible} $h$-semiclassical pseudodifferential operator $\mathbf{A} \in \Psi^m_{h, \mathbf{k}}(F,\mathbf{L})$ is a family of operators such that for $h > 0$, $\mathbf{k} \in \widehat{G}$ and $h|\mathbf{k}| \leq 1$,
\begin{equation}
\label{equation:commutation-holomorphic}
[\mathbf{A},\Pi_{\mathbf{k}}], \quad \Pi_{\mathbf{k}}\mathbf{A}(\mathbbm{1}-\Pi_{\mathbf{k}}), \quad (\mathbbm{1}-\Pi_{\mathbf{k}}) \mathbf{A}~\Pi_{\mathbf{k}} \quad \in h^\infty\Psi^{-\infty}_{h,\mathbf{k}}(F,\mathbf{L}).
\end{equation}
In other words, $\mathbf{A}$ preserves fiberwise holomorphicity modulo negligible remainders.
\end{definition}

Note that by Lemma \ref{lemma:eat-pik}, the first condition in \eqref{equation:commutation-holomorphic} implies the other two. Moreover, we note that negligibility of the first term in \eqref{equation:commutation-holomorphic} follows from that of the latter two. We can now define the pseudodifferential operators in the Borel-Weil calculus:

\begin{definition}[$h$-semiclassical Borel-Weil pseudodifferential operators]
For $m \in \R$, define the class of $h$-semiclassical Borel-Weil pseudodifferential operators as
\[
\Psi^m_{h, \mathrm{BW}}(P) := \{ \Pi_{\mathbf{k}} \mathbf{A}\Pi_{\mathbf{k}} ~|~ \mathbf{A} \in \Psi^m_{h, \mathbf{k}}(F,\mathbf{L}) \text{ admissible}\}.
\]
\end{definition}

We shall verify below that $\Psi^\bullet_{h,\mathrm{BW}}(P)$ is indeed an algebra satisfying all the standard properties of the pseudodifferential calculus. 

\begin{remark}
One can also define the Borel-Weil operators as $\Pi_{\mathbf{k}}\mathbf{A}\Pi_{\mathbf{k}}$, where $\mathbf{A} \in \Psi^m_{h,\mathbf{k}}(F,\mathbf{L})$ is \emph{any} pseudodifferential operator (not necessarily admissible). It can be verified that this would actually define the same classes. The advantage of working directly with admissible operators is twofold: first, it is clear that all natural geometric operators that appear in various problems (such as in Chapter \ref{chapter:flow} and \ref{chapter:hypoelliptic}) actually \emph{commute} with the holomorphic projection and are thus admissible; second, the algebra property is trivial to verify for admissible operators whereas it is not for general operators, and the proof of this fact would actually go through verifying that one can replace $\Pi_{\mathbf{k}}\mathbf{A}\Pi_{\mathbf{k}}$ (for an arbitrary $\mathbf{A} \in \Psi^m_{h,\mathbf{k}}(F,\mathbf{L})$) by $\Pi_{\mathbf{k}}\mathbf{A}'\Pi_{\mathbf{k}}$ where $\mathbf{A}'$ is admissible, modulo smoothing remainder. We point out that a similar discussion already appears in the foundational paper by Boutet de Monvel and Guillemin on Toeplitz operators, see \cite[Theorem 2.9, Proposition 2.13]{BoutetDeMonvel-Guillemin-81}.
\end{remark}

Operators in $\Psi^m_{h, \mathrm{BW}}(P)$ only act non-trivially on fiberwise holomorphic sections. Given a family $f_{h,\mathbf{k}} \in C^\infty_{\mathrm{hol}}(F,\mathbf{L}^{\otimes \mathbf{k}})$, since $h \overline{\partial}_{\mathbf{k}} f_{h,\mathbf{k}} = 0$ and $h \overline{\partial}_{\mathbf{k}}$ is elliptic outside of $\overline{\HH^*_F}$ (see \S \ref{sssection:examples}, Item (iv)), the microlocal mass of $f_{h,\mathbf{k}}$ can only accumulate on $\overline{\HH^*_F}$ (i.e. $\mathbf{A} f_{h, \mathbf{k}(h)} = \mc{O}_{C^\infty}(h^\infty)$ if $\WF(\mathbf{A}) \cap \overline{\HH^*_F} = \emptyset$). As a consequence, only the microlocal behaviour of operators on $\overline{\HH^*_F}$ is relevant in the calculus $\Psi^m_{h, \mathrm{BW}}(P)$. From now on, we drop the index $F$ for $\HH_F^*,\V_F^*$; unless explicitly mentioned, $\HH^*$ and $\V^*$ will always refer to $\HH_F^*$ and $\V_F^*$, respectively.

We will also further assume that a $G$-equivariant connection $\nabla$ has been chosen on $P$. This induces a connection on all associated bundles. In particular, we have a connection $\overline{\nabla}_{\mathbf{k}}$ on all $\mathbf{L}^{\otimes \mathbf{k}} \to F$ as defined in \eqref{equation:dynamical-lambda}. For each $m$, define the \emph{horizontal symbol space} $S^m_{h, \mathrm{BW}}(\HH^*)$ as the space of restrictions of symbols in $S^m_{h,\mathbf{k}}(T^*F)$ to $\HH^* \subset T^*F$.

\subsubsection{Relation to fiberwise Toeplitz quantization} We now discuss how the Borel-Weil quantization is related to a fiberwise Toeplitz quantization on the flag manifold. This is the content of the following:

\begin{theorem}
\label{theorem:penible}
	There exists $\psi \in C^\infty(P \times P)$ such that for all admissible $\mathbf{A}' \in \Psi_{h, \mathbf{k}}(F, \mathbf{L})$, $\mathbf{A} := \Pi_{\mathbf{k}} \mathbf{A}' \Pi_{\mathbf{k}} \in \Psi^m_{h, \mathrm{BW}}(P)$ can be written as
\begin{equation}
\label{equation:penible}
\begin{split}
\mathbf{A}f(w_{\ell}) & = \dfrac{1}{(2\pi h)^n} \int_M \int_{T^*_{x_\ell}M} e^{-\tfrac{i}{h}\xi\cdot\exp_{x_\ell}^{-1}(x_r)} \\
& \tau_{x_r \to x_\ell} \left(\Pi_{\mathbf{k}}(x_r) \left(\psi(w_\ell, \bullet) a_{h,\mathbf{k}}(w_{\ell}, \bullet, d\pi^{\top}\xi) f(\bullet)|_{F_{x_r}}\right)\right)(w_\ell) \, \dd \xi\dd x_r   + \mc{O}_{C^\infty}(h^\infty),
\end{split}
\end{equation}
where $a_{h,\mathbf{k}} \in S^m_{h,\mathbf{k}}(F \times \HH^*)$, satisfies for $w \in F, \xi \in T^*_xM$, $x=\pi(w)$,
\begin{equation}
\label{equation:symbol-expansion}
a_{h,\mathbf{k}}(w,w,d\pi^\top\xi) = \sigma_{\mathbf{A}'}|_{\HH^*}(w,d\pi^\top\xi) + \mc{O}_{S^{m-1}(\HH^*)}(h),
\end{equation}
and \eqref{equation:penible} holds for all $f \in C^\infty(F,\mathbf{L}^{\otimes \mathbf{k}})$. 
\end{theorem}

Conversely, we shall see in the next subsection how to define a quantization in the Borel-Weil calculus using an oscillatory integral as in \eqref{equation:penible}. We note that the expression under the integral sign \eqref{equation:penible} is \emph{not} automatically holomorphic in $w_\ell$, since $a_{h, \mathbf{k}}$ also depends on $w_\ell$.


\begin{proof}
Equality \eqref{equation:penible} is established using the stationary phase lemma. We identify operators and Schwartz kernels. Set $b_{\mathbf{k}} := d_{\mathbf{k}} \vol(G/T)^{-1}$. By \eqref{equation:partie-simple}, we can write locally near the diagonal
\begin{equation}
\label{equation:local-e}
\Pi_{\mathbf{k}}(w_\ell,w_r) = b_{\mathbf{k}} \mathbf{E}^{\otimes \mathbf{k}}(w_\ell,w_r)= b_{\mathbf{k}} e^{i \mathbf{k}\cdot\boldsymbol{\psi}(w_\ell,w_r)} \mathbf{s}^{\otimes \mathbf{k}}(w_\ell) \overline{\mathbf{s}}^{\otimes \mathbf{k}}(w_r),
\end{equation}
where $\mathbf{s}$ is a multisection of $\mathbf{L}$ defined locally with constant fiberwise norm equal to $1$. Let $\pi : F \to M$ be the footpoint projection. For $w \in F$, write $x=\pi(w)$.
 Also write:
\[
\begin{split}
&\mathbf{A}'(f \mathbf{s}^{\otimes \mathbf{k}})(w_\ell) \\
& = \dfrac{1}{(2\pi h)^{\dim F}}\int_{w, \xi} e^{-\tfrac{i}{h}\xi\cdot\exp^{-1}_{w_\ell}(w)} a(w_\ell,\xi) \tau_{w \to w_\ell}(f(w)\mathbf{s}^{\otimes \mathbf{k}}(w)) \, \dd w \dd \xi + \mc{O}_{C^\infty}(h^\infty),
\end{split}
\]
for some symbol $a \in S^{m}_{h,\mathbf{k}}(T^*F)$. (Here, we work in local coordinates close to $w_\ell$.) Define locally the real-valued (multi) $1$-form $\boldsymbol{\beta}$ by the identity $\overline{\nabla} \mathbf{s}(w) = i\boldsymbol{\beta}(w) \otimes \mathbf{s}(w)$ for $w \in F$. Then we have 
\begin{equation}\label{eq:parallel}
	\tau_{w \to w_\ell}(\mathbf{s}^{\otimes \mathbf{k}}(w)) = e^{-i\int_{w}^{w_\ell} \mathbf{k}\cdot \boldsymbol{\beta}} \mathbf{s}^{\otimes \mathbf{k}}(w_\ell).
\end{equation}
In particular, the Schwartz kernel of $\mathbf{A}'$ takes the form
\begin{equation}\label{eq:schwartz-kernel-A'}
\begin{split}
	\mathbf{A}'(w_\ell, w_r) = & (2\pi h)^{- \dim F} \int_{\xi \in T^*_{w_\ell}} e^{-\frac{i}{h} \xi \cdot \exp_{w_\ell}^{-1} (w_r)} e^{-i\int_{w_r}^{w_\ell} \mathbf{k}\cdot \boldsymbol{\beta}} a(w_\ell, \xi)\, d\xi\, \mathbf{s}^{\otimes \mathbf{k}}(w_\ell) \otimes \overline{\mathbf{s}}^{\otimes \mathbf{k}}(w_r) \\
	& \hspace{5cm} + \mc{O}_{C^\infty}(h^\infty).
	\end{split}
\end{equation}
In all the following computations, one would formally need to introduce cutoff functions for the various quantities to make sense; we refrain from doing so in order to keep the notation simple but the reader should bear this fact in mind.

Using \eqref{equation:commutation-holomorphic}, the Schwartz kernel of $\mathbf{A}$ at $(w_\ell,w_r) \in F^2$ satisfies
\[
 \mathbf{A}(w_\ell,w_r)=\Pi_{\mathbf{k}} \mathbf{A}' \Pi_{\mathbf{k}}(w_\ell,w_r) = \mathbf{A}' \Pi_{\mathbf{k}}(w_\ell,w_r) + \mc{O}(h^\infty)_{C^\infty(F\times F)}.
\]

Using \eqref{eq:parallel}, we thus obtain:
\begin{equation}
\label{equation:XX}
\begin{split}
 \mathbf{A}(w_\ell,w_r) &= \dfrac{b_{\mathbf{k}}}{(2\pi h)^{\dim F}} \int_{\xi \in T_{w_\ell}^*F,w \in F_{x_r}} e^{\tfrac{i}{h}( h\mathbf{k} \cdot (\boldsymbol{\psi}(w, w_r)-\int_{w}^{w_\ell} \boldsymbol{\beta})-\xi \cdot \exp_{w_\ell}^{-1}(w))} a(w_\ell,\xi)  \,  \dd w \dd \xi \\
& \hspace{2cm} \mathbf{s}^{\otimes \mathbf{k}}(w_\ell) \overline{\mathbf{s}}^{\otimes \mathbf{k}}(w_r) + \mc{O}_{C^\infty}(h^\infty).
\end{split}
\end{equation}
Indeed, observe that $\Pi_{\mathbf{k}}(w_1, w_2) = 0$ if $w_1$ and $w_2$ do not belong to the same fibers of $F \to M$, since locally $\Pi_{\mathbf{k}}$ is the tensor product of the identity on $M$ and the holomorphic projection in the fibers. This restricts $w$ to $F_{x_r}$. 
Define the phase
\[
\Omega_{w_\ell,w_r,\xi_{\HH^*}}(w,\xi_{\V^*}) := h\mathbf{k}\cdot\left(\boldsymbol{\psi}(w, w_r) - \int_{w}^{w_\ell}\boldsymbol{\beta}\right)-\xi \cdot \exp_{w_\ell}^{-1}(w).
\]
We shall write
\[
\xi \cdot \exp_{w_\ell}^{-1}(w) = \xi_{\HH^*} \cdot \pi_{\HH}  \exp_{w_\ell}^{-1}(w)  + \xi_{\V^*} \cdot \pi_{\V}  \exp_{w_\ell}^{-1}(w).
\]
Notice that $(w,\xi_{\V^*}) \mapsto \Omega_{w_\ell,w_r,\xi_{\HH^*}}(w,\xi_{\V^*})$ extends naturally to $\V^*_{\C}$ by simply allowing $\xi_{\V^*} \in \V^*_\C$ in the previous expression. (This corresponds to the analytic extension of $\Omega_{w_\ell,w_r,\xi_{\HH^*}}$ to $\V^*_{\C}$.) We also consider an \emph{almost analytic extension} of $a$ to $\V^*_{\C}$, see \cite[Section 1 and Theorem 1.3]{Melin-Sjostrand-75} for the definition. The critical point of the phase is denoted by $(w,\xi_{\V^*}) = Z(w_\ell,w_r,\xi_{\HH^*})$ and satisfies
\[
\pi_{\V}(\exp^{-1}_{w_\ell}(w)) = 0, \qquad d_{\V, w}\Omega = 0,
\]
where $d_{\V,w}$ stands for exterior derivative in the vertical directions in $w$. Also notice that for $w_\ell = w_r$, $\Omega_{w_\ell,w_r,\xi_{\HH^*}}(w_r,\xi_{\V^*})=0$ (in particular, the imaginary part of the phase vanishes).


\begin{lemma}
\label{lemma:critical-Omega}
One has:
\begin{equation*}
\begin{split}
	w &= \tau_{x_{\ell} \to x_r} w_\ell,\\
		\xi_{\V^*} \cdot \pi_{\V} (d \exp_{w_\ell})^{-1}(w) &= h \mathbf{k} \cdot (i \partial_{\V,L}\boldsymbol{\varphi}(w, w_r) + \alpha_{w_\ell}(w)) - \xi_{\HH^*} \cdot \pi_{\HH}  (d \exp_{w_\ell})^{-1}(w),
\end{split}
\end{equation*}
where $\alpha_{w_\ell}(w) \in \V^*(w)$ is the $1$-form defined as
\[
	\alpha_{w_\ell}(w)Z := \int_{0}^{d(w_\ell, w)} d \boldsymbol{\beta}_{\gamma(t)}(J(Z, t), \dot{\gamma}_0(t)) \, dt.
\]
Here, $t \mapsto \gamma(t)$ is the unit speed parametrization of the (unique) geodesic from $w_\ell$ to $w$, and, for $Z \in \V(w)$, $t \mapsto J(Z,t)$ is the Jacobi field along $\gamma(t)$ such that $J(Z,0) = 0$, $J(Z, d(w_\ell,w)) = Z$.

In particular, when $w_\ell = w_r$, we have
\[
	Z(w_\ell, w_\ell, \xi_{\HH^*}) = (w_\ell, 0), \quad \Omega_{w_\ell, w_\ell, \xi_{\HH^*}}(Z(w_\ell, w_r, \xi_{\HH^*})) = 0.
\]
\end{lemma}

Before proving this lemma, we note that \eqref{eq:xi_V} uniquely determines $\xi_{\V}^*$, because we may always assume that $\pi_{\V} (d \exp_{w_\ell})^{-1}(w): \V(w) \to \V(w_\ell)$ is an isomorphism (for $w_\ell$ and $w$ close enough). We first prove an auxiliary claim.

\begin{claim}
The following holds:
\begin{equation}
\label{equation:utile}
\begin{split}
&d_{\V} \boldsymbol{\psi}(w_\ell,w_r) = i(\partial_\ell \boldsymbol{\varphi}(w_\ell,w_r)+\overline{\partial}_r \boldsymbol{\varphi}(w_\ell,w_r)) + \boldsymbol{\beta}^{\V}(w_r) - \boldsymbol{\beta}^{\V}(w_\ell), \\
& d_{\V, \ell} \boldsymbol{\psi}(w_\ell,w_r) = i\partial_\ell \boldsymbol{\varphi}(w_\ell,w_r) - \boldsymbol{\beta}^{\V}(w_\ell).
\end{split}
\end{equation}
\end{claim}

Here, we denote by $\partial_{\ell}$ (resp. $\overline{\partial}_r$) the $(1,0)$ (resp. $(0,1)$) part of the vertical differential in the left (resp. right) variable. Also, $d_{\V, \ell}$ is the exterior vertical derivative in the left variable.

\begin{proof}
Differentiating $|\mathbf{E}(w_\ell,w_r)|^2 = e^{-\boldsymbol{\varphi}(w_\ell,w_r)}$, one finds
\begin{align*}
	d |\mathbf{E}|^2 = -d\varphi \cdot |\mathbf{E}|^2 = -(\partial_\ell + \overline{\partial}_\ell + \partial_r + \overline{\partial}_r + d_{\HH})\varphi \cdot |\mathbf{E}|^2.
\end{align*}
On the other hand, write $\nabla \mathbf{E} = \eta \otimes \mathbf{E}$ for some $1$-form $\eta$. In fact, using that $\mathbf{E} = e^{i \boldsymbol{\psi}(w_\ell, w_r)} \mathbf{s}(w_\ell) \overline{\mathbf{s}}(w_r)$, we get 
\[
	\eta(w_\ell, w_r) = i d\boldsymbol{\psi}(w_\ell, w_r) + i\boldsymbol{\beta}(w_\ell) - i\boldsymbol{\beta}(w_r). 
\]
Using that $\mathbf{E}$ is holomorphic in the left variable, antiholomorphic in the right variable, we get that the left $(0, 1)$ and the right $(1, 0)$ parts of $\eta$ are zero. Using that the induced connection on $\mathbf{L}^{\otimes \mathbf{k}} \otimes \overline{\mathbf{L}}^{\otimes \mathbf{k}}$ is unitary, we get $d|\mathbf{E}|^2 = (\eta + \overline{\eta}) \otimes |\mathbf{E}|^2$, and therefore
\[
	-\partial_\ell \boldsymbol{\varphi} = \eta^{(1, 0), \ell}, \quad -\overline{\partial}_r \boldsymbol{\varphi} = \eta^{(0, 1), r}, \quad -d_{\HH} \boldsymbol{\varphi} = 2\Re \eta^{\HH}.
\]
It follows that
\[
	-(\partial_\ell + \overline{\partial}_r) \varphi = \eta^{\V} = i d\boldsymbol{\psi}^{\V}(w_\ell, w_r) + i\boldsymbol{\beta}^{\V}(w_\ell) - i\boldsymbol{\beta}^{\V}(w_r).
\]
Then
\begin{equation}
\label{equation:prout}
D^{\mathrm{Chern}}_{\V} \mathbf{E} = - (\partial_\ell \boldsymbol{\varphi}+\overline{\partial}_r\boldsymbol{\varphi}) \otimes \mathbf{E}.
\end{equation}

\end{proof}

\begin{proof}[Proof of Lemma \ref{lemma:critical-Omega}]
Since $\pi: F \to M$ is a Riemannian submersion, horizontal lifts of geodesics in $M$ are geodesics in $F$ and, conversely, geodesics in $F$ with initial speed in the horizontal direction are horizontal curves which project to geodesics in $M$. (We note that $\pi$ projects geodesics to geodesics in this setting if and only if the curvature of $F$ is zero, see \cite{Vilms-70}.) This implies:
\[
w = \tau_{x_\ell \to x_r} w_\ell.
\]
Next, we claim that 
\begin{equation}\label{eq:xi_V}
	\xi_{\V^*} \cdot \pi_{\V} (d \exp_{w_\ell})^{-1}(w) = h \mathbf{k} \cdot (d_{\V, L} \boldsymbol{\psi} + \boldsymbol{\beta}^{\V}(w) + \alpha_{w_\ell}(w)) - \xi_{\HH}^* \cdot \pi_{\HH}  (d \exp_{w_\ell})^{-1}(w).
\end{equation}

To prove the formula \eqref{eq:xi_V}, consider a parametrisation $(s, t) \mapsto \gamma(s, t)$, where $\gamma(s, \bullet)$ is the unit speed geodesic starting at $w_\ell$ and ending at $\exp_{w}(tZ)$. Here $s \in (-\varepsilon, \varepsilon)$ for some $\varepsilon > 0$ and for $r > 0$ small enough we set
\[
	(s, t) \in D_r := \{(s, t) \mid 0 \leq s \leq r,\,\, 0 \leq t \leq d(w_\ell, \exp_{w}(rZ))\}.
\]
By Stokes' formula, we then get
\[
	\int_{w_\ell}^{\exp_{w}(rZ)} \boldsymbol{\beta} = \int_{\gamma(D_r)} d\boldsymbol{\beta} + \int_{w_\ell}^{w} \boldsymbol{\beta} + \int_{w}^{\exp_{w} (tZ)} \boldsymbol{\beta}.
\]
(Formally, by $\int_{\gamma(D_r)}$ we mean the integral of the pullback by $\gamma$ over $D_r$.) Using the parametrization $\gamma$, we get
\begin{align*}
	&\partial_r|_{r = 0} \int_{w_\ell}^{\exp_{w}(rZ)} \boldsymbol{\beta}\\ 
	&= \boldsymbol{\beta}(w)(Z) + \partial_r|_{r = 0} \int_{s = 0}^r \int_{t = 0}^{d(w_\ell, \exp_{w}(sZ))} d\boldsymbol{\beta}(\gamma(s, t))(\partial_t \gamma(s, t), \partial_s \gamma(s, t))\, ds dt\\
	&= \boldsymbol{\beta}(w)(Z) + \int_0^{d(w_\ell, w)} d\boldsymbol{\beta}(\gamma(t)) (\dot{\gamma}(t), \partial_s \gamma(0, t))\, dt.
\end{align*}
Since $\partial_s \gamma(0, t) = J(Z, t)$ by definition of Jacobi fields, this computes the first part of \eqref{eq:xi_V}. The derivative of $\xi \cdot \exp_{w_\ell}^{-1}(w)$ is immediately computed, and this proves \eqref{eq:xi_V}. 

Furthermore, by \eqref{equation:utile}, $d_{\V, \ell} \boldsymbol{\psi} + \boldsymbol{\beta}^{\V}(w) = i \partial_{\V,\ell}\boldsymbol{\varphi}$, which completes the proof. (As a side remark, we note that $d\boldsymbol{\beta} = -i \mathbf{F}_{\overline{\nabla}}$, and by Lemma \ref{lemma:vanishing-curvature} this vanishes on $\HH \times \V$. However, the Jacobi field $J(Z, t)$ need not not be vertical if the curvature is non-zero.)

\end{proof}

Observe then that
\[
\begin{split}
\Omega_{w_\ell,w_r,\xi_{\HH^*}}&(Z(w_\ell,w_r,\xi_{\HH^*}))\\
& = h \mathbf{k} \cdot \left(\boldsymbol{\psi}(\tau_{x_\ell \to x_r}w_\ell,w_r) - \int_{\tau_{x_\ell \to x_r} w_\ell}^{w_\ell} \boldsymbol{\beta}\right) - \xi_{\HH^*}\cdot(\exp^{-1}_{w_\ell}(\tau_{x_\ell \to x_r} w_\ell)).
\end{split}
\]
The last term simplifies to
\[
\xi_{\HH^*}\cdot\exp^{-1}_{w_\ell}(\tau_{x_\ell \to x_r}w_\ell) = \xi\cdot \exp^{-1}_{x_\ell}(x_r),
\]
where $\xi \in T^*_{x_\ell}M$ is such that $\xi_{\HH^*} = d\pi^\top \xi$. In what follows, we write $\mathrm{Hess}$ for the Hessian matrix (defined at a critical point).

\begin{claim}\label{claim:hessian}There exists a smooth, non-zero function $G = G(w_\ell, w_r)$ defined near $w_\ell = w_r$ such that
\[
	\det \mathrm{Hess} (\Omega_{w_\ell, w_r,\xi_{\HH^*}})(Z(w_\ell,w_r,\xi_{\HH^*})) = G(w_\ell, w_r).
\]
Moreover, $G(w_\ell, w_\ell) = 1$.
\end{claim}

\begin{proof}
Given $\xi \in \V_F^*(w_\ell)$ and $W \in \V_F(w)$, at $(w, \xi_{\V}^*) = Z$ we may compute
\begin{align*}
	 \mathrm{Hess} (\Omega_{w_\ell, w_r,\xi_{\HH^*}})(Z(w_\ell,w_r,\xi_{\HH^*}))(W, \xi) &= \frac{\partial^2}{\partial s \partial t} \Omega_{w_\ell, w_r, \xi_{\HH}^*}(\exp_{w}(sW), \xi_{\V}^* + t\xi)\big|_{s = t = 0}\\ 
	&= -\xi \cdot \pi_{\V} \circ (d\exp_{w_\ell}(\exp_{w_\ell}^{-1}(w)))^{-1} (W).
\end{align*}
Using that $\Omega_{w_\ell, w_r, \xi_{\HH}^*}(w, \xi_{\V}^*)$ is linear in $\xi_{\V}^*$, we get that the second derivative in this variable vanishes. Since $\pi_{\V} \circ d\exp_{w_\ell}(\exp_{w_\ell}^{-1}(w)))^{-1}: \V(w) \to \V(w_\ell)$ is an isomorphism, this proves the first claim. For the second claim, we have $Z(w_\ell, w_r, \xi_{\HH}^*) = (w_\ell, 0)$, so the preceding formula simplifies to
\[
	\mathrm{Hess} (\Omega_{w_\ell, w_r,\xi_{\HH^*}})(Z(w_\ell,w_r,\xi_{\HH^*}))(W, \xi) = -\xi \cdot W.
\]
The conclusion follows, i.e. $G(w_\ell, w_\ell) = 1$.
\end{proof}

At $w_\ell = w_r$, the critical point of the phase is given by $w = w_\ell = w_r$ and $\xi_{\V^*}=0$. Moreover, the imaginary part of the phase vanishes at this point. As a consequence, we can apply the complex stationary phase lemma (see \cite[Theorem 2.3]{Melin-Sjostrand-75}) to get that, near $w_\ell = w_r$,
\begin{equation}
\label{equation:temporaire}
\begin{split}& \mathbf{A}(w_\ell,w_r) \\
&  = \dfrac{b_{\mathbf{k}}}{(2\pi h)^{\dim F}} \int_{\xi_{\HH^*} \in \HH^*_{w_\ell}} e^{\tfrac{i}{h}\Omega_{w_\ell,w_r,\xi_{\HH^*}}(Z(w_\ell,w_r,\xi_{\HH^*}))}  (2\pi h)^{2\dim(G/T)/2} p_{h,\mathbf{k}}(w_\ell,w_r,\xi_{\HH^*})\, \dd \xi_{\HH^*} \\
&  \hspace{2cm} \mathbf{s}^{\otimes \mathbf{k}}(w_\ell) \overline{\mathbf{s}}^{\otimes \mathbf{k}}(w_r)+ \mc{O}_{C^\infty}(h^\infty) \\
& = \dfrac{b_{\mathbf{k}}}{(2\pi h)^{\dim M}} \int_{\xi \in T^*_{x_\ell}M} e^{-\tfrac{i}{h} \xi\cdot \exp^{-1}_{x_\ell}(x_r)}  p_{h,\mathbf{k}}(w_\ell,w_r,d\pi^\top\xi)\, \dd \xi \\
& \hspace{2cm} e^{i \mathbf{k} \cdot (\boldsymbol{\psi}(\tau_{x_{\ell}\to x_r}w_\ell,w_r)- \int_{\tau_{x_\ell \to x_r} w_\ell}^{w_\ell} \boldsymbol{\beta})} \mathbf{s}^{\otimes \mathbf{k}}(w_{\ell}) \overline{\mathbf{s}}^{\otimes \mathbf{k}}(w_r)+ \mc{O}_{C^\infty}(h^\infty) \\
&= \dfrac{b_{\mathbf{k}}}{(2\pi h)^{\dim M}} \int_{\xi \in T^*_{x_\ell}M} e^{-\tfrac{i}{h} \xi\cdot \exp^{-1}_{x_\ell}(x_r)}  p_{h,\mathbf{k}}(w_\ell,w_r,d\pi^\top\xi)\, \dd \xi \\
& \hspace{2cm} e^{i \mathbf{k} \cdot \boldsymbol{\psi}(\tau_{x_{\ell}\to x_r}w_\ell,w_r)} \tau_{(\tau_{x_\ell \to x_r}w_\ell) \to w_\ell }\mathbf{s}^{\otimes \mathbf{k}}(\tau_{x_\ell \to x_r}w_{\ell}) \overline{\mathbf{s}}^{\otimes \mathbf{k}}(w_r)+ \mc{O}_{C^\infty}(h^\infty) \\
& =  \dfrac{1}{(2\pi h)^{\dim M}} \int_{\xi \in T^*_{x_\ell}M} e^{-\tfrac{i}{h} \xi\cdot \exp^{-1}_{x_\ell}(x_r)}  p_{h,\mathbf{k}}(w_\ell,w_r,d\pi^\top\xi)\, \dd \xi ~ \tau^{\ell}_{x_r \to x_\ell} \left(\Pi_{\mathbf{k}}(\bullet, w_r)|_{x_r}\right)(w_\ell)\\
& \hspace{2cm} + \mc{O}_{C^\infty}(h^\infty).
\end{split}
\end{equation}
Here, in the last line we used \eqref{equation:local-e}, as well as that for sections $S$ of $H^0(F, \Lk) \to M$, we have for $w \in F_y$
\[
	\tau_{x \to y} (S|_x) (w) = \tau_{(\tau_{y \to x}w) \to w} S(\tau_{y \to x}w).
\]
Also, here $p_{h,\mathbf{k}} \in S^m_{h,\mathbf{k}}(F \times \HH^*)$ is a symbol such that
\[
p_{h,\mathbf{k}}(w_\ell,w_r,\xi_{\HH^*}) = |\det \mathrm{Hess}(\Omega_{w_\ell,w_r,\xi_{\HH^*}})(Z)|^{-1/2} a^{(m)}(w_\ell,w_r,\xi_{\HH^*})+ \mc{O}_{S^{m-1}(F \times \HH^*)}(h),
\]
and
\[
a^{(m)}(w_\ell,w_r,\xi_{\HH^*}) :=  a(w_\ell,\xi_{\HH^*},\xi_{\V^*}(Z(w_\ell,w_r,\xi_{\HH^*}))) \in S^m_{h,\mathbf{k}}(F \times \HH^*),
\]
and $\tau^\ell_{x_\ell \to x_r}$ denotes parallel transport in the left variable. The remainder term in $p_{h,\mathbf{k}}$ is indeed in $S^{m-1}(F \times \HH^*)$ because in the remainder term of the complex stationary phase lemma, using that $a(w_\ell, \xi)$ does not depend on $w$, there is at least one derivative in the direction $\V$ (which lowers the order by one). We emphasize that this is only valid for $w_\ell$ close to $w_r$, that is close to the diagonal in $F \times F$. As \eqref{equation:temporaire} corresponds exactly to \eqref{equation:penible}, and this concludes the proof; here, we use that 
\begin{align*}
	&\int_{w_r \in F_{x_r}} a_{h,\mathbf{k}}(w_{\ell}, w_r, d\pi^{\top}\xi) \tau_{x_r \to x_\ell}^{\ell} (\Pi_{\mathbf{k}} (\bullet, w_r)|_{x_r}) (w_\ell) f (w_r)\, \dd w_r\\ 
	&= \tau^{\ell}_{(\tau_{x_\ell \to x_r}w_\ell) \to w_\ell } \int_{w_r \in F_{x_r}} a_{h,\mathbf{k}}(w_{\ell}, w_r, d\pi^{\top}\xi) \Pi_{\mathbf{k}} (\tau_{x_\ell \to x_r} w_\ell, w_r) f (w_r)\, \dd w_r\\
	&= \tau^{\ell}_{(\tau_{x_\ell \to x_r}w_\ell) \to w_\ell } \left(\Pi_{\mathbf{k}}(a_{h,\mathbf{k}}(w_{\ell}, \bullet, d\pi^{\top}\xi) f(\bullet)) (\tau_{x_\ell \to x_r} w_\ell)\right)\\
	&= \tau_{x_r \to x_\ell} \Pi_{\mathbf{k}}(a_{h,\mathbf{k}}(w_{\ell}, \bullet, d\pi^{\top}\xi) f(\bullet)|_{F_{x_r}}) (w_\ell).
\end{align*}

Since the determinant of the Hessian matrix at $w_\ell = w_r$ is equal to $1$ by Claim \eqref{claim:hessian}, we obtain
\begin{equation}\label{eq:blablabla}
	p_{h,\mathbf{k}}(w_\ell, w_\ell, \xi_{\HH^*}) = a(w_\ell, \xi_{\HH^*}) + \mc{O}_{S^{m-1}(\HH^*)}(h) = \sigma_{\mathbf{A}'}|_{\HH^*}(w_\ell,\xi_{\HH^*})+ \mc{O}_{S^{m-1}(\HH^*)}(h).
\end{equation}
This proves \eqref{equation:symbol-expansion}.
\end{proof}

\begin{remark}
Observe that $\overline{\partial}_{\V,\ell} \mathbf{A}(w_\ell,w_r) = 0 = \partial_{\V,r} \mathbf{A}(w_\ell,w_r)$, since $\mathbf{A} = \Pi_{\mathbf{k}}\mathbf{A}'\Pi_{\mathbf{k}}$. By Proposition \ref{proposition:parallel-transport-fibrewise-holomorphic}, $\tau^{\ell}_{x_\ell \to x_r}$ preserves fiberwise holomorphicity (in the left variable); we thus find that on the diagonal
\begin{equation}
\label{equation:catue}
\overline{\partial}_{\V,\ell} p_{h,\mathbf{k}}(w_\ell, w_\ell, \xi_{\HH^*}) = \mc{O}_{S^{m - 1}(\HH^*)}(h) = \partial_{\V,r} p_{h,\mathbf{k}}(w_\ell, w_\ell, \xi_{\HH^*}).
\end{equation}
Indeed, applying a holomorphic covariant derivative $(D^{\mathrm{Chern}}_{\V})_{\overline{X}_\ell}$ to the Schwartz kernel of $\mathbf{A}$ in \eqref{equation:temporaire}, we obtain
\[
\begin{split}
& \dfrac{1}{(2\pi h)^{\dim M}} \int_{\xi \in T^*_{x_\ell}M} e^{-\tfrac{i}{h} \xi\cdot \exp^{-1}_{x_\ell}(x_r)}  \overline{X}_\ell p_{h,\mathbf{k}}(w_\ell,w_r,d\pi^\top\xi) \dd \xi ~ \tau^{\ell}_{x_r \to x_\ell} \left(\Pi_{\mathbf{k}}(\bullet, w_r)|_{x_r}\right) \\
& \hspace{2cm} = \mc{O}_{C^\infty}(h^\infty),
\end{split}
\]
which then implies by Fourier inversion the first equality in \eqref{equation:catue}. (The second equality is obtained similarly.)
\end{remark}

\begin{remark}
The function $a^{(m)}$ is obtained modulo $\mc{O}_{S^{m-1}}(h)$ locally near the diagonal as the unique (local) extension of $w \mapsto \sigma_{\mathbf{A}}^{\mathrm{BW}}(w,\xi_{\HH^*})$ which is holomorphic (resp. antiholomorphic) in the left (resp. right) variable, as in the proof of Theorem \ref{theorem:bergman}.
\end{remark}

More generally, it is possible to obtain a version of the preceding theorem for non-admissible operators.
\begin{theorem}
	With the notation of Theorem \ref{theorem:penible}, for \emph{any} $\mathbf{A}' \in \Psi^m_{h, \mathbf{k}}(F, \Lk)$, there exists $a_{h, \mathbf{k}} \in S^m(F \times \HH^*)$ such that
\begin{equation}
\label{equation:penible2}
\begin{split}
\mathbf{A}f(w_{\ell}) & = \dfrac{1}{(2\pi h)^n} \int_{x_2 \in M} \int_{T^*_{x_\ell}M} e^{-\tfrac{i}{h}\xi\cdot\exp_{x_\ell}^{-1}(x_2)} \Pi_{\mathbf{k}} \big(\tau_{x_2 \to x_\ell} \\
& \hspace{3cm} \Pi_{\mathbf{k}} \big(\psi(w_\ell, \bullet ) a_{h,\mathbf{k}}(w_1, \bullet, d\pi^{\top}\xi)  f(\bullet)|_{F_{x_2}})\big)\big)(w_\ell) \, \dd \xi \dd x_2   + \mc{O}_{C^\infty}(h^\infty).
\end{split}
\end{equation}
\end{theorem}
\begin{proof}
	Following the proof of Theorem \ref{theorem:penible}, we may write
	\begin{equation*}
\begin{split}
 \mathbf{A}(w_\ell,w_r) &= \dfrac{b_{\mathbf{k}}^2}{(2\pi h)^{\dim F}} \int_{\xi \in T_{w_1}^*F, w_1 \in F_{x_\ell}, w_2 \in F_{x_r}} e^{\tfrac{i}{h}( h\mathbf{k} \cdot (\boldsymbol{\psi}(w_\ell, w_1)+ \boldsymbol{\psi}(w_2, w_r) - \int_{w_2}^{w_1} \boldsymbol{\beta}) - \xi \cdot \exp_{w_1}^{-1}(w_2))}\\
& \hspace{2cm} a(w_1, \xi)  \,  \dd w_1 \dd w_2 \dd \xi\, \mathbf{s}^{\otimes \mathbf{k}}(w_\ell) \overline{\mathbf{s}}^{\otimes \mathbf{k}}(w_r) + \mc{O}_{C^\infty}(h^\infty).
\end{split}
	\end{equation*}
	The proof then follows by applying the complex stationary phase lemma as before with the phase 
	\[
		\Omega_{w_1, w_r, \xi_{\HH^*}} (w_2, \xi_{\V^*}) =  h\mathbf{k} \cdot \left(\boldsymbol{\psi}(w_2, w_r) - \int_{w_2}^{w_1} \boldsymbol{\beta}\right) - \xi \cdot \exp_{w_1}^{-1}(w_2).
	\]
\end{proof}


\subsubsection{Properties of $\Psi^\bullet_{h,  \mathrm{BW}}(P)$}

\label{sssection:properties-bw}


We now verify that operators in $\Psi^m_{h, \mathrm{BW}}(P)$ satisfy all required properties of semiclassical pseudodifferential calculus (algebra, invertibility of elliptic elements, etc.). 

\begin{lemma}
$\Psi^\bullet_{h, \mathrm{BW}}(P)$ is an algebra.
\end{lemma}

\begin{proof}
Let $\mathbf{A} := \Pi_{\mathbf{k}}\mathbf{A}' \Pi_{\mathbf{k}} \in \Psi^m_{h, \mathrm{BW}}(P)$, $\mathbf{B} := \Pi_{\mathbf{k}}\mathbf{B}' \Pi_{\mathbf{k}} \in \Psi^{m'}_{h,\mathrm{BW}}(P)$, with $\mathbf{A}',\mathbf{B}' \in \Psi^m_{h,\mathbf{k}}(F,\mathbf{L})$ admissible. Then, using \eqref{equation:commutation-holomorphic},
\[
\mathbf{A}\mathbf{B} = \Pi_{\mathbf{k}} \underbrace{\mathbf{A}' (\mathbf{B}' +[\Pi_{\mathbf{k}},\mathbf{B}'])}_{\in \Psi^{m+m'}_{h,\mathbf{k}}(F,\mathbf{L})}\Pi_{\mathbf{k}} \in \Psi^{m+m'}_{h,\mathrm{BW}}(P),
\]
by Lemma \ref{lemma:eat-pik}, which completes the proof.
\end{proof}


Let $\mathbf{I} \in \Psi^m_{h,\mathbf{k}}(F,\mathbf{L})$ be an operator such that $\mathbf{I} \equiv \mathbbm{1}$ microlocally in a conic neighborhood $\mc{C}$ of $\overline{\HH^*} \subset \overline{T^*F} = T^*F \cup \partial_\infty T^*F$, that is $\WF(\mathbf{I}-\mathbbm{1}) \subset \complement(\mc{C})$ (where $\complement$ denotes complement), and $\mathbf{I}^* = \mathbf{I}$. 

\begin{lemma}
\label{lemma:i}
One has $\Pi_{\mathbf{k}}\mathbf{I} \Pi_{\mathbf{k}} \in \Psi^0_{h, \mathrm{BW}}(P)$ and $\Pi_{\mathbf{k}}\mathbf{I} \Pi_{\mathbf{k}} - \mathbbm{1} \in h^{\infty}\Psi^{-\infty}_{h,\mathrm{BW}}(P)$. \end{lemma}

\begin{proof}
By definition, that $\Pi_{\mathbf{k}}\mathbf{I} \Pi_{\mathbf{k}} \in \Psi^0_{h, \mathrm{BW}}(P)$ amounts to showing that the following holds:
\begin{equation}
\label{equation:plaf}
[\mathbf{I}, \Pi_{\mathbf{k}}], \quad \Pi_{\mathbf{k}}\mathbf{I}(\mathbbm{1}-\Pi_{\mathbf{k}}), \quad (\mathbbm{1}-\Pi_{\mathbf{k}}) \mathbf{I}~\Pi_{\mathbf{k}} \quad \in h^\infty \Psi^{-\infty}_{h,\mathbf{k}}(F,\mathbf{L}). 
\end{equation}
By \S\ref{sssection:examples}, Item (iv), the operator $h\overline{\partial}_{\mathbf{k}}$ is elliptic outside of $\HH^*$. Hence, by Proposition \ref{proposition:ellipticity}, there exists an operator $\mathbf{B} \in \Psi^{-1}_{h,\mathbf{k}}(F,\mathbf{L})$ such that $\mathbf{B} h\overline{\partial}_{\mathbf{k}} = \mathbbm{1}-\mathbf{I} + \mc{O}_{\Psi^{-\infty}_{h,\mathbf{k}}(F,\mathbf{L})}(h^\infty)$. Composing on the right by $\Pi_{\mathbf{k}}$ and on the left by $\mathbbm{1}-\Pi_{\mathbf{k}}$, we therefore obtain $(\mathbbm{1}-\Pi_{\mathbf{k}}) \mathbf{I}~\Pi_{\mathbf{k}}= \mc{O}_{\Psi^{-\infty}_{h,\mathbf{k}}(F,\mathbf{L})}(h^\infty)$. Taking the adjoint, we also get $\Pi_{\mathbf{k}}\mathbf{I}(\mathbbm{1}-\Pi_{\mathbf{k}})= \mc{O}_{\Psi^{-\infty}_{h,\mathbf{k}}(F,\mathbf{L})}(h^\infty)$. Finally, this implies $[\mathbf{I}, \Pi_{\mathbf{k}}]\in h^\infty \Psi^{-\infty}_{h,\mathbf{k}}(F,\mathbf{L})$, which proves \eqref{equation:plaf}.

Note that the proof also gives
\begin{equation}
\label{equation:useful}
\mathbf{I}~\Pi_{\mathbf{k}} = \Pi_{\mathbf{k}} + \mc{O}_{\Psi^{-\infty}_{h,\mathbf{k}}(F,\mathbf{L})}(h^\infty), \qquad  \Pi_{\mathbf{k}}\mathbf{I} = \Pi_{\mathbf{k}} + \mc{O}_{\Psi^{-\infty}_{h,\mathbf{k}}(F,\mathbf{L})}(h^\infty).
\end{equation}
In particular, this gives $\Pi_{\mathbf{k}}\mathbf{I} \Pi_{\mathbf{k}} = \mathbbm{1} + \mc{O}_{\Psi^{-\infty}_{h,\mathrm{BW}}(P)}(h^{\infty})$ (on fibrewise holomorphic sections).
\end{proof}

\begin{definition}[Symbols, ellipticity, and wavefront set]\label{definition:BW-stuff}
Let $\mathbf{A} := \Pi_{\mathbf{k}}\mathbf{A}' \Pi_{\mathbf{k}} \in \Psi^m_{h, \mathrm{BW}}(P)$. The principal symbol $\sigma_{\mathbf{A}}^{\mathrm{BW}} \in S^m_{h, \mathrm{BW}}(\HH^*)$ in the calculus $\Psi^\bullet_{h, \mathrm{BW}}(P)$ is defined as the \emph{restriction} of the principal symbol of $\sigma_{\mathbf{A}'}$ in the calculus $\Psi^{\bullet}_{h,\mathbf{k}}(F,\mathbf{L}^{\otimes \mathbf{k}})$ to the subbundle $\HH^* \subset T^*F$. The elliptic set of $\mathbf{A}$ in $\Psi^\bullet_{h, \mathrm{BW}}(P)$ is defined as $\Ell^{\mathrm{BW}}(\mathbf{A}) := \Ell(\mathbf{A}') \cap \overline{\HH^*}$, while the wavefront set is defined to be $\WF^{\mathrm{BW}}(\mathbf{A}) := \WF(\mathbf{A}') \cap \overline{\HH^*}$.
\end{definition}

We need to show that the principal symbol is well-defined, that is independent of the operator $\mathbf{A}'$.

\begin{lemma}
	The principal symbol $\sigma^{\mathrm{BW}}_{\mathbf{A}}$ is well-defined.
\end{lemma}

\begin{proof}
	By definition, it suffices to take two admissible operators $\mathbf{A}_1$ and $\mathbf{A}_2$ such that $\Pi_{\mathbf{k}} \mathbf{A}_1 \Pi_{\mathbf{k}} = \Pi_{\mathbf{k}} \mathbf{A}_2 \Pi_{\mathbf{k}}$, and show that $\sigma_{\mathbf{A}_1}^{\mathrm{BW}} - \sigma_{\mathbf{A}_2}^{\mathrm{BW}} = \mc{O}_{S^{m-1}(\HH^*)}(h)$. By Theorem \ref{theorem:penible} applied to the Schwartz kernels of $\Pi_{\mathbf{k}} \mathbf{A}_i \Pi_{\mathbf{k}}$, we find:
	\[
\begin{split}
& \dfrac{1}{(2\pi h)^{\dim M}} \int_{\xi \in T^*_{x_\ell}M} e^{-\tfrac{i}{h} \xi\cdot \exp^{-1}_{x_\ell}(x_r)} (\widetilde{a}_1 - \widetilde{a}_2)(w_\ell,w_r,d\pi^\top\xi) \dd \xi ~ \tau^{\ell}_{x_r \to x_\ell} \left(\Pi_{\mathbf{k}}(\bullet, w_r)|_{x_r}\right) \\
& \hspace{2cm} = \mc{O}_{C^\infty}(h^\infty),
\end{split}
	\]
	for some $\widetilde{a}_1, \widetilde{a}_2 \in S^\bullet(F \times \HH^*)$ such that $\widetilde{a}_i(w, w, \xi)  = \sigma|_{\mathbf{A}_i}|_{\HH^*}(w,\xi) + \mc{O}_{S^{m-1}(\HH^*)}(h)$. Now by Fourier inversion, this implies that the symbol $b(w_\ell,\xi) := \widetilde{a}_1(w_\ell,w_\ell,\xi) - \widetilde{a}_2(w_\ell,w_\ell,\xi)$ satisfies $b= \mc{O}_{S^{m-1}(\HH^*)}(h)$. This proves the claim.
\end{proof}

Finally, the following lemma shows that elliptic elements of $\Psi^\bullet_{h, \mathrm{BW}}(P)$ are also invertible among $\Psi^\bullet_{h, \mathrm{BW}}(P)$.

\begin{lemma}
\label{lemma:parametrix-bw}
Let $\mathbf{A} \in \Psi^m_{h, \mathrm{BW}}(P)$ such that $\Ell^{\mathrm{BW}}(\mathbf{A}) = \overline{\HH^*}$. Then, there exists $\mathbf{B} \in \Psi^{-m}_{h, \mathrm{BW}}(P)$ such that $\mathbf{A}\mathbf{B} = \mathbbm{1} + \mc{O}_{\Psi^{-\infty}_{h,\mathrm{BW}}(P)}(h^\infty)$.
\end{lemma}

In particular, one can invert $\mathbf{A}\mathbf{B} = \mathbbm{1} + \mc{O}_{\Psi^{-\infty}_{h,\mathrm{BW}}(P)}(h^\infty)$ by Neumann series for $h$ small enough, and construct $\mathbf{B}' \in \Psi^{-m}_{h,\mathrm{BW}}(P)$ such that $\mathbf{A}\mathbf{B}' = \mathbbm{1}$. This proves that the inverse of elliptic operators belongs to the calculus $\Psi^\bullet_{h,\mathrm{BW}}(P)$. To avoid cumbersome notation in the proof, we will use $\equiv$ to denote equality modulo a term in $\mc{O}_{\Psi^{-\infty}_{h,\mathbf{k}}(F,\mathbf{L})}(h^\infty)$.

\begin{proof}
Let $\mathbf{A} = \Pi_{\mathbf{k}} \mathbf{A}' \Pi_{\mathbf{k}} \in \Psi^m_{h, \mathrm{BW}}(P)$ with $\mathbf{A}' \in \Psi^m_{h,\mathbf{k}}(F,\mathbf{L})$ admissible, and with elliptic principal symbol on $\overline{\mathbb{H}^*}$. Since the elliptic set is open, $\sigma_{\mathbf{A}'}$ is elliptic on a (conic) open neighborhood of $\overline{\mathbb{H}^*}$. Let $\mathbf{I} \in \Psi^0_{h,\mathbf{k}}(F,\mathbf{L})$ be an admissible pseudodifferential operator as above, which is microlocally equal to the identity on a neighborhood of $\overline{\HH^*}$ and further assume that $\mathbf{I}$ has small enough microsupport, that is $\WF(\mathbf{I}) \subset \Ell(\mathbf{A}')$. To avoid cumbersome notation, with use the letter $\mathbf{A}$ in place of $\mathbf{A}'$ in the rest of this proof.

By Proposition \ref{proposition:ellipticity}, there exists $\mathbf{B}', \mathbf{B} \in \Psi^{-m}_{h,\mathbf{k}}(F,\mathbf{L})$ such that $\mathbf{A}\mathbf{B}\equiv \mathbf{B}'\mathbf{A} \equiv \mathbf{I}$. We now verify that $\mathbf{B}$ satisfies \eqref{equation:commutation-holomorphic}. We compute
\begin{equation}
\label{eq:000}
	\mathbf{A} \Pi_{\mathbf{k}} \mathbf{B} \equiv \Pi_{\mathbf{k}} \mathbf{A}\mathbf{B} \equiv \Pi_{\mathbf{k}}~\mathbf{I} \equiv \Pi_{\mathbf{k}} \equiv \mathbf{I}~\Pi_{\mathbf{k}} \equiv \mathbf{A}\mathbf{B} \Pi_{\mathbf{k}},
\end{equation}
where we used \eqref{equation:commutation-holomorphic} for $\mathbf{A}$ in the first equivalence, $\mathbf{A} \mathbf{B} \equiv \mathbf{I}$ in the second and last ones, and \eqref{equation:useful} elsewhere. Then we compute
\begin{equation}
\label{equation:step0}
\mathbf{I} ~\Pi_{\mathbf{k}}\mathbf{B} \equiv \mathbf{B}'\mathbf{A}\Pi_{\mathbf{k}}\mathbf{B} \equiv \mathbf{B}' \mathbf{A}\mathbf{B}\Pi_{\mathbf{k}} \equiv  \mathbf{I}\mathbf{B}\Pi_{\mathbf{k}},
\end{equation}
where we used $\mathbf{B}' \mathbf{A} \equiv \mathbf{I}$ in the first and last equivalences, and \eqref{eq:000} in the middle one. This brings us to:
\begin{equation}
\label{equation:first-step}
\Pi_{\mathbf{k}}\mathbf{B} \equiv \Pi_{\mathbf{k}}\mathbf{B}\Pi_{\mathbf{k}},
\end{equation}
where we multiplied \eqref{equation:step0} by $\Pi_{\mathbf{k}}$ on the left and used \eqref{equation:useful}.

Observe that $\mathbf{B}'\mathbf{A}\mathbf{B} \equiv \mathbf{I} \mathbf{B} \equiv \mathbf{B}'\mathbf{I}$. Since $\mathbf{I}$ is microlocally equal to $\mathbbm{1}$ near $\overline{\HH^*}$, this implies that $\mathbf{B}$ is microlocally equal to $\mathbf{B}'$ near $\overline{\HH^*}$, that is $\WF(\mathbf{B}-\mathbf{B}') \cap \overline{\HH^*} = \emptyset$. Hence, writing $\mathbf{R} := \mathbf{B}' - \mathbf{B} \in \Psi^{-m}_{h, \mathbf{k}}(F, \mathbf{L})$, with $\WF(\mathbf{R})$ disjoint from $\overline{\HH^*}$, we have $\mathbf{B}\mathbf{A} \equiv \mathbf{I}-\mathbf{R}\mathbf{A}$. Using again Proposition \ref{proposition:ellipticity} and the ellipticity of $h\overline{\partial}_{\mathbf{k}}$ outside of $\overline{\HH^*}$, there exists $\mathbf{Q} \in \Psi^{-1}_{h,\mathbf{k}}(F,\mathbf{L})$ such that $\mathbf{R}\mathbf{A} \equiv \mathbf{Q} h\overline{\partial}_{\mathbf{k}}$, and therefore
\[
	\mathbf{R}\mathbf{A} \Pi_{\mathbf{k}} \equiv \mathbf{Q} h\overline{\partial}_{\mathbf{k}} \Pi_{\mathbf{k}} \equiv 0.
\]
Similarly, $\Pi_{\mathbf{k}}\mathbf{R}\mathbf{A} \equiv 0$ by using that $(\Pi_{\mathbf{k}}\mathbf{R}\mathbf{A})^* = \mathbf{A}^*\mathbf{R}^*\Pi_{\mathbf{k}} \equiv 0$ and a similar argument as above. Hence,
\begin{equation}\label{eq:111}
	\Pi_{\mathbf{k}} \mathbf{B} \mathbf{A} \equiv \Pi_{\mathbf{k}}~ \mathbf{I} \equiv \Pi_{\mathbf{k}} \equiv ~ \mathbf{I}~\Pi_{\mathbf{k}} \equiv \mathbf{B} \mathbf{A} \Pi_{\mathbf{k}},
\end{equation}
where we also used \eqref{equation:useful} in the second and third equivalences. We conclude that
\[
	\Pi_{\mathbf{k}} \mathbf{B} \equiv  \mathbf{B} \Pi_{\mathbf{k}} \mathbf{A} \mathbf{B} \equiv \mathbf{B} \Pi_{\mathbf{k}}~ \mathbf{I} \equiv \mathbf{B} \Pi_{\mathbf{k}},
\]
where in the first equivalence we used \eqref{eq:111}, in the second one we used $\mathbf{A} \mathbf{B} \equiv \mathbf{I}$ and in the third one \eqref{equation:useful}. Combining the preceding equality and \eqref{equation:first-step} verifies \eqref{equation:commutation-holomorphic} for $\mathbf{B}$. Finally, going back to the initial notation $\mathbf{A} = \Pi_{\mathbf{k}} \mathbf{A}' \Pi_{\mathbf{k}}$, we obtain:
\[
\mathbf{A} \Pi_{\mathbf{k}}\mathbf{B} \Pi_{\mathbf{k}} \equiv \Pi_{\mathbf{k}}\mathbf{A}' \Pi_{\mathbf{k}} \Pi_{\mathbf{k}}\mathbf{B} \Pi_{\mathbf{k}} \equiv \Pi_{\mathbf{k}} \mathbf{I} \Pi_{\mathbf{k}} \equiv \mathbbm{1},
\]
by Lemma \ref{lemma:i}. This proves the claim.
%
\end{proof}

Finally, we conclude this paragraph with the sharp Gårding inequality:

\begin{lemma}
\label{lemma:garding}
Let $\mathbf{A} \in \Psi^m_{h,\mathrm{BW}}(P)$ such that $\Re(\sigma^{\mathrm{BW}}_{\mathbf{A}}) \geq 0$. Then, there exists a constant $C > 0$ such that for all $u \in C^\infty(F,\mathbf{L}^{\otimes \mathbf{k}})$, for all $h > 0$, $\mathbf{k} \in \widehat{G}$ with $h|\mathbf{k}|\leq 1$,
\[
\Re \langle \mathbf{A}u,u\rangle_{L^2(F,\mathbf{L}^{\otimes \mathbf{k}})} \geq -C h\|u\|^2_{H^{(m-1)/2}(F,\mathbf{L}^{\otimes \mathbf{k}})}.
\]
\end{lemma}

\begin{proof}
The proof follows \emph{verbatim} the usual proof in $\R^n$, see \cite[Theorem 9.11]{Zworski-12} for instance.
\end{proof}

\begin{remark}
Given fibrewise holomorphic vector bundles $E_1$ and $E_2$ over $F$, it is possible to define a calculus of pseudodifferential operators $\Psi^{\bullet}_{h, \mathrm{BW}}(P, E_1 \to E_2)$ where operators send (fibrewise holomorphic) sections of $\mathbf{L}^{\otimes \mathbf{k}} \otimes E_1$ to (fibrewise holomorphic) sections of $\mathbf{L}^{\otimes \mathbf{k}} \otimes E_2$. Principal symbols, ellipticity, and wavefront sets are defined analogously.
\end{remark}

\subsection{Quantization of functions}\label{ssection:quantization-symbols} So far, it is clear how to construct operators in $\Psi^\bullet_{h, \mathrm{BW}}(P)$. We shall see that all natural examples of geometric operators (e.g. connections, horizontal Laplacians, etc.) fit in $\Psi^\bullet_{h, \mathrm{BW}}(P)$ but these operators do not allow to microlocalize in compact regions of phase space, which is a convenient procedure for several purposes. The purpose of this paragraph is to introduce a way to quantize symbols $a \in S^m_{h,\mathbf{k}}(\HH^*)$ into operators $\Op^{\mathrm{BW}}(a) \in \Psi^m_{h,\mathrm{BW}}(P)$.
%

\subsubsection{Quantization procedure} Given $a \in S^m_{h,\mathbf{k}}(\HH^*)$, and $f \in C^\infty(F,\mathbf{L}^{\otimes \mathbf{k}})$, we introduce the following quantization procedure (here $w_\ell \in F$, $x_\ell := \pi(w_\ell)$):
\[
\begin{split}
\Op^{\mathrm{BW}}_h&(a)f(w_\ell) \\
& := \dfrac{1}{(2\pi h)^n} \Pi_{\mathbf{k}}(x_\ell) \int_{x_r \in M, \xi \in T^*_{x_\ell}M} e^{-\tfrac{i}{h}\exp^{-1}_{x_\ell}(x_r)} a(\bullet,\xi) \tau_{x_r \to x_\ell} \Pi_{\mathbf{k}}(x_r)f|_{F_{x_r}} \dd x_r \dd \xi.
\end{split}
\]
We now prove that this operator fits into the Borel-Weil calculus.

\begin{theorem}
\label{theorem:quantization-bw}
The map 
\[
 S^m_{h,\mathbf{k}}(\HH^*) \ni a \mapsto \Op^{\mathrm{BW}}_h(a) \in \Psi^m_{h,\mathrm{BW}}(P)
 \]
 is well-defined and continuous.
\end{theorem}

The proof is rather similar to that of Theorem \ref{theorem:penible} and is only sketched. Further details are given below in the specific case of pullback functions.

\begin{proof}[Sketch of proof]
We start by showing that it is well-defined. Consider a symbol $\widetilde{a} \in S^m_{h,\mathbf{k}}(T^*F)$ such that $\widetilde{a}(w,\xi_{\HH^*},\xi_{\V^*}) = a(x,\xi_{\HH^*})$ on a conic neighborhood of $\overline{\HH^*}$. Then, by the stationary phase lemma in the variable $w_3$ and $\xi_{\V^*}$, one finds
\[
\begin{split}
& \dfrac{1}{(2\pi h)^{\dim F}} \Pi_{\mathbf{k}}(x_\ell) \int_{w_3 \in F_{x_r},x_r \in M, \xi \in T^*_{x_\ell}M, \xi_{\V^*} \in \V^*_{w_3}} e^{-\tfrac{i}{h}\exp^{-1}_{x_\ell}(x_r)} e^{-\tfrac{i}{h}\xi_{\V^*}(w_3-w_r)} \\
& \hspace{6cm} \widetilde{a}(\bullet,d\pi^\top\xi) \tau_{x_r \to x_\ell} \Pi_{\mathbf{k}}(x_r)f|_{F_{x_r}} \dd w_3 \dd \xi_{\V^*} \dd x_r \dd \xi \\
& =  \dfrac{1}{(2\pi h)^n} \Pi_{\mathbf{k}}(x_\ell) \int_{x_r \in M, \xi \in T^*_{x_\ell}M} e^{-\tfrac{i}{h}\exp^{-1}_{x_\ell}(x_r)} a(\bullet,\xi) \tau_{x_r \to x_\ell} \Pi_{\mathbf{k}}(x_r)f|_{F_{x_r}} \dd x_r \dd \xi + \mc{O}(h^\infty).
\end{split}
\]
(Indeed, all the derivatives of the symbol in the $\xi_{\V^*}$ direction vanish at $\xi_{\V^*}=0$, so only the first term appears in the stationary phase expansion) The left-hand side is in the Borel-Weil calculus by construction and has principal symbol $a$. Finally, continuity (in the sense of Fréchet spaces) can be traced from the argument just exposed; we omit the details.
\end{proof}

\subsubsection{Pullback functions}

\label{sssection:pullbackfunctions}

Let $\pi : F \to M$ be the projection. Given $f \in C^\infty(T^*M)$, there is a natural pullback $\pi^* f \in C^\infty(T^*F)$ defined by
\[
	\pi^*f(z,d\pi^\top_z\xi + \eta) := f(\pi(z), \xi), \qquad z \in F, \xi \in T^*_{\pi(z)}M, \eta \in \V^*_z.
\]
Note that the decomposition $d\pi^\top_z\xi + \eta$ of an element in $T_z^*F$ is unique. Conversely, any function $h \in C^\infty(T^*F)$ that is \emph{equivariant} in the sense that
\begin{enumerate}[label=(\roman*)]
\item for all $z \in F, \xi \in T_z^*F, \eta \in \V^*_z$, we have $h(z,\xi+\eta) = h(z,\xi)$,
\item for all $x \in M$, $z,z' \in F_x$, $\xi \in T_x^*M$, we have $h(z, d\pi_z^{\top} \xi) = h(z',d\pi_{z'}^\top \xi)$,
\end{enumerate}
is the pullback $h = \pi^*f$ of a function $f \in C^\infty(T^*M)$. The first condition means that $h$ only depends on the $\mathbb{H}_z^*$ component, while the second one gives equivariance in the fibre, enabling one to identify $h$ with a function downstairs. 

\begin{remark}[Examples and non-examples of pullback symbols]
\label{remark:pullback}
	Observe that given a symbol $\sigma \in S^m_h(T^*M)$, its pullback $\pi^*\sigma$ might not be a symbol in $T^*F$. Indeed, by definition the pullback $\pi^*\sigma$ is constant in $\V^*$ directions, and so if $m = 0$ then $\pi^*\sigma$ should be of order zero. Heuristically, after differentiating in $\HH^*$ direction it remains constant in $\V^*$ directions, and so it remains in $S^0(T^*F)$, contradicting the symbol estimates. (This is the easiest to see when $\pi: F = \mathbb{R}^2 \to M = \mathbb{R}$ is the projection to the first coordinate, with $\HH$ tangent to the first copy of $\mathbb{R}$.) 
	
	Nevertheless, if $\sigma$ is \emph{polynomial} in $\xi \in T^*M$, its pullback $\pi^*\sigma$ is still a symbol in $T^*F$. More importantly for us, given a symbol $\psi \in S^0(T^*F)$ which is supported in a conic neighbourhood of $\HH^*$ (not containing $\V^*$), it is straightforward to show that $\psi \cdot \pi^* \sigma$ indeed is a symbol and belongs to $S_h^m(T^*F)$. In particular, if we take $\psi$ such that $\psi = 1$ on $\mathbb{H}^*$, then $\pi^* \sigma|_{\HH^*} \in S^m_{h, \mathrm{BW}}(\HH^*)$.
\end{remark}

\subsubsection{Quantization of pullback functions} This section gives a more pedestrian approach to the quantization of symbols in the Borel-Weil calculus in the specific case where the symbol is a pullback function.


Let $\psi \in C^\infty(M \times M)$ be a cutoff function equal to $1$ near the diagonal. For $a_{h,\mathbf{k}} \in S^m_{h,\mathbf{k}}(T^*M)$, we then define an operator acting on $f \in C^\infty_{\mathrm{(hol)}}(F,\mathbf{L}^{\otimes \mathbf{k}})$ by 
\begin{multline}
\label{equation:oph0}
\Op_h^{0}(a_{h,\mathbf{k}}) f(x, w)\\ 
:= \dfrac{1}{(2\pi h)^n} \int_M \int_{T^*_xM} e^{-\tfrac{i}{h}\xi\cdot\exp_x^{-1}(y)} a_{h,\mathbf{k}}(x,\xi) (\tau_{y\to x}(f|_{F_y}))(x, w) \psi(x,y)\, \dd \xi\dd y.
\end{multline}
We emphasize that $\Op_h^{0}(a_{h,\mathbf{k}})$ is \emph{not} an operator in $\Psi^\bullet_{h, \mathrm{BW}}(P)$ because, as we shall see below, this corresponds to the quantization of $\pi^*a_{h,\mathbf{k}}$ and this is \emph{not} a symbol in $S^m_{h,\mathbf{k}}(T^*F)$ in general, as pointed out above. By construction, and by using that $\Pi_{\mathbf{k}}$ and $\tau_{x \to y}$ commute thanks to \eqref{eq:parallel-transport-fibrewise-holomorphic}, we obtain that
\begin{equation}
\label{equation:commutation-construction}
[\Op_h^{0}(a_{h,\mathbf{k}}),\Pi_{\mathbf{k}}] = 0,
\end{equation}
We can form the operators

\begin{equation}
\label{equation:quantization-geometric}
\overline{\Op}_h(a_{h,\mathbf{k}}) := \mathbf{I} ~\Op_h^0(a_{h,\mathbf{k}})~ \mathbf{I}, \qquad \overline{\Op}_h^{\mathrm{BW}}(\pi^* a_{h,\mathbf{k}}) := \Pi_{\mathbf{k}} \overline{\Op}_h(a_{h,\mathbf{k}}) \Pi_{\mathbf{k}}.
\end{equation}
and the following holds:


\begin{lemma}
One has $\mathbf{A} := \overline{\Op}_h^{\mathrm{BW}}(\pi^* a_{h,\mathbf{k}}) \in \Psi^m_{h, \mathrm{BW}}(P)$. Moreover, the principal symbol of $\mathbf{A}$ is equivariant and given by $\sigma_{\mathbf{A}}^{\mathrm{BW}} = \pi^*a_{h,\mathbf{k}}|_{\mathbb{H}^*}$.
\end{lemma}

In other words, $\overline{\Op}_h^{\mathrm{BW}}(\pi^* a_{h,\mathbf{k}}) = \Op_h^{\mathrm{BW}}(\pi^* a_{h,\mathbf{k}}) + \mc{O}(h)$, where $\Op_h^{\mathrm{BW}}$ denotes the quantization of Theorem \ref{theorem:quantization-bw}. In what follows, we will simply write $\Op_h^{\mathrm{BW}}$ as there is no need to distinguish.


\begin{remark} \label{remark:identification-lol} To avoid cumbersome notation, we will also sometimes write $\sigma_{\mathbf{A}}^{\mathrm{BW}} = a_{h,\mathbf{k}}$. By this, it is meant that the principal symbol of $\mathbf{A}  \in \Psi^m_{h, \mathrm{BW}}(P)$ is given by restricting the pullback of $a_{h,\mathbf{k}} \in S^{m}_{h, \mathbf{k}}(T^*M)$ to $\HH^*$.

\end{remark}

\begin{proof}
That $\mathbf{A}$ satisfies \eqref{equation:commutation-holomorphic} follows from Lemma \ref{lemma:i} and \eqref{equation:commutation-construction}.
Let $U \subset M$ be an open subset and let $F|_{U} \simeq U \times G/T$ be a trivialisation. Then $\tau_{y\to x}(y,z) = (x,g(y,x)z)$, where $x,y \in U, z \in G/T$ and $g(y,x)\in G$ is a smooth function of $(x, y)$ (note that $g(y,x)=g(x,y)^{-1}$). If $s_1, \dotsc,s_d$ are local sections of $L_1, \dotsc, L_d$ respectively, of unit pointwise norm, defined over some open subset $V \subset F|_U$, write $\mathbf{s}^{\otimes \mathbf{k}}=s_1^{\otimes k_1} \otimes \dotsm \otimes s_d^{\otimes k_d}$. Define locally the $1$-forms $\beta_j \in C^\infty(V,\HH^*)$ such that $\nabla_j s_j = i \beta_j \otimes s_j$, where $\nabla_j$ are the induced (horizontal) connections on $L_j \to F$, for $j = 1, \dotsc, d$.

Then
\[
\tau_{y \to x}(\mathbf{s}^{\otimes \mathbf{k}}(y, w)) = \exp\left(-i \int_w^{g(y,x)w}\mathbf{k}\cdot\beta\right)\mathbf{s}^{\otimes \mathbf{k}}(x,g(y,x)w),
\]
where $\int_w^{g(y,x)w}\mathbf{k}\cdot\beta$ denotes the integral of $\mathbf{k}\cdot\beta=\sum_j k_j \beta_j$ along the horizontal lift, starting at $w$, of the unique geodesic joining $y$ to $x$. Let $f \in C^\infty(V,\Lk)$, and write $f(x,z) = u(x,z) \mathbf{s}^{\otimes \mathbf{k}}(x,z)$ for some function $u \in C^\infty(V)$. Write $\tau_{x \to y} (x, z) = (y, w) = (y, g(x, y)z)$ and let $\chi \in C^\infty_{\comp}(U)$. Then \small
\begin{align*}
	&\mathbf{s}^{- \otimes \mathbf{k}}(x,z)\chi(x)\Op_h^{0}(a_{h,\mathbf{k}}) (\chi~u~\mathbf{s}^{\otimes \mathbf{k}})(x,z) \\
& = (2\pi h)^{-n} \chi(x) \int_U \int_{T^*_xM} e^{-\tfrac{i}{h}\xi\cdot\exp_x^{-1}(y)} a_{h,\mathbf{k}}(x,\xi) \tau_{y \to x}\big(u(y, w) \mathbf{s}^{\otimes \mathbf{k}}(y, w)\big)(x, z) \chi(y)\psi(x,y)\, \dd\xi\dd y \\
& = (2\pi h)^{-n} \chi(x) \int_U \int_{T^*_xM} e^{-\tfrac{i}{h}\xi\cdot\exp_x^{-1}(y)} a_{h,\mathbf{k}}(x,\xi) \exp\left(-i\int_w^{g(y, x)w} \mathbf{k}\cdot\beta\right) u(y, w) \chi(y) \psi(x, y)\, \dd\xi\dd y \\
& = (2\pi h)^{-n} \chi(x) \int_U \int_{T^*_xM} e^{-\tfrac{i}{h}\big(\xi + h \mathbf{k} \cdot W(x, y, z)\big) \cdot \exp_x^{-1}(y)} a_{h,\mathbf{k}}(x,\xi) u(y, w) \chi(y) \psi(x, y)\, \dd\xi\dd y \\
& = (2\pi h)^{-n} \chi(x) \int_U \int_{T^*_xM} e^{-\tfrac{i}{h}\xi \cdot \exp_x^{-1}(y)} a_{h,\mathbf{k}}(x,\xi - h \mathbf{k} \cdot W(x, y, z)) u(y, g(x, y)z) \chi(y) \psi(x, y)\, \dd\xi\dd y \\
& := A^0 u (x,z),
\end{align*} \normalsize
where we changed variables using $\xi' = \xi + h \mathbf{k} \cdot W(x, y, z)$ in the fourth equality, $W(x, y, z) \in (T^*_xM)^{\oplus d}$ depends smoothly on $(x, y, z)$ and is defined implicitly by 
\[
	F(x, y, z) := \int_w^{g(y, x)w} \beta = -\int_z^{g(x, y) z} \beta = W(x, y, z) \cdot \exp_x^{-1} y.
\]
Indeed, we can define $W(x, y, z)$ by Taylor expansion as follows
\begin{equation}\label{eq:taylor-expansion-W}
\begin{split}
	F(x, y, z) &= \int_0^1 \partial_t F(x, \exp_x(t \exp_x^{-1}y), z)\, dt = \int_0^1 d_yF \cdot d\exp_x(t\exp_x^{-1}y) (\exp_x^{-1}y)\, dt\\ 
 &= \underbrace{\left(\int_0^1 G(t, x, y) d_yF\, dt\right)}_{W(x, y, z) := } \cdot \exp_x^{-1}y,
\end{split}
\end{equation}
where $G(t, x, y) = [d\exp_x(t\exp_x^{-1}y)]^{\top}$. Then also $G(t, x, x) \equiv \id$, and for $v \in T_xM$
\begin{equation}
\label{equation:idk}
	W(x, x, z)(v) = d_y F(x, x, z) (v) = -\partial_t|_{t = 0} \int_0^t  \beta(\gamma^{\HH}(s))(\dot{\gamma^{\HH}}(s))\, ds = -\beta(x, z)(v^{\HH}_z),
\end{equation}
where for some small $\varepsilon > 0$, $(\gamma^{\HH}(s))_{s \in (-\varepsilon, \varepsilon)}$ is the horizontal lift to $F$, starting at $z$, of the unit speed geodesic $\gamma$ generated by $(x, v)$. 

We now put the operator $A^0$ in a standard pseudodifferential form. Given a smooth function $S$ compactly supported on a small open subset of $G/T$, we can write:

\[
	L_{g(x, y)}^*S(z) = S(g(x,y)z) = (2\pi h)^{-\dim G/T} \int_{G/T} \int_{T_z^*(G/T)} e^{-\frac{i}{h} (w - L_{g(x, y)} z) \cdot \eta} S(w)\, \dd w \dd\eta,
\]
where it is understood that $G/T$ is locally near $z$ identified with an open subset in $\R^{\dim(G/T)}$ and $T^*_z(G/T) \simeq \R^{\dim(G/T)}$. We then obtain, using previous computations:
\[
\begin{split}
A^0 u (x,z) & = (2\pi h)^{-\dim F} \chi(x) \int_U \int_{T^*_xM}  \int_{G/T} \int_{T_z^*(G/T)} e^{-\tfrac{i}{h}(\xi \cdot \exp_x^{-1}(y)+ (w - L_{g(x, y)} z) \cdot \eta)} \\
& \hspace{3cm} \cdot a_{h,\mathbf{k}}(x,\xi - h \mathbf{k} \cdot W(x, y, z)) u(y, w) \chi(y) \psi(x, y)\, \dd\xi\dd y \dd w \dd\eta.
\end{split}
\]
We now consider the Taylor expansion of $w-L_{g(x,y)}z$ at $x=y$, similarly to \eqref{eq:taylor-expansion-W}, and we obtain
\[
	H(x, y, z) := L_{g(x, y)}z - z = \underbrace{\left(\int_0^1 G(t, x, y) d_yH\, dt\right)}_{Q(x, y, z) := } \cdot \exp_x^{-1}y,
\]
where $Q(x, y, z) \in (T^*_xM)^{\oplus \dim G/T}$ depends smoothly on $(x, y, z)$. Therefore
\[
	(w - L_{g(x,y)}z)\cdot \eta = \big(w- z - Q(x,y,z)\cdot \exp_x^{-1}(y)\big)\cdot \eta = (w-z)\cdot \eta - \exp_x^{-1}(y) \cdot Q(x,y,z)^{\top}\eta.
\]
This allows to rewrite the phase as
\[
	\xi \cdot \exp_x^{-1}(y)+ (w - L_{g(x, y)} z) \cdot \eta = (\xi - Q(x,y,z)^\top\eta)\cdot \exp_x^{-1}(y) + (w-z)\cdot\eta.
\]
Hence, setting $\xi' := \xi - Q(x,y,z)^\top\eta$ and making the change of variable in $\xi$, we finally obtain:
\[
\begin{split}
A^0 u (x,z) & = (2\pi h)^{-\dim F} \chi(x) \int_V \int_{T^*V} e^{-\tfrac{i}{h}(\xi\cdot \exp_x^{-1}(y) + (w-z)\cdot\eta)} \\
& \cdot a_{h,\mathbf{k}}(x,\xi - h \mathbf{k} \cdot W(x, y, z) + Q(x,y,z)^\top\eta) u(y, w) \chi(y) \psi(x, y)\, \dd\xi\dd y \dd w \dd\eta,
\end{split}
\]
which brings $A^0$ to a form with the usual phase for a pseudodifferential operator. In general $A^0$ is \emph{not} a pseudodifferential operator in $\Psi^m_{h,\mathbf{k}}(F)$ because the symbol
\begin{equation}
\label{equation:symbol-tot}
(x,y,z,\xi,\eta) \mapsto a_{h,\mathbf{k}}(x,\xi-h\mathbf{k}W(x,y,z) + Q(x,y,z)^\top\eta)
\end{equation}
is not in $S^m(F \times T^*F)$. (This can be easily seen in the following case: if the connection on $P$ is flat, then $g(x,y)=\mathbf{1}_G$ and thus $Q \equiv 0$. The symbol is then invariant by translation in the $\eta \in \V^*$ variable, and therefore cannot be in $S^m(F \times T^*F)$, see Remark \ref{remark:pullback}. In other words, the obstruction for $A^0$ to being pseudodifferential lies in its microlocal behaviour as $|\xi_{\V^*}| \to \infty$.)

However, multiplication by $\mathbf{I}$ on both sides to define $\mathbf{A} := \mathbf{I} \Op_h^{0}(a_{h,\mathbf{k}}) \mathbf{I}$ microlocalizes near $\HH^*$, where $ \Op_h^{0}(a_{h,\mathbf{k}})$ is indeed a semiclassical pseudodifferential operator. Hence $\mathbf{A} \in \Psi^m_{h,\mathbf{k}}(F,\mathbf{L})$.

Finally, we compute the principal symbol. First, observe that $\HH^*$ corresponds to the set $\{\eta=0\}$ and that $\xi \in T^*U$ is canonically identified with $d\pi^{\top}\xi \in \HH^*$. Hence, by \eqref{equation:symbol-tot}, the principal symbol (relative to a family $h \mapsto \mathbf{k}(h)$) of $A^0$ on $\HH^*$ is given by $a_{h,\mathbf{k}(h)}(x,\xi - h\mathbf{k}(h) \cdot W(x,x,z))$. The principal symbol of $\mathbf{A}$ in the calculus $\Psi^m_h(F,\mathbf{L})$ is then obtained using both \eqref{equation:symbole-principal} and \eqref{equation:idk} as follows, where $\xi \in \HH^*$:
\[
	\sigma_{\mathbf{A}}(x,z,\xi) = a_{h,\mathbf{k}(h)}(x,\xi-h\mathbf{k}(h)\cdot W(x,x,z) - h\mathbf{k}(h)\cdot \beta_{\HH^*}) = \pi^*a_{h,\mathbf{k}(h)}(x,\xi).
\]
This completes the proof.
\end{proof}

\subsubsection{Examples}

\label{sssection:examples2}
We now provide some examples of operators in this calculus and compute their principal symbol. Further examples will be provided in the next chapters. As we shall see, all natural geometric operators have a principal symbol which is a pullback function in the sense of \S\ref{sssection:pullbackfunctions}.

Fix a connection $\nabla$ on $P$. Recall that by construction, for $\mathbf{k} \in \widehat{G}$, there is an induced connection
\[
\nabla_{\mathbf{k}} : C^\infty_{\mathrm{hol}}(F,\mathbf{L}^{\otimes \mathbf{k}}) \to C^\infty_{\mathrm{hol}}(F,\mathbf{L}^{\otimes \mathbf{k}}\otimes (\HH_F^*)_{\C}),
\]
see Proposition \ref{prop:nabla-preserves-holomorphicity}.


\begin{enumerate}[label=(\roman*), itemsep=5pt]
\item Set $\mathbf{P} := h \nabla_{\mathbf{k}}$ such that $h|\mathbf{k}| \leq 1$. Then $\mathbf{P}$ belongs to $\Psi^1_{h,\mathrm{BW}}(P,\C \to (\HH_F^*)_{\C})$ (where $\C$ denotes the trivial line bundle) since it commutes with the fibrewise holomorphic projection (see \eqref{equation:commutation-nablak-pik}). Similarly to \S \ref{sssection:examples}, items (i) and (iv), the principal symbol is a section $\sigma^{\mathrm{BW}}_{\mathbf{P}} \in S^1_{h, \mathrm{BW}}(\HH^*, \HH_{\C}^*)$ given for $(z,\xi) \in \HH^*$ by $\sigma^{\mathrm{BW}}_{\mathbf{P}}(z,\xi) = i\xi_{\HH^*}$. Equivalently, it can be seen as a section $\sigma^{\mathrm{BW}}_{\mathbf{P}} \in S^1_{h, \mathbf{k}}(T^*M,H^0(\HH_{\mathbb{C}}^*))$ (where $H^0(\HH_{\mathbb{C}}^*)$ stands for fiberwise holomorphic sections of $\HH^*_{\mathbb{C}}$) given for $(x,\xi) \in T^*M, z \in F_x$, by
\begin{equation}
\label{equation:symbol-nabla}
\sigma_{\mathbf{P}}^{\nabla}(x,\xi; z) = i d\pi_{z}^\top \xi.
\end{equation}
Observe that the principal symbol is independent of the choice of family $h \mapsto \mathbf{k}(h)$, that is $\sigma_{\mathbf{P},\mathbf{k}(h)}^{\nabla}= \sigma_{\mathbf{P}}^\nabla$ for all families $h \mapsto \mathbf{k}(h)$.

\item Given a vector field $X \in C^\infty(M,TM)$, let $X^{\HH} \in C^\infty(F,\HH_F)$ be its horizontal lift to $F$, and set $\mathbf{P} := h \X_{\mathbf{k}} = \iota_{X^{\HH}} \nabla_{\mathbf{k}}$. Then $\mathbf{P}$ belongs to $\Psi^1_{h,\mathrm{BW}}(P)$ as it commutes with $\Pi_{\mathbf{k}}$ using \eqref{equation:commutation-nablak-pik} (and its proof). Moreover, similarly to \S \ref{sssection:examples}, items (iii) and (iv), the principal symbol $\sigma_{\mathbf{P}}^{\mathrm{BW}} \in S^1_{h, \mathbf{k}}(T^*M)$ is given by,
\[
	\sigma^{\nabla, \mathrm{BW}}_{\mathbf{P}}(x,\xi) = i\xi (X(x)).
\]
Once again, following Remark \ref{remark:identification-lol}, the principal symbol is equivariant and identified with a function on $T^*M$ in the previous equation.

\item Finally, consider $\mathbf{P} := h^2\nabla^*_{\mathbf{k}}\nabla_{\mathbf{k}} = h^2\Delta_{\mathbf{k}}$ for $h|\mathbf{k}| \leq 1$. Then by \eqref{equation:commutation-nablak-pik}, we have $\mathbf{P} \in \Psi^2_{h,\mathrm{BW}}(P)$ and similarly to \S \ref{sssection:examples}, item (v), $\sigma^{\mathrm{BW}}_{\mathbf{P}} \in S^2_{h, \mathbf{k}}(T^*M)$ is given by
\[
\sigma^{\nabla, \mathrm{BW}}_{\mathbf{P}}(x,\xi) = |\xi_{\HH^*}|^2_g.
\]
\end{enumerate}


\subsection{Further properties}

\subsubsection{Sobolev spaces II}

\label{sssection:sobolev-spaces-2}

We state the usual Sobolev embedding estimate. 

\begin{lemma}
\label{lemma:sobolev-embedding}
The following holds: for all $h, \mathbf{k}$ such that $h|\mathbf{k}| \leq 1$, for all $\alpha > n/2$, there exists $C > 0$ such that for all $f \in C^\infty_{\mathrm{hol}}(F,\mathbf{L}^{\otimes \mathbf{k}})$, one has:
\begin{equation}
\label{equation:sobolev-embeddin-c0}
\|f\|_{C^0(F,\mathbf{L}^{\otimes \mathbf{k}})} \leq C h^{- \dim F/2}  \|f\|_{H^{\alpha}_h(F,\mathbf{L}^{\otimes \mathbf{k}})}.
\end{equation}
\end{lemma}



\begin{proof}

The proof is verbatim the same as for the usual Sobolev embedding estimate.
%
\end{proof}

\begin{remark}
	We note that the estimate in Lemma \ref{lemma:sobolev-embedding} can be slightly improved by using the Borel-Weil calculus and Theorem \ref{theorem:penible}, by replacing the exponent $h^{-\dim F/2}$ with $h^{-n/2}d_{\mathbf{k}}$, which in turn can be estimated by $Ch^{-n/2 - \dim(G/T)/4}$. Note that $\dim(G/T) = 2\dim_{\mathbb{C}}(G/T)$.
\end{remark}

\subsubsection{Propagation of singularities} We now expand the discussion of \S\ref{sssection:propagation1} to the case of $F$ and the line bundles coming from the Borel-Weil theory. Write $\omega_0$ and $\omega_{h, \mathbf{k}}$ for the standard and twisted symplectic forms (with respect to the line bundles $(L_i, \nabla_i) \to F$), respectively. Let $Y \in C^\infty(M,TM)$ be a vector field on $M$ generating a flow $\phi$, and $Y^{\HH} \in C^\infty(F,\HH)$ be the horizontal lift of $Y$ to $F$ generating a flow $\psi$. Let $\Phi^{\omega_0}$ and $\Phi^{\omega_{h,\mathbf{k}}}$ be the Hamiltonian flows on $T^*F$ generated by the Hamiltonian $p_{Y^{\HH}}$, with respect to symplectic structures $\omega_0$ and $\omega_{h, \mathbf{k}}$.


\begin{lemma}
\label{lemma:invariance-h}
The flow $\Phi^{\omega_{h,\mathbf{k}}}$ preserves $\HH^*$. Moreover, the Hamiltonian flow of $|\xi_{\HH^*}|^2$ preserves $\HH^*$.
\end{lemma}

\begin{proof}

We first observe that $\Phi^{\omega_0}$ preserves $\HH^*$. Indeed, note that $\pi \circ \psi_t = \phi_t \circ \pi$ and thus $d\pi^\top \circ d\phi_t^\top  = d\psi_t^\top \circ d\pi^\top$. This shows that the flow $\Phi^{\omega_0}_t(x,z,\xi,\eta) = (\psi_t(z), d\psi_t^{-\top}\xi)$ preserves $\HH^*$.

Next, let $U \subset F$ be an open set over which there are trivialising pointwise unit sections $(s_j)_{j = 1}^d$ of $(L_j)_{j = 1}^d$, such that $\nabla_j s_j = i \beta_j \otimes s_j$ for $j = 1, \dotsc d$. By \eqref{eq:hamiltonian-flow-conjugacy'}, it suffices to show that $\psi_q^*(\iota_{Y^{\HH}} d\beta_j) \in \HH^*$ for each $q \in \mathbb{R}$ and $j = 1, \dotsc, d$. Indeed, by the previous paragraph $d\psi_q$ preserves $\V$, and since $id\beta_j$ is the curvature of $(L_j, \nabla_j)$, $\iota_{Y^{\HH}} d\beta_j$ is zero on $\V$ by Lemma \ref{lemma:vanishing-curvature}. This completes the proof if the first claim.

For the second claim, we argue locally by taking a local orthonormal frame $(\e_i)_{i = 1}^n$ and lift it horizontally to $\HH$ writing $(\e_i^{\HH})_{i = 1}^n$. Then $p := |\xi_{\HH^*}|^2 = \sum_{i = 1}^n p_i^2$, where $p_i = \xi_{\HH^*}(\e_i^{\HH})$. By definition of the Hamiltonian flow
\[
	H_p^{\omega_{h, \mathbf{k}}} = \sum_{i = 1}^n 2p_i H_{p_i}^{\omega_{h, \mathbf{k}}}.
\]
The claim then follows from the first part of the lemma by observing that $p_i = p_{\e_i^{\HH}}$; thus $H_{p_i}^{\omega_{h, \mathbf{k}}}$ is tangent to $\HH^*$.
\end{proof}

%

We let $\mathbf{Y} := \iota_{Y^{\HH}} \nabla_{\mathbf{k}} : C^\infty_{\mathrm{(hol)}}(F,\mathbf{L}^{\otimes \mathbf{k}}) \to C^\infty_{\mathrm{(hol)}}(F,\mathbf{L}^{\otimes \mathbf{k}})$. Similarly to Proposition \ref{proposition:propagation}, one can state a lemma on propagation of singularities within the calculus $\Psi^\bullet_{h, \mathrm{BW}}(P)$. The important point is that, thanks to Lemma \ref{lemma:invariance-h}, singularities only live and propagate in $\HH^* \subset T^*F$. Then the following holds:

\begin{proposition}
\label{proposition:propagation2}
Let $Y \in C^\infty(M,TM)$ be a vector field, $Y^\HH$ its horizontal lift to $F$, and further assume that $\iota_{Y^{\HH}} \mathbf{F}_{\overline{\nabla}} = 0$. Let $\mathbf{A}, \mathbf{A}' \in \Psi^{\comp}_{h,\mathrm{BW}}(P)$ be compactly microlocalised such that the following holds: for all $(x,\xi) \in \WF^{\mathrm{BW}}(\mathbf{A})$, there exists $t \in \R$ such that $\Phi_t^{\omega_0}(x,\xi) \in \Ell^{\mathrm{BW}}(\mathbf{A}')$. Then, there exists $\mathbf{B} \in \Psi^{\comp}_{h, \mathrm{BW}}(P)$, compactly microlocalised such that for all $s \in \R, N > 0$, there exists $C >0$ such that for all $f_{h,\mathbf{k}} \in C^\infty_{\mathrm{hol}}(M,\Lk)$, the following holds:
\begin{equation}
\label{equation:propagation-classique2}
\|\mathbf{A} f_{h,\mathbf{k}}\|_{H^s_h} \leq C\left(\|\mathbf{B} \mathbf{Y}f_{h,\mathbf{k}}\|_{H^s_h} + \|\mathbf{A}' f_{h,\mathbf{k}}\|_{H^s_h} + h^N \|f_{h,\mathbf{k}}\|_{H^{-N}_h} \right).
\end{equation}
Moreover, if $f_{h,\mathbf{k}}$ is merely a distribution and the right-hand side is finite, then $\mathbf{A}f_{h,\mathbf{k}} \in H^s_h$ and \eqref{equation:propagation-classique2} holds.
\end{proposition}
\begin{proof}
	The proof is a consequence of Proposition \ref{proposition:propagation}.
\end{proof}

Notice that $\iota_{Y^{\HH}} \mathbf{F}_{\overline{\nabla}}$ already vanishes on $\V$ by Lemma \ref{lemma:vanishing-curvature}, so the condition $\iota_{Y^{\HH}} \mathbf{F}_{\overline{\nabla}} = 0$ of the preceding proposition should only be verified on horizontal vector fields.

\chapter{Rapid mixing for frame flows}

\label{chapter:flow}

In this chapter, we apply the tools developed in Chapter \ref{chapter:analysis} to study rapid mixing of extensions of Anosov flows to principal bundles.

 \minitoc
 
 \newpage

\section{Introduction}

Let $M$ be a smooth closed connected $n$-dimensional manifold equipped with a volume-preserving Anosov flow $\varphi := (\varphi_t)_{t \in \R}$ (that is, it preserves a smooth measure). Recall that the Anosov property means that the tangent bundle $TM$ splits as a sum of three continuous flow-invariant subbundles
\[
TM = \R X \oplus E_s \oplus E_u,
\]
where $X$ is the generator of the flow, and there exist $C, \lambda > 0$ such that:
\begin{equation}
\label{equation:anosov}
\begin{array}{l}
\forall t \geq 0, \forall \xi \in E_s, ~~ |\dd\varphi_t(\xi)| \leq Ce^{-t\lambda}|\xi|, \\
\forall t \leq 0, \forall \xi \in E_u, ~~ |\dd\varphi_t(\xi)| \leq Ce^{-|t|\lambda}|\xi|,
\end{array}
\end{equation}
where the norm $|\bullet|$ in the fibres of $TM$ is arbitrary. 

Let $\pi : P \to M$ be a $G$-principal bundle, where $G$ is a compact connected Lie group, and let $\psi := (\psi_t)_{t \in \R}$ be an extension of $\varphi$ in the sense that:
\begin{equation}
\label{equation:g-flow}
R_g \circ \psi_t = \psi_t \circ R_g, \quad \pi \circ \psi_t = \varphi_t \circ \pi, \qquad \forall g \in G, t \in \R,
\end{equation}
where $R_g$ denotes the right multiplication in the fibres by $g \in G$. Let $\mu$ be the smooth probability measure on $P$ obtained by the product of the smooth flow-invariant probability measure on $M$ and a bi-invariant probability measure on $G$. We will say that the flow $\psi$ is \emph{rapid mixing} with respect to $\mu$ if the following holds: for every $k \in \mathbb{Z}_{\geq 0}$, there exist $C, s > 0$, such that for all $f_1,f_2 \in C^\infty(P)$, 
\begin{equation}\label{eq:rapid-mixing}
\left|\int_{P} (f_1 \circ \psi_t) \cdot f_2\, \dd \mu - \int_P f_1 \,\dd \mu \int_P f_2\, \dd \mu\right| \leq  C\|f_1\|_{H^s}\|f_2\|_{H^s} t^{-k}, \quad t \geq 1,
\end{equation}
where $\|\bullet\|_{H^s}$ denotes the Sobolev norm of degree $s$. To show \eqref{eq:rapid-mixing} it suffices to consider only $f_1$ and $f_2$ with zero average, i.e. $\int_P f_1\, \dd\mu = \int_P f_2\, \dd\mu = 0$. 

In order to state the main theorem, we now introduce some necessary vocabulary in Lie group theory and the theory of connections. Let $\mathfrak{g}$ denote the Lie algebra of $G$, let $\mathfrak{z}$ denote the centre of $\mathfrak{z}$, set $a := \dim \mathfrak{z} \in \mathbb{Z}_{\geq 0}$, and denote by $[\mathfrak{g}, \mathfrak{g}]$ the Lie subalgebra generated by all Lie brackets of elements in $\mathfrak{g}$. Then by structure theory of Lie groups $\mathfrak{g}$ is a direct sum of Lie algebras (see \S \ref{ssection:borel-weil})
\begin{equation}\label{eq:splitting-centre-commutator}
	\mathfrak{g} = \mathfrak{z} \oplus [\mathfrak{g}, \mathfrak{g}].
\end{equation}
Denote the \emph{commutator subgroup} of $G$ by $[G, G]$, the group generated by all expressions $xyx^{-1}y^{-1}$ for $x, y \in G$. Then $[G, G] \leqslant G$ is a normal Lie subgroup with Lie algebra $[\mathfrak{g}, \mathfrak{g}]$, and $G/[G, G]$ is an Abelian group. There is a quotient $G/[G, G]$-principal bundle $P/[G, G]$ and the flow $\psi$ descends to a flow on $P/[G, G]$. Denote by $T \leqslant G$ a \emph{maximal torus} of $G$, that is, a connected Abelian Lie subgroup of maximal dimension. Then, write $F := P/T$ for \emph{flag bundle} over $M$, and denote the quotient flow on $F$ by $\psi^F$. Finally, denote the \emph{adjoint} representation of $G$ on $\mathfrak{g}$ by $\Ad$ (defined by the derivative of the conjugation action).

Let $\alpha$ be the Anosov $1$-form, defined by $\ker \alpha = E_u \oplus E_s$ and $\alpha(X) = 1$; it is well-known that $\alpha$ is $\eps$-Hölder regular for some $\eps > 0$. The flow $\psi$ induces a \emph{dynamical} principal connection $\nabla^{\mathrm{dyn}}$ on the bundle $P$ whose horizontal distribution is the sum of stable, unstable, and flow directions of $\psi$ (see \S\ref{ssection:dynamical-connection}). The curvature $F_{\nabla^{\mathrm{dyn}}}$ of $\nabla^{\mathrm{dyn}}$ is an $\Ad(P)$-valued \emph{distributional} $2$-form on $M$, where $\Ad(P)$ denotes the vector bundle associated to the adjoint representation (whose each fibre can be identified with $\mathfrak{g}$). Since the splitting \eqref{eq:splitting-centre-commutator} is $\Ad$-invariant and the adjoint representation on $\mathfrak{z} \cong \mathbb{R}^a$ is trivial, $\Ad(P) \cong \mathbb{R}^a \oplus E$, where $\mathbb{R}^a$ denotes the trivial bundle $M \times \mathbb{R}^a$ over $M$ and $E$ is associated to $P$ by restricting the $\Ad$-action to $[\mathfrak{g}, \mathfrak{g}]$. The projection of $F_{\nabla^{\mathrm{dyn}}}$ to the $\mathbb{R}^a$ component can thus be identified with the vector of (real-valued) distributional $2$-forms $(F_{1}, \dotsc, F_{a})$ over $M$.

We are now in shape to state the main theorem.


\begin{theorem}
\label{theorem:main2}
Assume that the flow $\psi^F$ is ergodic. The following holds:
\begin{enumerate}[label=\emph{(\roman*)}]
\item The flow $\psi$ is rapid mixing if $(d\alpha, F_{1}, \dotsc, F_{a})$ are linearly independent over $\R$. Moreover, writing $\vartheta := \max(\dim F + 10, 3\dim F + 4) + 2$, for any $k \in \mathbb{Z}_{> 0}$, there exists $C > 0$, such that for any $f_1, f_2 \in C^\infty(P)$ with zero average,
\begin{equation}\label{eq:quantitative-rapid-mixing}
	\left|\int_P (f_1 \circ \psi_t)\cdot f_2\, \dd \mu\right| \leq C t^{-k} \|f_1\|_{H^{\vartheta (k + 1) + 4}(P)} \|f_2\|_{H^{\vartheta (k + 1) + 4}(P)}, \quad t \geq 1. 
\end{equation}
\item The flow $\psi$ on $P$ is rapid mixing if and only if the flow induced on the quotient $P/[G,G]$ is rapid mixing.
\end{enumerate}
\end{theorem}


We expect that a similar result holds in the related case of extensions of Anosov diffeomorphisms, where in Item (i), instead, we simply ask that $(F_{1}, \dotsc, F_{a})$ are linearly independent (and $\nabla^{\mathrm{dyn}}$ is similarly defined). The linear independence condition in (i) should be interpreted as a non-integrability condition; when $a = 0$, the condition $d \alpha \neq 0$ is equivalent to the non-joint integrability of $E_s$ and $E_u$, and implies rapid mixing of $\varphi$, see \S\ref{ssection:joint-integrability} where this is further discussed. The converse in item (i) is to the best of our knowledge still an open question related to Plante's conjecture \cite{Plante-72} and to the Anosov alternative which, in the volume-preserving case, ensures that the following statements are equivalent: (i) $\varphi$ is not a suspension by a constant roof function; (ii) $\varphi$ is mixing. Finally, we did not optimize the estimate in \eqref{eq:quantitative-rapid-mixing} with respect to Sobolev norms of $f_1$ and $f_2$, but we remark that interpolating between \eqref{eq:quantitative-rapid-mixing} and the trivial estimate for $k = 0$, for any $\alpha \in (0, k)$ it holds that
\begin{equation*}
	\left|\int_P (f_1 \circ \psi_t)\cdot f_2\, \dd \mu\right| \leq C t^{-\frac{k \alpha}{\vartheta(k + 1) + 4}} \|f_1\|_{H^{\alpha}(P)} \|f_2\|_{H^{\alpha}(P)}, \quad t \geq 1,
\end{equation*}
for some $C > 0$. (In particular, by taking $k \to \infty$ we see that for regularity $H^\alpha$, the best exponent we get is asymptotically $t^{-\frac{\alpha}{\vartheta}}$.) Using H\"older-Zygmund instead of Sobolev spaces as in \cite{Bonthonneau-Lefeuvre-23}, it is likely that we would obtain sharper estimates, and in particular we would obtain a dimension-independent value of $\vartheta$. 

Recall that $G$ is called \emph{semisimple} if $a = 0$, that is, if $\mathfrak{g} = [\mathfrak{g}, \mathfrak{g}]$. As an immediate corollary of our main theorem, we obtain:

\begin{corollary}
\label{corollary:ss}
If $G$ is semisimple and $d\alpha \neq 0$, then the flow $\psi$ is rapid mixing if and only if it is ergodic.
\end{corollary}

The ergodicity of $\psi$ is detected using the \emph{transitivity group} $H \leqslant G$, see \S\ref{sssection:transitivity-group} for further details (and also \cite{Brin-75-2,Lefeuvre-23}). In some specific cases such as frame flows over negatively curved manifolds, more can be said on the ergodicity of $\psi$, see \S\ref{ssection:frame-flows}. It should be also noted that, if $\psi$ is rapid mixing, then the induced flow on any associated bundle to $P$ with fibers isomorphic to a $G$-homogeneous space is also rapid mixing, see \S\ref{section:representation} where the associated bundle construction is recalled.

When $G=\mathrm{U}(1)$, the condition that $F_1 = F_{\nabla^{\mathrm{dyn}}}$ and $d\alpha$ are linearly independent can be sometimes easily verified by a topological argument. Recall that a $\mathrm{U}(1)$-bundle $P \to M$ is \emph{torsion} if one of its tensor powers is isomorphic to the trivial bundle $\mathrm{U}(1) \times M$. As a corollary to Theorem \ref{theorem:main2}, we obtain:

\begin{corollary}
\label{corollary1}
If $d \alpha \neq 0$, $G = \mathrm{U}(1)$, and $P \to M$ is not torsion, then $\psi$ is rapid mixing. 
\end{corollary}

That the assumptions imply ergodicity of $\psi$ follows from Lemma \ref{lemma:no-resonances}; then Corollary \ref{corollary1} is an immediate consequence of Theorem \ref{theorem:main2}, Item (i), since $P \to M$ is not torsion if and only if the cohomology class of the curvature $[F_{1}] \in H^2(M, \C)$ does not vanish. Moreover, a similar result holds for arbitrary extensions under the condition that $[F_{1}], \dotsc, [F_{a}]$ are linearly independent in $H^2(M, \C)$. Finally, while writing this monograph, we learnt that Pollicott-Zhang \cite{Pollicott-Zhang-24} had obtained a similar result to Corollary \ref{corollary:ss} for general equilibrium states.

\subsection{Application to frame flows}

\label{ssection:frame-flows}

Recall that a closed Riemannian manifold $(N,g)$ of dimension $n$ is \emph{Anosov} if its geodesic flow $(\varphi_t)_{t \in \mathbb{R}}$, defined on its unit tangent bundle $SN \subset TN$, is Anosov. Typical examples are provided by negatively curved manifolds. A natural extension of the geodesic flow is given by the \emph{frame flow} on the \emph{frame bundle}, i.e. the bundle of orthonormal frames $FM$, defined by
\[
	\psi_t(x, v, \e_2, \dotsc, \e_n) = (\varphi_t(x, v), \mc{P}_t \e_2, \dotsc, \mc{P}_t \e_n), \quad x \in M, t \in \mathbb{R},
\]
where $(v, \e_2, \dotsc, \e_n)$ is an orthonormal basis of $T_xM$, and $\mc{P}_t$ denotes parallel transport along the geodesic at $x$ in the direction of $v$.

As a corollary of Theorem \ref{theorem:main2}, we obtain the following result for frame flows:

\begin{corollary}
\label{corollary}
Let $(N,g)$ be a closed oriented Anosov Riemannian manifold of dimension $n \geq 3$. Then the frame flow over $N$ is rapid mixing if and only if it is ergodic.
\end{corollary}

Corollary \ref{corollary} is straightforward to prove from Theorem \ref{theorem:main2}, with $M := SN$, $\varphi$ the geodesic flow (note that $d\alpha \neq 0$ since $\varphi$ is a contact flow), and applied to the frame bundle $P := FN \to SN$ (which is a principal $\mathrm{SO}(n-1)$-bundle). The proof of Corollary \ref{corollary} is given in \S\ref{ssection:application-frame-flow}.

Let us briefly comment on known results on frame flow ergodicity. It was conjectured by Brin \cite[Conjecture 2.9]{Brin-82} that the frame flow should be ergodic in negative curvature, unless $(M,g)$ admits a holonomy reduction to a strict subgroup of $\mathrm{SO}(n)$. Recall that a negatively curved manifold is $\delta$-pinched for some $\delta > 0$ (resp. strictly $\delta$-pinched) if there exists a constant $C > 0$ such that the sectional curvatures take values in the interval $[-C, -C \delta]$ (resp. $(-C,-C\delta)$). In negative curvature, all holonomy reductions appear with $\delta \leq 0.25$; it was thus also conjectured by Brin \cite[Conjecture 2.6]{Brin-82} that the frame flow of strictly $0.25$-pinched manifold should be ergodic. This problem is still open in full generality; however, ergodicity of the frame flow is known in the following cases:
\begin{enumerate}[label=(\roman*)]
\item When $(N,g)$ is Anosov, and $n$ is odd and different from $7$, see \cite{Brin-Gromov-80};
\item When $(N,g)$ has negative sectional curvature, $n$ is even or equal to $7$ and the metric is $\delta$-pinched with $\delta > \delta(n)$; the values of $\delta(n)$ are given by $\delta(4k+2)\sim 0.28...$, $\delta(4k) \sim 0.56...$ except for the three exotic dimensions $n=7,8,134$ for which $\delta(7)=0.49...$, $\delta(8)=0.62...$, $\delta(134)=0.57...$, see \cite{Cekic-Lefeuvre-Moroianu-Semmelmann-21,Cekic-Lefeuvre-Moroianu-Semmelmann-22} (and also \cite{Burns-Pollicott-03, Brin-Karcher-83} for previous results).
\end{enumerate}

A similar corollary holds on Kähler manifolds and the \emph{unitary frame flow}, that is the flow of frames compatible with the complex structure. We say that $(N, g, J)$ is an Anosov K\"ahler manifold if it is K\"ahler (where $J$ denotes a complex structure which is compatible with $g$) and the underlying Riemannian manifold $(N, g)$ is Anosov.

\begin{corollary}
\label{corollary:kahler}
Let $(N, g, J)$ be a closed Anosov Kähler manifold of complex dimension $\geq 2$. Moreover, if the complex dimension of $N$ equals $2$, assume that $(N, g)$ is negatively curved. Then the unitary frame flow over $N$ is rapid mixing if and only if it is ergodic.
\end{corollary}

We refer to \cite{Cekic-Lefeuvre-Moroianu-Semmelmann-23} for an introduction to unitary frame flows. On (complex) $2$-dimensional and odd-dimensional negatively curved Kähler manifolds, the unitary frame flow is ergodic by work of Brin-Gromov \cite{Brin-Gromov-80}. On even-dimensional Kähler manifolds (complex dimension $m \geq 6$ and $m \neq 28$), it is known to be ergodic for $\lambda(m)$-pinched holomorphic sectional curvature with $\lambda(m)\sim0.92...$, see \cite{Cekic-Lefeuvre-Moroianu-Semmelmann-23}. Finally, on hyperbolic (resp. complex hyperbolic) manifolds, the frame flow (resp. unitary frame flow) is exponentially mixing \cite{Moore-87, Guillarmou-Kuster-21}. This should hold conjecturally in variable curvature as well whenever the frame flow is ergodic. However, we are still unable to prove it using the techniques of the present chapter.

Finally, let us mention that more general frame flows can be considered as in \cite{Cekic-Lefeuvre-22} corresponding to an arbitrary vector bundle $E$ over $N$ equipped with a fibrewise inner product and a compatible connection. When the rank of $E$ is small compared to $n$, more precisely equal to $\mc{O}(\sqrt{n})$ as $n \to \infty$ (recall $n = \dim N$), the frame flow is ergodic if and only if the connection on $E$ has a full holonomy group, see \cite[Theorem 1.3]{Cekic-Lefeuvre-22}; so it is also rapid mixing thanks to Theorem \ref{theorem:main2} when the fibre is a semisimple Lie group (e.g. always true for $\rank(E) \geq 3$).

\subsection{Previous work}

Dolgopyat \cite[Theorem 3]{Dolgopyat-98} established rapid mixing for non-jointly integrable Anosov flows with respect to an arbitrary Gibbs measure. In the case $P = M$, our Theorem \ref{theorem:main2} recovers this in the case of volume-preserving flows, and provides an alternative proof. Next, Dolgopyat \cite{Dolgopyat-02} for the related case of extensions of topologically mixing Axiom A diffeomorphisms, establishes rapid mixing with respect to an arbitrary Gibbs measure under a generic assumption, and gives a reduction to the Abelian case as in Theorem \ref{theorem:main2}, Item (ii).  Still in the setting of Axiom A diffeomorphisms, \cite{Field-Melbourne-Torok-07} show that rapid mixing (on each basic set) is $C^2$-open and $C^\infty$-dense.

Using methods from harmonic analysis, the stronger property than rapid mixing, \emph{exponential mixing} (i.e., the one obtained by replacing $t^{-k}$ in \eqref{eq:rapid-mixing} by $e^{-\kappa t}$ for some $\kappa > 0$) is known to hold for frame flows over hyperbolic manifolds \cite{Moore-87, Pollicott-92} (this is also valid in the unitary frame flow setting over complex hyperbolic manifolds); see \cite{Guillarmou-Kuster-21} for a more recent approach in dimension $3$ based on spectral theory.

Compared to the mentioned works which all use \emph{symbolic coding}, it is evident that our approach is novel and spectral theoretic in nature. Moreover, determining the speed of mixing for frame flows is a highly relevant geometric questions, see for instance \cite{Kahn-Markovic-12}.

\subsection{Perspectives}

In Theorem \ref{theorem:main2}, exponential mixing is expected in place of rapid mixing. The techniques developed in the present paper appear to show this (combined with the approach of \cite{Cekic-Guillarmou-21}) when the stable/unstable foliation of the flow $\psi$ on $P$ is $C^{1+\eps}$ for some $\eps > 0$, which seems to be a very restrictive condition. In particular, for frame flows it was conjectured by Kanai \cite{Kanai-93} that the regularity $C^{1 + \varepsilon}$ is never achieved unless the manifold is locally symmetric (in which case it is $C^\infty$). We believe that the solution could lie in better understanding of the curvature of the dynamical connection seen as a distribution.



It is also interesting to investigate rapid mixing for measures obtained as a product of a Gibbs measure on $M$ (to start with, an SRB measure) with a bi-invariant measure on $G$ (in fact, in the related case of Anosov diffeomorphisms it is known that the property of rapid mixing is independent of which Gibbs measure we choose \cite[Corollary 4.10]{Dolgopyat-02}). Moreover, it might be possible to apply our techniques in the setting of Anosov actions to show rapid mixing, where the theory of Pollicott-Ruelle resonances was recently developed in \cite{Bonthonneau-Guillarmou-Weich-24, Bonthonneau-Guillarmou-Hilgert-Weich-23}.

Another related question is to extend the results on stochastic stability and Brownian motion to the setting of frame flows, following the recent sequence of papers \cite{Kolb-Weich-Wolf-20, Ren-Tao-22a, Ren-Tao-22b}, as well as \cite{Drouot-17, Dyatlov-Zworski-15}. These results recover the resonance or the Laplace spectrum as a limit of spectra of suitable hypoelliptic operators.

\subsection{Proof strategy}

The strategy of proof of Theorem \ref{theorem:main2} consists in proving high-frequency estimates on the resolvent $(X_P + z)^{-1}$ when $z \in \mathbb{C}$ belongs to the imaginary axis, i.e. as $|z| \to \infty$, and in a norm defined using spaces of anisotropic regularity tailored to the dynamics. Using Fourier analysis on the Lie group, this is reduced to studying the resolvent on Fourier modes, that is, on vector bundles associated to irreducible representations of the group. The main difficulty here is that there are ``two infinities'' to handle at the same time: the spectral infinity (as the imaginary part of $z$ goes to $\pm \infty$) and the infinity in Fourier modes. In the calculus of Chapter \ref{chapter:analysis}, this accounts for the two parameters introduced: $h$ will be the spectral parameter and $\mathbf{k}$ the Fourier mode. Arguing by contradiction, i.e. taking a suitable family of approximate solutions to $(X_{\mathbf{k}(h)} + z_h)u_{\mathbf{k}} = 0$ as $h \to 0$, we will show that $u_{\mathbf{k}}$ exhibit invariance in both stable and unstable directions (horocyclic invariance), and that it is compactly microlocalised in phase space (at the trapped set). Using \emph{Diophantine} properties of the group $G$ (see \ref{ssection:diophantine}), we will obtain a contradiction with the non-ingerability assumption in Item (i) of Theorem \ref{theorem:main2}.

\subsection{Organization of the chapter} 

In \S \ref{section:dynamics} we will summarise the preliminary dynamical concepts needed in the proof: the notions of the dynamical connection and the transitivity group (illustrated by giving explicit examples). In \S \ref{section:analytic-preliminaries} we outline the analytical preliminaries: we construct the anisotropic spaces tailored to dynamics and give a criterion for mixing in \S \ref{ssection:spectral-analysis}, and state the required propagation estimates in \S \ref{ssection:radial-estimates}. We deduce Theorem \ref{theorem:main2}, Item (i), from the spectral theorem assuming high-frequency resolvent estimates in \S \ref{section:proof-of-main-results}. Finally, in \S \ref{section:hf} we prove the high-frequency estimates in the case of $\mathrm{U}(1)$-extensions, and then in \S \ref{section:hf2} also in the case of arbitrary principal bundle extensions.

\section{Dynamical preliminaries}

\label{section:dynamics}

The purpose of this section is to understand stable and unstable holonomies on $P \to M$, the related concept of the \emph{dynamical connection}, and the associated bundles induced by the extension $\psi$. Finally, we discuss the Anosov $1$-form and the joint integrability of $E_s$ and $E_u$.

\subsection{Dynamical connection on the principal bundle}

\label{ssection:dynamical-connection}


We first describe the \emph{dynamical connection} $\nabla^{\mathrm{dyn}}$ induced by $\psi$. Recall that $G$ is a connected compact Lie group, $\pi : P \to M$ a $G$-principal bundle equipped with an extension $\psi$ of $\varphi$. Equip $M$ with an arbitrary metric and $P$ with an arbitrary smooth principal bundle connection. For $x,y \in M$ sufficiently close, write $\mc{P}_{x \to y}: P_x \to P_y$ for the parallel transport along the shortest geodesic between $x$ and $y$ induced by this connection. Here and below, $W^{s/u}_M(x)$ denotes the stable/unstable manifold of $x$.

\begin{definition}[Stable/unstable holonomies]
Let $x \in M, y \in W^{s}_M(x)$. For $w \in P_x$, define the stable holonomy $h^s_{x \to y} : P_x \to P_y$ as:
\begin{equation}
\label{equation:stable-holonomy}
h^s_{x\to y} w := \lim_{t \to \infty}  \psi_{-t}\mc{P}_{\varphi_t x \to \varphi_t y} \psi_t w.
\end{equation}
The unstable holonomy $h^u_{x \to y}$ is defined similarly.
\end{definition}

It can be checked that \eqref{equation:stable-holonomy} is well-defined (using the Ambrose-Singer formula, see \cite{Cekic-Lefeuvre-21-2}), independent of the choice of metric on $M$ and smooth principal bundle connection on $P$. Moreover, since both $\psi$ and the parallel transport $\mc{P}_{\bullet \to \bullet}$ commute with the right $G$-action, from \eqref{equation:stable-holonomy} we conclude that $h^{s/u}_{\bullet \to \bullet}$ is also $G$-equivariant: 
\[
	h^{s/u}_{x \to y} R_g = R_g h^{s/y}_{x \to y}.
\]
This defines stable/unstable bundles on $P$ by
\[
	E^{s/u}_P := \mathrm{Span}\left( \partial_t h^{s/u}_{x \to y_t} w|_{t = 0} ~|~ (y_t)_{t \in (-1,1)} \subset W^{s/u}(x),\, y_0 = x\right) \subset TP,
\]
where $(y_t)_{t \in (-1, 1)}$ denotes a smooth curve such that $y_0 = x$. We thus obtain a horizontal distribution
\[
\HH^\psi := \R X_P \oplus E^s_P \oplus E^u_P,
\]
invariant by the $G$-action, where $X_P$ is the vector field on $P$ generating $\psi$. The space $TP$ therefore splits as
\begin{equation}
\label{equation:splitting-p}
TP =  \V \oplus \HH^\psi = \V \oplus E^s_P \oplus E^u_P \oplus \R X_P,
\end{equation}
where $\V := \ker \dd \pi$ is the kernel of the basepoint projection $\pi : P \to M$. The flow $\psi$ is \emph{partially hyperbolic} (it acts isometrically on the $\V$ component in \eqref{equation:splitting-p}). Note that $\V$ can be naturally identified with $\mathfrak{g}$, the Lie algebra of $G$ via the isomorphism
\begin{equation}
\label{equation:identification}
	\mathfrak{g} \to \V_w, \qquad \xi \mapsto \partial_t (w \cdot e^{t\xi})|_{t = 0}, \quad w \in P,\, \xi \in \mathfrak{g}.
\end{equation}
Without further notice, we will use this identification and, in particular, write $\xi$ for the vertical vector field obtained by $w \mapsto \partial_t (w \cdot e^{t\xi})|_{t = 0}$. We also emphasize at this stage that, most of the time, $\HH^\psi$ is merely Hölder-continuous, see \cite[Theorem 4.11]{Crovisier-Potrie-notes}

The space $\HH^\psi$ is the horizontal space associated to a principal $G$-connection $\nabla^{\mathrm{dyn}}$ on $P$ called the \emph{dynamical connection}. Let $\Theta \in C^\eps(P,T^*P \otimes \mathfrak{g})$ be the connection $1$-form induced by the dynamical connection and defined by
\begin{equation}
\label{equation:Theta}
	\Theta|_{\HH^\psi} = 0, \qquad  \Theta(\xi) = \xi , \quad \forall \xi \in \mathfrak{g}.
\end{equation}
(Here again, we use the identification $\mathfrak{g} \simeq \V$ via \eqref{equation:identification}.) Write $F_{\nabla^{\mathrm{dyn}}}$ for the curvature of the dynamical connection $\nabla^{\mathrm{dyn}}$, which a distributional $2$-form with values in the vector bundle $\Ad(P) \to M$ associated to $P$ via the adjoint representation. We will denote by $\Lie_Y$ the Lie derivative along the vector field $Y$.

\begin{lemma}
\label{lemma:computations}
The following identities are satisfied:
\begin{enumerate}[label=\emph{(\roman*)}]
	\item $\iota_X F_{\nabla^{\mathrm{dyn}}} = 0$, 
	\item $\iota_{X_P} d \Theta = 0$,
	\item $\mc{L}_{X_P} \Theta = 0$, 
	\item $\mc{L}_\xi\Theta = - [\xi, \Theta]$, $\forall \xi \in \mathfrak{g}$.
\end{enumerate}
\end{lemma}

\begin{proof}
(iv) This follows from the general theory of connections, see \cite[Proposition 1.1, Chapter II]{Kobayashi-Nomizu-63}, which gives $R_g^* \Theta = \operatorname{Ad}(g^{-1}) \Theta$ for any $g \in G$, where $\Ad: G \to \End(\mathfrak{g})$ is the adjoint representation. Therefore:
\[
	\Lie_{\xi} \Theta(w)(v) = \partial_t|_{t = 0}\big(\Ad(e^{-t\xi}) \Theta\big)(w)(v) = - [\xi, \Theta(w)(v)], 
\]
for any $w \in P$, $v \in T_{w} P$, where in the first equality we used the mentioned formula and in the second one we used that the derivative of the adjoint action is the Lie bracket. \\

(iii) To show this formula, it suffices to show that for any $t \in \mathbb{R}$, we have $\psi_t^*\Theta = \Theta$. By definition of pullback we have
\[
	\psi_t^*\Theta(w)(v) = \Theta(\psi_t(w))(d\psi_t(w) v), \quad w \in P,\, v \in T_wP.
\]
If $v \in \HH^\psi(w)$, by the flow invariance of $\HH^\psi$ and definition of $\Theta$, the above quantity vanishes. Moreover, if we identify $v \in \V(w)$ with $\eta$ for some $\eta \in \mathfrak{g}$ via \eqref{equation:identification}, then
\[
	d\psi_t(w)v = \partial_{s}|_{s = 0}\psi_t(w \cdot e^{s\eta}) = \partial_{s}|_{s = 0} \psi_t(w) \cdot e^{s \eta},
\] 
and so $\psi_t^* \Theta(w)(v) = \eta$. This proves the claim. \\

(ii) This follows readily from Item 3 by the use of Cartan's magic formula since $\iota_{X_P} d\Theta = \mc{L}_{X_P} \Theta - d\iota_{X_P}\Theta = 0$, where the second term is zero as $\Theta$ vanishes on horizontal vectors and $X_P$ is horizontal. \\

(i) Finally, for the first equation, recall that 
\begin{equation}\label{eq:curvature-Ad(P)}
	F_{\nabla^{\mathrm{dyn}}}(x)(Y_1, Y_2) = p\big(\Omega(p)(Y_1^{\HH}(p), Y_2^{\HH}(p))\big),\quad x \in M, Y_1, Y_2 \in T_xM, 
\end{equation}
where $p \in P_x$ is arbitrary, $\Omega = d\Theta + [\Theta, \Theta]$ is the curvature $2$-form with values in $\mathfrak{g}$, $Y_1^{\HH}(p)$ and $Y_2^{\HH}(p)$ are horizontal lifts of $Y_1$ and $Y_2$ at $p$, respectively, and $p: \mathfrak{g} \to \Ad(P)(x)$ is seen as a linear isomorphism (note that the curvature is a distributional section so strictly speaking the above identity holds when integrated against smooth functions). That $\iota_{X} F_{\nabla^{\mathrm{dyn}}} = 0$ follows from Item 2 and from the fact that $\iota_{X_P}\Theta = 0$.
%
\end{proof}

Let $T \leqslant G$ be a maximal torus. As described in \S\ref{section:connection-principal-bundles}, the dynamical connection induces a connection on all associated bundles, and in particular on $F = P/T$, as well as on the line bundles $L_1, \dotsc, L_d \to F$ appearing in Borel-Weil theory (see \S \ref{ssection:borel-weil}). This implies that the tangent space to the flag bundle $F \to M$ splits as
\[
	TF = \HH^\psi_F \oplus \V_F = \R X_F \oplus E^s_F \oplus E^u_F \oplus \V_F,
\]
and there is a dual decomposition
\begin{equation}\label{eq:F-horizontal-vertical}
	T^*F = (\HH^\psi_F)^* \oplus \V_F^* = \R \pi^*\alpha \oplus (E^s_F)^* \oplus (E^u_F)^* \oplus \V_F^*,
\end{equation}
where $\pi : F \to M$ is the projection and $\alpha$ the Anosov $1$-form defined by the equalities:
\[
	\V_F^*(\HH_F^\psi)=0= (\HH^{\psi}_F)^*(\V_F), \quad (E^s_F)^*(\R X_F \oplus E^s_F \oplus \V_F)=0= (E^u_F)^*(\R X_F \oplus E^u_F \oplus \V_F)
\]
The connection $\nabla_{\mathbf{k}}$ associated to $\nabla^{\mathrm{dyn}}$ on $\mathbf{L}^{\otimes \mathbf{k}} \to F$ is only a partial connection and it is completed using the vertical Chern connection to a total connection. The vector of distributional curvatures of these total connections on $L_1, \dotsc, L_d$ is denoted by $\mathbf{F}_{\nabla^{\mathrm{dyn}}} \in \mc{D}'(F,(\Lambda^2 T^*F)^{\oplus d})$. Recall that $X_F$ is the generator of the flow induced on $F$. The following holds:

\begin{lemma}
One has $\iota_{X_F} \mathbf{F}_{\nabla^{\mathrm{dyn}}} = 0$.
\end{lemma}

\begin{proof}
By the proof of Lemma \ref{lemma:vanishing-curvature}, the total connection on any of $L_j \to F$ is the connection associated to the connection $1$-form $\widetilde{\Theta} = \Pi_{\mathfrak{t}} \Theta$ on $P \to F$, where $\Pi_{\mathfrak{t}}: \mathfrak{g} \to \mathfrak{t}$ is the projection along the real part of $\mathfrak{n}^+ \oplus \mathfrak{n}^-$. We then have 
\[
	\widetilde{\Omega} := d \widetilde{\Theta} + [\widetilde{\Theta}, \widetilde{\Theta}] = d\widetilde{\Theta} = \Pi_{\mathfrak{t}} d\Theta,
\] 
and using Lemma \ref{lemma:computations}, Item 2, we get that $\iota_{X_P}\widetilde{\Omega} = 0$. Since $\mathbf{F}_{\nabla^{\mathrm{dyn}}}$ is obtained through an associated construction from $\widetilde{\Theta}$ as in \eqref{eq:curvature-Ad(P)}, and $X_P$ is the horizontal lift to $P$ of $X_F$, the conclusion follows. 
\end{proof}

\subsubsection{Transitivity group}\label{sssection:transitivity-group}

For principal bundle extensions it is possible to detect ergodicity using the concept of the \emph{transitivity group} which we briefly recall. Namely, fix an arbitrary periodic point $x_\star \in M$ and consider the free monoid $\mathbf{G}$ generated by all orbits which are homoclinic to $\gamma_\star$ (that is, those orbits which accumulate to $\gamma_\star$ in both past and future). For a given homoclinic orbit $\gamma$, choose arbitrary $x_{s/u} \in \gamma \cap W^{s/u}(x_\star)$. Then $x_s = \psi_T(x_u)$ for some $T \in \mathbb{R}$, and we define \emph{Parry's representation} as
\[
	\rho_{\mathrm{Parry}}(\gamma) := h_{x_s \to x_\star}^s \circ \psi_t \circ h_{x_\star \to x_u}^u.
\]
We then extend $\rho_{\mathrm{Parry}}$ to $\mathbf{G}$ using the monoid structure. After an arbitrary identification of $P_{x_\star}$ with $G$, we may write $\rho_{\mathrm{Parry}}: \mathbf{G} \to G$. Define the \emph{transitivity} and the \emph{pre-transitivity} groups as the closure $H := \overline{\rho_{\mathrm{Parry}}(\mathbf{G})} \leqslant G$ and $\rho_{\mathrm{Parry}}(\mathbf{G}) \leqslant G$, respectively. It is shown in \cite{Brin-75-1} and \cite{Lefeuvre-23} that $\psi$ is ergodic if and only if $H = G$ (see also \cite{Cekic-Lefeuvre-22}).

Alternatively, we may define the transitivity group by using arbitrary piecewise smooth loops in $M$ at a fixed basepoint, obtained by concatenations of stable, unstable, flowlines, and stable, unstable holonomies and by applying the flow.

\subsubsection{$\mathrm{U}(1)$ extensions and flat connections}\label{sssection:U(1)} Let us first focus on the case $G=\mathrm{U}(1)$. Since $F=M$, the vertical space $\V_F = \{0\}$ is trivial and $\HH^\psi_F = TM$. Similarly to \eqref{equation:dynamical-lambda}, the action of $\psi$ on the complex line bundle $L \to M$ whose underlying circle bundle is $P$ (introduced in \S\ref{ssection:u1}) induces a dynamical connection
\begin{equation}
\label{equation:dynamical-connection-u1}
	\nabla^{\mathrm{dyn}} : C^\infty(M,L) \to C^\eps(M,L \otimes T^*M),
\end{equation}
where $\varepsilon > 0$ is the H\"older regularity of the dynamical connection on $P$. Note that the dynamical connection $\nabla_k^{\mathrm{dyn}}$ on $L^{\otimes k}$ (for $k \in \Z$) is simply obtained by taking the connection induced by \eqref{equation:dynamical-connection-u1} on the $k$-th tensor power of $L$.

In this case, the connection $1$-form $\Theta \in C^\varepsilon(P,T^*P)$ can be identified with a $1$-form on $P$ such that $\Theta(V)=1$ and $\Theta|_{\HH^\psi} = 0$, where $V$ is the vertical vector field generating the right $\mathrm{U}(1)$ action (we make the identification $\operatorname{U}(1) = \mathbb{R}/2\pi \mathbb{Z}$). Since $\Theta$ is only Hölder-continuous, the curvature form $d\Theta$ is merely a distribution on $P$. Equivalently, writing $\nabla^{\mathrm{dyn}} = d + i \beta_{\mathrm{dyn}}$ in a local trivialisation, where $\beta_{\mathrm{dyn}}$ is a real-valued Hölder-continuous $1$-form, the curvature $F_{\nabla^{\mathrm{dyn}}} := id \beta_{\mathrm{dyn}} \in \mc{D}'(M,\Lambda^2 T^*M)$ is a distribution. For $k \geq 1$, $\nabla^{\mathrm{dyn}}_k = d+ik\beta_{\mathrm{dyn}}$. It is enlightening to have in mind the case of trivial circle extensions.

\begin{example}[Trivial circle extension]
\label{example:trivial-extension}
Let $P := M \times \mathrm{U}(1)$ be the trivial circle bundle. Then
\[
\psi_t(x,\theta) := \left(\varphi_t x, \theta + \int_0^t a(\varphi_q x)\, \dd q \text{ mod } 2\pi\right),
\]
for some smooth function $a \in C^\infty(M)$ and $X_P(x,\theta) = X_M(x) + a(x) \partial_\theta$. A straightforward computation shows that
\begin{align*}
	h^s_{x\to y} \theta &= \theta + \int_0^\infty (a(\varphi_qx) - a(\varphi_q y))\, dq \mod 2\pi,\\ 
	h^u_{x\to y} \theta &= \theta - \int_0^\infty (a(\varphi_{-q}x) - a(\varphi_{-q} y))\, dq \mod 2\pi,
\end{align*}
where in the first and second equalities $x, y$ belong to stable and unstable manifolds, respectively. Differentiating the holonomy expressions when $y$ is varied over a path in stable/unstable manifolds starting at $x$, it is then immediate that
\begin{align*}
	E_P^s(x, \theta) &= \left\{v -  \left(\int_0^\infty da \circ d\varphi_q(v) \, dq\right) \partial_\theta \mid v \in E_s(x)\right\},\\
	E_P^u(x, \theta) &= \left\{v + \left(\int_0^\infty da \circ d\varphi_{-q}(v) \, dq\right)\partial_\theta \mid v \in E_u(x)\right\}.
\end{align*}
This data uniquely determines the horizontal space of the dynamical connection, and in particular it determines the horizontal lift $Y^{\HH}$ of vector fields $Y$ on $M$. The equivariant lift of a section $s$ of $L \to M$ (note that $L$ is the line bundle associated to $P$ via the representation $\rho: \mathrm{U}(1) \to \mathrm{U}(1), z \mapsto z$), to a section of $P \times \mathbb{C} \to P$ is checked to be $\overline{s}(x, z) = z^{-1}s(x)$. The associated connection $\nabla^{\mathrm{dyn}}$ on $L$ is by definition determined via
\[
	\nabla^{\mathrm{dyn}}_Y s(x) = (Y + i\beta_{\mathrm{dyn}}(Y)) s(x) = z Y^{\HH} \overline{s} (x, z),\quad x \in M,
\]
for any $z \in \mathbb{S}^1$, where we identified $L$ with $M \times \mathbb{C}$, and $z: \mathbb{C} \to L_x, \xi \mapsto [(x, z), \xi]$ with the multiplication by $z$. Writing $Y^{\HH} = Y + f \partial_{\theta}$ for some $f \in C^\infty(P)$ and using $\partial_\theta z = iz$, we get $\beta_{\mathrm{dyn}}(Y) = -f$. This then gives $\beta_{\mathrm{dyn}}(X_M) = -a$ and for $\xi_{s/u} \in E^{s/u}_M$ we compute:
%
\begin{equation}
\label{equation:integral0}
\beta_{\mathrm{dyn}}(\xi_s) = \int_0^{\infty} da \circ d\varphi_t(\xi_s)\, \dd t, \qquad \beta_{\mathrm{dyn}}(\xi_u)=-\int_0^{\infty} da \circ d\varphi_{-t}(\xi_u)\, \dd t.
\end{equation}
On $P = M \times \mathrm{U}(1)$, $V = \partial_\theta$ and the horizontal space is then given by
\[
	\HH^\psi = \{Y - \beta_{\mathrm{dyn}}(Y)\partial_\theta ~|~ Y \in C^\infty(M,TM)\} \subset TP.
\]
The connection $1$-form on $P$ is given by $\Theta = \pi^*\beta_{\mathrm{dyn}} + d\theta$ (note that the connection $1$-form on $P$ is real-valued because we identified $\mathrm{U}(1)$ with $\mathbb{R}/2\pi \mathbb{Z}$), where $\pi: P \to M$ is the projection. The curvature of the dynamical connection on $L$ is given by $id \beta_{\mathrm{dyn}}$. (We will need this computation in the proof of Lemma \ref{lemma:no-resonances} below.)

The curvature of the dynamical connection seen as a $2$-form with values in $\Ad(P) = M \times \mathbb{R}$ is by definition 
\[
	F_{\nabla^{\mathrm{dyn}}}(x)(X, Y) = p\big(d\Theta(p)(X^{\HH}, Y^{\HH})\big) = d\pi^*\beta_{\mathrm{dyn}}(p)(X^{\HH}, Y^{\HH}) = d\beta_{\mathrm{dyn}}(x)(X, Y),
\]
for $x\in M$, $X, Y \in T_xM$.

\end{example}

\begin{example}[Example of non-mixing]\label{example:trivial-extension-concrete}
	Still in the setting of Example \ref{example:trivial-extension}, consider the case $a \equiv 1$. Then $E^{s/u}_P$ agrees with $E^{s/u}$ and so $\Theta = -\alpha + d\theta$. Therefore, the condition of Theorem \ref{theorem:main2}, item (i), is not satisfied;. This is in agreement with \cite[Lemma 3.9]{Lefeuvre-23} which shows that $\psi$ is ergodic, but not mixing (for instance, the self-correlation function of $f(x, \theta) = e^{i \theta}$ does not decay as $t \to \infty$). It is tempting to conjecture that $\psi$ is not mixing if and only if $a$ is cohomologous to a constant function.
\end{example}

\begin{example}[Flat connections]
\label{example:flat-connection}
Let $G$ be a compact Lie group and $\rho: \pi_1(M) \to G$ be a representation. This determines a flat connection on a $G$-principal bundle over $M$ as follows. Let $\widetilde{M}$ be the universal cover of $M$, and define
\[
	P := \widetilde{M} \times G/\sim, \quad (\widetilde{x} \cdot [\gamma], g) \sim (\widetilde{x}, \rho([\gamma]) g), \quad \forall \widetilde{x} \in \widetilde{M}, g \in G, [\gamma] \in \pi_1(M), 
\]
so the product connection on $\widetilde{M} \times G$ descends to $P \to M$ (here, $[\gamma]$ acts on $\widetilde{M}$ by covering transformations). Then the holonomy group $\operatorname{Hol}(P,\nabla)$ of $P$ is equal to $\rho(\pi_1(M))$. Consider the horizontal lift $X_P$ of $X$. Since the connection is flat, holonomy along contractible loops is trivial, and hence we get that from \eqref{equation:stable-holonomy} that $h^s_{x\to y}: P_x \to P_y$ (where $y \in W^s(x)$) is simply parallel transport along \emph{any} path in $W^s(x)$ from $x$ to $y$ (the global leaves are diffeomorphic to $\mathbb{R}^k$ for some $k$); similar statement holds for unstable holonomy. In particular, this shows that the dynamical and the flat connections on $P$ agree. Moreover, since any loop in $M$ can be approximated by flow, unstable, and stable paths, we get that the holonomy group contains the pre-transitivity group $K$ and is contained in the transitivity group $\overline{K}$ (see \S \ref{sssection:transitivity-group} for a definition), i.e. $K \leqslant \mathrm{Hol}(P) \leqslant \overline{K}$.

In particular, $\psi$ is ergodic if and only if $\rho(\pi_1(M)) \leqslant G$ is dense, and moreover, if $G$ is semisimple, then Theorem \ref{theorem:main2} shows that if $\psi$ is ergodic and $d\alpha \neq 0$ then $\psi$ is rapid mixing. When $G$ has a non trivial centre (i.e. of dimension $\geq 1$), since the dynamical connection is flat, Theorem \ref{theorem:main2}, Item (i), does not give any information about rapid mixing.
\end{example}

\subsection{Joint integrability}

\label{ssection:joint-integrability}

Let $\alpha \in C^\eps(M,T^*M)$ be the Anosov $1$-form. We briefly discuss the condition $d\alpha=0$, its relation to joint integrability of $E^s_M$ and $E^u_M$. Recall that $E^s_M$ and $E^u_M$ are said to be jointly integrable if there exists a (smooth) foliation $\mc{F}$ on $M$ with smooth leaves $L \in \mc{F}$ such that $TL = E^s_M \oplus E^u_M$.

\begin{lemma}\label{lemma:joint-integrability}
Assume that $d\alpha=0$ in the sense of distributions. Then $\alpha \in C^\infty(M,T^*M)$ is smooth, and $E^s_M$ and $E^u_M$ are jointly integrable. Conversely, if $E^s_M$ and $E^u_M$ are jointly integrable, then $d\alpha = 0$.
\end{lemma}

\begin{proof}
Write $\alpha = \omega + df$, where $\omega$ is a smooth harmonic $1$-form and $f \in C^{1+\eps}(M)$. Then $\iota_X \alpha = 1 = \iota_X \omega + Xf$, that is $Xf \in C^\infty(M)$. By radial estimates (see \cite[Theorem 1.3]{Bonthonneau-Lefeuvre-23} for instance), we obtain that $f \in C^\infty(M)$, and thus $\alpha$ is smooth. By Frobenius' theorem it then follows that $d\alpha = 0$ is equivalent to the integrability of $E^s_M \oplus E^u_M$. 
\end{proof}

Up to our knowledge, it is still a conjecture that joint integrability (in the volume-preserving case) only occurs for suspensions of Anosov diffeomorphisms by constant roof functions. This is known as Plante's conjecture \cite{Plante-72}.

\section{Analytic preliminaries}
\label{section:analytic-preliminaries}

\subsection{Spectral analysis}

\label{ssection:spectral-analysis}

In what follows, $X$ is a volume-preserving vector field on $M$ that generates an Anosov flow $\varphi$, $G$ is a compact connected Lie group, $T$ a maximal torus in $G$, and $\psi$ is an extension of $\varphi$ to a $G$-principal bundle $P \to M$ generated by $X_P$. We denote by $\Phi$ the symplectic lift of $\varphi$ to $T^*M$ and by $H \in C^\infty(T^*M,T(T^*M))$ its generator.

We will implement the analysis developed in Chapter \ref{chapter:analysis}. We denote by $F := P/T$ the flag bundle and $\mathbf{L}^{\otimes {\mathbf k}} := L_1^{\otimes k_1} \otimes \dotsm L_d^{\otimes k_d} \to F$ (for $\mathbf{k} \in \widehat{G}$ identified with $\mathbb{Z}^a \times \mathbb{Z}^b_{\geq 0}$ via \eqref{equation:ecriture}) the fiberwise holomorphic line bundles that appear in the Borel-Weil theorem (see \S \ref{ssection:borel-weil}). We let $\nabla$ be an arbitrary smooth connection on $P$. This induces a natural connection on $\mathbf{L}^{\otimes \mathbf{k}}$, see \S \ref{section:connection-principal-bundles}. Denote for $i = 1, \dotsc, d$, $F_{\overline{\nabla}_i}$ the (purely imaginary) curvature $2$-forms $L_i \to F$ (see \S \ref{sssection:curvature} and \S \ref{sssection:homogeneous-line}), and set $\mathbf{F}_{\overline{\nabla}} = (F_{\overline{\nabla}_1}, \dotsc, F_{\overline{\nabla}_d})$.

We recall that $\X_{\mathbf{k}}$ is the first order differential operator introduced in \eqref{equation:def-conn}, and that the Fourier transform intertwines $X_P$ and $\X_{\mathbf{k}}$ by \eqref{eq:ft-intertwines-infinitesimal}. Note that the multiplicity of the representation $\mathbf{k} \in \widehat{G}$ appearing in the Fourier decomposition \eqref{eq:ft-bw-isomorphism} is $d_{\mathbf{k}}$. We will therefore write $f_{\mathbf{k}, i}$ for $i = 1, \dotsc, d_{\mathbf{k}}$ to denote the corresponding component of $f \in C^\infty(P)$; however, when the context is clear we will drop the index $i$.

\subsubsection{Resolvents. Meromorphic extension} \label{sssection:anisotropic-space}

For all $\mathbf{k} \in \widehat{G}$, the resolvent
\begin{equation}
\label{equation:resolvent}
(-\X_{\mathbf{k}} - z)^{-1} := -\int_0^{\infty} e^{-t z} e^{-t \X_{\mathbf{k}}}\, \dd t,
\end{equation}
is well-defined and holomorphic on both $L^2(M, H^0(F, \Lk))$ and $C^0(M, H^0(F, \Lk))$ for $\Re(z) > 0$ (here, we write $H^0(F, \Lk) \to M$ for the vector bundle whose fibre at $x \in M$ consists of holomorphic sections of $\Lk|_{F_x} \to F_x$). It is now a classical result due to Faure-Sjöstrand \cite{Faure-Sjostrand-11} that \eqref{equation:resolvent} admits a meromorphic extension to $\C$ as an operator $C^\infty(M, H^0(F, \Lk)) \to \mc{D}'(M, H^0(F, \Lk))$; its poles are called the \emph{Pollicott-Ruelle resonances} of $\X_{\mathbf{k}}$, see Theorem \ref{theorem:faure-sjostrand} below. To state the theorem, we need to introduce a scale of anisotropic spaces.

Consider a Riemannian metric $g$ on $M$; this induces a natural fibrewise inner product on $T^*M$ denote by the same later. Fix an \emph{order function} $m$, that is a symbol in $S^0(T^*M)$, which is $0$-homogeneous in the $\xi$-variable for $|\xi|\gg 1$ large enough such that $m \equiv 1$ near $\overline{E_s^*} \cap \partial_\infty T^*M$, $m \equiv -1$ near $\overline{E_u^*} \cap \partial_\infty T^*M$, and the function
\[
G_m(x,\xi) := m(x,\xi) \log \langle \xi \rangle_g, \qquad \langle\xi\rangle_g := (1+|\xi|^2_g)^{1/2},
\]
satisfies the following:
\begin{enumerate}[label=(\roman*)]
\item $HG_m(x,\xi) \leq 0$ near $\partial_\infty T^*M$;
\item $HG_m(x,\xi) \leq -C < 0$ near $\partial_\infty T^*M \cap (\overline{E_s^*} \cup \overline{E_u^*})$, for some constant $C >0$.
\end{enumerate}
The existence of such pairs $(g, m)$ for Anosov flows can be found in \cite{Faure-Sjostrand-11} or \cite[Chapter 9, Section 9.1]{Lefeuvre-book} for instance.


Next, we observe that the construction in Chapter \ref{chapter:analysis} carries over to the case of \emph{anisotropic symbols}, that is it is possible to define a calculus $\Psi_{h, \operatorname{BW}}^{m(x, \xi)}(P)$ satisfying similar properties (see \cite[Appendix A]{Faure-Roy-Sjostrand-08}). By the construction of \S\ref{ssection:quantization-symbols}, there exists a family $\mathbf{A}(s) := \Op^{\mathrm{BW}}_h(\pi^*e^{sG_m}) \in \Psi^{s m(x,\xi)}_{h, \operatorname{BW}}(P)$
of operators (where $\pi^*e^{sG_m}$ denotes the function pulled back to $\HH^*_F$ as in \S\ref{sssection:pullbackfunctions}) which map $C^\infty_{\mathrm{hol}}(F,\mathbf{L}^{\otimes \mathbf{k}})$ into itself 
and for $(w,\xi) \in \HH^*_F$, $x := \pi(w) \in M$:
\[
	\sigma^{\nabla, \operatorname{BW}}_{\mathbf{A}_{\mathbf k}(s)}(w,d\pi^{\top}_x\xi) = \langle\xi\rangle^{sm(x,\xi)} = \exp(sG_m(x,\xi)).
\]
Up to lower-order modifications, we can also assume that $\mathbf{A}(s)$ is invertible (for all $h,\mathbf{k}$). For $s > 0$, and $f_{\mathbf k} \in C^\infty_{(\mathrm{hol})}(F,\mathbf{L}^{\otimes \mathbf{k}})$, define:
\begin{equation}
\label{equation:norm}
\|f_{\mathbf{k}}\|_{\mc{H}^s_{h}(F,\mathbf{L}^{\otimes \mathbf{k}})} := \|\mathbf{A}_{h, \mathbf{k}}(s)f_{\mathbf k}\|_{L^2(F,\mathbf{L}^{\otimes \mathbf{k}})},
\end{equation}
and $\mc{H}^s_{h,(\mathrm{hol})}(F,\mathbf{L}^{\otimes \mathbf{k}})$ to be the completion of $C^\infty_{(\mathrm{hol})}(F,\mathbf{L}^{\otimes \mathbf{k}})$ with respect to the norm \eqref{equation:norm}. Note that by density of $C^\infty_{\mathrm{hol}}(F,\mathbf{L}^{\otimes \mathbf{k}})$ inside $\mc{H}^s_{h, \mathrm{hol}}(F,\mathbf{L}^{\otimes \mathbf{k}})$, we have that $\mc{H}^s_{h, \mathrm{hol}}(F,\mathbf{L}^{\otimes \mathbf{k}}) \subset \mc{D}'_{\mathrm{hol}}(F,\mathbf{L}^{\otimes \mathbf{k}})$, where $\mc{D}'_{\mathrm{hol}}$ denotes distributions which are fibrewise holomorphic.

The following holds:

\begin{theorem}
\label{theorem:faure-sjostrand}
There exists a constant $c > 0$ such that for all $\mathbf{k} \in \widehat{G}$, the operator
\[
(-\X_{\mathbf{k}}-z)^{-1} : \mc{H}^s_{1/\langle\mathbf{k}\rangle,\mathrm{hol}}(F,\mathbf{L}^{\otimes \mathbf{k}}) \to \mc{H}^s_{1/\langle\mathbf{k}\rangle,\mathrm{hol}}(F,\mathbf{L}^{\otimes \mathbf{k}}),
\]
initially defined and holomorphic for $\Re(z) \gg 0$ admits a meromorphic extension to $\{\Re(z)>-cs\}$. The poles are contained in the half-plane $\{\Re(z) \leq 0\}$ and independent of choices made in the construction.
\end{theorem}
\begin{proof}
	The proof is verbatim the same as in \cite{Faure-Sjostrand-11} or \cite[Chapter 9, Section 9.1]{Lefeuvre-book}. 
\end{proof}

We will say that $z \in \mathbb{C}$ is a \emph{resonance} of $\X_{\mathbf{k}}$ if there exists $s$ large enough such that $\ker(\X_{\mathbf{k}} - z) \neq \{0\}$ on $\mc{H}^s_{1/\langle\mathbf{k}\rangle, \mathrm{hol}}(F,\mathbf{L}^{\otimes \mathbf{k}})$. Moreover, $u \in \mc{H}^s_{1/\langle\mathbf{k}\rangle, \mathrm{hol}}(F,\mathbf{L}^{\otimes \mathbf{k}})$ will be called a \emph{resonant state} or a \emph{generalised resonant state} at $z$ if $u \in \ker(\X_{\mathbf{k}} - z)$ or $u \in \ker(\X_{\mathbf{k}} - z)^J$ for some $J \geq 1$, respectively.

The key technical result in what follows will be to establish \emph{high-frequency} estimates on the norm of the resolvent $\|(-\X_{\mathbf{k}}-z)^{-1}\|_{\mc{H}^s_{1/\langle\mathbf{k}\rangle,\mathrm{hol}}(F,\mathbf{L}^{\otimes \mathbf{k}})}$ as $|\mathbf{k}| \to \infty$ and $\Im z \to \infty$, see \S\ref{ssection:statement-estimates} and \S\ref{section:hf}.

\subsubsection{Mixing. Absence of resonances of the imaginary axis}

We now prove that the operators $\X_{\mathbf{k}}$ have no resonances on the imaginary axis, except $z=0$ for $\mathbf{k}=0$ associated to constant functions or, equivalently, that $\psi$ is mixing under the assumptions of Theorem \ref{theorem:main2}, Item (i). Before stating the lemma, we remark that by Lemma \ref{lemma:abelian-curvature} this assumption is equivalent to linear independence of $(\pi^*d\alpha, -iF_{\overline{\nabla}_{1}}, \dotsc, -iF_{\overline{\nabla}_{a}})$, where $\pi: F \to M$ is the projection. Denote by $\psi^F$ the quotient flow on the flag bundle $F \to M$.

\begin{lemma}
\label{lemma:no-resonances}

Assume that $\psi^F$ is ergodic, and that $(\pi^*d\alpha, -iF_{\overline{\nabla}_1}, \dotsc, -iF_{\overline{\nabla}_a})$ are linearly independent over $\mathbb{R}$. Then the following holds:
\begin{enumerate}[label=\emph{(\roman*)}]
\item $\X_{\mathbf{k}}$ has no resonances on $\{\Re(z)=0\}$ for $\mathbf{k} \neq 0$;
\item $\X_0 = X$ has a unique simple resonance at $z=0$ on the imaginary axis, and the space of generalised resonant states is equal to $\C\cdot\mathbf{1}_M$ (constant functions).
\end{enumerate}
In particular, $\psi$ is mixing on $P$. 
\end{lemma}

%

\begin{proof}
We first show that (i-ii) implies that $\psi$ is mixing. For any $f, g \in C^\infty(P)$ with zero average, we have 
\[
	\langle{e^{tX_P}f, g}\rangle_{L^2} = \sum_{\mathbf{k} \in \widehat{G}} d_{\mathbf{k}} \sum_{i = 1}^{d_{\mathbf{k}}} \langle{e^{t\X_{\mathbf{k}}}f_{\mathbf{k}, i}, g_{\mathbf{k}, i}}\rangle_{L^2},
\]
where $f_{\mathbf{k}, i}, g_{\mathbf{k}, i} \in C^\infty_{\mathrm{hol}}(F,\Lk)$ are the Fourier modes of $f, g$, respectively, and we used the inverse Fourier transform \eqref{equation:inversion}, as well as that the Fourier transform takes $X_P$ to $\X_{\mathbf{k}}$ (see \eqref{eq:ft-intertwines-infinitesimal}), and Plancherel's theorem (see Lemma \ref{lemma:ft-isometry}). Next, we may identify $\X_0$ with $X$ since for $\mathbf{k} = 0$, $\Lk \to F$ is naturally identified with a pullback bundle from the base, see \S \ref{sssection:abelian-rep}; also, $f_0, g_0$ have zero mean. Following the proof of \cite[Lemma 3.7, bottom of page 22]{Lefeuvre-23}, an application of Stone's formula and the spectral theorem shows that for each $\mathbf{k}$, $i = 1, \dotsc, \mathbf{k}$, $\langle{e^{t\X_{\mathbf{k}}}f_{\mathbf{k}, i}, g_{\mathbf{k}, i}}\rangle_{L^2} \to 0$ as $t \to \infty$, which upon using the Dominated Convergence Theorem proves the claim.

For condition (ii), note that by the same argument as in the preceding paragraph (see the proof of \cite[Lemma 3.7]{Lefeuvre-23}), the absence of non-zero resonances of $X$ on the imaginary axis is equivalent to mixing of $\varphi$ which, in turn, by the Anosov alternative is equivalent to $\varphi$ not being a suspension by a constant roof function (see \cite[Appendix A]{Lefeuvre-23} for instance). That we are not in the latter scenario is ensured by Lemma \ref{lemma:joint-integrability} and the assumption that $d\alpha \neq 0$, proving (ii).



It remains to show that $\X_{\mathbf{k}}$ has no resonances on the imaginary axis for $\mathbf{k} \neq 0$. Assume that there exists a non-zero $f_{\mathbf{k}} \in \mc{H}^s_{1/\langle\mathbf{k}\rangle,\mathrm{hol}}(F,\Lk)$ such that $\X_{\mathbf{k}} f_{\mathbf{k}} = i\lambda  f_{\mathbf{k}}$ for some $\lambda \in \R$. Then, by \cite[Lemma 2.3]{Dyatlov-Zworski-17} (see also the proof of \cite[Lemma 4]{Faure-Roy-Sjostrand-08} and \cite[Proposition 9.3.3]{Lefeuvre-book}), $f_{\mathbf{k}} \in C^\infty_{\mathrm{hol}}(F,\Lk)$. Since $\X_{\mathbf{k}}$ is unitary, we obtain $X_F|f_{\mathbf{k}}|^2 = \langle\X_{\mathbf{k}}f_{\mathbf{k}},f_{\mathbf{k}}\rangle+\langle f_{\mathbf{k}},\X_{\mathbf{k}}f_{\mathbf{k}}\rangle = 0$, where $\langle{\bullet, \bullet}\rangle$ denotes the inner product in the fibres of $\Lk \to F$. By ergodicity of $\psi^F$, this implies that $|f_{\mathbf{k}}|$ is a non-zero constant so $\mathbf{L}^{\otimes \mathbf{k}} \to F$ is topologically trivial. We then split to cases according to the value of $\mathbf{k}$. \medskip

\emph{Case 1: $(k_{a + 1}, \dotsc, k_d)$ is non-zero.} By Proposition \ref{prop:line-bundle-topology}, then $\Lk \to F$ is topologically a non-trivial line bundle, which is a contradiction.
\medskip

\emph{Case 2: $k_{a + 1} = \dotsc = k_{d} = 0$}. By Lemma \ref{lemma:extension}, Lemma \ref{lemma:1d-rep}, and \S \ref{sssection:abelian-rep}, we view $f_{\mathbf{k}}$ as a section of $\pi^*\Lk_M$, which is holomorphic and hence constant in the fibres of $F$. Here, we recall $\Lk_M := P \times_{\widetilde{\gamma}} \mathbb{C}$ where $\widetilde{\gamma}: G \to \mathbb{S}^1$, $\widetilde{\gamma}|_{T} = \gamma$. Thus, we may write $f_{\mathbf{k}} = \pi^*f_{\mathbf{k}, M}$, where $f_{\mathbf{k}, M}$ is a section of $\Lk_M$ and by Lemma \ref{lemma:pullback-equivalence} it satisfies that $\X_{\mathbf{k}, M} f_{\mathbf{k}, M} = i\lambda f_{\mathbf{k}, M}$. The section $s := f_{\mathbf{k}, M}$ of $\Lk_M$ is nowhere vanishing since $f_{\mathbf{k}}$ is and re-normalising, we may assume $s$ has pointwise unit norm and a section of the underlying circle bundle $\mc{C}_{\mathbf{k}, M} \subset \Lk_M$. 

Using $s$, we identify $\Lk_M$ and $\mc{C}_{\mathbf{k}, M}$ with $M \times \mathbb{C}$ and $M \times \operatorname{U}(1)$, respectively. Then, the operator $\X_{\mathbf{k}, M}$ is conjugated to $X + i\lambda$, and the flow induced on $M \times \mathrm{U}(1)$ is generated by $X - \lambda\partial_\theta$ (in the notation of Example \ref{example:trivial-extension} it corresponds to $a \equiv -\lambda$). Since the construction of the dynamical connection behaves well under associated construction (i.e. the dynamical connection of the associated flow is the associated dynamical connection), and using Example \ref{example:trivial-extension} (in particular \eqref{equation:integral0}), the $1$-form of the dynamical connection on $M \times \mathbb{C}$ is given by $\beta_{\mathrm{dyn}} = \lambda \alpha$, where $\alpha$ is the Anosov $1$-form. Thus the curvature of the connection on $M \times \mathbb{C}$ and so of $\pi^*\nabla_{\mathbf{k}, M}$ on $\pi^*\Lk_M$ is given by $\mathbf{k} \cdot \mathbf{F}_{\overline{\nabla}} = i\lambda \pi^*d\alpha$, which contradicts the assumption of the lemma, and completes the proof. 
\end{proof}


\begin{remark}
	We observe that in the proof of Lemma \ref{lemma:no-resonances} we actually needed a much weaker assumption than linear independence of $(\pi^*d\alpha, -iF_{\overline{\nabla}_1}, \dotsc, -iF_{\overline{\nabla}_a})$, that is, we only required that for all $\mathbf{k} \in \mathbb{Z}^a \times \{0\}^b$ we have $\mathbf{k} \cdot \mathbf{F}_{\overline{\nabla}} \not \in i \mathbb{R} \pi^*d\alpha$. Therefore, in particular we get that this weaker assumption, as well as ergodicity of $\psi^F$, implies mixing of $\psi$ on $P$.
\end{remark}

\subsection{Radial source/sink estimates} \label{ssection:radial-estimates} Radial source and sink estimates were first introduced by Melrose \cite{Melrose-94} in scattering theory. In the context of Anosov flows, they were discovered by Dyatlov and Zworski \cite{Dyatlov-Zworski-16}.


We now work specifically with the Anosov vector field $X$ and $\X := \nabla^{\mathrm{dyn}}_X$. However, what follows also holds in greater generality (i.e. of first order differential operators with source/sink structure), and we stick to the dynamical setting for simplicity. We introduce the thresholds:
\begin{align}
\label{equation:omega+}
\omega_+(X) &:= \inf\left\{\rho > 0 ~\big|~ \sup_{x \in M} \lim_{t \to +\infty} \tfrac{1}{t}\log\big(\mathrm{Jac}_{\varphi_{-t}x}(\varphi_t)^{1/2}\|\dd \varphi_{-t}|_{E_u(x)}\|^\rho\big) < 0\right\}, \\
 \label{equation:omega-}
\omega_-(X) &:=  \inf\left\{\rho > 0 ~\big|~ \sup_{x \in M} \lim_{t \to +\infty} \tfrac{1}{t}\log\big(\mathrm{Jac}_{x}(\varphi_t)^{1/2}\|\dd \varphi_{t}|_{E_s(x)}\|^\rho\big) < 0\right\}.
\end{align}
Here, the norms are computed with respect to an arbitrary background Riemannian metric $g$ on $M$, and we denote by $\mathrm{Jac}_x(\varphi_t)$ the determinant of $d\varphi_t(x): T_xM \to T_{\varphi_t x}M$ with respect to $g$. Note that, if $X$ is volume-preserving, then $\omega_\pm(X) = 0$.
In the next proposition, elliptic and wavefront set are computed with respect to $\nabla^{\mathrm{dyn}}$.

\begin{theorem}[Source/sink estimates]
\label{theorem:source-sink}
The following holds:
\begin{enumerate}[label=\emph{(\roman*)}, itemsep=5pt]
\item \emph{\textbf{Source estimate.}} For all $s > \omega_+(X)$, $N > 0$ large enough, for all $\mathbf{A} \in \Psi^0_{h, \mathrm{BW}}(P)$ with wavefront set near $\overline{(E^s_F)^*} \cap \partial_\infty (\HH^\psi_F)^*$, there exists $\mathbf{B} \in \Psi^0_{h, \mathrm{BW}}(P)$ with wavefront set and elliptic on a slightly larger neighborhood than $\WF^{\mathrm{BW}}(\mathbf{A})$, and a constant $C > 0$ such that for all $h > 0$, $\mathbf{k}\in \widehat{G}$, such that $h|\mathbf{k}| \leq 1$, for all families $f_{h,\mathbf{k}} \in C^\infty_{\mathrm{hol}}(F,\mathbf{L}^{\otimes \mathbf{k}})$,
\begin{equation}
\label{equation:source-estimate}
\|\mathbf{A} f_{h,\mathbf{k}}\|_{H^s_h} \leq C\left( \|\mathbf{B} \X_{\mathbf{k}} f_{h,\mathbf{k}}\|_{H^s_h} + h^N\|f_{h,\mathbf{k}}\|_{H^{-N}_h} \right).
\end{equation}
If $f_{h,\mathbf{k}} \in \mc{D}'_{\operatorname{hol}}(F,\mathbf{L}^{\otimes \mathbf{k}})$ is merely a family of distributions such that $\mathbf{B}\X_{\mathbf{k}} f_{h,\mathbf{k}} \in H^{s_0}_h$ for some $s_0 > \omega_+(X)$, then $\mathbf{A}f_{h,\mathbf{k}} \in H^s_h$ for all $s > \omega_+(X)$ and \eqref{equation:source-estimate} holds.

\item \emph{\textbf{Sink estimate.}} For all $s < \omega_-(X)$, $N > 0$, for all $\mathbf{A} \in \Psi^0_{h, \mathrm{BW}}(P)$ with wavefront set near $\overline{(E^u_F)^*} \cap \partial_\infty (\HH^\psi_F)^*$, there exists $\mathbf{B} \in \Psi^0_{h, \mathrm{BW}}(P)$ with wavefront set and elliptic on a slightly larger neighborhood than $\WF^{\mathrm{BW}}(\mathbf{A})$, $\mathbf{A}' \in \Psi^0_{h, \mathrm{BW}}(P)$ such that $\WF^{\mathrm{BW}}(\mathbf{A}') \subset \WF^{\mathrm{BW}}(\mathbf{B})$ and $\WF^{\mathrm{BW}}(\mathbf{A}') \cap \overline{(E^u_F)^*} \cap \partial_\infty (\HH^\psi_F)^* = \emptyset$, and a constant $C > 0$ such that for all $h > 0$, $\mathbf{k}\in \widehat{G}$, such that $h|\mathbf{k}| \leq 1$, for all families $f_{h,\mathbf{k}} \in C^\infty_{\mathrm{hol}}(F,\mathbf{L}^{\otimes \mathbf{k}})$,
\begin{equation}
\label{equation:sink-estimate}
\|\mathbf{A} f_{h,\mathbf{k}}\|_{H^s_h} \leq C\left( \|\mathbf{B} \X_{\mathbf{k}} f_{h,\mathbf{k}}\|_{H^s_h} + \|\mathbf{A}' f_{h,\mathbf{k}}\|_{H^s_h}+h^N\|f_{h,\mathbf{k}}\|_{H^{-N}_h} \right).
\end{equation}
If $f_{h,\mathbf{k}} \in \mc{D}'_{\operatorname{hol}}(F,\mathbf{L}^{\otimes \mathbf{k}})$ is merely a family of distributions such that $\mathbf{B}\X_{\mathbf{k}} f_{h,\mathbf{k}} \in H^{s_0}_h$ and $\mathbf{A}' f_{h,\mathbf{k}} \in H^{s_0}_h$ for some $s_0 < \omega_-(\X)$, then $\mathbf{A} f_{h,\mathbf{k}} \in H^s_h$ for all $s < \omega_-(X)$ and \eqref{equation:sink-estimate} holds.
\end{enumerate}
\end{theorem}

Once again, we emphasize that working with the dynamical connection (to define principal symbols, elliptic/wavefront sets) is actually irrelevant in the previous theorem. More precisely, the inequality \eqref{equation:source-estimate} holds \emph{regardless} of the connection chosen to compute the symbols and elliptic/wavefront sets, similarly to the parametrix construction in Proposition \ref{proposition:ellipticity} for instance. However, it is convenient for practical purposes to work with the dynamical connection because the underlying Hamiltonian flow associated to the operator $\X$ is merely the symplectic lift of $(\varphi_t)_{t \in \R}$ to $T^*M$. If one works with another (smooth) connection $\nabla'$, then the new Hamiltonian flow of $\X$ computed with respect to this new connection is topologically conjugate to the previous one and the source/sink structure of the flow is preserved.

Typically, the operator $\mathbf{A}$ in the source estimate (i) is constructed by $\mathbf{A} := \Op^{\mathrm{BW}}_h(a_{h,\mathbf{k}})$ (see \eqref{equation:quantization-geometric}) and $a_{h,\mathbf{k}}$ is chosen equal to $1$ near $\overline{E_s^*} \cap \partial_\infty T^*M$. The only (slight) issue at this stage when working with the dynamical connection is that the $\Op^{\mathrm{BW}}_h$ quantization relies on a choice of \emph{smooth} connection (see \eqref{equation:quantization-geometric}) and one cannot take the dynamical connection to define it here. Nevertheless, taking another smooth connection $\nabla'$ (arbitrarily close to $\nabla^{\mathrm{dyn}}$), it is straightforward to verify that the principal symbol of ${\Op^{\mathrm{BW}}_h}^{\nabla'}(a_{h,\mathbf{k}})$ computed with respect to $\nabla^{\mathrm{dyn}}$ is still equal to $1$ near $\overline{E_s^*} \cap \partial_\infty T^*M$.

The proof of Theorem \ref{theorem:source-sink} is a (non-trivial) consequence of the Egorov-type theorem stated in Proposition \ref{proposition:egorov}. Within the calculus $\Psi^\bullet_{h, \mathrm{BW}}(P)$, it is \emph{verbatim} the same proof as in the non-twisted case so we omit the details and refer the reader to \cite[Chapter 9, Section 2]{Lefeuvre-book} for a proof.

Regarding the thresholds \eqref{equation:omega+} and \eqref{equation:omega-}, the line bundle $\mathbf{L}^{\otimes \mathbf{k}}$ does not play any role since the connection is unitary. The threshold computed in \cite[Chapter 9, Section 2]{Lefeuvre-book} is slightly weaker and we refer to \cite[Appendix B]{Bonthonneau-Lefeuvre-23} for a proof of the radial estimates with the right thresholds \eqref{equation:omega+} and \eqref{equation:omega-} (in the main body of \cite{Bonthonneau-Lefeuvre-23} the authors work with thresholds on H\"older-Zygmund spaces).

Finally, one can replace $\X_{\mathbf{k}}$ by $\X_{\mathbf{k}} + \lambda$ in \eqref{equation:source-estimate} for $\lambda \in \C$. The threshold is then changed to $\max(\omega_+(X)-\Re(\lambda),0)$ (the $\Re(\lambda)$ factor comes from estimating the propagator $e^{-t(\X_{\mathbf{k}} + \lambda)}$ pointwise, see \cite[Appendix B]{Bonthonneau-Lefeuvre-23}).

\section{Proof of main results}
\label{section:proof-of-main-results}

\subsection{Estimates on the resolvent}

\label{ssection:statement-estimates}

Let us fix $s \geq 0$ and $m \in \mathbb{R}$. We now introduce a bi-graded scale of anisotropic spaces $\mc{H}^{m,s}(P)$ defined as follows. For $f \in C^\infty(P)$, decomposing its Fourier transform $\mc{F}f = (f_{\mathbf{k}, i})_{\mathbf{k}\in \widehat{G}, i = 1, \dotsc, d_{\mathbf{k}}}$ with $f_{\mathbf{k}, i} \in C^\infty_{\mathrm{hol}}(F,\mathbf{L}^{\otimes \mathbf{k}})$, we set


\begin{equation}
\label{equation:space}
\|f\|_{\mc{H}^{m,s}(P)}^2 := \sum_{\mathbf{k} \in \widehat{G}} d_{\mathbf{k}} \sum_{i = 1}^{d_{\mathbf{k}}} \langle \mathbf{k} \rangle^{2m} \|f_{\mathbf{k}, i}\|^2_{\mc{H}^s_{1/\langle\mathbf{k}\rangle}(F,\mathbf{L}^{\otimes \mathbf{k}})},
\end{equation}
where we recall that the anisotropic space $\mc{H}^s_{1/\langle\mathbf{k}\rangle}(F,\mathbf{L}^{\otimes \mathbf{k}})$ was introduced in \eqref{equation:norm}. The space $\mc{H}^{m,s}(P)$ is then defined as the completion of $C^\infty(P)$ with respect to the norm \eqref{equation:space}. We also set $\mc{H}^{m,s}_0(P) := \mc{H}^{m,s}(P) \cap (\C\mathbf{1})^\perp$ ($L^2$-orthogonal to constant functions). The following lemma asserts that the $\mc{H}^{m,s}(P)$ norm is bounded by some isotropic Sobolev norm.

\begin{lemma}\label{lemma:anisotropic-isotropic}
Let $s, m\in \R$. Define $p:=\max(m+|s|, |s|)$. Then, there exists $C > 0$ such that for all $f \in C^\infty(P)$,
\[
	\|f\|_{\mc{H}^{m,s}(P)}^2 \leq C \|f\|_{H^{p}(P)}^2.
\]
\end{lemma}

\begin{proof}
By construction and Proposition \ref{proposition:proprietes}, Item (v), the space $H^{s}_{1/\langle\mathbf{k}\rangle}(F,\mathbf{L}^{\otimes \mathbf{k}})$ embeds continuously into the space $\mc{H}^{|s|}_{1/\langle\mathbf{k}\rangle}(F,\mathbf{L}^{\otimes \mathbf{k}})$, and the embedding is uniform as $|\mathbf{k}| \to \infty$. Namely, there exists a constant $C>0$ such that for all $\mathbf{k} \in \widehat{G}$,
\[
	\|f\|_{\mc{H}^s_{1/\langle\mathbf{k}\rangle}(F,\mathbf{L}^{\otimes \mathbf{k}})} \leq C \|f\|_{H^{|s|}_{1/\langle\mathbf{k}\rangle}(F,\mathbf{L}^{\otimes \mathbf{k}})}, \qquad \forall f \in C^\infty_{\mathrm{hol}}(F,\mathbf{L}^{\otimes \mathbf{k}}).
\]
Recall that the norm on $H^{s}_{1/\langle\mathbf{k}\rangle}(F,\mathbf{L}^{\otimes \mathbf{k}})$ is defined using the operator $(1 + \langle{\mathbf{k}}\rangle^{-2} (\Delta_{\mathbf{k}} + \overline{\partial}_{\mathbf{k}}^*\overline{\partial}_{\mathbf{k}}))^{\tfrac{s}{2}}$ (see \eqref{eq:sobolev-norm}). By \eqref{equation:ck-asymptotic}, $ \langle\mathbf{k}\rangle^2 = \mc{O}(c(\mathbf{k}))$. By Lemmas \ref{lemma:ft-isometry} and \ref{lemma:ft-horizontal-laplacian}, the Fourier transform is an $L^2$ isometry and it intertwines the horizontal Laplace operators, and respects the vertical Laplacian. In the case $m \in 2\mathbb{Z}_{\geq 0}$ and $s \in 2\mathbb{Z}$, we thus obtain:
\[
\begin{split}
\|f\|_{\mc{H}^{m,s}(P)}^2 &= \sum_{\mathbf{k} \in \widehat{G}} d_{\mathbf{k}} \sum_{i=1}^{d_{\mathbf{k}}} \langle\mathbf{k}\rangle^{2m} \|f_{\mathbf{k},i}\|^2_{\mc{H}^s_{1/\langle\mathbf{k}\rangle}} \\
&  \leq C\sum_{\mathbf{k} \in \widehat{G}} d_{\mathbf{k}} \sum_{i=1}^{d_{\mathbf{k}}} \langle\mathbf{k}\rangle^{2m} \|(1+\langle{\mathbf{k}}\rangle^{-2}\Delta_{\mathbf{k}})^{|s|/2} f_{\mathbf{k},i}\|^2_{L^2} \\
& \leq C\sum_{\mathbf{k} \in \widehat{G}} d_{\mathbf{k}} \sum_{i=1}^{d_{\mathbf{k}}} \|\langle\mathbf{k}\rangle^{m}(1+\Delta_{\mathbf{k}})^{|s|/2} f_{\mathbf{k},i}\|^2_{L^2} \\
& \leq C \|\langle\Delta_{\V}\rangle^{m/2}(1+\Delta_{\HH})^{|s|/2}f\|^2_{L^2(P)}\\
& \leq  C\|f\|_{H^{m+|s|}(P)}^2.
\end{split}
\]
where $C > 0$ could vary from line to line. In the case $m \leq 0$, the proof is the same up to the fourth line, where we now simply use that $\langle{\mathbf{k}}\rangle^m \leq 1$. 

The general case $s, m \in \mathbb{R}$ then follows by interpolation; this completes the proof.

\end{proof}



Define the following band in the complex plane
\begin{equation}\label{eq:band}
	\B:=\{z \in \C ~|~ 0 \leq \Re(z) \leq 1\}.
\end{equation}
The following bound is key to Theorem \ref{theorem:main2}:
\begin{theorem}
\label{theorem:real}
Assume that $\psi^F$ is ergodic and that $(\pi^*d\alpha, -iF_{\overline{\nabla}_1}, \dotsc, -iF_{\overline{\nabla}_a})$ are linearly independent over $\mathbb{R}$. Fix $s > 0$. There exist $C, \vartheta > 0$ such that for all $z \in \B$, $\mathbf{k} \in \widehat{G}$,
\begin{equation}
\label{equation:uniform-bound}
\|(-\X_{\mathbf{k}} - z)^{-1}\|_{\mc{H}^s_{1/\langle\mathbf{k}\rangle,\mathrm{hol}}(F,\mathbf{L}^{\otimes \mathbf{k}}) \circlearrowleft} \leq C \langle \Im(z) \rangle^\vartheta \langle \mathbf{k} \rangle^\vartheta,
\end{equation}
with the exception that for $\mathbf{k}=0$, \eqref{equation:uniform-bound} holds in restriction to distributions with $0$ average. Moreover, we may take $\vartheta = \max(\dim F + 10,3\dim F + 4) + 2$ in \eqref{equation:uniform-bound}.
\end{theorem}

This bound is trivial if $\mathbf{k}$ and $z$ are both in a fixed compact window. The proof of Theorem \ref{theorem:real} is postponed to \S\ref{section:hf} ($\mathrm{U}(1)$ case) and \S\ref{section:hf2} (general case). Denote by $\mc{L}(A, B)$ the space of continuous linear maps between normed spaces $A$ and $B$ equipped with operator norm. Then Theorem \ref{theorem:real} implies as a corollary:

\begin{corollary}
\label{corollary:hf}
Assume that $\psi^F$ is ergodic and that $(\pi^*d\alpha, -iF_{\overline{\nabla}_1} \dotsc, -iF_{\overline{\nabla}_a})$ are linearly independent over $\mathbb{R}$. Fix $p \geq 0, m \in \R, s > 0$. Let $\vartheta > 0$ be the constant given by Theorem \ref{theorem:real}. Then
\begin{enumerate}[label=\emph{(\roman*)}]
\item The resolvent
\[
\R \ni \lambda \mapsto (- X_P - i\lambda)^{-1} \in \mc{L}\left(\mc{H}^{p\vartheta+m,s}_0(P),\mc{H}^{m,s}_0(P)\right)
\]
is well-defined on the imaginary axis and $C^p$-regular with respect to $\lambda$.

\item For all $p \geq 0, m \in \R$, there exists $C > 0$ such that
\[
\|\partial^p_\lambda(- X_P-i\lambda)^{-1}\|_{\mc{H}^{(p + 1)\vartheta+m,s}_0\to\mc{H}^{m,s}_0} \leq C \langle \lambda \rangle^{\vartheta (p + 1)}.
\]
\end{enumerate}
\end{corollary}

\begin{proof}
(i) Fix $\lambda \in \R$ and consider for $u \in C^\infty(P)$, $\eps > 0$,
\begin{equation}
\label{equation:feps}
f^{(\eps)} := (-X_P - (i \lambda + \eps))^{-1} u.
\end{equation}
Since $X_P$ is skew-adjoint on $L^2(P)$, this is well-defined for $\eps > 0$ and $\|f^{(\eps)}\|_{L^2(P)} \leq 1/\eps$. We claim that there exists a unique $f_\star \in \cap_{m \in \R} \mc{H}^{m,s}_0(P)$ such that for any $m \in \R$,
\[
f^{(\eps)} \to_{\eps \to 0} f_\star \qquad \text{ in } \mc{H}^{m,s}(P).
\]
Moreover, $\zeta = f_\star \in \mc{H}^{m,s}(P)$ is the unique solution in $\cap_{m \in \R} \mc{H}^{m,s}_0(P)$ to the equation $(-X_P-i\lambda) \zeta = u$. This defines $f_\star =: (-X_P-i\lambda)^{-1}u$.

First, to show uniqueness, consider $\zeta \in \cap_{m \in \R} \mc{H}^{m,s}_0(P)$ such that $(-X_P-i\lambda)\zeta = 0$. Decomposing in Fourier modes, this is equivalent to $(-\X_{\mathbf{k}}-i\lambda) \zeta_{\mathbf{k}} = 0$ for all ${\mathbf{k}} \in \widehat{G}$ and $\zeta_{\mathbf{k}} \in \mc{H}^s_{\mathrm{hol}}(F,\Lk)$.
By Lemma \ref{lemma:no-resonances}, there are no resonances on the imaginary axis, so this forces $\zeta_{\mathbf{k}} \equiv 0$ and thus $\zeta \equiv 0$.

We now show that $f_\star$ is well-defined. Observe that \eqref{equation:feps} is equivalent to
\[
(-\X_{\mathbf{k}} -(i\lambda+\eps))f^{(\eps)}_{\mathbf{k}} = u_{\mathbf{k}}, \qquad \forall {\mathbf{k}} \in \widehat{G}.
\]
By Theorem \ref{theorem:faure-sjostrand} the resolvent $(-\X_{\mathbf{k}}-z)^{-1}$ is meromorphic on a strip $\Re(\lambda) > - cs$ (where $c$ is independent of ${\mathbf{k}}$, it only depends on the flow $\varphi$ on $M$) as a bounded operator in $\mc{H}^s_{1/|{\mathbf{k}}|,\mathrm{hol}}(F,\Lk)$. By Lemma \ref{lemma:no-resonances}, there are no poles on the imaginary axis, hence
\[
f^{(\eps)}_{\mathbf{k}} = (-\X_{\mathbf{k}}-(i\lambda+\eps))^{-1}u_{\mathbf{k}} \to_{\eps \to 0} (-\X_{\mathbf{k}}-i\lambda)^{-1}u_{\mathbf{k}} =: w_{\mathbf{k}} \in \mc{H}^s_{1/|{\mathbf{k}}|,\mathrm{hol}}(F,\Lk),
\]
with bound
\[
\|w_{\mathbf{k}}\|_{\mc{H}^s_{1/|{\mathbf{k}}|}(F,\Lk)} \leq C \langle \lambda\rangle^\vartheta \langle {\mathbf{k}} \rangle^\vartheta  \|u_{\mathbf{k}}\|_{\mc{H}^s_{1/|{\mathbf{k}}|}(F,\Lk)},
\]
by \eqref{equation:uniform-bound}. Also note that \eqref{equation:uniform-bound} implies $\|f^{(\eps)}_{\mathbf{k}}\|_{\mc{H}^s_{1/|{\mathbf{k}}|}} \leq 2 C \langle \lambda\rangle^\vartheta \langle {\mathbf{k}} \rangle^\vartheta  \|u_{\mathbf{k}}\|_{\mc{H}^s_{1/|{\mathbf{k}}|}}$ uniformly as $\eps \to 0$.

As a consequence, defining $f_\star := \mc{F}^{-1}(w_{\mathbf{k}})_{{\mathbf{k}}\in\widehat{G}}$ (where for simplicity we drop the index $i = 1,\dotsc, d_{\mathbf{k}}$), we see that $f_\star \in \mathcal{H}^{m,s}_0(P)$ since $u \in \mc{H}^{\vartheta +m,s}_0(P)$ with bound $\|f_\star\|_{\mc{H}^{m,s}(P)} \leq C \langle \lambda \rangle^\vartheta  \|u\|_{\mc{H}^{m+\vartheta ,s}(P)}$, and by dominated convergence that $f^{(\eps)} \to f_\star$ in $\mc{H}^{m,s}(P)$. This shows that
\[
(-X_P-i\lambda)^{-1} \in \mc{L}\left(\mc{H}^{\vartheta +m,s}_0(P),\mc{H}^{m,s}_0(P)\right)
\]
is well-defined on the imaginary axis and
\begin{equation}\label{eq:poly-estimate}
\|(-X_P-i\lambda)^{-1}\|_{\mc{H}^{\vartheta +m,s}_0(P) \to \mc{H}^{m,s}_0(P)} \leq C \langle \lambda \rangle^\vartheta .
\end{equation}

We now want to show smoothness with respect to the parameter $\lambda$. For that, it suffices to use the resolvent formula
\[
(-X_P-i\lambda)^{-1} - (-X_P-i\lambda')^{-1} = i(\lambda-\lambda')(-X_P-i\lambda)^{-1}(-X_P-i\lambda')^{-1},
\]
which shows that
\[
\lambda \mapsto (-X_P-i\lambda)^{-1} \in C^1\left(\R, \mc{L}(\mc{H}^{2\vartheta +m,s}_0(P),\mc{H}^{m,s}_0(P))\right)
\]
with derivative
\[
\partial_\lambda (-X_P-i\lambda)^{-1} = i (-X_P-i\lambda)^{-2}.
\]
A similar argument for higher order derivatives gives that
\[
\lambda \mapsto (-X_P-i\lambda)^{-1} \in C^p\left(\R, \mc{L}(\mc{H}^{(p + 1)\vartheta +m,s}_0,\mc{H}^{m,s}_0)\right),
\]
with derivatives 
\begin{equation}\label{eq:derivative-formula}
	\partial^p_\lambda (-X_P-i\lambda)^{-1} = i^{p} p! (-X_P-i\lambda)^{-(p+1)}.
\end{equation}

(ii) This follows immediately by iterating the estimate \eqref{eq:poly-estimate} and using the formula \eqref{eq:derivative-formula}.
\end{proof}

\subsection{Proof of main results}

\label{ssection:implication}

Using Corollary \ref{corollary:hf}, we can now complete the proof of Theorem \ref{theorem:main2}.

\begin{proof}[Proof of Theorem \ref{theorem:main2}, Item \emph{(i)}]
The proof is based on the spectral theorem. The (unbounded) operator $-iX_P$ on $L^2(P)$ with domain
\[
\mc{D}_{L^2} := \{ f \in L^2(P) ~|~ X_Pf \in L^2(P)\}
\]
is self-adjoint. Hence, it admits a spectral measure $\dd P(\lambda)$ such that $\mathbbm{1}_{L^2(P)} = \int_{\R} \dd P(\lambda)$.

We claim that $\dd P(\lambda)$ is smooth with respect to $\lambda$ and given by
\begin{equation}
\label{equation:p}
\dd P(\lambda) = -\dfrac{1}{2\pi}(R_+(-i\lambda) + R_-(i\lambda))\, \dd \lambda.
\end{equation}
Indeed, first $\dd P(\lambda)$ has no atoms by Lemma \ref{lemma:no-resonances} on the space of $L^2$ functions with zero average (and the only atom on $L^2$ is at $\lambda = 0$ corresponding to constant functions). Then, by using Stone's formula, we get that for $f,g \in C^\infty(P)$ with zero average,
\[
\int_a^b \langle\dd P(\lambda)f, g\rangle_{L^2(P)} = - \dfrac{1}{2\pi} \int_a^b \langle (R_+(-i\lambda) + R_-(i\lambda)) f, g \rangle_{L^2(P)}\, \dd \lambda,
\]
for all $a < b$, where $R_{\pm}(i\lambda) := (\pm X_P + i\lambda)^{-1}$. This proves \eqref{equation:p}. Moreover, by Corollary \ref{corollary:hf},
\[
\lambda \mapsto \langle (R_+(-i\lambda) + R_-(i\lambda)) f, g \rangle_{L^2(P)}
\]
is smooth.

Now, for $f,g \in C^\infty(P)$ such that $\int f\, \dd \mu = 0$, we have:
\[
\int_{P} (f \circ \psi_t) g ~\dd\mu = \langle e^{tX_P}f,g \rangle_{L^2(P)} = \int_{\R} e^{it\lambda} \langle \dd P(\lambda)f,g\rangle_{L^2(P)}.
\]
We will use $X_P\dd P(\lambda) = \dd P(\lambda) X_P = i\lambda \dd P(\lambda)$ and $D_\lambda e^{it\lambda} = t e^{it\lambda}$ with $D_\lambda := -i\partial_{\lambda}$. Write $\langle{D_\lambda}\rangle^2 = 1 + D_\lambda^2$ and $\langle{X_P}\rangle^2 = 1 - X_P^2$. Taking $N_1 \gg 1$ arbitrary and even, and $N_2 > \vartheta (N_1+1)+1$ even, we get:
\[
\begin{split}
 \int_{\R} e^{it\lambda} \langle \dd P(\lambda)f,g\rangle_{L^2(P)} & = \langle t\rangle^{-N_1} \int_{\R} \langle D_\lambda\rangle^{N_1} (e^{it\lambda}) \langle \lambda\rangle^{-N_2} \langle \dd P(\lambda)\langle X_P\rangle^{N_2}f,g\rangle_{L^2(P)} \\
 & = \langle t\rangle^{-N_1} \int_{\R} e^{it\lambda} \langle D_\lambda\rangle^{N_1}\left( \langle \lambda\rangle^{-N_2} \langle \dd P(\lambda)\langle X_P\rangle^{N_2}f,g\rangle_{L^2(P)}\right),
 \end{split}
\]
where the last equality follows by integration by parts. Indeed, we claim that the previous integral converges absolutely. This can be seen by expanding the term
\begin{multline*}
\langle D_\lambda\rangle^{N_1}\left( \langle \lambda\rangle^{-N_2} \langle \dd P(\lambda)\langle X_P\rangle^{N_2}f,g\rangle_{L^2(P)}\right) = (-1)^{N_1/2}\langle \lambda \rangle^{-N_2} \langle \partial_\lambda^{N_1} \dd P(\lambda) \langle X_P \rangle^{N_2}f,g\rangle_{L^2}\\ 
+ \mathrm{l.o.t},
\end{multline*}
where $\mathrm{l.o.t}$ are lower-order terms in the $\lambda$ variable. Note that
\[
\begin{split}
\partial_\lambda^{N_1}\dd P(\lambda) & = - \dfrac{1}{2\pi} \partial_\lambda^{N_1} \left((-X_P+i\lambda)^{-1}+(X_P-i\lambda)^{-1}\right)\, \dd \lambda \\
& = -\dfrac{i^{N_1} N_1!}{2\pi}\left( (-1)^{N_1}(-X_P+i\lambda)^{-(N_1+1)} + (X_P-i\lambda)^{-(N_1+1)}\right)\, \dd \lambda,
\end{split}
\]
Write for any $s \geq 0$ and $m \in \mathbb{R}$, $\mc{H}^{m, s}_\pm := \mc{H}^{m, \pm s}$ and  $(\mc{H}^{m, s}_\pm)'$ for the $L^2$-dual space. It is straightforward to check that $(\mc{H}^{m, s}_\pm)'$ can be identified with $\mc{H}^{-m, \mp s}$. Also, by assumption we may write $-N_2+\vartheta (N_1+1) = -(1 + \varepsilon)$ for some $\varepsilon \in (0, 2]$. Therefore we get
\begin{equation}\label{eq:absolute-convergence}
\begin{split}
 &|\langle \lambda \rangle^{-N_2} \langle \partial_\lambda^{N_1} \dd P(\lambda) \langle X_P \rangle^{N_2}f,g\rangle_{L^2}|\\ 
 &\leq C \langle \lambda \rangle^{-(1+\eps)}\left(\|\langle X_P\rangle^{N_2}f\|_{\mc{H}^{m+\vartheta (N_1+1),s}_+} \|g\|_{(\mc{H}^{m,s}_+)'} +\|\langle X_P\rangle^{N_2}f\|_{\mc{H}^{m+\vartheta (N_1+1),s}_-} \|g\|_{(\mc{H}^{m,s}_-)'}\right) \\
&\leq C \langle \lambda \rangle^{-(1+\eps)}\left(\|\langle{X_P}\rangle^{N_2}f\|_{\mc{H}^{m+\vartheta (N_1+1),s}} \|g\|_{\mc{H}^{-m,-s}}  + \|\langle X_P\rangle^{N_2}f\|_{\mc{H}^{m+\vartheta (N_1+1), -s}} \|g\|_{\mc{H}^{-m,s}}\right)\\
& \leq C \langle \lambda \rangle^{-(1+\eps)} \|f\|_{H^{s + \vartheta(N_1 + 1) + 1 + \varepsilon}(P)} \|g\|_{H^{s + \vartheta(N_1 + 1) + 1 + \varepsilon}(P)},
\end{split}
\end{equation}
where $C > 0$ varies from line to line, and in last line we specified to $m = -N_2$ and used Lemma \ref{lemma:anisotropic-isotropic}. This proves the absolute convergence of the integral.

Moreover, we see from the previous bound that
\begin{equation}\label{eq:quantitative}
\int_{P} (f \circ \psi_t) g \,\dd\mu = \mc{O}\left(\langle t\rangle^{-N_1} \|f\|_{H^{s + \vartheta(N_1 + 1) + 1 + \varepsilon}(P)} \|g\|_{H^{s + \vartheta(N_1 + 1) + 1 + \varepsilon}(P)}\right),\quad t \to \infty,
\end{equation}
which proves the claim (in \eqref{eq:quantitative-rapid-mixing} we apply this together with $s \leq 1$ and $\varepsilon \leq 2$; note that any $s > 0$ would work).
\end{proof}

\begin{remark}
	A different choice of $m \in \mathbb{R}$ in \eqref{eq:absolute-convergence} gives a different estimate in \eqref{eq:quantitative}. For instance, setting $m = 0$ we get
\begin{equation*}
	\int_{P} (f \circ \psi_t) g \,\dd\mu = \mc{O}\left(\langle t\rangle^{-N_1} \|f\|_{H^{s + 2\theta (N_1 + 1) + 1 + \varepsilon}(P)} \|g\|_{H^{s}(P)}\right), \quad t \to \infty.
\end{equation*}
\end{remark}

%
%

\subsection{Application to rapid mixing of frame flows}

\label{ssection:application-frame-flow}

\begin{proof}[Proof of Corollary \ref{corollary}] 
It is clear that rapid mixing implies ergodicity; we will now show the converse. The frame bundle is a $G$-principal bundle over the unit tangent bundle $SN$ over $N$, where $G = \mathrm{SO}(n-1)$. If $n \geq 4$, this is a semisimple Lie group and by Theorem \ref{theorem:main2}, Item (i), the frame flow is rapid mixing if $d\alpha \neq 0$; the latter follows from the contact property of the geodesic flow. For $n = 3$, we have $G = \mathrm{SO}(2) = \mathrm{U}(1)$. Hence by Corollary \ref{corollary1}, it suffices to prove that the frame bundle $FN \to SN$ is not a torsion circle bundle over $SN$. In particular, it suffices to show that for any $x_0 \in N$, $FN_{x_0} \to SN_{x_0} \cong \mathbb{S}^2$ is not torsion. This is isomorphic to the standard fibration $\mathrm{SO}(3) \to \mathbb{S}^2 = \mathrm{SO}(3)/\mathrm{SO}(2)$ corresponding to the unit circle bundle of $\mathbb{S}^2$, and this is clearly not torsion (for instance, because $H^2(\mathbb{S}^2, \mathbb{Z}) \cong \mathbb{Z}$, so there are no non-trivial torsion circle bundles, and it is non-trivial as there are no non-vanishing vector fields on $\mathbb{S}^2$).
%
%
%
%
\end{proof}

\begin{proof}[Proof of Corollary \ref{corollary:kahler}]
As above, it is clear that rapid mixing implies ergodicity; we will now show the converse. If $\dim_{\mathbb{C}} N = m$, the unitary frame bundle is a $G$-principal bundle over $SN$, where $G = \mathrm{U}(m-1)$. For $m \geq 3$, $G$ is semisimple; hence by Theorem \ref{theorem:main2}, item (i), the unitary frame flow is rapid mixing if $d\alpha \neq 0$ (which as before holds for geodesic flows). For $m = 2$, by Corollary \ref{corollary1}, it suffices to show that $F_{\C}N \to SN$ is not torsion. Note that for any $x_0 \in N$, the bundle $\operatorname{U}(2) \to \operatorname{U}(2)/\operatorname{U}(1) \cong \mathbb{S}^3$ is trivial, so an argument as above cannot work. This is where the negatively curved assumption appears: the circle bundle $F_{\C}N \to SN$ is not torsion by the argument on \cite[p. 13]{Cekic-Lefeuvre-Moroianu-Semmelmann-23}. (Indeed, in the notation of \cite{Cekic-Lefeuvre-Moroianu-Semmelmann-23}, this is the underlying circle bundle of the complex line bundle $\mc{N} \to SN$, which is isomorphic to $\pi^*\Lambda^{2, 0} T^*M$. Since $\pi^*$ is injective on the second cohomology, it suffices to show that the canonical line bundle $\Lambda^{2, 0} T^*M$ is not torsion. This follows from the negative curvature assumption.)
\end{proof}

%
%
%
%

\section{High-frequency estimates I. $G = \mathrm{U}(1)$}

\label{section:hf}

It remains to show Theorem \ref{theorem:real}. For the sake of clarity, we start with the case $G = \mathrm{U}(1)$; then by \S \ref{ssection:u1} we have $\widehat{G} \simeq \Z$. In this case, we will write $L$ for the Hermitian line bundle whose underlying circle bundle is $P$; write $k$ instead of $\mathbf{k}$ and $\nabla^{\mathrm{dyn}}$ for the associated dynamical connection on $L$. Also, note that we have that the flag bundle $F$ is equal to $M$. Recall that the anisotropic spaces $\mc{H}^s_h(M, L^{\otimes k})$ were introduced in \S \ref{sssection:anisotropic-space}.

\subsection{Technical reductions} Define
\[
\ell := \max(n + 10, 2n + 2) + 1.
\]
This choice ensures that the following inequalities are satisfied:
\begin{equation}
\label{equation:inegalites}
\ell/2- n -1 > 0, \qquad \ell/2- n/2 - 1 > 4.
\end{equation}
Recall that the notation $\mathbb{B} = [0, 1] \times \mathbb{R} \subset \mathbb{C}$ was introduced in \eqref{eq:band}. We will show the following more technical result:

\begin{proposition}
\label{proposition:technical-reduction}
Assume that $(\pi^*d\alpha, -iF_{\overline{\nabla}_1})$ are linearly independent over $\mathbb{R}$. There exists $C > 0$ such that for all $z \in \B, k \in \Z$, setting $\lambda := \Im(z)$:
\begin{align}
\label{equation:bound1}
& \|(-\X_k - z)^{-1}\|_{\mc{H}^s_{\langle{k}\rangle^{-1}}(M,L^{\otimes k}) \to \mc{H}^s_{\langle{k}\rangle^{-1}}(M,L^{ \otimes k})}\leq C \langle k \rangle^\ell, \qquad |k|\geq |\lambda|, \\
\label{equation:bound2}
& \|(-\X_k - z)^{-1}\|_{\mc{H}^s_{\langle{\lambda}\rangle^{-1}}(M,L^{\otimes k}) \to \mc{H}^s_{\langle{\lambda}\rangle^{-1}}(M,L^{\otimes k})}\leq C \langle \lambda \rangle^\ell, \qquad |\lambda| \geq |k|.
\end{align}
\end{proposition}

Note that for $k$ and $\lambda$ lying both in a compact window, estimates \eqref{equation:bound1} and \eqref{equation:bound2} are trivial. These estimates are only meaningful when at least one of the parameters $k$ or $\lambda$ tends to $\pm \infty$. Let us show that Proposition \ref{proposition:technical-reduction} implies Theorem \ref{theorem:real}.

\begin{proof}[Proposition \ref{proposition:technical-reduction} implies Theorem \ref{theorem:real}]
It is obvious that \eqref{equation:bound1} implies \eqref{equation:uniform-bound} for the range $|k| \geq |\lambda|$. In the complementary range, we have that \eqref{equation:bound2} implies \eqref{equation:uniform-bound} based on equivalence of norms as follows. The spaces $\mc{H}^s(M, L^{\otimes k})$ for varying $h$ are all the same; only their norms differ, but we can say the following. There exists a constant $C > 0$ such that for all $h,h' \leq 1/k$, for all $f \in C^\infty(M,L^{\otimes k})$,
\[
	\|f\|_{\mc{H}^s_h(M,L^{\otimes k})} \leq C\max\left(\left(\frac{h}{h'}\right)^s, \left(\frac{h'}{h}\right)^s\right) \|f\|_{\mc{H}^s_{h'}(M,L^{\otimes k})}.
\]
This can be proved directly by using local coordinates and the formula for quantization, see \cite[Lemma 4.3]{Guillarmou-Kuster-21} where the same bound is established. (More precisely, this follows from the observation that $\overline{\Op}_{h', k}(a) = \overline{\Op}_{h, k}(\widetilde{a})$ for any symbol $a$, where $\widetilde{a}(x, \xi) := a\left(x, \tfrac{h}{h'}\xi\right)$ (by changing variables in the quantisation formula). Then, the bound is established by using the composition formula in the $\Psi^{\bullet}_{h, \mathrm{BW}}(P)$ calculus and Proposition \ref{proposition:proprietes}, item (v).) This completes the proof, and as we may take $s > 0$ arbitrarily small, shows furthermore that the estimate is satisfied with $\vartheta = \ell + 1$ (in fact $\vartheta = \ell + \varepsilon$ for any $\varepsilon > 0$ would also work).
\end{proof}

The proof of both inequalities in Proposition \ref{proposition:technical-reduction} is based on an argument by contradiction: if this estimate does not hold, one can construct quasi-modes associated to a certain semiclassical operator (a renormalization of $\X_k$). In turn, these quasi-modes give rise to a non-trivial semiclassical defect measure with specific properties (compact support in phase space, invariance by the flow $\varphi$, etc.). Using the assumptions $\dd \alpha \neq 0, \omega \neq c \cdot \dd \alpha$, we show that such a measure cannot exist.

\subsection{Proof of estimate \eqref{equation:bound1}}

\label{ssection:dur}

Assume that \eqref{equation:bound1} does not hold. Then, there exists a sequence $z_n \in \B$, $k_n \in \Z $ with $|k_n| \geq |\Im (z_n)|$ and $f_n \in \mc{H}^s_{\langle{k_n}\rangle^{-1}}(M,L^{\otimes k_n})$ such that
\begin{equation}
\label{equation:quasimode1}
	\|f_n\|_{\mc{H}^s_{\langle{k_n}\rangle^{-1}}(M,L^{\otimes k_n})}=1, \qquad \|(-\X_{k_n}-z_n)f_n\|_{\mc{H}^s_{\langle{k_n}\rangle^{-1}}(M,L^{\otimes k_n})} = o(\langle k_n\rangle^{-\ell}).
\end{equation}
We can always assume that $(k_n)_{n \geq 0}$ is unbounded, otherwise $(k_n)_{n \geq 0}$ and $(\Im(z_n))_{n \geq 0}$ are both bounded and the existence of a quasimode in \eqref{equation:quasimode1} is impossible since \eqref{equation:bound1} is trivially satisfied. Hence, after passing to a subsequence, we may assume $k_n \to_{n \to \infty} \pm \infty$ and $|k_n| \geq 1$ for all $n$.

In order to simplify notation, we will further assume $k_n \geq 0$, $\Im(z_n) \geq 0$. The symmetric cases are treated similarly. Up to passing to a subsequence, we can also assume that $\Re(z_n) \to \nu \in [0,1]$ and $\Im(z_n)/k_n \to \eta \in [0,1]$.

To be consistent with the notation of \S\ref{chapter:analytic}, we drop the dependence on $n$ and write $h := \langle{k_n}\rangle^{-1}$ (one can bear in mind that $h$ only belongs to a subset of $\{(1 + k^2)^{-1/2} ~|~ k \in \mathbb{Z}_{\geq 0}\}$ but this is actually quite irrelevant in what follows) and take the function $k(h) := \sqrt{h^{-2} - 1}$. Hence, we can rewrite \eqref{equation:quasimode1} as:
\begin{equation}
\label{equation:quasimode2}
\|f_h\|_{\mc{H}^s_h(M,L^{\otimes k(h)})} = 1, \qquad \|\mathbf{P}_h f_h\|_{\mc{H}^s_{h}} = o(h^{\ell+1}),
\end{equation}
where 
\begin{align*}
	\mathbf{P}_h &:= -h\X_{k(h)} - h(\nu+o(1)) - i(\eta +o(1))(1 + \mc{O}(h^3))\\
	&= -h\X_{k(h)} - h(\nu+o(1)) - i(\eta +o(1)),
\end{align*}
and the ``$o$" and ``$\mc{O}$" in the expression of $\mathbf{P}_h$ are \emph{real} constants (depending on $h$). (Note that the $\mc{O}(h^3)$ comes from the expression $-1 + \tfrac{k_n}{\langle{k_n}\rangle} = \mc{O}(h^3)$.)

Throughout this paragraph, elliptic and wavefront sets (as introduced in \S\ref{chapter:analytic}) are all computed with respect to $\nabla^{\mathrm{dyn}}$.

\subsubsection{Microsupport of quasimodes}

Define $\mc{T}_\eta := \{(x,- \eta \alpha(x)) \in T^*M ~|~ x \in M\}$. The following holds:

\begin{lemma}
\label{lemma:microsupport}
Let $\mathbf{E}_h \in \Psi_{h, k}^{\comp}(M,L)$. Then 
\[
\WF_h(\mathbf{E}_h) \cap \mc{T}_\eta = \emptyset \implies \mathbf{E}_h f_h = o_{L^2}(h^{\ell/2}) = o_{C^0}(h^{\ell/2 - n/2}).
\]
\end{lemma}

\begin{proof}
\emph{Step 1.} First, the operator $\mathbf{P}_h$ has principal symbol $\sigma_{\mathbf{P}_h}(x,\xi) = -i(\langle \xi,X(x)\rangle + \eta + o(1))$. Hence, if $\WF_h(\mathbf{E}_h) \cap \{\langle\xi,X(x)\rangle=-\eta\} = \emptyset$, one obtains by ellipticity of $\PP_h$ on $\WF_h(\mathbf{E}_h)$ (see Proposition \ref{proposition:ellipticity}) that $\|\mathbf{E}_h f_h\|_{L^2} = o(h^{\ell+1})$. (We emphasise that the principal symbol of $\mathbf{P}_h$ depends on $h$ through the $o(1)$ term. In this chapter we will always work with a subsequence $h \mapsto k(h)$, and so here and in what follows the elliptic and wavefront sets are computed with respect to a subsequence. By definition these sets depend only on properties as $h \to 0$ and so using that any point in the complement of $\{\langle\xi,X(x)\rangle=-\eta\}$ is in the elliptic set of $\mathbf{P}_h$ for $h$ small enough, we obtain the claimed result.)
\medskip

\emph{Step 2.} Second, assume that $\WF_h(\mathbf{E}_h) \cap (\mc{T}_\eta + E_u^*) = \emptyset$. We claim that $\mathbf{E}_h f_h = o_{L^2}(h^\ell)$. Indeed, we apply the source estimate  \eqref{equation:source-estimate}. The threshold is $\omega_+ = 0$ so we may apply the estimate with exponent $s$ which appears in the anisotropic space $\mc{H}^s_h$. Taking $\mathbf{A}_h \in \Psi^0_{h, k}(M,L)$ elliptic and localized near $E_s^* \cap \partial_\infty T^*M$, we get:
\[
\|\mathbf{A}_h f_h\|_{H^s_h} \leq C\left( h^{-1}\|\mathbf{B}_h \PP_h f_h\|_{H^s_h} + h^N\|f_h\|_{H^{-N}_h} \right) = o(h^\ell),
\]
with $\mathbf{B}_h$ as in \eqref{equation:source-estimate}.

Now, for all $(x,\xi) \in \{\langle \xi,X(x)\rangle=-\eta\}$ such that $(x,\xi) \notin \mc{T}_\eta + E_u^*$, there is a finite time $T \in \R$ such that $\Phi_T(x,\xi) \in \Ell_h(\mathbf{A}_h)$. (We recall that $\Phi_t(x, \xi) = (\varphi_tx, \xi \circ (d\varphi_t)^{-1}(x))$ denotes the symplectic lift of the flow $\varphi$.) In particular, this holds for any $(x,\xi) \in \WF_h(\mathbf{E}_h)$. As a consequence, by standard propagation of singularities (Proposition \ref{proposition:propagation}), we deduce that $\mathbf{E}_h f_h = o_{L^2}(h^\ell)$.
\medskip

\emph{Step 3.} Finally, it remains to deal with $\mc{T}_\eta + E_u^*$. Consider formally self-adjoint $\mathbf{A}_h \in \Psi_{h, k}^{\comp}(M,L)$ such that $\WF_h(\mathbf{A}_h)$ is localized near $\mc{T}_\eta + E_u^*$, $\sigma_{\mathbf{A}_h} \geq 0$, and $\sigma_{\mathbf{A}_h} \equiv 1$ on a neighborhood of $\mc{T}_\eta$. Up to modifying $\mathbf{A}_h$ by lower order terms, we can also further assume that $\mathbf{A}_h \geq 0$ in the sense that for all $f_h \in C^\infty(M,L^{\otimes k(h)})$, $\langle \mathbf{A}_h f_h,f_h\rangle_{L^2} \geq 0$.

Since $\mathbf{P}_h f_h = o_{\mc{H}^{s}_h}(h^{\ell+1})$ and the propagator $t \mapsto e^{t \mathbf{P}_h/h} \in \mc{L}(\mc{H}^{s}_h)$ is uniformly bounded for $t \in [0,T]$ in a fixed compact interval, by Egorov's Theorem (Proposition \ref{proposition:egorov}), one gets
\[
e^{t\mathbf{P}_h/h} f_h = f_h + o_{\mc{H}^{s}_h}(h^{\ell}),\quad t\in [0, T].
\]
Hence, for $t \in [0, T]$ we obtain
\[
\begin{split}
\langle \mathbf{A}_hf_h, f_h\rangle_{L^2} & = \langle \mathbf{A}_h e^{t\mathbf{P}_h/h} f_h, e^{t\mathbf{P}_h/h} f_h\rangle_{L^2} + o(h^\ell) \\
& = e^{-2(\nu+o(1))t} \langle e^{t\X_k} \mathbf{A}_h e^{-t\X_k}f_h,f_h\rangle_{L^2} +  o(h^\ell).
\end{split}
\]
Using the dynamics of $\Phi$, for $t$ large enough one can write
\[
\mathbf{A}_h - e^{-2(\nu+o(1))t} e^{t\X_k} \mathbf{A}_h e^{-t\X_k} = \mathbf{C}_h^* \mathbf{C}_h + \mathbf{D}_h,
\]
where $\WF_h(\mathbf{D}_h) \cap (\mc{T}_\eta + E_u^*) = \emptyset$, $\WF_h(\mathbf{C}_h)$ is close to $\mc{T}_\eta + E_u^*$, and $\mathbf{C}_h$ is elliptic on an open subset $\Omega \subset \mc{T}_\eta + E_u^*$ which contains a non-empty region homeomorphic to an annulus, such that:
\begin{enumerate}[label=(\roman*)]
\item If $\nu > 0$, then $\mc{T}_\eta \subset \Omega$;
\item If $\nu = 0$, then $\Omega \cap \mc{T}_\eta = \emptyset$.
\end{enumerate}
Applying Step 2 with $\mathbf{D}_h$ and using Step 1, one has $\langle\mathbf{D}_hf_h,f_h\rangle_{L^2}=o(h^\ell)$. Hence $\|\mathbf{C}_hf_h\|^2_{L^2} = o(h^\ell)$ and so
\[
	\|\mathbf{C}_h f_h\|_{L^2} = o(h^{\ell/2}).
\] 
Then, using elliptic estimates (Proposition \ref{proposition:ellipticity}) and standard propagation of singularities (Proposition \ref{proposition:propagation}), it is straightforward to control $\|\mathbf{E}_h f_h\|_{L^2}$ by $\|\mathbf{C}_h f_h\|_{L^2}$ modulo $o(h^\ell)$ remainders. This proves the claim.

Finally, it remains to prove the $C^0$-estimate. Observe that, by Sobolev embeddings (see Lemma \ref{lemma:sobolev-embedding}), for all $N > n/2$, there exists a constant $C > 0$ such that for all $u_h \in C^\infty(M,L^{\otimes k(h)})$, $\|u_h\|_{C^0} \leq C h^{-n/2} \|u_h\|_{H^N_h}$. Now, since $\mathbf{E}_h$ has compact support, the $L^2$-norm in the previous estimates is irrelevant and one has $\|\mathbf{E}_hf_h\|_{H^N_h} = o(h^{\ell/2})$ similarly. Therefore, this yields
\[
\|\mathbf{E}_hf_h\|_{C^0} \leq Ch^{-n/2}\|\mathbf{E}_hf_h\|_{H^N_h} = o(h^{\ell/2 - n/2}).
\]
This concludes the proof.
\end{proof}

Next, we have the following:

\begin{lemma}
\label{lemma:2}
Let $\mathbf{E}_h \in \Psi^{\comp}_{h, k}(M, L)$ such that $\mathbf{E}_h \equiv \mathbbm{1}$ microlocally near $\mc{T}_\eta$. Then, there exists $c > 1$ such that:
\begin{equation}
\label{equation:bounds-p}
\PP_h \mathbf{E}_h f_h = o_{L^2}(h^{\ell/2+1}) = o_{C^0}(h^{\ell/2 + 1 - n/2}), \quad 1/c \leq \|\mathbf{E}_h f_h\|_{L^2} \leq c.
\end{equation}
Furthermore, $\nu = 0$ and $\mathbf{P}_h = -h\X_k - z_h$, where $\Re(z_h) \geq 0$ and
\begin{equation}\label{eq:z_h}
	z_h = i(\eta + o(1)) + o(h^{\ell/2+1}).
\end{equation}	
\end{lemma}

The previous lemma asserts that quasimodes have microsupport on $\mc{T}_\eta$.

\begin{proof}
 The first estimate is a consequence of Lemma \ref{lemma:microsupport} since
\[
	\PP_h \mathbf{E}_h f_h = \mathbf{E}_h \PP_h f_h + [\PP_h,\mathbf{E}_h] f_h = o_{L^2}(h^{\ell+1}) + o_{L^2}(h^{\ell/2+1}) = o_{L^2}(h^{\ell/2+1}),
\]
where in the second equality we used \eqref{equation:quasimode2}, as well as that $[\PP_h,\mathbf{E}_h] \in h \Psi_{h, k}^{\comp}(M,L)$ and $\WF_h(h^{-1}[\PP_h,\mathbf{E}_h]) \cap \mc{T}_{\eta} = \emptyset$ by construction. The $C^0$-estimate comes from Sobolev embedding as in the proof of Lemma \ref{lemma:microsupport}.

It remains to establish the bounds $1/c \leq \|\mathbf{E}_h f_h\|_{L^2} \leq c$. For that we use the sink estimate \eqref{equation:sink-estimate} (where the threshold $\omega_- = 0$). Taking $\mathbf{A}_h$ microlocally supported and elliptic near $E_u^*$, we obtain:
\[
\|\mathbf{A}_h f_h \|_{H^{-s}_h} \leq C\left(h^{-1}\|\mathbf{B}_h \PP_{h} f_h\|_{H^{-s}_h} + \|\mathbf{A}'_hf_h\|_{H^{-s}_h} + h^N\|f_h\|_{H^{-N}_h} \right),
\]
where we recall that $s > 0$ is the exponent in the anisotropic space, and where $\mathbf{B}_h,\mathbf{A}'_h$ are as in \eqref{equation:sink-estimate}. Since $\mathbf{B}_h \in \Psi^0_{k, h}(M, L)$ is microlocalised near $E_u^*$, where $\mc{H}^s_h$ is microlocally equivalent to $H^{-s}_h$, we get
\[
h^{-1}\|\mathbf{B}_h \PP_{h} f_h\|_{H^{-s}_h} \leq C h^{-1}\|\PP_h f_h\|_{\mc{H}^s_h} =  o(h^{\ell}).
\]
The operator $\mathbf{A}'_h$ has wavefront set disjoint from $E_u^* \cap \partial_\infty T^*M$. Hence, using standard propagation of singularities (Proposition \ref{proposition:propagation}) and the estimates in the proof of Lemma \ref{lemma:microsupport}, we get that $\|\mathbf{A}'_hf_h\|_{H^{-s}_h} = o(h^{\ell/2})$. As a consequence,
\begin{equation}
\label{equation:sink2}
\|\mathbf{A}_h f_h \|_{H^{-s}_h} = o(h^{\ell/2}).
\end{equation}

Now, we write
\[
\|f_h\|_{\mc{H}^s_h} = 1 \leq \|\mathbf{E}_hf_h\|_{\mc{H}^s_h} + \|(\mathbbm{1}-\mathbf{E}_h)f_h\|_{\mc{H}^s_h}.
\]
Observe that $\mathbbm{1}-\mathbf{E}_h$ has wavefront set outside of a small neighborhood of $\mc{T}_\gamma$. Hence, using \eqref{equation:sink2} and the estimates from the proof of Lemma \ref{lemma:microsupport}, we get that $\|(\mathbbm{1}-\mathbf{E}_h)f_h\|_{\mc{H}^s_h} = o(h^{\ell/2})$. Taking $h > 0$ small enough, we therefore obtain $\|\mathbf{E}_hf_h\|_{L^2} > 1/c_1$ for some $c_1 > 0$. Moreover, since $\mathbf{E}_h$ has compact support, there exists a constant $c_2 > 0$ such that $\|\mathbf{E}_hf_h\|_{L^2} \leq c_2\|f_h\|_{\mc{H}^s_h} \leq c_2$. Hence, the claim is verified with $c := \max(c_1, c_2)$.

We now show that $\nu = 0$. If $\nu > 0$, by point (ii) in the proof of Lemma \ref{lemma:microsupport} and elliptic estimates, we get that $\|\mathbf{E}_h f_h\|_{L^2} = o(h^{\ell/2})$ but this contradicts $\|\mathbf{E}_hf_h\|_{L^2} > 1/c$ if $h$ is small enough. Hence $\nu = 0$. 

It remains to show that $\PP_h$ admits the desired expression. Write $w_h \geq 0$ for the $o(h)$ in the definition of $\PP_h$, that is 
\[
	\mathbf{P}_h = -h\X_k - w_h - i(\eta+ o(1)). 
\]
We want to show that $w_h = o(h^{\ell/2+1})$ (we already have $w_h = o(h)$). By unitarity of the propagator of $\X_k$ on $L^2(M,L^{\otimes k(h)})$, for $\Re(z) > 0$, we can bound the resolvent
\[
(-\X_h-z)^{-1} = - \int_0^{+\infty} e^{-t\X_k}e^{-tz}\, \dd t,
\]
by
\begin{equation}
\label{equation:bound-res-l2}
\|(-\X_k - z)^{-1}\|_{L^2 \to L^2} \leq \int_0^{+\infty} \|e^{-t\X_k}\|_{L^2 \to L^2} e^{-t\Re(z)}\, \dd t = (\Re z)^{-1}.
\end{equation}
Hence, if $w_h = o(h^{\ell/2+1})$ does not hold, it means that $w_{h_n} \geq C {h_n}^{\ell/2+1}$ for some positive constant $C >0$ along a subsequence $(h_n)_{n = 1}^\infty$. Hence, along that subsequence (we drop the index $n$ for the sake of simplicity), we get by applying \eqref{equation:bound-res-l2} to \eqref{equation:bounds-p} that
\[
\|\mathbf{E}_h f_h\|_{L^2} \leq w_h^{-1} \|\PP_h \mathbf{E}_h f_h\|_{L^2} = o(1),
\]
and this contradicts the bound $\|\mathbf{E}_h f_h\|_{L^2} > 1/c$ established earlier.
\end{proof}

\subsubsection{Horocyclic invariance} Recall that the sequence $z_h$ was defined in \eqref{eq:z_h}. We now prove that the following holds:

\begin{lemma}
\label{lemma:parallel}
Let $\mathbf{E}_h \in \Psi^{\comp}_{h, k}(M,L)$ such that $\mathbf{E}_h \equiv \mathbbm{1}$ microlocally near $\mc{T}_\eta$. Then 
\begin{equation}
\label{equation:parallel}
(-h\nabla^{\mathrm{dyn}}_{k(h)} - z_h \alpha \wedge)\mathbf{E}_h f_h = o_{L^2}(h^{\ell/2}) = o_{C^0}(h^{\ell/2-n/2}).
\end{equation}
\end{lemma}

Note that $\mathbf{E}_h f_h \in C^\infty(M,L^{\otimes k(h)})$ since $\mathbf{E}_h$ has compact support and thus
\[
\nabla^{\mathrm{dyn}}_{k(h)} \mathbf{E}_h f_h \in C^\eps(M,T^*M \otimes L^{\otimes k(h)})
\]
is well-defined despite the non-smoothness of $\nabla^{\mathrm{dyn}}$.

\begin{proof}
Let
\[
	\nabla^{u/s}_{k(h)} : C^\infty(M,L^{\otimes k(h)}) \to C^\eps(M,E_{s/u}^* \otimes L^{\otimes k(h)}),
\]
be the restriction of the dynamical connection to $E_{s/u}^*$ (that is, we can plug in vectors in $E_{u/s}$ inside $\nabla^{u/s}_{k(h)} f$ for a section $f$, hence the use of the letter $u/s$). We claim that
\begin{equation}
\label{equation:todo}
	-h\nabla^{u}_{k(h)}\mathbf{E}_h f_h = o_{L^2}(h^{\ell/2}) = o_{C^0}(h^{\ell/2-n/2}).
\end{equation}
By a symmetric argument, a similar bound will hold for $\nabla^{s}_{k(h)}$ and, combining with \eqref{equation:bounds-p}, this will prove \eqref{equation:parallel}. Hence, it remains to show \eqref{equation:todo}.

The operator $\X_{k}$ acting on sections of $L^{\otimes k}$ extends naturally to an operator $\X_k^u$ acting on $E_s^* \otimes L^{\otimes k}$ by twisting with the Lie derivative on the $E_s^*$ factor. Hence $\mathbf{P}_h$ extends naturally to $\mathbf{P}_h^u$ in the same fashion. Moreover, it is straightforward to check that the following commutation relation holds:
\[
\X_k^u \nabla^u_k = \nabla^u_k \X_k, \qquad \mathbf{P}_h^u \nabla^u_{k(h)} = \nabla^u_{k(h)} \PP_h.
\]
Indeed, the second equality is a consequence of the first one, and the first is proved as follows. Choose a local frame $(\e_i)_{i = 1}^r$ for $E_u$, write $(\e_i^*)_{i = 1}^r$ for the dual frame of $E_s^*$, and for a local section $s$ compute
\begin{align*}
	\X_k^u \nabla^u_k s &= \X_k^u \sum_{i = 1}^r \e_i^* \otimes \nabla^{\mathrm{dyn}}_{\e_i} s = \sum_{i = 1}^r \Lie_X \e_i^* \otimes \nabla^{\mathrm{dyn}}_{\e_i} s + \sum_{i = 1}^r \e_i^* \otimes \nabla^{\mathrm{dyn}}_X \nabla^{\mathrm{dyn}}_{\e_i} s\\
	&= \sum_{i = 1}^r \Lie_X \e_i^* \otimes \nabla^{\mathrm{dyn}}_{\e_i} s + \sum_{i = 1}^r \e_i^* \otimes [\nabla^{\mathrm{dyn}}_{X}, \nabla^{\mathrm{dyn}}_{\e_i}] s + \nabla^{u}_k \X_k s.
\end{align*}
It suffices to show that the first two terms on the right hand side sum to zero. Plugging in $\e_j$ for some fixed $j$ and using the definition of curvature, we get
\[
	-\sum_{i = 1}^r \e_i^*([X, \e_j]) \nabla^{\mathrm{dyn}}_{\e_i} s + \nabla^{\mathrm{dyn}}_{[X, \e_j]} s + F_{\nabla^\mathrm{dyn}}(X, \e_j) s = 0,
\]
where $F_{\nabla^\mathrm{dyn}}$ denotes the curvature of $\nabla^{\mathrm{dyn}}$, and we used the fact that $\iota_X F_{\nabla^\mathrm{dyn}} = 0$ (see Lemma \ref{lemma:computations}, Item 1). Since $j$ was arbitrary this proves the claim.

As a consequence,
\begin{equation}\label{eq:as-a-consequence}
	-\PP_h^u h\nabla^u_{k(h)} \mathbf{E}_h f_h = - h\nabla^u_{k(h)} \PP_h \mathbf{E}_h f_h = - h\nabla^u_{k(h)} \mathbf{E}_h \PP_h f_h + h\nabla^u_{k(h)} [\mathbf{E}_h, \PP_h] f_h.
\end{equation}
The first term on the right hand side is bounded by
\[
	\|h\nabla^u_{k(h)} \mathbf{E}_h \PP_h f_h\|_{L^2} \leq C \|\mathbf{E}_h \PP_h f_h\|_{H^1_h} \leq C \|\PP_h f_h\|_{\mc{H}^s_h}= o(h^{\ell+1}),
\]
where we used \eqref{equation:quasimode2} in the last equality, while the second term is bounded by
\[
\|h\nabla^u_{k(h)} [\mathbf{E}_h, \PP_h] f_h\|_{L^2} \leq C\|[\mathbf{E}_h, \PP_h] f_h\|_{H^1_h} =  o(h^{\ell/2+1}),
\]
by Lemma \ref{lemma:microsupport} since $\WF_h(h^{-1}[\mathbf{E}_h, \PP_h]) \cap \mc{T}_\eta = \emptyset$. Hence 
\begin{equation}
\label{equation:petit}
-\PP_h^u h\nabla^u_{k(h)} \mathbf{E}_h f_h = o_{L^2}(h^{\ell/2+1})
\end{equation}

Similarly, in the $C^0$-topology, by Sobolev embeddings (using Lemma \ref{lemma:sobolev-embedding}), one has  for all $N > n/2+1$:
\[
\begin{split}
\|h\nabla^u_{k(h)} [\mathbf{E}_h, \PP_h] f_h\|_{C^0}  \leq C\|[\mathbf{E}_h, \PP_h] f_h\|_{C^1_h}\leq C h^{-n/2} \|[\mathbf{E}_h, \PP_h] f_h\|_{H^N_h} = o(h^{\ell/2+1-n/2}),
\end{split}
\]
where $C^1_h$ denotes the space $C^1$ with semiclassical norm $\|f\|_{C^1_h} := \|f\|_{C^0} + h \|df\|_{C^0}$, and again we used Lemma \ref{lemma:microsupport} and $\WF_h(h^{-1}[\mathbf{E}_h, \PP_h]) \cap \mc{T}_\eta = \emptyset$. Next, as above we have that
\[
	\|h\nabla^u_{k(h)} \mathbf{E}_h \PP_h f_h\|_{C^0} \leq C \|\mathbf{E}_h \PP_h f_h\|_{C^1_h} \leq C \|\PP_h f_h\|_{\mc{H}^s_h}= o(h^{\ell+1}).
\]
Combining the previous two inequalities and using \eqref{eq:as-a-consequence}, we thus get 
\begin{equation}
\label{equation:petit2}
-\PP_h^u h\nabla^u_{k(h)} \mathbf{E}_h f_h = o_{C^0}(h^{\ell/2+1-n/2}).
\end{equation}

Now, using the contraction property of the propagator $e^{-t\mc{L}_X}$ on $E_s^*$, it is straightforward to check that the resolvent
\[
	(-\X^u_{k(h)}-z)^{-1} = - \int_0^{+\infty} e^{-t\X^u_{k(h)}} e^{-tz}\, \dd t
\]
converges both in the $C^0$- and the $L^2$-topology on a half-space $\{z \in \C\mid \Re(z) > -\delta\}$ for some $\delta > 0$. Namely, given $u_h \in C^0(M,E_s^* \otimes L^{\otimes k(h)})$,
\[
\begin{split}
\|(-\X^u_{k(h)}-z)^{-1} u\|_{C^0,L^2} & \leq \int_0^{+\infty} \|e^{-t \X^u_{k(h)}}\|_{C^0 \to C^0,L^2 \to L^2} e^{-t\Re(z)}  \dd t~ \|u\|_{C^0,L^2} \\
& \leq \int_0^{+\infty} \|e^{-t \mc{L}_X|_{E_s^*}}\|_{C^0 \to C^0,L^2 \to L^2} e^{-t\Re(z)} \dd t \|u\|_{C^0,L^2} \\ 
& \leq \int_0^{+\infty} e^{-\delta t} e^{-t\Re(z)} ~ \dd t ~\|u\|_{C^0,L^2} \leq (\delta + \Re(z))^{-1}\|u\|_{C^0,L^2}.
\end{split}
\]
where $\delta$ is the contraction rate of the flow on $E_s^*$ (see \eqref{equation:anosov}). In particular, the resolvent is defined and uniformly bounded for $\{z \in \C \mid \Re(z) \geq 0\}$ on both $C^0$ and $L^2$. Hence, $\mathbf{P}^u_h$ is invertible on both $C^0$ and $L^2$ with norm bounded by $\leq Ch^{-1}$. Applying $({\mathbf{P}^u_h})^{-1}$ to \eqref{equation:petit} and \eqref{equation:petit2}, we thus obtain:
\[
-h\nabla^u_{k(h)} \mathbf{E}_h f_h = o_{L^2}(h^{\ell/2}) = o_{C^0}(h^{\ell/2- n/2}).
\]
This completes the proof.
\end{proof}

\subsubsection{Vanishing of the curvature}

Recall that $\mathbf{E}_h \in \Psi_{h, k}^{\comp}(M,L)$ satisfies $\mathbf{E}_h \equiv \mathbbm{1}$ microlocally near $\mc{T}_\eta$ and that we set $v_h :=  \mathbf{E}_h f_h$.

\begin{lemma}
We have that
\begin{equation}
\label{equation:d-small}
\dd |v_h|^2 = o_{C^0}(h^{\ell/2 - n- 1}).
\end{equation}
\end{lemma}

\begin{proof}
The estimate is equivalent to showing that $Y |v_h|^2 = o_{C^0}(h^{\ell/2-n-1})$ for all smooth vector fields $Y \in C^\infty(M,TM)$ such that $\|Y\|_{C^0} \leq 1$.

Since $\mathbf{E}_h$ has compact support, we get, using the $L^2$-bound $\|v_h\|_{L^2} < c$ from \eqref{equation:bounds-p} and Sobolev embeddings (see Lemma \ref{lemma:sobolev-embedding}), that $\|v_h\|_{C^0} \leq C h^{-n/2}$. Hence, by unitarity of the dynamical connection, we get:
\[
\begin{split}
Y |v_h|^2 & = \langle\nabla^{\mathrm{dyn}}_Yv_h,v_h\rangle + \langle v_h,\nabla^{\mathrm{dyn}}_Y v_h\rangle \\
& = \langle(\nabla^{\mathrm{dyn}}_Y + \tfrac{z_h}{h} \alpha(Y))v_h,v_h\rangle + \langle v_h,(\nabla^{\mathrm{dyn}}_Y + \tfrac{z_h}{h} \alpha(Y)) v_h\rangle - 2 \Re(\tfrac{z_h}{h}) \alpha(Y).
\end{split}
\]
By Lemma \ref{lemma:2}, $\Re(\tfrac{z_h}{h}) = o(h^{\ell/2})$. The conclusion then follows immediately by appying the $C^0$-bound of \eqref{equation:parallel} and $\|v_h\|_{C^0} \leq C h^{-n/2}$ to the previous equality.
\end{proof}

\begin{lemma}
\label{lemma:lower-bound}
There exists a constant $C > 0$ such that for all $x \in M, h > 0$, $|v_h(x)| > C$.
\end{lemma}

\begin{proof}
Since $M$ is compact, it has finite diameter (after equipping it with an arbitrary Riemannian metric), so we get by \eqref{equation:d-small} that uniformly for all $x,y \in M$,
\begin{equation}
\label{equation:bound-norm}
|v_h(x)|^2-|v_h(y)|^2 = o(h^{\ell/2- n-1}).
\end{equation}
Now, assume for the sake of a contradiction that (along a subsequence) there exists $x_h \in M$ such that $|v_h(x_h)|^2 \to 0$ as $h \to 0$. Then we get by \eqref{equation:bound-norm} that $\|v_h\|^2_{C^0} = o(1)$.

Hence
\[
\|\mathbf{E}_h f_h\|^2_{L^2} = \|v_h\|^2_{L^2}\leq C\|v_h\|^2_{C^0} = o(1),
\]
which contradicts the lower bound $\|\mathbf{E}_h f_h\|_{L^2} > 1/c$ from \eqref{equation:bounds-p}.  (In this lemma we use that $\ell/2 > n + 1$, that is, $\ell > 2n + 2$.)
\end{proof}

If $L$ is not a torsion bundle, we obtain immediately a contradiction proving estimate \eqref{equation:bound1}.

\begin{proof}[Proof of \eqref{equation:bound1} if $L$ is not torsion]
Since $L$ is not torsion and $v_h \in C^\infty(M,L^{\otimes k(h)})$ with $k(h)=\sqrt{h^{-2} - 1}$, there exists $x_h \in M$ such that $v_h(x_h) = 0$. This contradicts the lower bound of Lemma \ref{lemma:lower-bound}.
\end{proof}

When $L$ is torsion, the proof of \eqref{equation:bound1} is slightly more involved. (Note however that the fact that $L$ is torsion will not be used in what follows; the preceding proof was for motivational purposes only.) It relies on the following crucial observation:

\begin{lemma}
\label{lemma:contradiction}
The quasimode equation \eqref{equation:quasimode2} implies $-iF_{\nabla^{\mathrm{dyn}}} + \eta \dd \alpha = 0$.
\end{lemma}

The previous lemma easily allows to conclude:

\begin{proof}[Proof of \eqref{equation:bound1}]
Since $F_{\nabla^{\mathrm{dyn}}}$ is assumed not to be proportional to $\dd \alpha$ (assumption of Theorem \ref{theorem:main2}), we get a contradiction by applying Lemma \ref{lemma:contradiction}. This proves \eqref{equation:bound1}.
\end{proof}

 It thus remains to show Lemma \ref{lemma:contradiction}. Given a unitary Hölder-continuous $\nabla$ on the (smooth) complex line bundle $L \to M$, the curvature $F_{\nabla}$ is well-defined as a distributional $2$-form. Given any (smooth) path $\gamma$ contained in a local trivialisation where $\nabla = d+i\beta$, with $\beta$ Hölder-continuous, parallel transport along $\gamma$ is well-defined as $e^{i \int_\gamma \beta}$. Thus, for any closed loop $\gamma$ the holonomy $\mathrm{Hol}_\nabla(\gamma) \in \mathrm{U}(1)$ is well-defined. The following holds:

\begin{lemma}
\label{lemma:flat}
Let $\nabla$ be a $\varepsilon$-Hölder continuous connection on $L \to M$. Then $F_\nabla \equiv 0$ if and only if for all homotopically trivial closed loops $\gamma \subset M$, $\mathrm{Hol}_\nabla(\gamma) = 1$.
\end{lemma}
The proof in the case where $\nabla$ is $C^1$ is standard.
\begin{proof}
	The condition that the holonomy is trivial along homotopically trivial closed loops is equivalent to the condition that in any local chart where $\nabla = d + i\beta$, for any path $\gamma$ the integral $\int_\gamma \beta$ depends only on the homotopy class of $\gamma$.
	
	We first prove the easy direction: if $F_\nabla = 0$, then in a local trivialisation where $\nabla = d + i\beta$, we have $d\beta = 0$. Then by Hodge decomposition $\beta = df + h$, where $h$ is a harmonic $1$-form and $f \in C^{1 + \varepsilon}$ by elliptic regularity. Now clearly $\int_{\gamma} df$ depends only on the values of $f$ at the endpoints of $\gamma$, and $\int_{\gamma} h$ is independent of the homotopy class of $\gamma$ by Stokes' theorem.
	
	In the other direction, we work locally in a chart $\mc{U} \subset \mathbb{R}^n$ diffeomorphic to a ball, such that $0 \in \mc{U}$; write $\nabla = d + i\beta$. Define for $x \in \mc{U}$
	\[
		f(x) := \int_{\gamma_x} \beta,
	\]
	where $\gamma_x$ is an arbitrary path from $0$ to $x$. We claim that $f$ is $C^{1}$ regular. It suffices to check this at $x = 0$; write $\beta = \sum_{i = 1}^n \beta_i dx_i$. Indeed, we have for $x \neq 0$
	\begin{align*}
		&f(x) - f(0) - \sum_{i = 1}^n \beta_i(0) x_i = \sum_{i = 1}^n \int_0^{|x|} \beta_i \left(\tfrac{sx}{|x|}\right) \tfrac{x_i}{|x|}\, \dd s - \sum_{i = 1}^n\beta_i(0) x_i\\ 
		&=  \sum_{i = 1}^n \int^{|x|}_0 \left( \beta_i (\tfrac{sx}{|x|}) - \beta_i(0) \right) \tfrac{x_i}{|x|}\, \dd s = \sum_{i = 1}^n o(1) \int_0^{|x|} \tfrac{x_i}{|x|}\, \dd s = o(|x|),
	\end{align*}
	as $|x| \to 0$, which proves the claim and moreover shows that $df = \beta$. This implies $F_\nabla = 0$ and completes the proof. 
\end{proof}

We now prove Lemma \ref{lemma:contradiction}. 

%

\begin{proof}[Proof of Lemma \ref{lemma:contradiction}]
Our aim is to show that the connection $\nabla := \nabla^{\mathrm{dyn}} + i\eta \alpha$, satisfies $\mathrm{Hol}_\nabla(\gamma) = 1$ for all homotopically trivial loops $\gamma$. Then $F_\nabla = F_{\nabla^{\mathrm{dyn}}} + \eta \dd \alpha \equiv 0$ by Lemma \ref{lemma:flat}.

We start with a preliminary observation. Let $U \subset M$ be an arbitrary contractible open subset and write $\nabla^{\mathrm{dyn}} = d+i\beta_{\mathrm{dyn}}$ in this trivialization. Using \eqref{equation:parallel}, one has for every smooth loop $\gamma \subset U$ of length bounded by $1$ based at $x \in U$:
\[
v_h(x)  = \exp\left(-ik(h)\int_{\gamma} \beta_{\mathrm{dyn}} - \dfrac{z_h}{h} \int_\gamma \alpha \right) v_h(x) + o(h^{\ell/2 - n/2 - 1}).
\]
By Lemma \ref{lemma:2},
\[
\dfrac{z_h}{h} = \dfrac{i}{h}(\eta+ o(1)) + \Re\left(\dfrac{z_h}{h}\right) = i k(h) (\eta+\eps_h)  + o(h^{\ell/2}),
\]
where $\eps_h \to_{h \to 0} 0$, and we used that $k(h) = \sqrt{h^{-2} - 1}$. As a consequence, using $\|v_h\|_{C^0} \leq C h^{-n/2}$, we write
\[
	v_h(x)  = \exp\left(-ik(h)\int_{\gamma} \beta_{\mathrm{dyn}} -  ik(h)(\eta+\eps_h) \int_\gamma \alpha \right) v_h(x) + o(h^{\ell/2 - n/2 - 1}).
\]
Using the lower bound of Lemma \ref{lemma:lower-bound}, this yields
\begin{equation}
\label{equation:upper-bound}
\left|1-\exp\left(-i k(h) \int_{\gamma} (\beta_{\mathrm{dyn}} +(\eta+\eps_h) \alpha)\right)\right| = o(h^{\ell/2 - n/2 - 1}).
\end{equation}
Our aim is to show that $\int_\gamma (\beta_{\mathrm{dyn}} +\eta \alpha) = 0$; this would show indeed that $\mathrm{Hol}_\nabla(\gamma) = 1$ (for all $\gamma \subset U$). (We point out that the fact that $\eps_h$ is not necessarily $0$ creates some slight technical difficulty in the proof; it might be good to have in mind the case $\eps_h \equiv 0$ at first.)
\medskip

Assume for the sake of a contradiction that there exists $\gamma \subset U$ such that $c := \int_\gamma (\beta_{\mathrm{dyn}} +\eta \alpha) \neq 0$. Let $(\gamma_s)_{s \in [0,1]}$ be a continuous family of (smooth) curves such that $\gamma_1 = \gamma$ and $\gamma_0$ is a point and further assume $\cup_{s \in [0,1]} \gamma_s \subset U$. 

For $h > 0$, define the continuous functions $F,F_h \in C^0([0,1],\R)$ by
\[
F(s) := \int_{\gamma_s} (\beta_{\mathrm{dyn}} +\eta \alpha), \qquad F_h(s) :=\int_{\gamma_s}( \beta_{\mathrm{dyn}} +(\eta +\eps_h)\alpha)
\]
By construction $F(0)=0, F(1)=c \neq 0$. By continuity, we can find $0 < s_0 < s_1 < 1$ such that $F(s_0)$ and $F(s_1)$ are non-zero, $\mathrm{x} := F(s_0)/F(s_1)$ is an irrational algebraic number, and $F(s_0), F(s_1) \in (\min F, \max F)$. By Roth's Theorem, there exists a constant $C > 0$ such that for all integers $p,q \in \Z$ (and $q$ non-zero),
\begin{equation}
\label{equation:irrationnel}
|\mathrm{x}-p/q| > Cq^{-5/2}.
\end{equation}
(Actually, we could replace $5/2$ by $2+\eps$ for any $\eps > 0$. In order to simplify the discussion, we do not optimize the parameters here.)

Again, by continuity, for all $h > 0$ small enough, we can find $s_0(h),s_1(h) \in (0, 1)$ such that 
\begin{equation}
\label{equation:sol}
F_h(s_0(h)) = F(s_0), \qquad F_h(s_1(h))  = F(s_1).
\end{equation}
Now, we claim that the combination of \eqref{equation:irrationnel} and \eqref{equation:sol} implies that there exists a constant $C > 0$ such that for all $h > 0$ small enough:
\begin{equation}
\label{equation:bound}
\left| \exp\left(-i k(h) F_h(s_0(h))\right) + \exp\left(-i k(h) F_h(s_1(h))\right) - 2 \right| > Ch^4.
\end{equation}
Indeed, assume for the sake of a contradiction that, along a subsequence $h_n \to 0$, one has:
\[
\left| \exp\left(-i k(h_n) F_{h_n}(s_0(h_n))\right) + \exp\left(-i k(h_n)F_{h_n}(s_1(h_n))\right) - 2 \right| \leq h_n^4.
\]
(We drop the parameter $n$ for the sake of simplicity.) This implies that
\begin{align*}
\left|\cos\left( k(h) F_h(s_0(h))\right) + \cos\left( k(h) F_h(s_1(h))\right) - 2\right| &\leq h^4\\ 
\left|\sin\left( k(h) F_h(s_0(h))\right) + \sin\left( k(h) F_h(s_1(h))\right)\right| &\leq h^4.
\end{align*}
In turn, this yields for $i=0,1$ (using Taylor expansion), for some uniform constant $C > 0$
\[
k(h) F_h(s_i(h)) = 2\pi p_i(h) + r_i(h), \qquad p_i(h) \in \Z, r_i(h) \in [-\pi,\pi), |r_i(h)|\leq Ch^{2}.
\]
As a consequence, we obtain by \eqref{equation:sol} and \eqref{equation:irrationnel}:
\[
\mathrm{x}-\dfrac{p_0(h)}{p_1(h)} = \dfrac{F_h(s_0(h))}{F_h(s_1(h))} - \dfrac{p_0(h)}{p_1(h)} = \dfrac{1}{2\pi} \dfrac{p_1(h)r_0(h) - p_0(h)r_1(h)}{p_1(h)(p_1(h)+r_1(h)/2\pi)} = \mc{O}(h^3), 
\]
using that $Ch^{-1} > p_i(h) > \tfrac{1}{C} h^{-1}$ for some $C > 0$, as well as $r_i = \mc{O}(h^2)$. Hence, we get for some constants $C, C' > 0$ independent of $h > 0$:
\[
C p_1(h)^{-5/2} \leq \left|\mathrm{x}-\dfrac{p_0(h)}{p_1(h)}\right| = \mc{O}(h^3) \leq C' p_1(h)^{-3},
\]
and this is a contradiction for $h > 0$ small enough. This proves the lower bound \eqref{equation:bound}.

In order to conclude, it now suffices to compare \eqref{equation:upper-bound} and \eqref{equation:bound}. This yields, for some $C > 1$:
\[
\tfrac{1}{C} h^4 \leq \left| \exp\left(-i k(h) F_h(s_0(h))\right) + \exp\left(-i k(h)F_h(s_1(h))\right) - 2 \right| \leq C h^{\ell/2 - n/2 - 1},
\]
and this is a contradiction for $h > 0$ small enough. (One needs $\ell/2 - n/2 - 1>4$, hence $\ell > n +10$.) 
\end{proof}

\subsection{Proof of estimate \eqref{equation:bound2}} The proof is almost the same as for \eqref{equation:bound1}.

\begin{proof}[Proof of \eqref{equation:bound2}] We argue by contradiction as in the proof of \eqref{equation:bound1}. Assume that \eqref{equation:bound2} does not hold. Then, there exists a sequence $z_n \in \mathbb{B} = \{z \in \C \mid 0 \leq \Re(z) \leq 1\}$, $k_n \in \Z $ with $|k_n| \leq |\Im (z_n)|$ and $f_n \in \mc{H}^s_{\langle{\lambda}\rangle^{-1}_n}(M,L^{\otimes k_n})$ (with $\lambda_n := \Im(z_n)$) such that
\begin{equation}
\label{equation:quasimode1bis}
\|f_n\|_{\mc{H}^s_{\langle{\lambda_n}\rangle^{-1}}(M,L^{\otimes k_n})}=1, \qquad \|(-\X_{k_n}-z_n)f_n\|_{\mc{H}^s_{\langle{\lambda_n}\rangle^{-1}}(M,L^{\otimes k_n})} = o(\langle \lambda_n\rangle^{-\ell}).
\end{equation}
As before, we will further assume $k_n \geq 0,\Im(z_n) \geq 0$ (the other cases are treated similarly). Moreover, we can also assume that $(\Im(z_n))_{n \geq 0}$ is unbounded otherwise both $(k_n)_{n \geq 0}$ and $(\Im(z_n))_{n \geq 0}$ are both bounded and so \eqref{equation:quasimode1bis} is immediately contradicted. Finally, up to taking a subsequence, we can also assume $\Re(z_n) \to \nu \in [0,1]$ and $k_n/\langle{\lambda_n}\rangle \to \eta \in [0,1]$. We set $h_n := \langle{\lambda_n}\rangle^{-1}$. We will drop the index $n$ in what follows and see $k_n$ as a function of $h$, namely $h \mapsto k(h)$ (one can keep in mind that $h > 0$ belongs to a discrete subset of $(0,1]$ but this is irrelevant for the argument). In this case, the relevant calculus is $\Psi^\bullet_{h, k}(M,L)$ for this specific sequence $h \mapsto k(h)$, as introduced in \S\ref{ssection:quantization-line-bundles}.

We can then rewrite \eqref{equation:quasimode1bis} as
\begin{equation}
\label{equation:quasimode2bis}
\|f_h\|_{\mc{H}^s_h(M,L^{\otimes k(h)})} = 1, \qquad \|\mathbf{P}_h f_h\|_{\mc{H}^s_{h}} = o(h^{\ell+1}),
\end{equation}
where 
\[
	\mathbf{P}_h := -h\X_{k(h)} - h(\nu+o(1)) - i(1 + \mc{O}(h^3))
\] 
and the ``$o$" and ``$\mc{O}$" in the expression of $\mathbf{P}_h$ are real constants (depending on $h$). Observe that it is almost the same expression as \eqref{equation:quasimode2} with the exception that $k(h)$ is now an arbitrary function of $h$ (it could be distinct from $k(h) = \sqrt{h^{-2} - 1}$ or $k(h) = 1/h$) such that $h \mapsto hk(h)$ is bounded and $hk(h) \to \eta$. (Note in particular, that $k(h) \equiv 0$ is an admissible function in the sense of \S \ref{ssection:quantization-line-bundles}. In this case, this amounts to proving high-frequency estimates for the resolvent of the Anosov vector field acting on functions.)

It is then immediate to verify that all the arguments developed in \S\ref{ssection:dur} apply \emph{verbatim} the same up to Lemma \ref{lemma:contradiction}. The only slight difference lies in the proof of Lemma \ref{lemma:contradiction} where one has to use that 
\[
	\frac{z_h}{h} = \frac{i}{h}(1 + \varepsilon_h) + o(h^{\ell/2}),
\]
and also that $k(h) = h^{-1}(\eta + \delta_h)$ for some sequence $\delta_h \to 0$ as $h \to 0$, so instead of \eqref{equation:upper-bound} we would get
\[
	\left|1-\exp\left(-\frac{i}{h} \int_{\gamma} ((\eta + \delta_h)\beta_{\mathrm{dyn}} +(1 + \eps_h) \alpha)\right)\right| = o(h^{\ell/2 - n/2 - 1}).
\]
Then the remainder of the argument applies in a similar way to give the vanishing of the curvature $\eta F_{\nabla_{\mathrm{dyn}}} + \dd \alpha \equiv 0$, which contradicts the assumptions of Theorem \ref{theorem:main2}. This proves \eqref{equation:bound2}. \end{proof}

\section{High-frequency estimates II. Arbitrary group $G$}

\label{section:hf2}

We now consider the general case, where $G$ is a compact connected Lie group. Recall that the anisotropic spaces $\mc{H}_h^s(F, \Lk)$ were introduced in \S \ref{sssection:anisotropic-space}, and that we write $\mc{H}_{h, \mathrm{hol}}^s(F, \Lk)$ for the subspace of fibrewise holomorphic sections in $\mc{H}^s_h(F, \Lk)$. We also recall that $\mathbb{B} = [0, 1] \times \mathbb{R} \subset \C$. Compared to the preceding section, we will use the BW-calculus, that is $\Psi^{\bullet}_{h, \mathrm{BW}}(P)$ which we introduced in \S \ref{ssection:twiste-line-bundles}, which by construction respects the fibrewise holomorphicity (up to negligible remainders). Recall also that we identified $\widehat{G}$ with a quotient of $\mathbb{Z}^a \times \mathbb{Z}^b_{\geq 0}$ via \eqref{equation:ecriture}.

The following holds:

\begin{proposition}
\label{proposition:technical-reduction2}
Assume that $\psi^F$ is ergodic and that $(\pi^*d\alpha, -iF_{\overline{\nabla}_1}, \dotsc, -iF_{\overline{\nabla}_a})$ are linearly independent over $\mathbb{R}$, and set $\ell := \max\big(3\dim F + 4, \dim F + 10\big) + 1$. For any $s > 0$, there exists $C > 0$ such that for all $z \in \B$, $\mathbf{k} \in \widehat{G}$, setting $\lambda := \Im(z)$:
\begin{align}
\label{equation:bound12}
& \|(-\X_{\mathbf{k}} - z)^{-1}\|_{\mc{H}^s_{\langle\mathbf{k}\rangle^{-1}, \mathrm{hol}}(F,\mathbf{L}^{\otimes \mathbf{k}}) \circlearrowleft} \leq C \langle \mathbf{k} \rangle^\ell, \qquad |\mathbf{k}| \geq |\lambda|, \\
\label{equation:bound22}
& \|(-\X_{\mathbf{k}} - z)^{-1}\|_{\mc{H}^s_{\langle{\lambda}\rangle^{-1}, \mathrm{hol}}(F,\mathbf{L}^{\otimes \mathbf{k}}) \circlearrowleft}\leq C \langle \lambda \rangle^\ell, \qquad |\lambda| \geq |\mathbf{k}|.
\end{align}
\end{proposition}

We remark that the value of $\ell$ in Proposition \ref{proposition:technical-reduction2} is slightly worse than in the circle extension case (Proposition \ref{proposition:technical-reduction}); the reason is that in the general Lie group extension case we will use some estimates coming from Diophantine action theory. Similarly to Proposition \ref{proposition:technical-reduction}, this implies Theorem \ref{theorem:real}.

\subsection{Proof of \eqref{equation:bound12}}

We first prove \eqref{equation:bound12}. Assume for the sake of a contradiction that \eqref{equation:bound12} does not hold. Then, there exists a sequence $\mathbf{k}_n \in \widehat{G}$, $z_n \in \B$ and $f_n \in \mc{H}^s_{\langle\mathbf{k}_n\rangle^{-1}}(F,\mathbf{L}^{\otimes \mathbf{k}_n})$ such that $\|f_n\|_{\mc{H}^s_{\langle\mathbf{k}_n\rangle^{-1}}(F,\mathbf{L}^{\otimes \mathbf{k}_n})} = 1$ and
\[
\|(-\X_{\mathbf{k}_n} - z_n)f_n\|_{\mc{H}^s_{\langle\mathbf{k}_n\rangle^{-1}}(F,\mathbf{L}^{\otimes \mathbf{k}_n})} = o(\langle \mathbf{k}_n \rangle^{-\ell}), \quad \overline{\partial}_{\mathbf{k}_n} f_n = 0.
\]
We can always assume that $|\mathbf{k}_n| \to \infty$ (otherwise both $|\mathbf{k}_n|$ and $|\lambda_n|$ are bounded and the statement is trivial). As in \S\ref{ssection:dur}, up to passing to a subsequence, we can also assume that $\Re(z_n) \to \nu \in [0,1]$ and $\Im(z_n)/\langle{\mathbf{k}_n}\rangle \to \eta \in [0,1]$. We change variables by setting $h := \langle\mathbf{k}_n\rangle^{-1}$ and write $\mathbf{k}(h)$ from now on (we drop the letter $n$ for simplicity). 

Setting
\[
\mathbf{P}_h := -h\X_{\mathbf{k}(h)} - h(\nu+o(1)) - i(\eta +o(1)) \in \Psi^1_{h,\mathrm{BW}}(P),
\]
we can rewrite the previous equations as
\begin{equation}
\label{equation:quasi-mode-twist}
\|f_h\|_{\mc{H}^s_{h}(F,\mathbf{L}^{\otimes \mathbf{k}(h)})} = 1, \quad \|\mathbf{P}_hf_h\|_{\mc{H}^s_{h}(F,\mathbf{L}^{\otimes \mathbf{k}(h)})} = o(h^{\ell+1}), \quad \overline{\partial}_{\mathbf{k}(h)} f_h = 0.
\end{equation}

%

Recall from \eqref{eq:F-horizontal-vertical} that $T^*F$ admits a splitting into horizontal and vertical bundles induced by the dynamical connection. In the following, we recall $\nabla_{\mathbf{k}}^{\mathrm{dyn}}$ denotes the (partial) dynamical covariant derivative on $\Lk \to F$ as introduced in \S\ref{ssection:dynamical-connection} and \eqref{equation:connection-horizontale-lk}. Write $\pi_F: F \to M$ for the projection, and define
\[
\mc{T}_\eta := \{(x,- \eta~ \pi_F^* \alpha(x)) \in T^*F ~|~ x \in F\}.
\]
The following holds:

\begin{lemma}\label{lemma:microlocalisation}
Let $\mathbf{E}_h \in \Psi^{\mathrm{comp}}_{h,\mathrm{BW}}(P)$. Then:
\begin{enumerate}[label=\emph{(\roman*)}]
\item We have that
\[
\WF_{\mathrm{BW}}(\mathbf{E}_h) \cap \mc{T}_\eta = \emptyset \implies \mathbf{E}_h f_h = o_{L^2}(h^{\ell/2}) = o_{C^0}(h^{\ell/2 - \dim F/2}).
\]
\item Let $e \in C^\infty_{\comp}(T^*M)$ such that $e \equiv 1$ near
\[
\{(x,-\eta\alpha(x)) \in T^*M ~|~ x \in M\},
\]
and define $\mathbf{E}_h := \Op^{\mathrm{BW}}_h(e)$ by \eqref{equation:quantization-geometric}, which is microlocalized near $\mc{T}_\eta$. Then there exists $c > 1$ such that $1/c \leq \|\mathbf{E}_h f_h\|_{L^2} \leq c$. Furthermore, $\nu = 0$, $\mathbf{P}_h = -h\X_h - z_h$ where $z_h = i(\eta+o(1)) + o(h^{\ell/2+1})$, $\Re(z_h) \geq 0$, and 
\begin{equation}
\label{equation:parallel2}
\begin{split}
& (-h\nabla_{\mathbf{k}(h)}^{\mathrm{dyn}} - z_h \pi_F^*\alpha \wedge)\mathbf{E}_h f_h = o_{L^2}(h^{\ell/2}) = o_{C^0}(h^{\ell/2-\dim F/2}).
\end{split}
\end{equation}
\end{enumerate}
\end{lemma}

\begin{proof}
(i)  Note that on the complement of $\mathbb{H}^*_F$ we have $\mathbf{E}_h f_h = 0$ by $\overline{\partial}_{\mathbf{k}(h)} \mathbf{E}_h f_h = 0$ and ellipticity; so it remains to study the microsupport on $\HH^*_F$. The argument is the same as in Lemma \ref{lemma:microsupport} based on source estimates (but for the partially hyperbolic flow on $F$ this time, see Theorem \ref{theorem:source-sink}). 

\medskip

(ii) The same proofs as in Lemmas \ref{lemma:2} (using the sink estimate from Theorem \ref{theorem:source-sink}) and \ref{lemma:parallel} go through (again, by using the partially hyperbolic flow on $F$).
\end{proof}

We can now complete the proof of \eqref{equation:bound12}. Define $v_h := \mathbf{E}_h f_h$, where $\mathbf{E}_h$ is as in (ii) of the previous lemma.

\begin{proof}[Proof of \eqref{equation:bound12}]
As a preliminary observation, note that ergodicity of $\psi^F$ and the (non-integrability) assumption in Theorem \ref{theorem:main2}, Item (i), imply ergodicity (and mixing) of $\psi$ by Lemma \ref{lemma:no-resonances}. Similarly to \eqref{equation:d-small}, we deduce from \eqref{equation:parallel2} and the fact that $\Re(z_h/h) = o(h^{\ell})$:
\begin{equation}\label{eq:quasimode-normed-invariance}
	d|_{\HH_F} (|v_h|^2) = o_{C^0}(h^{\ell/2-\dim F-1}).
\end{equation}
Let us fix an arbitrary point $x_0 \in M$ and identify $F_{x_0}$ and $P_{x_0}$ with $G$ and $G/T$, respectively. We obtain that $F_{x_0} \ni z \mapsto |v_h|^2(x_0,z)$ is nearly invariant by all elements $g \in G$ obtained by doing stable/unstable/flow holonomies over loops based at $x_0$ obtained by concatenating stable/unstable/flow paths over $M$ (of bounded length). More precisely, fix $\beta > 0$. If $\tau \subset M$ is an oriented piecewise smooth path based at $x_0$, obtained by concatenation of stable/unstable/flow paths, and of total length $\ell(\tau) \leq h^{-\beta}$, then setting $g_\tau := \mathrm{Hol}^{\nabla^{\mathrm{dyn}}}_\tau \in G$ (holonomy of the dynamical connection on $P$ along $\tau$), we get that over $F_{x_0}$
\[
	g_\tau^*|v_h|^2(x_0, \bullet) - |v_h|^2(x_0, \bullet) = o_{C^0}(h^{\ell/2-(\dim F+1)-\beta}).
\]
Now, take a finite subset $W \subset \rho_{\mathrm{Parry}}(\mathbf{G})$ that is $\eps_0$-dense for $\eps_0 > 0$ sufficiently small (this is always possible since $G = \overline{\rho_{\mathrm{Parry}}(\mathbf{G})}$ by assumption). By Corollary \ref{corollary:diophantine-flag-manifold}, for $\delta \in (0, 1)$ fixed (which may be taken arbitrarily close to $0$), we know that there exists a constant $C > 0$ such that $\mc{W}_n$ is $Cn^{-(1 - \delta)}$-dense in $G$. Here, $\mc{W}_n$ is defined in \eqref{eq:diophantine-w_n-def}; note that an element $g \in \mc{W}_n$ is obtained by holonomy (with respect to the dynamical connection $\nabla^{\mathrm{dyn}}$) along an oriented path of total length $\leq Cn$. We take $N := \lceil{h^{-\beta}}\rceil$. This gives us that over $F_{x_0}$:
\begin{equation}
\label{equation:inv1}
g^*|v_h|^2(x_0, \bullet) - |v_h|^2(x_0, \bullet) = o_{C^0}(h^{\ell/2 -(\dim F + 1)-\beta}), \qquad \forall g \in \mc{W}_{N},
\end{equation}
and the set $\mc{W}_{N}$ is $C h^{\beta (1 - \delta)}$-dense in $G$.

On the other hand, we know by Sobolev embedding that 
\[
	\||v_h|^2\|_{C^1} \leq C h^{-(\dim F/2 + 1)} \||v_h|^2\|_{H^N_h} = \mc{O}(h^{-(\dim F/2 + 1)})
\] 
for all $N > n/2+1$. Hence for all $z, z' \in F_{x_0}$ such that $d(z,z') \leq C h^{\beta (1 - \delta)}$, we get
\begin{equation}
\label{equation:inv2}
|v_h|^2(x_0,z) - |v_h|^2(x_0,z') = \mc{O}(d_{F_{x_0}}(z,z')h^{-(\dim F/2 + 1)}) = \mc{O}(h^{\beta (1 - \delta) - (\dim F/2+1)}).
\end{equation}
Combining \eqref{equation:inv1} and \eqref{equation:inv2}, we therefore see that there exists a constant $c_h = |v_h(x_0, z_0)|^2 \geq 0$ for some $z_0 \in F_{x_0}$, such that over $F_{x_0}$
\begin{equation}\label{eq:c_h}
	|v_h|^2(x_0,\bullet) = c_h + \mc{O}_{C^0}(h^\theta), 
\end{equation}
where 
\[
	\theta = \min\big(\beta (1 - \delta) - (\dim F/2+1), \ell/2-(\dim F+1)-\beta\big). 
\]
Thus for any $\ell > 3\dim F + 4$, there exist $\beta > \dim F/2+1$ and $\delta > 0$ small enough such that that $\theta > 0$. We next split the proof according to the properties of the weight $\gamma: T \to \mathbb{S}^1$ associated to $\mathbf{k}(h) = (k_1, \dotsc, k_a, k_{a+1}, \dotsc, k_{a+b}) \in \widehat{G}$.
\medskip

\emph{Case 1: $(k_{a + 1}, \dotsc, k_{a + b}) \neq 0$.} Assume we have this property for a sequence of $h \to 0$. By Proposition \ref{prop:line-bundle-topology}, the line bundle $\mathbf{L}^{\otimes \mathbf{k}}|_{F_{x_0}}$ is topologically non-trivial, so we deduce that $v_h|_{F_{x_0}}$ has to vanish at some point. We thus get $c_h = \mc{O}(h^\theta)$ and so $|v_h|^2|_{F_{x_0}} = \mc{O}_{C^0}(h^\theta)$. Using \eqref{eq:quasimode-normed-invariance} again, and using that any $x \in M$ is connected to $x_0$ by a concatenation of stable/unstable/flow paths $\tau$ (of uniformly bounded length), we get $\|v_h\|^2_{C^0} = o(1)$. This contradicts the fact that $\|v_h\|_{L^2(F)} > 1/c > 0$ obtained in Lemma \ref{lemma:microlocalisation}, Item (ii). 
\medskip

\emph{Case 2: $(k_{a + 1}, \dotsc, k_{a + b}) = 0$.} Assume we have this property for all but finitely many $h$. Then, by Lemma \ref{lemma:extension}, we deduce that $\gamma$ extends to a homomorphism $\widetilde{\gamma}: G \to \mathbb{S}^1$. As in \S\ref{sssection:abelian-rep}, define the line bundle $\mathbf{L}^{\otimes \mathbf{k}}_M := P \times_{\widetilde{\gamma}} \C$ over $M$; using the assumption $(k_{a + 1}, \dotsc, k_{a + b}) = 0$, $\pi_F^*\Lk_M$ is identified with $\Lk$ (see Lemma \ref{lemma:pullback-equivalence}). By Lemma \ref{lemma:pullback-equivalence}, \eqref{eq:pushforward-identity} and \eqref{equation:parallel2}, the pushforward $u_h := (\pi_F)_* v_h \in C^\infty(M,\mathbf{L}^{\otimes \mathbf{k}}_M)$ satisfies (note that as $v_h$ is fibrewise holomorphic, it is fibrewise constant)
\begin{equation}
\begin{split}
\label{equation:parallel3}
(-h\nabla^{\mathrm{dyn}}_{\mathbf{k}, M} - z_h \alpha \wedge)u_h = o_{L^2}(h^{\ell/2}) = o_{C^0}(h^{\ell/2 - \dim F/2}), \quad \|u_h\|_{L^2} &\geq c_1, 
\end{split}
\end{equation}
for some $c_1 > 0$, where $\nabla^{\mathrm{dyn}}_{\mathbf{k}, M}$ denotes the associated dynamical connection on $\Lk_M$. For $i = 1, \dotsc, a$, as in \S \ref{sssection:abelian-rep} denote by $\nabla_{i, M}^{\mathrm{dyn}}$ the associated dynamical connection on $\mathbf{L}_{i, M} := \mathbf{L}_M^{\e_i}$, where $\e_i$ is the standard vector consisting of zeroes and a single $1$ at the $i$th position. Then, we have

\begin{lemma}\label{lemma:torus-bundle-curvature}
	Assume \eqref{equation:parallel3} holds and $h \mathbf{k}(h) \to \mathbf{l} = (\ell_1, \dotsc, \ell_a, 0, \dotsc, 0)$ as $h \to 0$. Then
	\[
		-i\sum_{j = 1}^a \ell_j F_{\nabla^{\mathrm{dyn}}_{j, M}} + \eta \dd\alpha = 0.
	\]
\end{lemma}
\begin{proof}
	The proof is completely analogous to Lemma \ref{lemma:contradiction} and we just outline the differences. First of all, similarly to \ref{equation:d-small}, from \eqref{equation:parallel3} it follows that $\dd |u_h|^2 = \mc{O}_{C^0}(h^{\ell/2 - \dim F - 1})$ (here we use that $\ell > 2\dim F + 2$). As a consequence, similarly to Lemma \ref{lemma:lower-bound}, there is a $C > 0$ such that $|u_h| > C$ for all $h$.
	
	Consider a contractible open set $U \ni x$ and a closed loop $\gamma \subset U$ based at $x$. For $j = 1, \dotsc, a$, we trivialize the line bundles $\mathbf{L}_{i, M}$ over $U$ and write $\nabla_{j, M} = d + i\beta_{j, M}$. Denote by 
	\[
		\boldsymbol{\beta}_{\mathrm{dyn}} := (\beta_{1, M}, \dotsc, \beta_{a, M})
	\] 
	the vector of connection $1$-forms. Then by \eqref{equation:parallel3}
	\[
		u_h(x)  = \exp\left(-i \int_{\gamma} \mathbf{k} \cdot \boldsymbol{\beta}_{\mathrm{dyn}} - \dfrac{z_h}{h} \int_\gamma \alpha \right) u_h(x) + o(h^{\ell/2 - \dim F/2 - 1}).
	\]
	Rewriting, using $|u_h| > C$, as well as $h = \langle{\mathbf{k}}\rangle^{-1}$, and $z_h = i(\eta + o(1)) + o(h^{\ell/2 + 1})$, we get
	\begin{equation*}
		\left|1-\exp\left(-\frac{i}{h}\int_{\gamma}  \left((\mathbf{l} + o(1)) \cdot \boldsymbol{\beta}_{\mathrm{dyn}} + (\eta + o(1)) \alpha\right)\right)\right| = o(h^{\ell/2 - \dim F/2 - 1}).
	\end{equation*}
	The remainder of the argument in Lemma \ref{lemma:contradiction} applies, and we obtain that 
	\[
		0 = \int_\gamma (\mathbf{l} \cdot \boldsymbol{\beta}_{\mathrm{dyn}} + \eta \alpha).
	\]
	(Here we use that $\ell > \dim F + 10$.) The result then follows from Lemma \ref{lemma:flat}.
\end{proof}

Since $h|\mathbf{k}(h)|$ is bounded, we may assume without loss of generality that $h \mathbf{k}(h) \to \mathbf{l}$ for some $\mathbf{l} = (\ell_1, \dotsc, \ell_a, 0, \dotsc, 0)$ where $\ell_i \in [0, 1]$ for $i = 1, \dotsc, a$. Then, the result of Lemma \ref{lemma:torus-bundle-curvature} contradicts the assumptions of the proposition (see Lemma \ref{lemma:abelian-curvature}) and completes the proof once we have $\ell > \max\big(3\dim F + 4, \dim F + 10\big)$.


\end{proof}


\subsection{Proof of \eqref{equation:bound22}} The proof is similar to \eqref{equation:bound12}, just like the proof of \eqref{equation:bound2} is similar to that of \eqref{equation:bound1}.

\subsection{Equivalence of rapid mixing}

Denote by $\mathbb{A} := P/[G, G]$ the Abelianisation of $P$ (which is  a principal torus bundle over $M$), by $\mathrm{pr}_{\mathbb{A}}: P \to P/[G, G]$ the quotient projection, by $d\mu_{\mathbb{A}}$ the quotient volume form on $\mathbb{A}$ preserved by the quotient flow $\psi^{\mathbb{A}}$. Write $X_{\mathbb{A}}$ for the vector field generating $\psi^{\mathbb{A}}$. Note that for any $t \in \mathbb{R}$, by definition, $\mathrm{pr}_{\mathbb{A}} \circ \psi_t = \psi_t^{\mathbb{A}} \circ \mathrm{pr}_{\mathbb{A}}$. Here, we prove Theorem \ref{theorem:main2}, Item (ii), that is, the following statement:

\begin{theorem}
	The flow $\psi$ on $P$ is rapid mixing if and only if the quotient flow $\psi^{\mathbb{A}}$ on the Abelianisation $P/[G, G]$ is rapid mixing.
\end{theorem}
\begin{proof}
We will prove the equivalence by first showing that rapid mixing on $P$ implies rapid mixing on $\mathbb{A}$, and then showing the converse.
 \medskip
 
	\emph{Step 1: $P \implies \mathbb{A}$.} If $f, g \in C^\infty(\mathbb{A})$ with zero mean, and $k \in \mathbb{N}$, there exists a positive $s$ such that
	\begin{align*}
		&\vol([G, G]) \int_{\mathbb{A}} f \cdot (\psi_t^{\mathbb{A}})^*g\, \dd \mu_{\mathbb{A}} = \int_{P} \pr_{\mathbb{A}}^*f \cdot \pr_{\mathbb{A}}^* (\psi_t^{\mathbb{A}})^*g\, \dd\mu \\
		& = \int_{P} \pr_{\mathbb{A}}^*f \cdot \psi_t^* \pr_{\mathbb{A}}^*g\, \dd \mu = \mc{O}(\|f\|_{H^s} \|g\|_{H^s} t^{-k}),
	\end{align*}
	as $t \to \infty$, where in the final line we used that $\pr_{\mathbb{A}}^*: H^\alpha(\mathbb{A}) \to H^\alpha(P)$ is continuous for any $\alpha > 0$, as well as the hypothesis. This completes the proof.
	\medskip
	
	\emph{Step 2: $\mathbb{A} \implies P$.} Let $f, g \in C^\infty(\mathbb{A})$ with zero mean. Then
	\[
		\int_0^\infty e^{-zt} \int_{\mathbb{A}} (\psi_{-t}^{\mathbb{A}})^*f \cdot g\, d\mu_{\mathbb{A}}\, dt = \langle{(X_{\mathbb{A}} + z)^{-1}f, g}\rangle_{L^2}, 
	\]
	for $\Re z > 0$, as the time integral converges (e.g. in $L^2$). Moreover, by the rapid mixing hypothesis, the left hand side of this formula still makes sense for $\Re z = 0$, and this defines $(X_{\mathbb{A}} + z)^{-1}$ on smooth functions on $\mathbb{A}$ with zero mean with values in distributions. In fact, using that the norm of correlations is bounded by $\mc{O}_k(\|f\|_{H^s} \|g\|_{H^s} t^{-k})$ as $t \to \infty$, this defines by duality for $\Re z \geq 0$ and $z \neq 0$ (or when $z = 0$ after restricting to zero mean functions)
	\[
		(X_{\mathbb{A}} + z)^{-1}: H^s(\mathbb{A}) \to H^{-s}(\mathbb{A}),
	\]
	where we took for instance $k = 2$ (and used that $\int_1^\infty t^{-2}\, dt < \infty$). 
	
	
	 
	We follow the proof of Theorem \ref{theorem:main2}, Item (i). Let us discuss the proof of \eqref{equation:bound12} and outline the differences to the present situation. Everything remains the same up until Case 2 (the case when $k_{a + 1} = \dotsb = k_{a + b} = 0$). Recall that we have $u_h$, a smooth section of $\Lk_M = P \times_{\widetilde{\gamma}} \mathbb{C}$, where $\widetilde{\gamma}: G \to \mathbb{S}^1$ is the extension of the global weight $\gamma$ to $G$. Moreover, $u_h$ is a quasimode, i.e. it satisfies \eqref{equation:parallel3}. Applying the inverse Fourier transform \eqref{equation:inversion} to $u_h$ yields a smooth function $v_h \in C^\infty(\mathbb{A})$
	\[
		q_h(x, w) := w^{-1} u_h(\widetilde{\gamma}, x),\quad (x, w) \in \mathbb{A},
	\]
	where $w$ is seen as a (unitary) isomorphism $w: \mathbb{C} \to \Lk_M(x)$. Moreover, thanks to the fact that the Fourier transform commutes with the flow (see \eqref{eq:ft-intertwines-infinitesimal}), $q_h$ satisfies (we only use the flow direction in \eqref{equation:parallel3})
	\begin{equation}\label{eq:derived-XXX}
\begin{split}
	(-hX_{\mathbb{A}} - z_h)q_h & = o_{L^2}(h^{\ell/2}) = o_{H^N_h}(h^{\ell/2}) = o_{H^N}(h^{\ell/2 - N}),\\
	\|q_h\|_{L^2} &\geq c, 
\end{split}
\end{equation}
	for some $c > 0$, where $N > 0$ is arbitrary. Here, we are free to upgrade the norm to $H^N_h$ since the analogous bound holds for $u_h$ by the proof of Lemma \ref{lemma:2}; we also used that $\|f\|_{H^N_h} \geq Ch^{N} \|f\|_{H^N}$ for some $C>0$ and for any smooth $f$. 
%
	Apply the resolvent to \eqref{eq:derived-XXX}, with $N = s$, to get 
	\[
		q_h = o_{H^{-s}}(h^{\ell/2 - s - 1}).
	\] 
	On the other hand, we have $\|q_h\|_{H^s} = \mc{O}(h^{-s})$, and by interpolation we get
	\[
		\|q_h\|_{L^2}^2 \leq C \|q_h\|_{H^s} \|q_h\|_{H^{-s}} = o(h^{\ell/2 - 2s - 1}).
	\]
	 When $\ell > 4s + 2$ (and $\ell > \max(3\dim F + 4, \dim F + 10) + 1$ as previously required), this contradicts the fact that $\|v_h\|_{L^2} \geq c > 0$ and establishes \eqref{equation:bound12}; the bound \eqref{equation:bound22} is similarly obtained. That this implies rapid mixing is shown in \S \ref{ssection:implication} (with explicit dependence of the exponent $\vartheta$ on $s$, e.g. any $\vartheta > 4\ell + 4$ would work). This completes the proof. 
\end{proof}

 \newpage 
\hspace{1cm}

 \chapter{Spectral theory of horizontal Laplacians}
 
 \label{chapter:hypoelliptic}
 
In this chapter, we study sub-elliptic Laplacians obtained as horizontal Laplacians on a principal bundle constructed from a $G$-equivariant connection. We prove several spectral properties for these operators. First, we establish that they are globally hypoelliptic if the group $G$ is semisimple and the connection has dense holonomy group. We also provide a lower bound for the first eigenvalue of the induced Laplacian on fiberwise holomorphic sections of the corresponding line bundles (over the flag bundle) when the curvature is non-degenerate. Finally, we prove a quantum ergodicity result for flat connections, as well as in the setting of perturbations of flat connections.
 
  \minitoc
 
 \newpage
 
 \section{Introduction}\label{sec:horizontal-lapl-intro}
 
 Let $(M,g)$ be a smooth closed connected Riemannian manifold and $G$ be a connected compact Lie group. Throughout this chapter, we shall consider a $G$-principal bundle $\pi : P \to M$ equipped with $\nabla$, a $G$-connection on $P$, which gives rise to a splitting $TP = \HH \oplus \V$ into horizontal $\HH$ and vertical $\V$ subspaces. As in \S\ref{section:connection-principal-bundles}, we define
\[
	d_{\HH} : C^\infty(P) \to C^\infty(P,\HH^*), \qquad d_{\HH}f := df|_{\HH},
\]
and the horizontal Laplace operator $\Delta_{\HH} := (d_{\HH})^*d_{\HH}$. (Here, we recall that $\HH^*$ is the annihilator of $\V$.) The purpose of this chapter is to describe the spectral properties of the operator $\Delta_{\HH}$. 

Recall that $F := P/T \to M$ is the flag bundle, where $T \leqslant G$ is a maximal torus, and that $\Lk \to F$ is a fibrewise holomorphic line bundle defined by a highest weight $\mathbf{k} \in \widehat{G}$ of an irreducible representation (see \eqref{equation:ecriture}). By Lemma \ref{lemma:ft-horizontal-laplacian}, the operator $\Delta_{\HH}$ induces for all $\mathbf{k} \in \widehat{G}$ an operator $\Delta_{\mathbf{k}}$ on $C^\infty_{\mathrm{hol}}(F,\mathbf{L}^{\otimes \mathbf{k}})$ and this operator has discrete spectrum contained in $[0,\infty)$. More precisely, recalling that $\overline{\partial}_{\mathbf{k}}$ is the fibrewise holomorphic derivative, the operator $\Delta_{\mathbf{k}} + (\overline{\partial}_{\mathbf{k}})^*\overline{\partial}_{\mathbf{k}} + 1$ is elliptic and invertible for fixed $\mathbf{k}$, commutes with $\Pi_{\mathbf{k}}$ by \eqref{equation:commutation-nablak-pik}, and its inverse defines a compact self-adjoint operator $(\Delta_{\mathbf{k}} + 1)^{-1}: L^2_{\mathrm{hol}}(F, \Lk) \to L^2_{\mathrm{hol}}(F, \Lk)$, which allows to define the spectrum of $\Delta_{\mathbf{k}}$ (counted with multiplicity and sorted in non-decreasing order):
 \[
 	\mathrm{spec}_{L^2}(\Delta_{\mathbf{k}}) := (\lambda_j(\mathbf{k}))_{j \geq 1} \subset [0,\infty).
 \]
 
 \subsection{Lower bound on the first eigenvalue}
 
We first aim to establish a lower bound on the first eigenvalue of the operator $\Delta_{\mathbf{k}}$. Recall from \eqref{equation:non-degenerate-curvature} that the curvature $F_{\mathrm{Ad}(P)}$ of the connection $\nabla$ on $P$ is said to be non-degenerate at $x \in M$ if
\[
\mathrm{Span}\left(F_{\mathrm{Ad}(P)}(x)(X,Y) \mid X,Y \in T_xM\right) = \mathrm{Ad}(P_x),
\]
and globally non-degenerate if it is non-degenerate at every $x \in M$. (Here, $\Ad(P) \to M$ is the vector bundle associated to the adjoint representation of $G$.) Also recall the following definitions from \S\ref{sssection:non-degenerate-curvature}: for $\mathbf{l} \in \partial_\infty \mathfrak{a}_+ \cong \mathbb{S}^{d-1} \cap (\R^a \times \R^b_+)$, 
\[
F_{\mathrm{min}}(\mathbf{l}) = \min_{(x,w) \in F} \sup_{\substack{X,Y \in T_xM \\ |X|=|Y|=1}} -i\times \mathbf{l} \cdot\mathbf{F}_{\overline{\nabla}}(X^{\HH_F},Y^{\HH_F}),
\]
where $X^{\HH_F},Y^{\HH_F}$ are the horizontal lifts of $X,Y$ respectively to $F$ and
\[
F_{\mathrm{min}} := \min_{\mathbf{l} \in \partial_\infty \mathfrak{a}_+} F_{\mathrm{min}}(\mathbf{l}).
\]
Since $\partial_\infty \mathfrak{a}_+$ is compact, this minimum is indeed achieved. It was shown in Lemma \ref{lemma:enfin} that the curvature is globally non-degenerate if and only if $F_{\mathrm{min}} > 0$, and in the discussion preceding the lemma it was mentioned that by the Ambrose-Singer theorem, in turn this is if and only if the local holonomy group of $G$ equals $G$ at every point $x \in M$.

Next, fix $x_0 \in M$, and define $\mathrm{Hol}(P,\nabla) \leqslant G$ to be the holonomy group of the connection on $P$ at $x_0$. Notice that one obtains conjugate holonomy groups by changing $x_0$ by another point, so the choice of point $x_0 \in M$ is irrelevant. We also emphasize that in general $H$ may \emph{not} be closed (see Example \ref{example:su2} below for instance). 

Finally, recall (see also \eqref{eq:splitting-centre-commutator} and the subsequent paragraphs) that the adjoint action of $G$ on the Lie algebra $\mathfrak{g}$ is diagonal according to the splitting $\mathfrak{g} = \mathfrak{z} \oplus [\mathfrak{g}, \mathfrak{g}]$. Thus the adjoint bundle $\Ad(P) \to M$ splits as $\Ad(P) = \mathbb{R}^a \oplus E$, where $a = \dim \mathfrak{z}$ and $\mathbb{R}^a$ denote the trivial vector bundle $\mathbb{R}^a \times M$. Write $(F_1, \dotsc, F_a)$ for the components of $F_{\Ad(P)}$ in the direction of $\mathbb{R}^a$ (these are real-valued differential two-forms).

Under certain assumptions on the local or global holonomy of the connection on $P$, we shall prove the following lower bounds on $\lambda_1(\Delta_{\mathbf{k}})$:

\begin{theorem}
\label{theorem:lambda1}
The following holds:
\begin{enumerate}[label=\emph{(\roman*)}]
\item Assume that the curvature is globally non-degenerate. Then for all $\eps > 0$, there exists $R > 0$ such that:
\[
\lambda_1(\Delta_\mathbf{k}) \geq (F_{\mathrm{min}}/2-\eps) |\mathbf{k}|, \qquad \forall \mathbf{k} \in \widehat{G}, |\mathbf{k}| > R.
\]
\item Assume $(F_1, \dotsc, F_a)$ are linearly independent over $\mathbb{R}$, and that the holonomy group $H$ is dense in $G$. Then, for all $\nu > \max(4\dim F + 2, \dim F + 8)$ there exists $C > 0$ such that:
\begin{equation}
\label{equation:lower-bound-ss}
\lambda_1(\Delta_{\mathbf{k}}) \geq C |\mathbf{k}|^{-\nu}, \qquad \forall \mathbf{k} \in \widehat{G}.
\end{equation}
\end{enumerate}
\end{theorem}



The constant $F_{\mathrm{min}}/2$ is sharp, as can be seen from examples such as the magnetic Laplacian on a hyperbolic surface (see \cite{Charles-Lefeuvre-24} for instance). The previous theorem is illustrated in Figure \ref{figure:lambda1}. The curvature condition appearing in Item (ii) is needed as one can otherwise construct counterexamples for arbitrary Abelian extensions, see \S\ref{ssection:not-hypo} below. It is the same kind of condition which appears in the statement of Theorem \ref{theorem:main2}, Item (i). As for frame flows (see Chapter \ref{chapter:flow}), the discrepancy between Abelian and semisimple Lie groups lies in the fact that semisimple Lie groups satisfy a Diophantine property (see \S\ref{ssection:diophantine}) and this turns out to be crucial in the proof. The main consequence of the lower bound \eqref{equation:lower-bound-ss} is that it guarantees global hypoellipticity of the operator $\Delta_{\mathbf{k}}$, see Theorem \ref{theorem:hypoellipticity} below. When $G$ is semisimple, the first condition in Item (ii) of the preceding theorem is automatically satisfied, and we immediately obtain the following corollary:
\begin{corollary}
	Assume that $G$ is semisimple and that the holonomy group $\mathrm{Hol}(P,\nabla)$ is dense in $G$. Then \eqref{equation:lower-bound-ss} holds.
\end{corollary} 

\begin{figure}[htbp!]
\centering 
\includegraphics{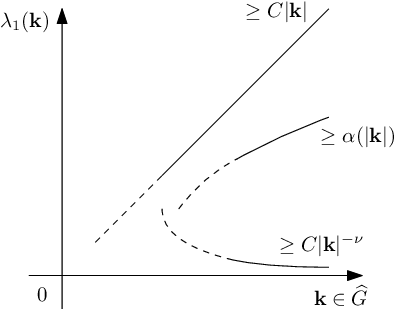}
\caption{The upper curve corresponds to the case of locally non-degenerate curvature, Theorem \ref{theorem:lambda1}, Item (i). In the middle curve, $\alpha \in C^\infty(\R_{\geq 0}, [C_0,\infty))$ is a smooth non-decreasing function such that $\alpha(x) \to_{x \to \infty} \infty$ and $C_0 > 0$; it corresponds to Theorem \ref{theorem:hypoellipticity}, Item (i) (below) and implies that $\Delta_{\HH}$ has compact resolvent on $L^2(P)$ with discrete spectrum. The lower curve corresponds to Theorem \ref{theorem:lambda1}, Item (ii), and Theorem \ref{theorem:hypoellipticity}, Item (ii); it implies that $\Delta_{\HH}$ is globally hypoelliptic and is verified when $G$ is semisimple with dense holonomy group.}
\label{figure:lambda1}
\end{figure}

 \subsection{Discrete spectrum and global hypoellipticity}

 \subsubsection{Motivation: Hörmander's hypoelliptic theorem}
 
 Let $M$ be a smooth closed (oriented) $n$-dimensional manifold equipped with a volume form $\mu$. Consider a family $X_0, X_1, \dotsc, X_d \in C^\infty(M,TM)$ of smooth vector fields and $c \in C^\infty(M)$. Define the sum of squares operator
 \begin{equation}
 \label{equation:l}
 L = \sum_{i=1}^d X_i^* X_i + X_0 + c.
 \end{equation}
Recall that the operator $L$ is (globally) \emph{hypoelliptic} if all solutions $u \in \mc{D}'(M)$ to the equation $Lu = f$, with $f \in C^\infty(M)$, are smooth, that is $u \in C^\infty(M)$. Following standard terminology, we say that the $X_i$'s satisfy the \emph{bracket condition} if for every $x \in M$, the iterated brackets 
\[
\left\{[X_{i_1},[X_{i_2}, \dotsc]] \mid i_1, \dotsc, i_k \in \{0, 1, \dotsc, d\}, k \in \Z_{\geq 0}\right\}
\]
span the tangent space $T_xM$. We first recall the standard hypoelliptic theorem due to Hörmander \cite{Hormander-67}:

\begin{theorem}[Hörmander]
Assume that the vector fields $X_0, X_1, \dotsc, X_d$ satisfy the bracket condition. Then the operator $L$ is hypoelliptic.
\end{theorem}

Following Hörmander's seminal work, we are interested in understanding when a general operator $L$ of the form \eqref{equation:l} is hypoelliptic. The bracket condition should be interpreted as saying that the \emph{local holonomy group} generated by the vector fields $X_1, \dotsc, X_d$ is full. In other words, starting from an arbitrary point $x \in M$, one can reach any other point in its neighborhood (hence in $M$ by compactness) by a concatenation of flowlines of the vector fields $X_i$. This statement turns out not to be always verifiable in practice: the bracket condition may either fail at some points in $M$, or the vector fields themselves may not be sufficiently explicit (or regular) to be able to compute the Lie brackets. Motivated by questions in dynamical systems such as ``Are essentially accessible systems ergodic?" (see \cite{Pugh-Shub-96, Wilkinson-11}), we are looking for global conditions that could entail hypoellipticity of $L$. For that, set
\[
H := \{ e^{t_{i_1}X_{i_1}} \circ \dotsb \circ e^{t_{i_m}X_{i_m}} ~|~ m \in \Z_{\geq 0}, 0 \leq i_j \leq d, t_{i_j} \in \R \} \subset \mathrm{Diff}^0(M),
\]
where $\mathrm{Diff}^0(M)$ denotes the space of smooth diffeomorphisms of $M$ isotopic to the identity (here, $e^{tY}$ denotes the flow of a vector field $Y$ at time $t$). We are interested in the following question: \\

\noindent \emph{Question:} Suppose that $H$ admits a dense orbit on $M$. What extra assumptions need to be made to ensure that $L$ is hypoelliptic? \\

That $H$ has a dense orbit means that there exists $x \in M$ such that its orbit $Hx \subset M$ is dense. That $H$ admits a dense orbit on $M$ cannot imply that $L$ is hypoelliptic in full generality as we shall see in a simple example in \S\ref{ssection:not-hypo}. We also emphasize that full transitivity (i.e. the property that $Hx = M$ for some $x \in M$) is conjectured to yield hypoellipticity \cite{Omori-91,Omori-Kobayashi-99}. Nonetheless, we believe that this is too strong of an assumption in all practical purposes related to dynamical systems.

\begin{remark} The horizontal Laplacian $\Delta_{\HH}$ introduced previously can be written in the form \eqref{equation:l} with $X_0, X_1, \dotsc, X_d$ horizontal lifts of vector fields on $M$ and $c = 0$. Indeed, locally we may write $\Delta_{\HH} = \sum_{i = 1}^n (\e_i^{\HH})^* \e_i^{\HH}$, where $(\e_i)_{i = 1}^n$ is a local orthonormal frame of $TM$, whose horizontal lift to $TP$ is $(\e_i^{\HH})_{i = 1}^n$. The claim (with possibly $d > n$) then follows by using a partition of unity. In particular, we see that the bracket condition is related to non-integrability of $\HH$, i.e. to brackets of horizontal vector fields (and in turn, to the curvature of the connection on $P$). Nevertheless, we will establish hypoellipticity also in the integrable case (i.e. when the connection is flat), see Theorem \ref{theorem:hypoellipticity} and Example \ref{example:su2} below.
\end{remark}

 \subsubsection{Statement} We will establish the following result. We assume the notation from \S \ref{sec:horizontal-lapl-intro}.


 \begin{theorem}
 \label{theorem:hypoellipticity}
  The following holds: 
 \begin{enumerate}[label=\emph{(\roman*)}]
 \item Assume that $\lambda_1(\Delta_{\mathbf{k}}) \to_{|\mathbf{k}|\to\infty} +\infty$. Then $\Delta_{\HH}$ has compact resolvent and discrete spectrum.
 \item Assume that $\lambda_1(\Delta_{\mathbf{k}}) \geq C|\mathbf{k}|^{-\nu}$ for some $C, \nu > 0$. Then $\Delta_{\HH}$ is hypoelliptic.
\end{enumerate}
 In particular, if Item \emph{(ii)} of Theorem \ref{theorem:lambda1} is satisfied (e.g. if $G$ is semisimple and the holonomy group is dense in $G$), then $\Delta_{\HH}$ is hypoelliptic.
 \end{theorem}
 
 The condition in Item (i) is satisfied if the curvature is globally non-degenerate by Theorem \ref{theorem:lambda1}, Item (i). Nevertheless, there exist cases where $\lambda_1(\Delta_{\mathbf{k}})$ is bounded from below by a function growing to infinity slower than linearly in $|\mathbf{k}|$. For instance, if $G = \mathrm{U}(1)$ (magnetic case) and the curvature vanishes to order $r \geq 0$ on a hypersurface in $M$, it was established in \cite[Theorem 1.2]{Helffer-Kordyukov-09}, that $\lambda_1(\Delta_{\mathbf{k}}) \gtrsim |\mathbf{k}|^{2/(r+2)}$ (for $r=0$, one retrieves the linear bound).
Regarding Item (ii), we point out that Omori \cite{Omori-91} had established hypoellipticity of $\Delta_{\HH}$ under the assumption that the holonomy group is \emph{equal} to $G$, which is a much stronger assumption (in particular, this is never satisfied for flat bundles).

Examples of flat $G$-principal bundles with dense holonomy group (but not equal to $G$) can be easily constructed as follows (see also Example \ref{example:flat-connection}), using representations of the fundamental group $\rho : \pi_1(M) \to G$ with dense image $\rho(\pi_1(M)) \leqslant G$.
 
 \begin{example}\label{example:su2}
 Let $\Sigma$ be a genus $2$ closed surface. Recall that (here, $[\bullet, \bullet]$ denotes the commutator)
 \[
 \pi_1(\Sigma) = \langle a_1, b_1, a_2, b_2 ~|~ [a_1, b_1][a_2, b_2] = 1\rangle.
 \]
 Then send $b_1, b_2$ to the neutral element in $\mathrm{SU}(2)$ and $a_1, a_2$ to the generators of a dense subgroup in $\mathrm{SU}(2)$. This defines a representation $\rho : \pi_1(\Sigma) \to \mathrm{SU}(2)$, hence an $\mathrm{SU}(2)$-principal bundle over $\Sigma$ with flat structure and dense holonomy group given by the image $\rho(\pi_1(\Sigma))$ of the representation. Theorem \ref{theorem:hypoellipticity} applies in this case.
 
More generally, if $G$ is a compact Lie group, then for any $g \geq \mathrm{dim}(G)-\mathrm{rk}(G)$, one can construct on $\Sigma_g$, the oriented surface of genus $g$, a flat $G$-bundle with dense holonomy group. Indeed, $\pi_1(\Sigma_g)$ surjects onto the free group $\mathbb{F}_{g}$ of rank $g$. It then suffices to map the $g$ generators $a_1, \dotsc, a_g$ of $\mathbb{F}_{g}$ to elements $\rho(a_1), \dotsc,\rho(a_g) \in G$ which generate a dense subgroup. This is possible as soon as $g \geq \mathrm{dim}(G) - \mathrm{rk}(G) + 1$. (To see that there is a set of cardinality $\mathrm{dim}(G) - \mathrm{rk}(G) + 1$ generating a dense subgroup of $G$, roughly speaking it suffices to take an element generating a dense subset of $T$, and then suitably inductively add elements generating the semisimple part. We did not try to optimise this constant.)
 \end{example}
 
 We next give a complementary example where $\Delta_{\HH}$ can be explicitly computed and $G$ is not semisimple, so we are in the setting of Theorem \ref{theorem:hypoellipticity}, Item (ii), and we can understand the density of the holonomy group.

\begin{example}\label{example:surface1}
	Assume $(M, g)$ is a closed Riemannian surface, and let $\pi: P = SM \to M$ be its unit sphere bundle. On $P$, it is well known that there is a canonical orthonormal frame \cite{Merry-Paternain-11, Paternain-Salo-Uhlmann-23}: let $X$ be the geodesic vector field, let $H(x, v)$ be the horizontal lift the rotation $iv$ by $\frac{\pi}{2}$ in the positive direction of $v$, and let $V$ be the generator of the rotations $(R_\theta)_{\theta \in [0, 2\pi)}$ in the fibres. We let $\alpha$, $\beta$, and $\psi$ denote the dual $1$-forms: $\alpha(X) = 1$, $\alpha(H) = \alpha(V) = 0$, and similarly for $\beta$ and $\psi$. Also, we have $d\psi = -K_g \alpha \wedge \beta$, where $K_g$ denotes the Gaussian curvature of $(M, g)$, and $\alpha \wedge \beta = \pi^*(d\vol_g)$, where $d\vol_g$ is the volume form on $(M, g)$.  
	We then see immediately that
	\[
		d_{\HH}u = du|_{\HH} = Xu\, \alpha + Hu\, \beta, \quad d_{\HH}^*(a\alpha + b \beta) = - (Xa + Hb), \quad u, a, b \in C^\infty(P).
	\]
	Then $\Delta_{\HH} = d_{\HH}^* d_{\HH} = -(X^2 + H^2)$, and we see that H\"ormander's bracket condition is satisfied at the points where $K_g$ does not vanish. 
	
	Assume that $K_g$ is not identically zero. Then we claim that the conditions of Theorem \ref{theorem:lambda1}, Item (ii), are met, and so $\Delta_{\HH}$ is hypoelliptic. Indeed, it is easily checked that $F_1 = K_g d\vol_g$, and so $F_1$ is not identically zero. Also, by the Ambrose-Singer theorem, at a point of non-vanishing curvature the local holonomy group is full, proving the claim. (And so in this particular case the result of \cite{Omori-91} also applies.) In the case where $K_g \equiv 0$, we are on a flat torus and the holonomy group is trivial; however in \S \ref{ssection:not-hypo} we will see an example of a flat circle bundle with dense holonomy whose $\Delta_{\HH}$ is not hypoelliptic.
	
	We remark that when $K_g$ is nowhere zero, conditions of Theorem \ref{theorem:lambda1}, Item (i), are met, so by Theorem \ref{theorem:hypoellipticity}, Item (i), we get that $\Delta_{\HH} = -(X^2 + H^2)$ has compact resolvent and discrete spectrum, and its spectrum by Remark \ref{remark:horizontal-laplacian} is determined by the spectra of $(\nabla_k)^* \nabla_k$ acting on the powers $\mathcal{K}^{\otimes k}$  of the canonical bundle (see also Example \ref{example:surface} below).
\end{example}
 
 \subsection{Quantum ergodicity}

We now discuss a quantum ergodicity statement. 
We may then introduce the $\Omega \subset \widehat{G} \times [0,\infty)$ as
 \begin{equation}
 \label{equation:omega-qe}
\Omega = \{(\mathbf{k},\lambda) \in \widehat{G} \times [0,\infty) ~|~\exists u_{\mathbf{k},\lambda} \in C^\infty_{\mathrm{hol}}(F,\mathbf{L}^{\otimes \mathbf{k}}), \|u_{\mathbf{k},\, \lambda}\|_{L^2}=1,\,  \Delta_{\mathbf{k}} u_{\mathbf{k},\lambda} = \lambda^2 u_{\mathbf{k},\lambda}\}.
 \end{equation}
Eigenvalues are counted with multiplicity in $\Omega$. The $L^2$-normalised eigenstate corresponding to $(\mathbf{k},\lambda) \in \Omega$ will be denoted by $u_{\mathbf{k},\lambda}$.

 In the following, $B(0,R)$ is the unit ball for the Euclidean norm in
 \[
 \widehat{G} \times [0,\infty) \subset \mathfrak{a} \times \R \simeq \R^{d+1},
 \]
 where $d$ is the rank of $G$. We shall prove the following:
 
 \begin{theorem}[Flat case]
 \label{theorem:quantum-ergodicity}
 Let $(M,g)$ be a Riemannian manifold with Anosov geodesic flow, and let $P \to M$ be a flat $G$-principal bundle with dense holonomy group. Then there exists a density $1$ subset $\Lambda \subset \Omega$, that is we have
 \[
 \dfrac{\sharp (\Lambda \cap B(0,R))}{\sharp(\Omega \cap B(0,R))} \to_{R \to \infty} 1,
 \]
 such that for all sequences $(\mathbf{k}_j,\lambda_j)_{j \geq 0} \in \Lambda^{\Z_{\geq 0}}$, one has: for all $a \in C^\infty(F)$,
 \begin{equation}
 \label{equation:bafi}
\langle a u_{\mathbf{k}_j,\lambda_j},u_{\mathbf{k}_j,\lambda_j}\rangle_{L^2(F,\mathbf{L}^{\otimes \mathbf{k}_j})} \to_{j \to \infty} \dfrac{1}{\vol(F)} \int_F a(w)\, \dd w,
 \end{equation}
 where $\dd w$ stands for the Riemannian measure on $F$.
 \end{theorem}
 
 As in the standard case (Laplacian on functions), the ergodicity of a certain Hamiltonian flow is responsible for \eqref{equation:bafi}. This flow can be described as follows. The (semiclassical) principal symbol of $h^2\Delta_{\mathbf{k}}$ is defined on $\HH^*_F \to F$ and given by $p(x,\xi) := |\xi|^2$. The (twisted) symplectic form (on $T^*F$) was defined in \eqref{equation:omega-twisted} and is given by $\omega_{h,\mathbf{k}} = \omega_0 + i h\mathbf{k} \cdot \pi^*\mathbf{F}_{\overline{\nabla}}$, where $\pi : T^*F \to F$ is the projection and $\mathbf{F}_{\overline{\nabla}}$ is the curvature on $F$ of the corresponding (multi) line bundle $\mathbf{L} \to F$. Here, the connection is flat so $\omega_{h,\mathbf{k}} = \omega_0$ (see \ref{equation:courbure-horizontale}). The induced flow $(\Phi_t)_{t \in \mathbb{R}}$ on $\HH^*_F$ is then Hamiltonian flow of $p$ computed with respect to $\omega_0$. We shall see that under the assumptions of the theorem, this flow is ergodic on the energy layers $S_E\HH^* = \{\xi \in \HH^* \mid |\xi|=E\}$ for $E \neq 0$ (see Lemma \ref{lemma:ergodic-layer}). 
 
 We emphasise that Theorem \ref{theorem:quantum-ergodicity} also holds if we only assume that the flow $(\Phi_t)_{t \in \mathbb{R}}$ on $S_1\HH^*$ is ergodic; we decided to state the theorem this way for simplicity. This more general result is also contained in Theorem \ref{theorem:quantum-ergodicity2}, as flat connections are $0$-admissible (see below for details). That this result is strictly more general is shown in Example \ref{example:surface3} below.
 
We now discuss the non-flat case. In what follows $P$ is a fixed $G$-bundle over $M$ and $G$ is a semisimple Lie group; we denote by $\mathbf{A}_P$ the space of all $G$-connections on $P$ modulo gauge equivalence. The subspace $\mathbf{A}_P^{\mathrm{flat}} \subset \mathbf{A}_P$ of all flat equivariant connections on $P$ (modulo gauge equivalence) is a finite dimensional algebraic variety (see for instance \cite[Chapter 5]{Labourie-13} in the surface case and \cite{Maret-22} for a more extensive discussion). It can be verified that the set of connections with dense holonomy group is open in $\mathbf{A}_P^{\mathrm{flat}}$. For such a connection, the Hamiltonian vector field $H^{\omega_{h,\mathbf{k}}}_p$ on the unit ball $S_1\HH^*_F := \{ (w,\xi) \in \HH^*_F ~|~ |\xi|=1\}$ generates an ergodic flow, as mentioned in the previous paragraph. It turns out that this property remains true for \emph{all} nearby volume preserving vector fields (in the $C^1$ topology) by \emph{stable ergodicity}, see \cite[Theorem B]{Burns-Wilkinson-99}. We note that the theorem is established there for Anosov maps but the proof goes through in the flow case, as indicated to us by the authors in private communication. (The semisimplicity assumption on $G$ is also needed for stable ergodicity to hold.) As a consequence, starting with a flat connection (with dense holonomy group) and perturbing the (twisted) symplectic form $\omega_{h,\mathbf{k}} = \omega_0 + ih\mathbf{k} \cdot \pi^*\mathbf{F}_{\overline{\nabla}}$ with $\mathbf{F}_{\overline{\nabla}} \neq 0$ but small enough, we can still obtain an ergodic Hamiltonian flow $(\Phi_t^{\omega_{h,\mathbf{k}}})_{t \in \R}$ on a range of energy shells in $\HH^*$ (here $h|\mathbf{k}|\leq 1$). In the following, we set $\omega(y) := \omega_0 + iy \cdot \mathbf{F}_{\overline{\nabla}}$ for $y \in \mathfrak{a}_+$ (positive Weyl chamber) with norm $|y| \leq 1$. We introduce the following terminology:

\begin{definition}\label{def:admissible-connection}
Let $\nabla \in \mathbf{A}_P$ be a $G$-equivariant connection, and let $\eps \geq 0$. We say that the connection is $\eps$-admissible if for all $y \in \mathfrak{a}_+$ with norm $|y| \leq 1$, for all $\eps < E \leq 1$, the Hamiltonian vector field $H^{\omega(y)}_p$ generates an ergodic flow on the energy layer $S_E\HH^*_F := \{ (x,\xi) \in \HH^*_F ~|~ |\xi|=E\}$.
\end{definition}

As mentioned above, for a fixed $\eps > 0$, $\eps$-admissible connections can be obtained on Anosov manifolds by simply perturbing a flat connection (in $C^\infty$, on a principal bundle with semisimple structure group) with dense holonomy group and using the stable ergodicity statement of \cite[Theorem B]{Burns-Wilkinson-99}. Note that for any $\eps > 0$, there is an infinite dimensional space of such connections. In the flat case we have some homogeneity with respect to the energy layers, and Lemma \ref{lemma:ergodic-layer} asserts that one can take $\eps = 0$. We do not know if there exist non-flat connections for which one can take $\eps =0$, but some evidence in Example \ref{example:surface} below is given that there might be none; in the same example the vector fields $H^{\omega(y)}_p$ are expressed explicitly in terms of geometric data. The quantum ergodicity statement we obtain in this case is the following:

 \begin{theorem}
 \label{theorem:quantum-ergodicity2}
 Let $(M,g)$ be a Riemannian manifold, and let $P \to M$ be $G$-principal bundle equipped with an $\eps$-admissible connection for some $\eps > 0$. Then there exists a subset $\Lambda \subset \Omega$ such that
 \[
 \dfrac{\sharp (\Lambda \cap B(0,R) \cap \{\lambda \geq \eps R\})}{\sharp(\Omega \cap B(0,R)\cap \{\lambda \geq \eps R\})} \to_{R \to \infty} 1,
 \]
 such that for all sequences $(\mathbf{k}_j,\lambda_j)_{j \geq 0} \in \Lambda^{\Z_{\geq 0}}$, one has: for all $a \in C^\infty(F)$,
 \begin{equation}
 \label{equation:bafi-non-flat}
\langle a u_{\mathbf{k}_j,\lambda_j},u_{\mathbf{k}_j,\lambda_j}\rangle_{L^2(F,\mathbf{L}^{\otimes \mathbf{k}_j})} \to_{j \to \infty} \dfrac{1}{\vol(F)} \int_F a(w)\, \dd w,
 \end{equation}
 where $\dd w$ stands for the Riemannian measure on $F$.
 \end{theorem}

We conclude this section with a discussion of previous results. See \S\ref{ssection:weyl} for further perspectives on this result. In the flat case (Theorem \ref{theorem:quantum-ergodicity}), a similar result was established recently and independently in \cite[Theorems 1.4 and 4.9]{Ma-Ma-23}. We note that their result seems to be complementary to ours as it features a notion of uniformity in the parameter $\mathbf{k} \in \widehat{G}$, but only deals with flat connections; in other words, they are allowed to take $\Lambda \subset \Omega$ which has uniform density one for a set of weights of the form $\{p\mathbf{k} \mid p \in \mathbb{Z}_{\geq 0}\}$, where $\mathbf{k} \in \widehat{G}$ is fixed (so \emph{not} considering all $\mathbf{k} \in \widehat{G}$ simultaneously). Their approach is based on a fibrewise Toeplitz quantisation which bears similarities with our approach. We also refer to \cite{Bismut-Ma-Zhang-17, Ma-24, Puchol-23} for some very limited similarity in approach in the context of Toeplitz operators and analytic torsion.

In the non-flat case (Theorem \ref{theorem:quantum-ergodicity2}), very recently, \cite{Ben-Ovadia-Ma-Rogdriguez-Hertz-24} prove a similar result for perturbations, featuring a notion of uniformity as in the previous paragraph, and also using a form of stable ergodicity in the proof.

We next give a remark about the spectrum of the horizontal Laplacian.
 
 \begin{remark}\label{remark:horizontal-laplacian}
The point spectrum of $\Delta_{\HH}$ can be defined as 
\[
	\spec(\Delta_{\HH}) := \{\lambda \in \mathbb{R}_{\geq 0} \mid \exists u \in L^2(P),\,u \neq 0,\, \Delta_{\HH}u = \lambda^2 u\},
\]
counted with multiplicity. Using the Fourier transform on $P$ (see \eqref{eq:ft-bw-isomorphism}), as well as Lemma \ref{lemma:ft-horizontal-laplacian}, we see that
\[
	\spec(\Delta_{\HH}) = \cup_{\mathbf{k} \in \widehat{G}} \cup_{i = 1}^{d_{\mathbf{k}}} \spec(\Delta_{\mathbf{k}, i}),
\]
where $\spec(\Delta_{\mathbf{k}, i})$ denotes the (discrete) spectrum of $\Delta_{\mathbf{k}, i} = \Delta_{\mathbf{k}}$ acting on $L^2_{\mathrm{hol}}(F, \Lk)$
\[
	\spec(\Delta_{\mathbf{k}, i}) := \{\lambda \in \mathbb{R}_{\geq 0} \mid \exists u \in C^\infty_{\mathrm{hol}}(F, \Lk),\,u \neq 0,\, \Delta_{\mathbf{k}} u = \lambda^2 u\}.
\]
In particular, we see that the pair $(\mathbf{k}, \lambda)$ in $\Omega$ (with multiplicity) is counted $d_{\mathbf{k}}$ times in $\spec(\Delta_{\HH})$. Therefore the spectrum $\spec(\Delta_{\HH})$ and eigenfunctions of $\Delta_{\HH}$ are uniquely determined from $\Omega$ and eigenfunctions of $\Delta_{\mathbf{k}}$, and vice versa.
\end{remark}

We finish this section by two examples. We first give a geometrically appealing example (unit tangent bundle of a surface) where the condition of $\varepsilon$-admissibility appearing in Theorem \ref{theorem:quantum-ergodicity2} is further clarified and some evidence towards its optimality is given. In particular, we explicitly compute the Hamiltonian flows appearing in the problem in terms of geometric data.

\begin{example}\label{example:surface}	
	
	Assume we are in the setting of Example \ref{example:surface1}. As mentioned earlier in \S \ref{ssection:u1}, the Borel-Weil theory for $P$ is simple. The irreducible representations of $\mathbb{S}^1$ are classified by integer $k \in \mathbb{Z}$, and are given by $\rho_k(z) = z^k$, $\rho_k: \mathbb{S}^1 \to \mathrm{U}(1)$. The flag manifold $F = P/T = M$ is trivial in this case, and the line bundle associated to $P$ via $\rho_k$ is $\mathcal{K}^{\otimes k}$, the $k$th power of the canonical line bundle, $\mathcal{K} = T^*M^{1, 0}$ (i.e. the bundle locally spanned by $dz$, for a complex coordinate $z$). (This is easily verified by computing transition functions for $\mc{K}$ and $SM$ in local isothermal coordinates.) 
	
	 The Levi-Civita connection $\nabla^{\mathrm{LC}}$ on $P$ is given by the $1$-form $\psi$; indeed its kernel is the horizontal space, and it satisfies $R_\theta^* \psi = \psi$ so it is equivariant. (Note that here $\mathbb{S}^1$ is identified with $\mathbb{R}/(2\pi \mathbb{Z})$.) The connection $\nabla_k$ associated to $\nabla^{\mathrm{LC}}$ on $\mathcal{K}^{\otimes k}$ is then simply the Levi-Civita connection on $\mathcal{K}^{\otimes k}$. Note that $\nabla_k$ is an actual connection, so there is no need to add the vertical Chern connection as in \S \ref{section:connection-principal-bundles} and \eqref{equation:dynamical-lambda}. Its curvature is computed using \eqref{equation:courbure-horizontale} to be 
	 \[
	 	F_{\nabla_k}(v, iv) = ik d\psi(X(v), H(v)) = -ik K_g d\vol_g(v, iv), \quad v \in SM,
	 \] 
	 where $d\vol_g$ is the volume form of $(M, g)$ (here $\alpha_k$ in \eqref{equation:courbure-horizontale} is simply the multiplication by $ik$). Thus $F_{\nabla_k} = -ik K_g d\vol_g$. Hence the twisted symplectic form $\omega_{h, k}$ on $T^*M$ introduced in \eqref{equation:omega-twisted} is given by
	 \[
	 	\omega_{h, k} = \omega_0 + ihk \pi^*F_{\nabla_1} = \omega_0 + h k \pi^*(K_g d\vol_g),
	 \]
	 where $\omega_0$ is the canonical symplectic form on $\pi: T^*M \to M$. Note that $\alpha$ extends to a $1$-form on $TM$ by setting 
	 \[
	 	\alpha(x, v)(\xi) := g_x(v, d\pi(x, v) \xi), \quad (x, v) \in TM,\quad \xi \in T_{(x, v)}TM,
	 \] 
	 where by slight abuse of notation we write $\pi$ also for the projection $TM \to M$. Also, the geodesic vector field $X$ extends naturally to $TM$, and so does the vertical vector field $V$. Using the musical isomorphism $\hat{g}: TM \to T^*M$, it is easy to see that $\hat{g}^* \omega_0 = -d\alpha$. Using \cite[Section 1.3 and Lemma 7.7]{Merry-Paternain-11}, the Hamiltonian vector field of $\frac{1}{2}|\xi|^2$ with respect to $\hat{g}^* \omega_{h, k}$ on $TM$ is
	 \[
	 		H = X + hk K_g V.
	 \]
	 For $r > 0$, denote by $SM_r := \{(x, v) \in TM \mid |v|_g = r\}$ the bundle of tangent vectors of length $r$, and by $\ell_r: SM \to SM_r, v \mapsto rv$ the re-scaling map. Then it can be checked that $\ell_r^*X = rX$ and $\ell_r^*V = V$. Write 
	 \[
	 	H_{r, y} := \frac{1}{r} \ell_r^*(X + y K_g V) = X + \frac{y K_g}{r} V, \quad y \in \mathbb{R}, r > 0,
	 \]
	 for the family of vector fields on $SM$. The flows generated by such vector fields (of the form $X + \lambda V$ where $\lambda$ is a pullback of a function on $M$) are called \emph{magnetic} because they model trajectories of charged particles influenced by a magnetic field.
	 
	  In Definition \ref{def:admissible-connection}, recall that we introduced the notion of $\varepsilon$-admissibility of a connection; by definition, the Levi-Civita connection on $P$ is $\varepsilon$-admissible if $H_{r, y}$ is ergodic for all $|y| \leq 1$ and $r \in (\varepsilon, 1]$. In particular, when $(M, g)$ is hyperbolic, i.e. $K_g \equiv -1$, $H_{r, y} = X - \frac{y}{r}V$, we see that $H_{r, y}$ is not $\varepsilon$-admissible for any $\varepsilon < 1$. (This is because it is well-known that for hyperbolic surfaces $X + cV$ for $c$ a constant generates an ergodic flow if and only if $|c| \leq 1$.) 
	 
	 However, in general, addressing the question of ergodicity for a range of $(y, r)$ is known to be difficult. For instance, there are examples for general $\lambda$ in $X + \lambda V$ where one can enter and leave the space of Anosov flows an arbitrary number of times by re-scaling $\lambda$ with a constant  \cite{Burns-Paternain-02}; it is also known that $H_{r, y}$ is \emph{not} Anosov for $\frac{y}{r}$ large enough \cite{Paternain-Paternain-96}. It is expected that ergodicity of $H_{r, y}$ also fails for $\frac{y}{r}$ large enough, and so Theorem \ref{theorem:quantum-ergodicity2} is in this sense optimal.
\end{example}

Finally, we give an example of a flat bundle where the conclusion of Theorem \ref{theorem:quantum-ergodicity} applies, under the mentioned weaker dynamical assumption.

\begin{example}\label{example:surface3}
	Let $(M, g)$ be a Riemannian surface and $P = M \times \mathbb{S}^1$. Write $\psi$ for the canonical $1$-form on the $\mathbb{S}^1$ factor, and let $\eta$ be a $1$-form on $M$. Then $\psi + \eta$ is a connection $1$-form on $P$, whose curvature is $d\eta$. Note that the line bundles in Borel-Weil theory are powers of $L = M \times \mathbb{C}$. Similarly to Example \ref{example:surface}, $\omega_{h, k} = \omega_0 - hk \pi^*d\eta$ is the twisted symplectic form.
	
	 Assume now the connection is flat, i.e. $d\eta = 0$. Then the flow Hamiltonian flow of $\frac{1}{2}|\xi|^2$ with respect to $\omega_0$ is just the geodesic flow. The conclusions of Theorem \ref{theorem:quantum-ergodicity} thus hold for the magnetic Laplacian (twisted with a closed $1$-form $\eta$), on surfaces with ergodic geodesic flow (which exist e.g. on the torus $\mathbb{T}^2$, see \cite[Theorem 1]{Donnay-88}).
\end{example}

 \section{Bottom of the spectrum}

 \subsection{Lower bound on $\lambda_1$ for globally non-degenerate curvature}
 


In this paragraph, we prove Theorem \ref{theorem:lambda1}, Item (i). We first claim that it reduces to the following statement:

\begin{proposition}
\label{proposition:lower-bound}
Let $\mathbf{l} \in \partial_\infty \mathfrak{a}_+$ such that $F_{\mathrm{min}}(\mathbf{l}) > 0$. Then for all $\eps > 0$, there exists $R > 0$ such that
\begin{equation}
\label{equation:open-set}
\lambda_1(\Delta_{\mathbf{k}}) > (F_{\mathrm{min}}(\mathbf{l})/2-\eps)|\mathbf{k}|, \qquad \forall \mathbf{k} \in \widehat{G}, |\mathbf{k}| > R, \left|\mathbf{k}/|\mathbf{k}|-\mathbf{l}\right| < 1/R.
\end{equation}
\end{proposition}

Let us establish Theorem \ref{theorem:lambda1}, Item (i), from the previous proposition.

\begin{proof}[Proof of Theorem \ref{theorem:lambda1}, Item (i)]
The topology on the compactified Weyl chamber $\overline{\mathfrak{a}_+} = \mathfrak{a}_+ \sqcup \partial_\infty \mathfrak{a}_+$ defined in \S\ref{sssection:compactification} makes it a compact topological space. The condition on $\mathbf{k}$ in \eqref{equation:open-set} defines an open neighborhood of the point $\mathbf{l} \in \partial_\infty \mathfrak{a}_+$ for the topology defined in $\overline{\mathfrak{a}_+}$. Since $\partial_\infty \mathfrak{a}_+ \subset \mathfrak{a}_+$ is compact, it can be covered by a finite number of open subsets such as the ones defined in \eqref{equation:open-set}; the conclusion is therefore immediate.
\end{proof}
 
 The rest of this paragraph is devoted to the proof of Proposition \ref{proposition:lower-bound}. Since $\pi : F \to M$ is a Riemannian submersion, any vector field $X \in C^\infty(M,TM)$ can be lifted in a unique way to a horizontal vector field $X^{\HH_F} \in C^\infty(F,\HH_F)$ such that $d\pi(X^{\HH_F}) = X$. Following standard terminology, we call these vector fields horizontal basic vector fields. The following holds.

\begin{lemma}
Fix $C > 0$, $\chi \in C^\infty(F,[0,1])$. Let $X^{\HH_F},Y^{\HH_F} \in C^\infty(F,\HH_F)$ be horizontal basic vector fields such that $|X|,|Y| \leq 1$ on $M$, and set $\X_{\mathbf{k}} := \iota_{X^{\HH_F}} \nabla_{\mathbf{k}}, \mathbf{Y}_{\mathbf{k}}=\iota_{Y^{\HH_F}} \nabla_{\mathbf{k}}$. Let $(u_{\mathbf{k}})_{\mathbf{k} \in \widehat{G}}$ be a sequence of sections $u_{\mathbf{k}} \in C^\infty_{\mathrm{hol}}(F,\mathbf{L}^{\otimes \mathbf{k}})$ such that
\begin{equation}
\label{equation:assumption-delta-k}
\langle \Delta_{\mathbf{k}} u_{\mathbf{k}}, u_{\mathbf{k}}\rangle_{L^2} \leq C|\mathbf{k}|, \qquad \|u_{\mathbf{k}}\|_{L^2}=1.
\end{equation}
Then the following equality holds:
\begin{equation}
\label{equation:ima}
\begin{split}
2 \Im \langle \X_{\mathbf{k}} u_{\mathbf{k}}, \chi \mathbf{Y}_{\mathbf{k}} u_{\mathbf{k}}\rangle_{L^2} = -i \langle \mathbf{k} \cdot \mathbf{F}_{\overline{\nabla}}(X^{\HH_F},Y^{\HH_F}) u_{\mathbf{k}}, \chi u_{\mathbf{k}}\rangle_{L^2} + \mc{O}_{C,X,Y,\chi}(|\mathbf{k}|^{1/2}).
\end{split}
\end{equation}
\end{lemma}

\begin{proof}
We start with the identity:
\[
\X_{\mathbf{k}} \mathbf{Y}_{\mathbf{k}} - \mathbf{Y}_{\mathbf{k}} \X_{\mathbf{k}} = \mathbf{k} \cdot \mathbf{F}_{\overline{\nabla}}(X^{\HH_F},Y^{\HH_F}) + (\overline{\nabla}_{\mathbf{k}})_{[X^{\HH_F}, Y^{\HH_F}]}.
\]
Applying this identity to $u_{\mathbf{k}}$, pairing against $\chi u_{\mathbf{k}}$, and using $\X_{\mathbf{k}}^*=-\X_{\mathbf{k}} -\Div(X^{\HH_F})$, we obtain:
\begin{equation}
\label{equation:long}
\begin{split}
& \langle \X_{\mathbf{k}} u_{\mathbf{k}}, \chi  \mathbf{Y}_{\mathbf{k}} u_{\mathbf{k}}\rangle_{L^2} - \langle \mathbf{Y}_{\mathbf{k}} u_{\mathbf{k}}, \chi \X_{\mathbf{k}} u_{\mathbf{k}}\rangle_{L^2} \\
& + \langle \X_{\mathbf{k}} u_{\mathbf{k}}, (Y^{\HH_F}\chi) u_{\mathbf{k}}\rangle_{L^2} - \langle \mathbf{Y}_{\mathbf{k}} u_{\mathbf{k}}, (X^{\HH_F}\chi) u_{\mathbf{k}}\rangle_{L^2} \\
& +  \langle \X_{\mathbf{k}} u_{\mathbf{k}}, \Div(Y^{\HH_F}) \chi u_{\mathbf{k}}\rangle_{L^2}- \langle \mathbf{Y}_{\mathbf{k}} u_{\mathbf{k}}, \Div(X^{\HH_F}) \chi u_{\mathbf{k}}\rangle_{L^2} \\
& \qquad = \langle\mathbf{k} \cdot \mathbf{F}_{\overline{\nabla}}(X^{\HH_F},Y^{\HH_F}) u_{\mathbf{k}}, \chi u_{\mathbf{k}}\rangle_{L^2}+ \langle (\nabla_{\mathbf{k}})_{[X, Y]^{\HH_F}}u_{\mathbf{k}},\chi u_{\mathbf{k}}\rangle_{L^2}\\
&\qquad \qquad + \langle (D_{\mathbf{k}}^{\mathrm{Chern}})_{[X^{\HH_F}, Y^{\HH_F}]^{\V_F}}u_{\mathbf{k}},\chi u_{\mathbf{k}}\rangle_{L^2}.
\end{split}
\end{equation}

Writing $\Delta_{\mathbf{k}} = (\nabla_{\mathbf{k}})^*\nabla_{\mathbf{k}}$, we obtain by \eqref{equation:assumption-delta-k} that $\|\nabla_{\mathbf{k}}u_{\mathbf{k}}\|^2_{L^2} \leq C|\mathbf{k}|$. For $Z \in C^\infty(F,\HH_F)$ and $f \in C^\infty(F)$, we therefore get by the Cauchy-Schwarz inequality:
\begin{equation}
\label{equation:bonui}
|\langle \iota_Z \nabla_{\mathbf{k}} u_{\mathbf{k}}, f u_{\mathbf{k}}\rangle_{L^2}| \leq C^{1/2} \|Z\|_{L^\infty} \|f\|_{L^\infty} |\mathbf{k}|^{1/2}.
\end{equation}
In \eqref{equation:long}, the last four terms on the left-hand side and the middle term on the right-hand side have the form \eqref{equation:bonui}. Then, writing $D_{\mathbf{k}}^{\mathrm{Chern}}= \overline{\partial}_{\mathbf{k}} + \partial_{\mathbf{k}}$ (given by the splitting into $(0, 1)$- and $(1, 0)$-parts, respectively), and $Z := [X^{\HH_F}, Y^{\HH_F}]^{\V_F}$, the last term on the right hand side can be re-written as:
\[
	\langle (D_{\mathbf{k}}^{\mathrm{Chern}})_{Z}u_{\mathbf{k}},\chi u_{\mathbf{k}}\rangle_{L^2} = \langle (\partial_{\mathbf{k}})_Z u_{\mathbf{k}}, \chi u_{\mathbf{k}}\rangle_{L^2} = -\langle u_{\mathbf{k}}, (Z^{0, 1} \chi + \chi \Div Z^{0, 1}) u_{\mathbf{k}}\rangle_{L^2} = \mc{O}(1),
\]
where in the first equality we used $\overline{\partial}_{\mathbf{k}} u_{\mathbf{k}} = 0$, and in the second one that $(\partial_{\mathbf{k}})_Z^* = -\overline{\partial}_Z - \Div Z^{0, 1}$; by $Z^{0, 1}$ we denote the $(0, 1)$-part of $Z$. Finally, the sum of the first two terms on the left-hand side is equal to $2 i \Im \langle \X_{\mathbf{k}} u_{\mathbf{k}}, \chi \mathbf{Y}_{\mathbf{k}} u_{\mathbf{k}}\rangle_{L^2}$. This proves the claim.
\end{proof}

\begin{proof}[Proof of Proposition \ref{proposition:lower-bound}]
By definition of $F_{\mathrm{min}}(\mathbf{l})$, the following holds: for all $\eps > 0$, for all $w_0 \in F$, there exists an open neighborhood $U$ of $w_0$, $X^{\HH_F}, Y^{\HH_F} \in C^\infty(U,\HH_F)$ of basic horizontal vector fields such that $|X^{\HH_F}|=|Y^{\HH_F}|=1$ pointwise, and $R > 0$ such that:
 \begin{equation}
 \label{equation:bound-uniform-curvature}
\dfrac{-i}{|\mathbf{k}|} \mathbf{k} \cdot {\mathbf{F}_{\overline{\nabla}}}(w)(X^{\HH_F},Y^{\HH_F}) > F_{\mathrm{min}}(\mathbf{l})-\eps, \quad \forall \mathbf{k} \in \widehat{G}, |\mathbf{k}| > R, |\mathbf{k}/|\mathbf{k}|-\mathbf{l}| < 1/R, \forall w \in U.
 \end{equation}
Notice that $U, X^{\HH_F}, Y^{\HH_F}, R$ all depend on $\eps$.
 
We now fix $\eps > 0$ and consider a finite cover $F = \cup_{j=1}^N U_{w_j}$, where $U_{w_j}$ are open subsets provided by \eqref{equation:bound-uniform-curvature}. We shall denote by $X^{\HH_F}_j, Y^{\HH_F}_j$ the corresponding horizontal basic vector fields. Let $\sum_j \chi_j = \mathbf{1}$ be a partition of unity subordinated to the sets $U_{w_j}$. Combining \eqref{equation:ima} with \eqref{equation:bound-uniform-curvature}, and writing $\X_{j,\mathbf{k}} = (\nabla_{\mathbf{k}})_{X^{\HH_F}_j}$, and $\Y_{j,\mathbf{k}} = (\nabla_{\mathbf{k}})_{Y^{\HH_F}_j}$, we find for all $\mathbf{k} \in \widehat{G}$ as in \eqref{equation:bound-uniform-curvature} (note that we rescale by $\|u_{\mathbf{k}}\|_{L^2}^2$):
\[
\begin{split}
\sum_{j=1}^N \dfrac{2}{|\mathbf{k}|} \Im \langle \X_{j,\mathbf{k}} u_{\mathbf{k}}, \chi_j \mathbf{Y}_{j, \mathbf{k}} u_{\mathbf{k}}\rangle_{L^2}& = \sum_{j=1}^N  \dfrac{-i}{|\mathbf{k}|} \langle \mathbf{k} \cdot \mathbf{F}_{\overline{\nabla}}(X^{\HH_F}_j,Y^{\HH_F}_j) u_{\mathbf{k}}, \chi_j u_{\mathbf{k}}\rangle_{L^2} + \|u_{\mathbf{k}}\|_{L^2}^2  \mc{O}(|\mathbf{k}|^{-1/2}) \\
& \geq (F_{\mathrm{min}}(\mathbf{l})-\eps) \underbrace{\sum_{j=1}^N \langle u_{\mathbf{k}}, \chi_j u_{\mathbf{k}}\rangle_{L^2}}_{= \|u_{\mathbf{k}}\|^2_{L^2}} + \|u_{\mathbf{k}}\|_{L^2}^2 \mc{O}(|\mathbf{k}|^{-1/2}) \\
& \geq (F_{\mathrm{min}}(\mathbf{l})-2\eps) \|u_{\mathbf{k}}\|^2_{L^2},
\end{split}
\]
by taking $|\mathbf{k}|$ large enough.
 Now, the left-hand side in the previous equality is bounded, using $|X^{\HH_F}_j|, |Y^{\HH_F}_j| \leq 1$ by:
 \[
 \left| \sum_{j=1}^N \dfrac{2}{|\mathbf{k}|} \Im \langle \X_{j,\mathbf{k}} u_{\mathbf{k}}, \chi_j \mathbf{Y}_{j, \mathbf{k}} u_{\mathbf{k}}\rangle_{L^2}\right| \leq 2 |\mathbf{k}|^{-1} \sum_{j=1}^N \int_F \chi_j |\nabla_{\mathbf{k}}u_{\mathbf{k}}|^2 \,\dd w = 2  |\mathbf{k}|^{-1} \|\nabla_{\mathbf{k}} u_{\mathbf{k}}\|^2_{L^2}.
 \]
Hence, we have established that for all $\eps > 0$, there exists $R > 0$ such that for all $\mathbf{k} \in \widehat{G}$, $|\mathbf{k}| > R$ and $|\mathbf{k}/|\mathbf{k}|-\mathbf{l}| < 1/R$, for all $u_{\mathbf{k}} \in C^\infty_{\mathrm{hol}}(F,\mathbf{L}^{\otimes \mathbf{k}})$,
 \[
 \|\nabla_{\mathbf{k}} u_{\mathbf{k}}\|^2_{L^2} \geq (F_{\mathrm{min}}(\mathbf{l})/2-\eps)  |\mathbf{k}| \|u_{\mathbf{k}}\|^2_{L^2}.
 \]
This proves the claim.
\end{proof}
 
 \subsection{Lower bound on $\lambda_1$ under global holonomy assumptions}
 
We start with the following:
 
 \begin{lemma}\label{lemma:first-eigenvalue-holonomy-group-dense}
 Assume that the holonomy group is dense in $G$. Then for all $\mathbf{k} \neq 0$, $\lambda_1(\mathbf{k}) > 0$. Moreover, $\lambda_2(\mathbf{0}) > 0$, and eigenfunctions associated to $\lambda_1(\mathbf{0}) = 0$ are the constant functions (see \S \ref{sssection:abelian-rep}).  
 \end{lemma}
A similar result on the level of the horizontal Laplacian was already observed in \cite[Theorem 8.2]{BerardBergery-Bourguignon-81}.
 
 \begin{proof}
Let $\mathbf{k} \neq 0$. Assume that there exists $u_{\mathbf{k}} \in C^\infty_{\mathrm{hol}}(F,\mathbf{L}^{\otimes k})$ such that $\Delta_{\mathbf{k}} u_{\mathbf{k}} = 0$. Then $\nabla_{\mathbf{k}} u_{\mathbf{k}} = 0$; thus $u_{\mathbf{k}}$ is a parallel section. Fix an arbitrary point $x_0 \in M$ such that
\[
\xi := u_{\mathbf{k}}(x_0) \in H^0(F_{x_0},\mathbf{L}^{\otimes \mathbf{k}}|_{F_{x_0}}) \simeq H^0(G/T,\mathbf{J}^{\otimes \mathbf{k}})
\]
is not zero. Let $H := \mathrm{Hol}(P,\nabla)$ be the holonomy group of the connection at $x_0$. That $u_{\mathbf{k}}$ is parallel implies that for every $g \in H$, $g \cdot \xi = \xi$. (More precisely, here $g$ acts via $\rho_{\mathbf{k}}(g)$, where $\rho_{\mathbf{k}}$ is the representation of $G$ on $H^0(G/T,\mathbf{J}^{\otimes \mathbf{k}})$.) Since $H \leqslant G$ is assumed to be dense in $G$, by continuity we obtain that for every $g \in G$, $g \cdot \xi = \xi$. But then this contradicts the fact that $H^0(G/T,\mathbf{J}^{\otimes \mathbf{k}})$ is an irreducible $G$-representation and that $\mathbf{k} \neq 0$.

For the second part, assuming $\Delta_{\mathbf{0}} u = 0$, as above we get $\nabla_{\mathbf{0}} u = 0$. We may think of $\mathbf{L}^{\bf \otimes 0}$ as the trivial bundle $F \times \mathbb{C}$. The fact that $H$ is dense in $G$ ensures that $u$ is fibrewise constant so it equals $u = \pi^*v$ for some $v \in C^\infty(M, \mathbb{C})$. Then using Lemma \ref{lemma:pullback-equivalence}, Item (ii), we get $d(\pi_* u) = \vol(G/T) dv = 0$, where $\pi: F \to M$ is the projection, and $\pi_*$ denotes integration in fibres. This implies that $v$, and so $u$, is constant and completes the proof.
 \end{proof}
 

We now establish Theorem \ref{theorem:lambda1}, Item (ii). The proof follows closely that of Proposition \ref{proposition:technical-reduction2} in Chapter \ref{chapter:flow}.

 \begin{proof}[Proof of Theorem \ref{theorem:lambda1}, Item \emph{(ii)}]
 Let $\nu > 2(\dim F + 1)$. Assume for the sake of contradiction that along a subsequence $\lambda_1(\mathbf{k}) = o(|\mathbf{k}|^{-\nu})$. Consider an eigenstate $u_{\mathbf{k}} \in C^\infty_{\mathrm{hol}}(F,\mathbf{L}^{\otimes \mathbf{k}})$ such that $\|u_{\mathbf{k}}\|_{L^2} = 1$ and $\Delta_{\mathbf{k}}u_{\mathbf{k}} = \lambda_1(\mathbf{k}) u_{\mathbf{k}}$. We rewrite this as:
 \[
 \mathbf{P} u_{\mathbf{k}} := 0,\quad  \mathbf{P} := h^2\Delta_{\mathbf{k}} -h^2\lambda_1(\mathbf{k}).
 \]
 The operator $\mathbf{P}$ belongs to $\Psi^2_{h, \mathrm{BW}}(P)$. We now further restrict to $h := |\mathbf{k}|^{-1}$ (we assume $\mathbf{k} \neq 0$ as we only care about asymptotic properties as $|\mathbf{k}| \to \infty$). The principal symbol $\sigma_{\mathbf{P}}^{\nabla}$ in the calculus $\Psi^\bullet_{h, \mathrm{BW}}(P)$ (computed with respect to the connection $\nabla$ itself on $P$) is then equal, along any family, to
 \[
 \sigma_{\mathbf{P}}^{\nabla}(x,\xi) = |\xi|^2_g,
 \]
 see \S\ref{sssection:examples}, Item (v). Hence, by ellipticity of $\mathbf{P}$ outside of $\{\xi_{\HH^*}=0\} \subset \HH^*$, we obtain that $u_{\mathbf{k}}$ is microlocally $\mc{O}(|\mathbf{k}|^{-\infty})$ outside of the null section $\{\xi_{\HH^*}=0\} \simeq F$ (the null section lies in $\HH^* \subset T^*F$).
 
Now, pairing against $u_{\mathbf{k}}$, we find $\nabla_{\mathbf{k}}u_{\mathbf{k}} = \mc{O}_{L^2}(\lambda_1(\mathbf{k})^{1/2}) = o_{L^2}(|\mathbf{k}|^{-\nu/2})$. Given $f \in C^\infty(F)$, we recall the notation $d_{\HH/\V}f := df|_{\HH^*/\V^*}$. Observe that, using Sobolev embeddings \eqref{equation:sobolev-embeddin-c0}, for any $\varepsilon > 0$ we have
\begin{equation}
\label{equation:dh0}
\begin{split}
\big|d_{\HH}|u_{\mathbf{k}}|^2\big| & = |\langle\nabla_{\mathbf{k}} u_{\mathbf{k}}, u_{\mathbf{k}}\rangle  + \langle u_{\mathbf{k}},\nabla_{\mathbf{k}}u_{\mathbf{k}}\rangle| \\
& \leq C \|u_{\mathbf{k}}\|_{C^0(F,\mathbf{L}^{\otimes \mathbf{k}})}\|\nabla_{\mathbf{k}}u_{\mathbf{k}}\|_{C^0(F, \HH^* \otimes\mathbf{L}^{\otimes \mathbf{k}})} \\
& \leq C |\mathbf{k}|^{\dim F} \|u_{\mathbf{k}}\|_{H^{n/2+\varepsilon}_{|\mathbf{k}|^{-1}}(F,\mathbf{L}^{\otimes \mathbf{k}})} \|\nabla_{\mathbf{k}}u_{\mathbf{k}}\|_{H^{n/2+ \varepsilon}_{|\mathbf{k}|^{-1}}(F,\mathbf{L}^{\otimes \mathbf{k}})} \\
& \leq C |\mathbf{k}|^{\dim F} \|u_{\mathbf{k}}\|_{L^2(F,\mathbf{L}^{\otimes \mathbf{k}})} \|\nabla_{\mathbf{k}}u_{\mathbf{k}}\|_{L^2(F, \HH^* \otimes \mathbf{L}^{\otimes \mathbf{k}})} \leq C |\mathbf{k}|^{\dim F - \nu/2},
\end{split}
\end{equation}
where $C > 0$ changes from line to line, and in the last line we used that $u_{\mathbf{k}}$ is compactly microsupported (so we can bound any semiclassical norm of $u_{\mathbf{k}}$ by a constant times the $L^2$ norm). (The previous inequality should be compared to \eqref{eq:quasimode-normed-invariance}.)

Fix $\beta > 0$, $\alpha < 1$, and pick an arbitrary $x_0 \in M$. Let $W$ be a finite set in the holonomy group over $x_0$, and recall that the set $\mc{W}_N$ of words in $W$ of length $N$ was defined in \eqref{eq:diophantine-w_n-def}. Choose $W$ sufficiently dense so that Corollary \ref{corollary:diophantine-flag-manifold} applies (here we use that $H \leqslant G$ is dense). Following the same argument as in the proof of \eqref{equation:bound12}, and using the Diophantine property of the semisimple Lie group $G$, as well as \eqref{equation:dh0}, we then find as in \eqref{equation:inv1} that the section $u_{\mathbf{k}}(x_0) \in H^0(F_{x_0},\mathbf{L}^{\otimes \mathbf{k}}|_{F_{x_0}}) \simeq H^0(G/T,\mathbf{J}^{\otimes \mathbf{k}})$ verifies:
\begin{equation}
\label{equation:tralala}
g^* |u_{\mathbf{k}}(x_0, \bullet)|^2 - |u_{\mathbf{k}}(x_0, \bullet)|^2 = O(|\mathbf{k}|^{\dim F - \nu/2+\beta}), \qquad \forall g \in \mc{W}_{N},
\end{equation}
where $N := \lceil{|\mathbf{k}|^{\beta}}\rceil$, and $\mc{W}_{N}$ is a $\mc{O}(|\mathbf{k}|^{-\alpha \beta})$-dense set in $G$. 

On the other hand, as in \eqref{equation:dh0}, and using Sobolev embeddings, for any $\varepsilon > 0$ we find that:
\[
\begin{split}
\big|d_{\V}|u_{\mathbf{k}}|^2\big| & = |\langle D_{\mathbf{k}}^{\mathrm{Chern}}u_{\mathbf{k}}, u_{\mathbf{k}}\rangle + \langle u_{\mathbf{k}}, D_{\mathbf{k}}^{\mathrm{Chern}}u_{\mathbf{k}}\rangle| \\
& \leq C \|D_{\mathbf{k}}^{\mathrm{Chern}}u_{\mathbf{k}}\|_{C^0} \|u_{\mathbf{k}}\|_{C^0} \\
& \leq C |\mathbf{k}|^{\dim F/2+1} \|u_{\mathbf{k}}\|_{H^{n/2+1+\varepsilon}_{|\mathbf{k}|^{-1}}}  |\mathbf{k}|^{\dim F/2} \|u_{\mathbf{k}}\|_{H^{n/2+\varepsilon}_{|\mathbf{k}|^{-1}}} \\
& \leq C |\mathbf{k}|^{\dim F + 1} \|u_{\mathbf{k}}\|^2_{L^2} = \mc{O}(|\mathbf{k}|^{\dim F + 1}).
\end{split}
\]

This yields, similarly to \eqref{equation:inv2} that
\[
|u_{\mathbf{k}}(x_0,z)|^2-|u_{\mathbf{k}}(x_0,z')|^2 = \mc{O}(|\mathbf{k}|^{\dim F + 1}d_{F_{x_0}}(z,z')) = \mc{O}(|\mathbf{k}|^{\dim F + 1 - \beta\alpha}),
\]
for all $z,z' \in F_{x_0}$ such that $d(z,z') \leq C |\mathbf{k}|^{-\beta \alpha}$. Combining the previous bound with \eqref{equation:tralala}, we therefore obtain the existence of a point $z_0 \in F_{x_0}$ such that, writing $a_{\mathbf{k}} := |u_{\mathbf{k}}(x_0, z_0)|^2$, we have 
\[
	|u_{\mathbf{k}}(x_0,\bullet)|^2 = a_{\mathbf{k}} + \mc{O}(|\mathbf{k}|^{-\theta}), 
\]
where $\theta = \min\big(-(\dim F + 1)+\beta\alpha, -\dim F + \nu/2 - \beta\big)$. In order to have $\theta > 0$, we can take $\beta > \dim F + 1$ arbitrarily close to $\dim F + 1$, $\alpha < 1$ arbitrarily close to $1$ and $\nu > 4\dim F + 2$ arbitrarily close to $4\dim F + 2$. We then split to cases according to the value of $\mathbf{k} = (k_1, \dotsc, k_a, k_{a + 1}, \dotsc, k_{a + b})$, similarly to the proof of \eqref{equation:bound12}.
\medskip

\emph{Case 1: $(k_{a + 1}, \dotsc, k_{a + b}) \neq 0$.} We assume to have this property along a subsequence of $h \to 0$. By Proposition \ref{prop:line-bundle-topology} we know that $\mathbf{L}^{\otimes \mathbf{k}}|_{F_{x_0}} \to F_{x_0}$ is non-trivial, so $u_{\mathbf{k}}(x_0,\bullet)$ vanishes somewhere and we find that $|u_{\mathbf{k}}(x_0,\bullet)|^2 = \mc{O}(|\mathbf{k}|^{-\theta})$. Using that $M$ is connected, as well as \eqref{equation:dh0}, we obtain that $\|u_{\mathbf{k}}\|_{C^0} =  \mc{O}(|\mathbf{k}|^{-\theta/2})$, thus contradicting $\|u_{\mathbf{k}}\|_{L^2}=1$, and proving the claim.
\medskip

\emph{Case 2: $(k_{a + 1}, \dotsc, k_{a + b}) = 0$.} Assume this property holds for all but finitely many $h$. By Lemma \ref{lemma:extension}, for those $h$ we have that the global weight $\gamma$ associated to $\mathbf{k}$ extends to $\widetilde{\gamma}: G \to \mathbb{S}^1$, and we may define the line bundle $\Lk_M := P \times_{\widetilde{\gamma}} \mathbb{C} \to M$, equipped with the connection associated to the connection on $P$ via $\widetilde{\gamma}$. By Lemma \ref{lemma:pullback-equivalence}, $\pi_F^* \Lk_M$ is identified with $\Lk$ (together with the natural connections on those bundles). Now by assumption we have $\nabla_{\mathbf{k}} u_{\mathbf{k}} = o_{L^2}(|\mathbf{k}|^{-\frac{\nu}{2}})$, and writing $v_{\mathbf{k}} := (\pi_F)_* u_{\mathbf{k}} \in C^\infty(M, \Lk_M)$, as well as by using Lemma \ref{lemma:pullback-equivalence} and \eqref{eq:pushforward-identity} (note that $u_{\mathbf{k}}$ is fibrewise constant since it is fibrewise holomorphic), we get
\[
	\nabla_{\mathbf{k}, M} v_{\mathbf{k}} = o_{L^2}(|\mathbf{k}|^{-\frac{\nu}{2}}) = o_{C^0}(|\mathbf{k}|^{\dim F/2 - \nu/2}), \quad 1/c_1 \geq \|v_{\mathbf{k}}\|_{L^2} \geq c_1 > 0,
\]
for some $c_1 > 0$. In particular, this implies that $d(|v_{\mathbf{k}}|^2) = o_{C^0}(|\mathbf{k}|^{\dim F - \nu/2})$. Then, it is straightforward to check that the proof of Lemma \ref{lemma:torus-bundle-curvature} applies (in fact, we need only a simplified version of it for $\alpha \equiv 0$ and $z_h \equiv 0$), under the assumption that $\nu > \dim F + 8$ (which is automatically satisfied if $\nu > 4\dim F + 2$ and $\dim F \geq 2$). Assuming that $h \mathbf{k}(h) \to \mathbf{l} = (\ell_1, \dotsc, \ell_a, 0, \dotsc, 0)$, this gives (see also Lemma \ref{lemma:abelian-curvature})
\[
	\sum_{j = 1}^a \ell_j F_j = 0,
\] 
which contradicts the starting assumption (note that $(\ell_1, \dotsc, \ell_a)$ has unit norm since $h = |\mathbf{k}|^{-1}$). This completes the proof.
\end{proof}

 \subsection{Hypoellipticity}
 
 The purpose of this subsection is to establish Theorem \ref{theorem:hypoellipticity}, Item (ii), and to give an example of a horizontal Laplacian which is not hypoelliptic.
 
 \subsubsection{Counterexample for Abelian extensions}

\label{ssection:not-hypo}

We first provide an example of an operator of the form \eqref{equation:l} that is \emph{not} hypoelliptic despite the holonomy group being dense. As pointed out to us by F. Nier, a similar construction had already been considered in \cite[Section 3.4]{BerardBergery-Bourguignon-81}.

Consider the relation $\sim$ on $\R \times \R/\Z$ given by
\[
	(x+1, \theta) \sim (x, \theta+\alpha \mod 1), \qquad \alpha \in [0,1].
\]
The quotient $P := (\R \times \R/\Z)/\sim$ is a circle bundle over $\R/\Z$, and it is naturally equipped with a flat connection (coming from the representation $\rho: \pi_1(\mathbb{S}^1) \cong \mathbb{Z} \to \mathbb{S}^1$, where the generator is mapped by $\rho: 1 \mapsto e^{2\pi i \alpha}$; see also \S \ref{sssection:associated-bundle}). Notice that in this case $P \simeq \T^2$, the $2$-torus. It can be equipped with a horizontal Laplacian $\Delta_{\HH}$, which if we write $\pi: \R \times \R/\Z \to P$ for the quotient map, can be shown to satisfy $\pi^*\Delta_{\HH} = -\partial_x^2 \pi^*$ where $\pi^*$ denotes pullback. Notice that a function $f \in C^\infty(P)$ is the same as a function $\bar{f} \in C^\infty(\R \times \R/\Z)$ satisfying $\bar{f}(x+1,\theta) = \bar{f}(x,\theta+\alpha)$. In the following, we will use this identification and drop the $\bar{f}$ notation. Given $f \in C^\infty(P)$, it can be decomposed as a sum of Fourier modes
\[
	f(x,\theta) = \sum_{k \in \Z} f_k = \sum_{k \in \Z} f_k(x) e^{2\pi ik\theta}, \qquad  f_k(x+1) = e^{2\pi i k \alpha} f_k(x).
\]
The Fourier modes $f_k$ can be identified with sections of the line bundle $L^{\otimes k} \to \R/\mathbb{Z}$ which is associated to $P$ via $L^{\otimes k} = P \times_{\rho_k} \mathbb{C}$, with $\rho_k: \mathbb{S}^1 \to \mathbb{S}^1$ given by $z \mapsto z^k$ (see also \S \ref{sssection:U(1)}). Denote by $\Delta_k$ the (connection) Laplacian on $L^{\otimes k}$. By Lemma \ref{lemma:ft-horizontal-laplacian}, the horizontal Laplacian then acts diagonally as 
\[
	\Delta_{\HH}f = \sum_{k \in \Z} \Delta_k f_k = - \sum_{k \in \Z} \partial_x^2 f_k(x) e^{2\pi i k \theta}.
\]
 
 \begin{lemma}
 \label{lemma:eigenvalues-deltak}
 The eigenvalues of $\Delta_k$ are equal to $\{4\pi^2 (k\alpha + q)^2 ~|~ q \in \Z\}$. In particular, if $\alpha$ is irrational, and $k \neq 0$, $\Delta_k$ is invertible.
 \end{lemma}
 
 \begin{proof}
 The solutions to the eigenvalue equation $-\partial^2_x f_k = \lambda f_k$ are given by $f_k(x) = a e^{i\sqrt{\lambda}x} + be^{-i\sqrt{\lambda}x}$ ($a,b \in \R$) and the equation $f_k(x+1) = e^{2\pi i k\alpha} f_k(x)$ forces $\sqrt{\lambda} = 2\pi (k\alpha + q)$, with $q \in \Z$. This proves the claim.
 \end{proof}
 
 We note that $\alpha$ is irrational if and only if the holonomy group is dense, in which case the invertibility of $\Delta_k$ for $k \neq 0$ was already established in Lemma \ref{lemma:first-eigenvalue-holonomy-group-dense}. For the next lemma, we recall that an irrational number $\alpha \in \mathbb{R}$ is called \emph{Liouville} if for any $N > 0$, there exists $(p, q) \in \mathbb{Z} \times \mathbb{Z}_{\geq 1}$, such that
 \[
 	0 < \left|\alpha - \frac{p}{q}\right| < \frac{1}{q^N}.
 \] 
It is known that Liouville numbers are dense in $\mathbb{R}$, but form a set of Hausdorff dimension zero.
\begin{lemma}
	Let $\alpha \in [0,1] \setminus \Q$ be a Liouville number. Then $\Delta_{\HH}$ is not hypoelliptic.
\end{lemma}
 
 \begin{proof}
 By the assumption on $\alpha$, we can find a sequence of integers $(p_j,k_j)_{j \geq 1} \in \mathbb{Z} \times \mathbb{Z}_{\geq 1}$ such that $k_j \geq j$, $(k_j)_{j \geq 1}$ is increasing, and $|k_j\alpha + p_j| \leq k_j^{-j}$. In particular, by Lemma \ref{lemma:eigenvalues-deltak}, along this subsequence, one obtains that the first eigenvalue $\lambda_1(k_j)$ of the Laplacian $\Delta_{k_j}$ satisfies 
 \begin{equation}\label{eq:first-eigenvalue}
 	0 < \lambda_1(k_j) \leq C k_j^{-2j} 
 \end{equation}
 for some $C > 0$ independent of $j$.
 
 We define a function $f$ on $P$ by the following
 \[
 	f := \sum_{j=1}^{\infty} f_{k_j}, \qquad \Delta_{k_j} f_{k_j} = \lambda_1(k_j) f_{k_j}, \qquad \|f_{k_j}\|_{L^2} = \lambda_1(k_j).
 \]
 Observe that
 \[
 \|f\|^2_{L^2(P)} = \sum_{j=1}^{\infty} \|f_{k_j}\|^2_{L^2} \leq C^2 \sum_{j=1}^{\infty} k_j^{-2j} \leq C^2 \sum_{j = 1}^{\infty} j^{-2} < \infty,
 \]
 so $f \in L^2(P)$. Computing the $L^2$ norm of $\Delta_{\HH}^k f$ and $\Delta_{\V}^k f$ for any $k \in \mathbb{Z}_{\geq 0}$ (where $\Delta_{\V}$ is the vertical Laplacian), and showing these expressions are finite for any $k \geq 0$ in a similar fashion (also using the estimate \eqref{eq:first-eigenvalue}), allows one to show that $f$ is smooth, i.e. $f \in C^\infty(P)$.
 
Set $u_{k_j} := \lambda_1(k_j)^{-1} f_{k_j}$ for all $j \in \mathbb{Z}_{\geq 1}$ and define $u := \sum_{j = 1}^\infty u_{k_j}$. By construction and Lemma \ref{lemma:ft-horizontal-laplacian} we have $\Delta_{\HH} u = f$. Moreover, $\|u_{k_j}\|_{L^2} = 1$, and so we get that $u \in \mc{D}'(P)$. (Indeed, to define $\langle{u, \varphi}\rangle_{L^2} = \sum_k \langle u_k, \varphi_k\rangle_{L^2}$ for a test function $\varphi \in C^\infty(P)$, it suffices to use that $\|\varphi_k\|_{L^2} = \mc{O}(k^{-\infty})$ as $k \to \infty$ by the smoothness of $\varphi$.) However, $u \not \in L^2(P)$, and so $\Delta_{\HH}$ is not hypoelliptic.
 \end{proof}

\subsubsection{Proof of Theorem \ref{theorem:hypoellipticity}, Item (ii)} Given $f,u \in \mc{D}'(P)$, write $f = \mc{F}^{-1}(f_{\mathbf{k},i})_{\mathbf{k},i}$ and $u= \mc{F}^{-1}(u_{\mathbf{k},i})_{\mathbf{k},i}$. By projecting onto each Fourier mode, the equation $\Delta_{\HH} u = f$ is equivalent to (using Lemma \ref{lemma:ft-horizontal-laplacian})
\begin{equation}
\label{equation:modes}
\Delta_{\mathbf{k}}  u_{\mathbf{k}, i} = f_{\mathbf{k}, i}, \qquad u_{\mathbf{k}, i},f_{\mathbf{k}, i} \in \mc{D}'_{\mathrm{hol}}(F,\mathbf{L}^{\otimes \mathbf{k}}), \qquad  \mathbf{k} \in \widehat{G}, i = 1, \dotsc, d_{\mathbf{k}}.
\end{equation}
If $f \in C^\infty(P)$, then $f_{\mathbf{k},i} \in C^\infty_{\mathrm{hol}}(F,\mathbf{L}^{\otimes \mathbf{k}})$ and hence $u_{\mathbf{k},i} \in C^\infty_{\mathrm{hol}}(F,\mathbf{L}^{\otimes \mathbf{k}})$ by ellipticity of $\Delta_{\mathbf{k}}$ in the direction of $\HH^*$ (we already know that $u_{\mathbf{k}, i}$ is fibrewise holomorphic and in particular smooth in the vertical direction). However, proving that $u$ is smooth amounts to proving high frequency estimates on the norms of $u_{\mathbf{k},i}$ in Sobolev spaces as $|\mathbf{k}| \to \infty$.

For this purpose, introduce for $p,q \in \mathbb{Z}_{\geq 0}$ the anisotropic norms defined for $u \in C^\infty(P)$ by:
\begin{equation}
\label{equation:hpq}
\|u\|_{H^{2p,2q}(P)}^2 := \|u\|^2_{L^2(P)} + \|\Delta_{\V}^pu\|^2_{L^2(P)} + \|\Delta_{\HH}^q u\|^2_{L^2(P)}.
\end{equation}
We then define the space $H^{2p,2q}(P)$ as the completion of $C^\infty(P)$ with respect to the previous norm. It can be easily verified that for $p = q$, this space is equivalent to the usual $H^{2p}(P)$ norm on $P$. Recall from Lemma \ref{lemma:ft-horizontal-laplacian} that $(\Delta_{\V} u)_{\mathbf{k}, i} = c(\mathbf{k}) u_{\mathbf{k}, i}$ and the Fourier transform intertwines the horizontal Laplacians. Hence from the Plancherel formula (Lemma \ref{lemma:ft-isometry}):
\[
	\|u\|^2_{H^{2p,2q}(P)} = \sum_{\mathbf{k} \in \widehat{G}} d_{\mathbf{k}} \sum_{i=1}^{d_{\mathbf{k}}}\left( (1 + c(\mathbf{k})^{2p}) \|u_{\mathbf{k},i}\|^2_{L^2} +  \|\Delta_{\mathbf{k}}^q u_{\mathbf{k},i}\|^2_{L^2}\right).
\]

Theorem \ref{theorem:hypoellipticity}, Item (ii) follows immediately from the following lemma.

\begin{lemma}
\label{lemma:key-lol}
Assume there exists $C, \nu > 0$, such that $\lambda_1(\Delta_{\mathbf{k}}) \geq C |\mathbf{k}|^{-\nu}$ for all non-zero $\mathbf{k} \in \widehat{G}$. Then, for all $p \in \mathbb{Z}_{\geq 0}$ and $q \in \mathbb{Z}_{\geq 1}$, there exists $C > 0$ such that for all $u \in C^\infty(P)$, orthogonal to the constant functions,
\begin{equation}
\label{equation:laval}
\|u\|_{H^{2p,2q}} \leq C \|\Delta_{\HH} u\|_{H^{2\lceil p+\nu/2 \rceil,2(q-1)}}.
\end{equation}
Moreover, if $u \in \mc{D}'(P)$ is orthogonal to constant functions, and we have $\Delta_{\HH} u \in H^{2\lceil p+\nu/2 \rceil,2(q-1)}(P)$, then $u \in H^{2p,2q}(P)$ and inequality \eqref{equation:laval} holds. In particular, if $\Delta_{\HH} u \in C^\infty(P)$ is smooth, then $u \in C^\infty(P)$.
\end{lemma}

 

\begin{proof}
Observe that $c(\mathbf{k})$ satisfies the asymptotic estimate \eqref{equation:ck-asymptotic}, and set $p' := \lceil p+\nu/2 \rceil$. For the rest of the proof, for simplicity, we write $\lambda_1(\mathbf{0})$ for the second eigenvalue of $\Delta_{\mathbf{0}}$ (it is positive by Lemma \ref{lemma:first-eigenvalue-holonomy-group-dense}). It follows that there exists a constant $C:=C(p) > 0$ such that for all $\mathbf{k} \in \widehat{G}$, the following inequality holds:
\begin{equation}
\label{equation:ehouai}
\lambda_1^2(\mathbf{k})(1+c(\mathbf{k})^{2p'}) \geq \lambda_1^2(\mathbf{k})(1+c(\mathbf{k})^{2(p+\nu/2)}) \geq C(1+c(\mathbf{k})^{2p}).
\end{equation}
Hence, if $u \in C^\infty(P) \cap (\C \mathbf{1})^{\perp}$ (i.e. it is orthogonal to constants) and solves $\Delta_{\HH} u = f$, we find that 
\begin{equation}
\label{equation:laval2}
\begin{split}
\|\Delta_{\HH} u\|_{H^{2p',2(q-1)}(P)}^2 & = \sum_{\mathbf{k} \in \widehat{G}} d_{\mathbf{k}} \sum_{i=1}^{d_{\mathbf{k}}} \left((1+c(\mathbf{k})^{2p'})\|\Delta_{\mathbf{k}} u_{\mathbf{k},i}\|^2_{L^2} + \|\Delta_{\mathbf{k}}^{q} u_{\mathbf{k},i}\|^2_{L^2}\right) \\
& \geq \sum_{\mathbf{k} \in \widehat{G}} d_{\mathbf{k}} \sum_{i=1}^{d_{\mathbf{k}}} \left(\lambda_1(\mathbf{k})^2 (1+c(\mathbf{k})^{2p'})\|u_{\mathbf{k},i}\|^2_{L^2} + \|\Delta_{\mathbf{k}}^{q} u_{\mathbf{k},i}\|^2_{L^2}\right)  \\
& \geq C \sum_{\mathbf{k} \in \widehat{G}} d_{\mathbf{k}} \sum_{i=1}^{d_{\mathbf{k}}} \left((1+c(\mathbf{k})^{2p})\|u_{\mathbf{k},i}\|^2_{L^2} + \|\Delta_{\mathbf{k}}^{q} u_{\mathbf{k},i}\|^2_{L^2}\right) \\
& = C \|u\|^2_{H^{2p,2q}(P)},
\end{split}
\end{equation}
where we used \eqref{equation:ehouai} in the second line. This proves the first claim.

Finally, it remains to prove the bootstrap argument, namely if $u \in \mc{D}'(P) \cap (\C \mathbf{1})^{\perp}$ and $\Delta_{\HH}u \in H^{2p', 2(q-1)}$, then $u \in H^{2p,2q}$ and the estimate \eqref{equation:laval} is satisfied. For $h > 0$, define
\begin{equation}
\label{equation:th}
T_h u := \mc{F}^{-1}(u_{\mathbf{k},i})_{\mathbf{k} \in \widehat{G}, |\mathbf{k}|\leq h^{-1},i=1, \dotsc,d_{\mathbf{k}}}.
\end{equation}
By standard elliptic regularity, $u_{\mathbf{k},i} \in H^{2q}_{\mathrm{hol}}(F,\mathbf{L}^{\otimes \mathbf{k}})$ since $\Delta_{\mathbf{k}}^q u_{\mathbf{k},i} \in L^2_{\mathrm{hol}}(F,\mathbf{L}^{\otimes \mathbf{k}})$, and thus $T_h u \in H^{2p,2q}(P)$. (In fact, we even have $T_hu \in H^{2p'', 2q}(P)$ for any $p'' \in \mathbb{Z}_{\geq 0}$ since the sum in \eqref{equation:th} is finite.) As a consequence, the same sequence of inequalities as in \eqref{equation:laval2} lead to:
\begin{equation}\label{eq:thh}
\begin{split}
\|T_h u\|^2_{H^{2p,2q}(P)} & = \sum_{\substack{\mathbf{k} \in \widehat{G}, \\ |\mathbf{k}|\leq h^{-1}}} d_{\mathbf{k}} \sum_{i=1}^{d_{\mathbf{k}}} \left((1+c(\mathbf{k})^{2p})\|u_{\mathbf{k},i}\|^2_{L^2} + \|\Delta_{\mathbf{k}}^{q} u_{\mathbf{k},i}\|^2_{L^2}\right) \\
& \leq C\|\Delta_{\HH} u\|_{H^{2p',2(q-1)}(P)}^2,
\end{split}
\end{equation}
for some $C> 0$ which is independent of $h > 0$. It follows that $(T_hu)_{h > 0}$ is a Cauchy sequence in $H^{2p, 2q}(P)$, and so as $h \to 0$, we have $T_h u \to v$ in $H^{2p, 2q}(P)$. Since the $L^2$-adjoint of $T_h$ is equal to $T_h$, we conclude that for any test function $\varphi \in C^\infty(P)$ we have 
\[
	\langle{v, \varphi}\rangle_{L^2} = \lim_{h \to 0} \langle{T_h u, \varphi}\rangle_{L^2} = \lim_{h \to 0} \langle{u, T_h \varphi}\rangle_{L^2} = \langle{u, \varphi}\rangle_{L^2},
\]	
as $T_h\varphi \to \varphi$ as $h \to 0$ in $C^\infty(P)$. This shows $v = u \in H^{2p,2q}(P)$, as well as that $\|u\|_{H^{2p,2q}(P)} \leq C \|\Delta_{\HH}u\|_{H^{2p',2(q-1)}}$, which concludes the proof of the lemma. (Alternatively, the Sobolev spaces $H^{2p,2q}(P)$ satisfy the Fatou property (see \cite[page 15]{Runst-Sickel-96}), and so \eqref{eq:thh} directly implies the result.) 
\end{proof}

\subsection{Discrete spectrum}

We now prove Theorem \ref{theorem:hypoellipticity}, Item (i). Let $\alpha : [C_0,+\infty) \to (0,\infty)$ be any function such that $\lambda_1(\Delta_{\mathbf{k}}) \geq \alpha(|\mathbf{k}|)$ (for $|\mathbf{k}| \gg 1$) and $\alpha(x) \to_{x \to +\infty} +\infty$. Define $H^{2p,2q}_\alpha(P)$ as the completion of $C^\infty(P)$ with respect to the norm
\begin{equation}
\label{equation:loulou}
\|u\|^2_{H^{2p,2q}_\alpha(P)} := \sum_{\mathbf{k} \in \widehat{G}} d_{\mathbf{k}} \sum_{i=1}^{d_{\mathbf{k}}}\left((1+c(\mathbf{k})^{2p})\alpha(|\mathbf{k}|)^2 \|u_{\mathbf{k},i}\|^2_{L^2} + \|\Delta_{\mathbf{k}}^{q} u_{\mathbf{k},i}\|^2_{L^2}\right).
\end{equation}

\begin{lemma}
For every $p \in \mathbb{Z}_{\geq 0}$ and $q \in \mathbb{Z}_{\geq 1}$, there exists a constant $C > 0$ such that for all $u \in C^\infty(P)$ orthogonal to the constants:
\begin{equation}
\label{equation:inverse-delta}
\|u\|_{H^{2p,2q}_\alpha} \leq C\|\Delta_{\HH}u\|_{H^{2p,2(q-1)}}.
\end{equation}
In particular, $\Delta_{\HH}: H^{2p, 2q}_{\alpha} \cap (\mathbb{C} \mathbf{1})^{\perp} \to H^{2p, 2(q - 1)} \cap (\mathbb{C} \mathbf{1})^{\perp}$ is an isomorphism. Moreover, the embedding $H^{2p,2q}_\alpha(P) \hookrightarrow H^{2p,2(q-1)}(P)$ is compact.
\end{lemma}


\begin{proof}
The proof of \eqref{equation:inverse-delta} is straightforward and follows that of \eqref{equation:laval}. Indeed, in \eqref{equation:laval2}, it suffices to replace $p'$ by $p$ and bound $\lambda_1(\mathbf{k})$ from below by $\alpha(\mathbf{k})$ (by assumption). That the embedding is compact follows from the fact that $\iota : H^{2p,2q}_\alpha(P)\hookrightarrow H^{2p,2(q-1)}(P) $ is approximated by the compact operators $T_h$ (defined in \eqref{equation:th}) in the topology $\mc{L}(H^{2p,2q}_\alpha(P), H^{2p,2(q-1)}(P))$ as $h \to 0$. Indeed, to see this, observe firstly that $T_h$ is compact since by definition it selects only finitely many Fourier modes, and since we have the compactness of the classical Sobolev spaces $H^{2q}_{\mathrm{hol}}(F, \Lk) \subset H^{2(q - 1)}_{\mathrm{hol}}(F, \Lk)$. To see that $T_h$ approximates $\iota$ in the operator norm, as in \eqref{equation:laval2}, we similarly obtain for any $u \in C^\infty(P)$ orthogonal to constants that 
\[
	\|T_h u - u\|_{H^{2p, 2(q - 1)}} \leq \frac{1}{|\alpha(h^{-1})|} \|u\|_{H_{\alpha}^{2p, 2(q - 1)}},
\]
using that $\alpha$ is non-decreasing. Using that $H^{2p, 2q}_\alpha(P) \subset H_\alpha^{2p, 2(q - 1)}(P)$, as well as that $\alpha(h^{-1})$ converges to $\infty$ as $h \to 0$ proves the claim, and completes the proof.
\end{proof}

\begin{proof}[Proof of Theorem \ref{theorem:hypoellipticity}, Item \emph{(i)}]
Applying the previous lemma with $p=0$, $q=1$, and using that $H^{0,2}_\alpha(P) \hookrightarrow L^2(P)$ is compact, with obtain that $\Delta_{\HH}^{-1} : L^2(P) \cap (\C \mathbf{1})^{\perp} \to L^2(P) \cap (\C \mathbf{1})^{\perp}$ is compact; thus $\Delta_{\HH}$ has discrete spectrum.
\end{proof}

\section{Quantum Ergodicity}

We now prove Theorem \ref{theorem:quantum-ergodicity}.

 \subsection{Dynamical preliminaries}
 
 Let $P \to M$ be a $G$-principal bundle. Let $\pi : T^*M \to M$ be the footpoint projection; let $\pi^* F \to T^*M$ be the pullback of $F \to M$ to $T^*M$. In what follows, $\HH^* := \HH^*_F$.

\begin{lemma}
\label{lemma:theta}
There exists a natural isomorphism $\theta :  \pi^*F \to \HH^*$.
\end{lemma}

\begin{proof}
Given $(x,\xi,w) \in \pi^*F$, where $x \in M, \xi \in T^*_xM, w \in F_x$, define $\theta(x,\xi,w) := (x,w,d\pi^\top\xi) \in \HH^*$. It is immediate that $\theta : \pi^* F \to \HH^*$ is an isomorphism.
\end{proof}

The geodesic flow $(\varphi_t)_{t \in \R}$ on $T^*M$ can be extended to a flow $(\psi_t)_{t \in \R}$ on $\pi^*F$ as follows:
\[
\psi_t(x,\xi,w) := (\varphi_t(x,\xi), \tau_{x \to \pi(\varphi_t(x,\xi))}w),
\]
where $\tau_{x \to \pi(\varphi_t(x,\xi))}w$ denotes the parallel transport of $w$ with respect to the induced connection on $F$, along the geodesic $(\pi \circ \varphi_s(x,\xi))_{s \in [0,t]}$. Recall from \eqref{equation:omega-twisted} that the curvature $\nabla$ on $P$ induces a twisted symplectic form on $T^*F$ given by
\[
\omega_{h,\mathbf{k}} = \omega_0 + ih\mathbf{k}\cdot\pi^*\mathbf{F}_{\overline{\nabla}},
\]
where $\pi : T^*F \to F$ is the projection and $\mathbf{F}_{\overline{\nabla}}$ is the (multi) curvature induced on the (multi) line bundle $\mathbf{L} \to F$. Define the symbol $p(w,\xi) := |\xi_{\HH^*}|^2$ on $T^*F$. Let $(\Phi^{\omega_{h,\mathbf{k}}}_t)_{t \in \R}$ be the symplectic flow on $T^*F$ generated by $p$ with respect to the $2$-form $\omega_{h,\mathbf{k}}$, that is the flow generated by the vector field $H_p^{\omega_{h,\mathbf{k}}}$ defined in \eqref{equation:hp-vfield}. Notice by Lemma \ref{lemma:invariance-h} that $(\Phi^{\omega_{h,\mathbf{k}}}_t)_{t \in \R}$ preserves $\HH^*$. Moreover, for any $E > 0$ it preserves the energy layers $S_E\HH^* := \{\xi \in \HH^* \mid |\xi| = E\}$ since it preserves the Hamiltonian function. In what follows, we will only consider the restriction of this flow to $\HH^*$.

\begin{lemma}
Let $a \in S^m(\HH^*)$. Then:
\begin{equation}
\label{equation:egorov-utile}
e^{ith\Delta_{\mathbf{k}}}\Op^{\mathrm{BW}}_h(a)e^{-ith\Delta_{\mathbf{k}}} = \Op^{\mathrm{BW}}_h(a \circ \Phi_t^{\omega_{h,\mathbf{k}}}) + \mc{O}_{S^{m-1}}(h).
\end{equation}
\end{lemma}

\begin{proof}
The principal symbol of $h^2\Delta_{\mathbf{k}}$ is $p(w,\xi)=|\xi_{\HH^*}|^2$ which generates the flow $(\Phi_t^{\omega_{h,\mathbf{k}}})_{t \in \R}$. The claim thus follows from Egorov's theorem in the twisted calculus, see \S\ref{sssection:propagation1} and Remark \ref{remark:egorov-general}. 
\end{proof}

We now investigate the flat case:

\begin{lemma}
Suppose that $\nabla$ is flat. Then for all $t \in \R$, $\theta \circ \psi_t = \Phi_t^{\omega_{h,\mathbf{k}}} \circ \theta$.
\end{lemma}

The curvature $\mathbf{F}_{\overline{\nabla}}$ does not appear in the flow $(\Phi_t^{\omega_{h,\mathbf{k}}})_{t \in \R}$; this is a consequence of the bundle being flat.

\begin{proof}
Let $U \subset M$ be a local patch of coordinates. Since the connection is flat, the bundle $F$ can be locally trivialized over $U$ as $F|_{U} \simeq U \times G/T$ in such a way that parallel transport along curves of $M$ (contained in $U$) of elements of $F$ corresponds to the identity map in the $G/T$ factor of the trivialization. Let $(x,w)$ denote a point in $U \times G/T$; let $(\xi,\eta) \in T^*U \times T^*(G/T)$ denote the dual variables. Then $\HH^* \simeq \{\eta=0\}$ and $\V^*\simeq\{\xi=0\}$. The symbol of the horizontal Laplacian is then given by $p(x,w,\xi,\eta) = |\xi|^2$. Since $P \to M$ is flat, the locally defined (multi) $1$-form $\boldsymbol{\beta} = (\boldsymbol{\beta}^{\HH}, \boldsymbol{\beta}^{\V})$ only admits components in the vertical directions (in the $G/T$ direction), that is $\boldsymbol{\beta} =(0,\boldsymbol{\beta}^{\V})$ and $\boldsymbol{\beta}^{\V}$ is independent of the $x$ variable. Using \eqref{equation:hvfield} in $T^*F$ for the expression of $H_p^{\omega_{h,\mathbf{k}}}$ in these coordinates, one finds:
\small
\[
\begin{split}
& H_p^{\omega_{h,\mathbf{k}}}  = \sum_{k=1}^n -\partial_{\xi_k}p~ \partial_{x_k} 
+  \sum_{k=1}^{\dim(G/T)}- \partial_{\eta_k}p~ \partial_{w_k} \\
& + \sum_{k=1}^n \left(\partial_{x_k}p + h\mathbf{k} \cdot \left( \sum_{\ell=1}^n \partial_{\xi_\ell}p(\partial_{x_\ell}\boldsymbol{\beta}^{\HH}_k-\partial_{x_k}\boldsymbol{\beta}^{\HH}_\ell) +  \sum_{\ell=1}^{\dim(G/T)} \partial_{\eta_\ell}p(\partial_{w_\ell}\boldsymbol{\beta}^{\HH}_k-\partial_{x_k}\boldsymbol{\beta}^{\V}_\ell)\right) \right)\partial_{\xi_k} \\
&  + \sum_{k=1}^{\dim(G/T)} \left(\partial_{w_k}p + h\mathbf{k} \cdot \left( \sum_{\ell=1}^n \partial_{\xi_\ell}p(\partial_{x_\ell}\boldsymbol{\beta}^{\V}_k-\partial_{w_k}\boldsymbol{\beta}^{\HH}_\ell) +  \sum_{\ell=1}^{\dim(G/T)} \partial_{\eta_\ell}p(\partial_{w_\ell}\boldsymbol{\beta}^{\V}_k-\partial_{w_k}\boldsymbol{\beta}^{\V}_\ell)\right) \right)\partial_{\eta_k}
\end{split} 
\]
\normalsize
Using $\boldsymbol{\beta}^{\HH}=0$, $\partial_{x_k} \boldsymbol{\beta}^{\V}=0$, $\partial_{w_k}p = \partial_{\eta_\ell} p = 0$, one obtains that this simplifies to
\[
H_p^{\omega_{h,\mathbf{k}}}  = \sum_{k=1}^n -\partial_{\xi_k}p \cdot \partial_{x_k}+ \sum_{k=1}^n \partial_{x_k}p \cdot \partial_{\xi_k}.
\]
One therefore retrieves the generator of the geodesic flow on $T^*M$; nevertheless, it should be interpreted here as the generator of the flow on $\HH^* \simeq T^*M \times G/T$ (locally) given by the geodesic flow in the $T^*M$ variable, and the identity in the $G/T$ variable. Notice that in our local trivialization of $F$, the latter actually corresponds to parallel transport in the fibers of $F \to M$. This proves the claim made in the lemma.
\end{proof}

In what follows, in the flat case, we will simply write $(\Phi_t)_{t \in \R}$ for this flow on $\HH^*$. Finally, we conclude with the following observation. 

\begin{lemma}
\label{lemma:ergodic-layer}
Assume that $(M,g)$ is a Riemannian manifold such that its geodesic flow is Anosov, and $P \to M$ is a flat $G$-principal bundle with dense holonomy group in $G$. Then for all $E > 0$, $(\Phi_t)_{t \in \R}$ is ergodic on $S_E\HH^*$ with respect to the smooth invariant measure.
\end{lemma}

\begin{proof}
By the isomorphism $\theta$ of Lemma \ref{lemma:theta}, the flow $(\Phi_t)_{t \in \R}$ is an isometric extension of the geodesic flow on $S_E^*M := \{(x,\xi) \in T^*M ~|~|\xi|=E\}$ to the bundle $\pi^*F \to S_E^*M$. This bundle is an associated bundle to $\pi^*P \to S_E^*M$ so it suffices to prove that the flow defined by parallel transport along geodesics on $\pi^*P \to S_E^*M$ is ergodic. Since this is an extension of the (Anosov) geodesic flow to a principal bundle, it suffices to check that the transitivity group $H$ equals $G$. Now, by Example \ref{example:flat-connection}, $H$ contains as a subgroup the holonomy group $\mathrm{Hol}(P,\nabla)$, which is dense in $G$ by assumption. Since $H$ is closed by definition, the claim is immediate.
\end{proof}

 \subsection{Spectral preliminaries}

\subsubsection{Hellfer-Sjöstrand's formula} The following holds:

\begin{lemma}
\label{lemma:hs}
Let $\chi \in C^\infty_{\mathrm{comp}}(\R)$, $m \geq 0$. Let $\mathbf{A} \in \Psi^m_{h,\mathrm{BW}}(P)$ be formally self-adjoint elliptic. Then $\chi(\mathbf{A}) \in \Psi^{-\infty}_{h,\mathrm{BW}}(P)$ with principal symbol $\sigma^{\mathrm{BW}}_{\chi(\mathbf{A})} = \chi(a(\bullet)) \in S^m(\HH^*)$.
\end{lemma}
Note that $\chi(\mathbf{A})$ is always well-defined on $C^\infty_{\mathrm{hol}}(F,\mathbf{L}^{\otimes \mathbf{k}})$ by the spectral theorem.

\begin{proof}
The proof is verbatim the same as \cite[Theorems 14.8 and 14.9]{Zworski-12} based on the formula
\[
\chi(\mathbf{A})  = \dfrac{1}{\pi i} \int_{\C} \overline{\partial}_z \widetilde{\chi}(z)(\mathbf{A}-z)^{-1}\, \dd z,
\]
where $\widetilde{\chi}$ denotes an almost-analytic extension of $\chi$ to $\C$ (more precisely, this means that $\bar{\partial}_{z}\widetilde{\chi}$ vanishes to infinite order at the real axis and is equal to $\chi$ on it, see \cite[Theorem 3.6]{Zworski-12}). Note that the resolvent $(\mathbf{A}-z)^{-1} : C^\infty_{\mathrm{hol}}(F,\mathbf{L}^{\otimes \mathbf{k}}) \to C^\infty_{\mathrm{hol}}(F,\mathbf{L}^{\otimes \mathbf{k}})$ is well-defined and in the Borel-Weil calculus for $z$ not in the spectrum of $\mathbf{A}$ by Lemma \ref{lemma:parametrix-bw} and the subsequent remark. Admissibility of $\chi(\mathbf{A})$ follows from admissibility of $(\mathbf{A} - z)^{-1}$.
\end{proof}

\subsubsection{Trace} We now establish an asymptotic formula for the trace of an operator in the Borel-Weil calculus.

\begin{lemma}
\label{lemma:trace}
Fix a constant $C > 0$. Let $\mathbf{A} \in \Psi^{-\infty}_{h,\mathrm{BW}}(P)$. Then $\mathbf{A} : C^\infty_{\mathrm{hol}}(F,\mathbf{L}^{\otimes \mathbf{k}}) \to C^\infty_{\mathrm{hol}}(F,\mathbf{L}^{\otimes \mathbf{k}})$ is trace class and for all $h, \mathbf{k}$ such that $h|\mathbf{k}| \leq C$, one has:
\begin{equation}
\label{equation:l1}
\Tr(\mathbf{A}) = \dfrac{d_{\mathbf{k}}}{\vol(G/T)(2\pi h)^n} \left( \int_{T^*M} \int_{F_x} \sigma_{\mathbf{A}}^{\mathrm{BW}}(x,w,d\pi^\top \xi)\, \dd w \dd x \dd \xi + \mc{O}(h)\right).
\end{equation}
In particular, if $\sigma_{\mathbf{A}}^{\mathrm{BW}}$ is a pullback symbol (see \S \ref{ssection:quantization-symbols}), then
\begin{equation}
\label{equation:l2}
\Tr(\mathbf{A}) = \dfrac{d_{\mathbf{k}}}{(2\pi h)^n} \left(\int_{T^*M} \sigma_{\mathbf{A}}^{\mathrm{BW}}(x,\xi)\, \dd x \dd \xi + \mc{O}(h)\right),
\end{equation}
where $\sigma^{\mathrm{BW}}_{\mathbf{A}}$ is here identified with a function on $T^*M$. The remainder terms in \eqref{equation:l1} and \eqref{equation:l2} both depend on $C$.
\end{lemma}

The principal symbol $\sigma_{\mathbf{A}}^{\mathrm{BW}}$ depends on a choice of family $h \mapsto \mathbf{k}(h)$ with $h|\mathbf{k}(h)| \leq 1$. In \eqref{equation:l1}, it is thus implicitly understood that such a choice has been made. The remainder $\mc{O}(h)$ is uniform among all choices of families $h \mapsto \mathbf{k}(h)$. Observe that, taking a constant family $\mathbf{k}(h) := \mathbf{k}_0$, and further assuming that $\sigma_{\mathbf{A}}^{\mathrm{BW}}$ is a pullback function, one retrieves the usual formula for the trace of a non-twisted semiclassical operator acting on the vector bundle $E := H^0(F,\mathbf{L}^{\otimes \mathbf{k}_0}) \to M$ with diagonal principal symbol $\sigma(x,\xi) := \sigma_{\mathbf{A}}^{\mathrm{BW}}(x,\xi) \mathbf{1}_{E_x}$.

\begin{proof}
We start by observing that \eqref{equation:l2} follows from \eqref{equation:l1} as $\sigma_{\mathbf{A}}^{\mathrm{BW}}$ being a pullback symbol simply means that it is independent of the variable $w$. To prove \eqref{equation:l1} apply \eqref{equation:penible}, \eqref{equation:symbol-expansion}, and Theorem \ref{theorem:bergman} (see also \eqref{equation:temporaire}, second equality, for the expression of the Schwartz kernel) which yields:
\[
\Tr(\mathbf{A}) = \dfrac{d_{\mathbf{k}}}{\vol(G/T)(2\pi h)^n}  \int_{T^*M} \int_{F_x} (\sigma^{\mathrm{BW}}_{\mathbf{A}}(x,w,d\pi^\top\xi) + \mc{O}_{S^{-\infty}}(h))\, \dd w \dd x \dd \xi.
\]
This proves the claim.
\end{proof}

\begin{remark}After some calculation, \eqref{equation:l1} could also be rewritten
\begin{equation}
\label{equation:l1bis}
\Tr(\mathbf{A}) = \dfrac{1}{(2\pi h)^n} \int_{T^*M} \Tr\left(\Pi_{\mathbf{k}}(x)\sigma_{\mathbf{A}}^{\mathrm{BW}}(x,d\pi^\top_{\bullet}\xi,\bullet)\Pi_{\mathbf{k}}(x)\right) \dd x \dd \xi + \mc{O}\left(\tfrac{d_{\mathbf{k}} h}{(2\pi h)^n}\right).
\end{equation}
\end{remark}

An immediate corollary of Lemma \ref{lemma:trace} is to provide a Weyl law for elliptic operators. For $0 \leq a < b$, define $N(h,\mathbf{k},a,b)$ to be the number of eigenvalues of $\Delta_{\mathbf{k}}$ in the interval $[h^{-2}a,h^{-2}b]$. The following holds:
 
 \begin{proposition}
 \label{proposition:weyl}
 Let $C > 0$. Then for all $h, \mathbf{k}$ such that $h|\mathbf{k}| \leq C$,
 \[
  N(h,\mathbf{k},a,b) = \dfrac{d_\mathbf{k}}{(2\pi h)^n}\dfrac{\vol(M) \vol(\Ss^{n-1})}{n}(b^{n/2}-a^{n/2})(1 + o_{h \to 0, C}(1)),
 \]
 where the $o_{h \to 0, C}(1)$ depends on the constant $C > 0$. 
 \end{proposition}

\begin{proof}[Proof of Proposition \ref{proposition:weyl}]
The proof of Proposition \ref{proposition:weyl} follows from Lemmas \ref{lemma:hs} and \ref{lemma:trace} by applying it with $\mathbf{A} := h^2\Delta_{\mathbf{k}}$ and considering cutoff functions $\chi_\eps \in C^\infty_{\comp}(\R)$ which approximate the characteristic function of the interval $[a,b]$, see \cite[Theorem 14.11]{Zworski-12} for a proof.
\end{proof}

\subsubsection{Local Weyl law estimate} Let $C(R)$ denote the `square' defined by
\[
	C(R) := \{(\mathbf{k},\lambda) \in \Omega ~|~ |\mathbf{k}| \leq R, 0 \leq \lambda \leq R\} \subset \Omega,
\]
where $\Omega$ was introduced in \eqref{equation:omega-qe} and set
\[
r := \dim(P)-\dim(G/T)/2 = n + (\dim(G)+d)/2,
\]
where $d := \mathrm{rank}(G)$.

\begin{lemma}
There exists a constant $C > 1$ such that for all $R > 1$:
\begin{equation}
\label{equation:size-cr}
R^r/C \leq \sharp C(R) \leq C R^{r}.
\end{equation}
\end{lemma}

\begin{proof}
By Proposition \ref{proposition:weyl}, there exists $C > 1$ such that for all $|\mathbf{k}| \leq R$, the set $C_{\mathbf{k}}(R) := \{\lambda \leq R ~|~  (\mathbf{k},\lambda) \in C(R)\}$ satisfies
\begin{equation}
\label{equation:ineq-dk}
R^{n}d_{\mathbf{k}}/C \leq \sharp C_{\mathbf{k}}(R) \leq C R^n d_{\mathbf{k}},\quad R \gg 1,
\end{equation}
The upper bound in \eqref{equation:size-cr} then follows from the bound
\begin{equation}
\label{equation:weyl-dim}
d_{\mathbf{k}} \leq C|\mathbf{k}|^{\dim(G/T)/2},
\end{equation}
for some uniform $C > 0$ (independent of $\mathbf{k}$), which itself follows from the Weyl dimension formula, and the formula $\sharp C(R) = \sum_{|\mathbf{k}| \leq R} \sharp C_{\mathbf{k}}(R)$. To obtain the lower bound, one writes using \eqref{equation:ineq-dk}:
\begin{equation}
\label{equation:jeudi-matin}
\begin{split}
\sharp C(R) & = \sum_{|\mathbf{k}| \leq R} \sharp C_{\mathbf{k}}(R)  \gtrsim R^n \sum_{|\mathbf{k}| \leq R} d_{\mathbf{k}}^2 d_{\mathbf{k}}^{-1} \gtrsim R^{n-\dim(G/T)/2} \sum_{|\mathbf{k}| \leq R} d_{\mathbf{k}}^2.
\end{split}
\end{equation}
Now, recall by \S\ref{sssection:laplace-eigenvalue} that an element $u \in H^0(G/T,\mathbf{J}^{\otimes \mathbf{k}})$ can be identified with an eigenfunction of the (non-negative) Laplacian $\Delta_G$ on $G$ (with respect to a bi-invariant Riemannian metric), associated to the eigenvalue $c(\mathbf{k})$ defined in \eqref{equation:definition-ck}. Furthermore, one has the asymptotic \eqref{equation:ck-asymptotic} for $c(\mathbf{k})$. This implies that there exists $\eps > 0$ small enough such that for $R \gg 1$ large enough, $\sum_{|\mathbf{k}| \leq R} d_{\mathbf{k}}^2 \geq \sum_{c(\mathbf{k}) \leq \eps R^2} d_{\mathbf{k}}^2$. By the Peter-Weyl theorem, the last sum corresponds to the number of eigenvalues of $\Delta_G$ less or equal than $\eps R^2$. By the Weyl law applied to $\Delta_G$, this number is bounded from below by $\gtrsim R^{\dim(G)}$. Going back to \eqref{equation:jeudi-matin}, we thus find:
\[
\sharp C(R) \gtrsim R^{n-\dim(G/T)/2} R^{\dim(G)} = R^{r},
\]
by definition of $r$. This proves the claim.
\end{proof}

In what follows, we write $B_E \HH^*$ for $E \geq 0$ to denote the ball bundle $\{ \xi \in \HH^* \mid |\xi| \leq E\}$ over $F$. The proof of Theorem \ref{theorem:quantum-ergodicity} relies on the following local Weyl law type of lemma (recall that $\Delta_{\mathbf{k}}u_{\mathbf{k},\lambda} = \lambda^2 u_{\mathbf{k},\lambda}$ and $\|u_{\mathbf{k},\lambda}\|_{L^2}=1$):

\begin{lemma}
\label{lemma:local-weyl}
There exists a constant $C > 0$ such that for all $a \in C^\infty_{\mathrm{comp}}(\HH^*)$ and all $R > 1,\eps > 0$,
\begin{equation}
\label{equation:local-weyl}
\begin{split}
R^{-r} \sum_{(\mathbf{k},\lambda) \in C(R)} \left|\langle \Op^{\mathrm{BW}}_{R^{-1}}(a) u_{\mathbf{k},\lambda},u_{\mathbf{k},\lambda}\rangle_{L^2(F,\mathbf{L}^{\otimes \mathbf{k}})}\right|^2 \leq C \|a\|^2_{L^2(B_{1+\eps}\HH^*)} + \mc{O}_{a,\eps}(R^{-1}).
\end{split}
\end{equation}
\end{lemma}

It should be possible to take $\eps = 0$ in \eqref{equation:local-weyl} but this will not be needed for the proof of Theorem \ref{theorem:quantum-ergodicity}.

\begin{proof}
We use the notation $h := 1/R$. Observe that, by the algebraic properties of the Borel-Weil calculus (see \S\ref{sssection:properties-bw}), one has:
\[
\begin{split}
|\langle \Op^{\mathrm{BW}}_h(a) u_{\mathbf{k},\lambda},u_{\mathbf{k},\lambda}\rangle_{L^2(F,\mathbf{L}^{\otimes \mathbf{k}})}|^2 &\leq \|\Op^{\mathrm{BW}}_h(a) u_{\mathbf{k},\lambda}\|^2_{L^2(F,\mathbf{L}^{\otimes \mathbf{k}})} \\
& = \langle \Op^{\mathrm{BW}}_h(a)^*\Op^{\mathrm{BW}}_h(a) u_{\mathbf{k},\lambda},u_{\mathbf{k},\lambda}\rangle_{L^2(F,\mathbf{L}^{\otimes \mathbf{k}})} \\
& = \langle (\Op^{\mathrm{BW}}_h(|a|^2) + \mc{O}_{\Psi^{-\infty}}(h))u_{\mathbf{k},\lambda},u_{\mathbf{k},\lambda}\rangle_{L^2(F,\mathbf{L}^{\otimes \mathbf{k}})} \\
& = \langle \Op_h^{\mathrm{BW}}(|a|^2)u_{\mathbf{k},\lambda},u_{\mathbf{k},\lambda}\rangle_{L^2(F,\mathbf{L}^{\otimes \mathbf{k}})} + \mc{O}_a(R^{-1}),
\end{split}
\]		
where we used in the last equality that $h=1/R$ and $\|u_{\mathbf{k},\lambda}\|_{L^2}=1$. Using \eqref{equation:size-cr}, we thus obtain:
\begin{equation}
\label{equation:coolos}
\begin{split}
R^{-r} \sum_{(\mathbf{k},\lambda) \in C(R)}& |\langle \Op^{\mathrm{BW}}_h(a) u_{\mathbf{k},\lambda},u_{\mathbf{k},\lambda}\rangle_{L^2(F,\mathbf{L}^{\otimes \mathbf{k}})}|^2 \\
&\leq R^{-r} \sum_{(\mathbf{k},\lambda) \in C(R)} \langle \Op_h^{\mathrm{BW}}(|a|^2)u_{\mathbf{k},\lambda},u_{\mathbf{k},\lambda}\rangle_{L^2(F,\mathbf{L}^{\otimes \mathbf{k}})} + \mc{O}_a(R^{-1}).
\end{split}
\end{equation}
By definition, for $(\mathbf{k},\lambda) \in C(R)$, one has $\Delta_{\mathbf{k}} u_{\mathbf{k},\lambda} = \lambda^2 u_{\mathbf{k},\lambda}$ and thus $(h^2 \Delta_{\mathbf{k}}- (h\lambda)^2)u_{\mathbf{k},\lambda} = 0$ with $h\lambda \leq 1$ (recall here that the principal symbol of $h^2\Delta_{\mathbf{k}}$ with respect to the connection $\nabla$ itself is independent of the family $h \mapsto \mathbf{k}(h)$). Hence $u_{\mathbf{k},\lambda}$ has microsupport in the unit ball $B_1 \HH^*$. Let $\eps > 0$, and let $\chi_\eps \in C^\infty_{\mathrm{comp}}( \HH^*, [0,1])$ be a smooth cutoff function equal to $1$ on $B_{1+\eps/2}\HH^*$ and $0$ on $\HH^* \setminus B_{1+\eps}$. One has:
\[
\langle \Op_h^{\mathrm{BW}}(|a|^2(1-\chi_\eps))u_{\mathbf{k},\lambda},u_{\mathbf{k},\lambda}\rangle_{L^2(F,\mathbf{L}^{\otimes \mathbf{k}})} = \mc{O}_{a,\eps}(h^\infty) = \mc{O}_{a,\eps}(R^{-\infty}), 
\]
which implies:
\[
\langle \Op_h^{\mathrm{BW}}(|a|^2)u_{\mathbf{k},\lambda},u_{\mathbf{k},\lambda}\rangle_{L^2(F,\mathbf{L}^{\otimes \mathbf{k}})} = \langle \Op_h^{\mathrm{BW}}(|a|^2\chi_\eps)u_{\mathbf{k},\lambda},u_{\mathbf{k},\lambda}\rangle_{L^2(F,\mathbf{L}^{\otimes \mathbf{k}})} +  \mc{O}_{a,\eps}(R^{-\infty}).
\]
As a consequence, we obtain the same inequality as \eqref{equation:coolos} where, in the right-hand side, the function $|a|^2$ is replaced by $|a|^2\chi_\eps$ and the $\mc{O}_a(R^{-1})$ is now equal to $\mc{O}_{a,\eps}(R^{-1})$. Since $|a|^2\chi_\eps \geq 0$, the sharp Gårding inequality (see Lemma \ref{lemma:garding} and its proof) implies the existence of an operator $\mathbf{R} \in \Psi^{-\infty}_{h, \mathrm{BW}}(P)$ such that $\Op_h^{\mathrm{BW}}(|a|^2\chi_\eps) + h\mathbf{R} \geq 0$ is a nonnegative operator, that is 
\[
	\langle (\Op_h^{\mathrm{BW}}(|a|^2\chi_\eps) + h\mathbf{R})u,u\rangle_{L^2(F,\mathbf{L}^{\otimes \mathbf{k}})} \geq 0, \quad u \in L^2(F,\mathbf{L}^{\otimes \mathbf{k}}).
\] 
We thus find:
\[
\begin{split}
R^{-r} & \sum_{(\mathbf{k},\lambda) \in C(R)} |\langle \Op^{\mathrm{BW}}_h(a) u_{\mathbf{k},\lambda},u_{\mathbf{k},\lambda}\rangle_{L^2(F,\mathbf{L}^{\otimes \mathbf{k}})}|^2 \\
& \leq R^{-r} \sum_{(\mathbf{k},\lambda) \in C(R)} \langle (\Op_h^{\mathrm{BW}}(|a|^2\chi_\eps)+ h\mathbf{R})u_{\mathbf{k},\lambda},u_{\mathbf{k},\lambda}\rangle_{L^2(F,\mathbf{L}^{\otimes \mathbf{k}})} + \mc{O}_{a,\eps}(R^{-1}) \\
& \leq R^{-r} \sum_{|\mathbf{k}|\leq R} \sum_{\lambda \in \mathrm{spec}(\Delta_{\mathbf{k}})} \langle (\Op_h^{\mathrm{BW}}(|a|^2\chi_\eps)+ h\mathbf{R})u_{\mathbf{k},\lambda},u_{\mathbf{k},\lambda}\rangle_{L^2(F,\mathbf{L}^{\otimes \mathbf{k}})} + \mc{O}_{a,\eps}(R^{-1}) \\
& = R^{-r} \sum_{|\mathbf{k}| \leq R} \Tr(\Op_h^{\mathrm{BW}}(|a|^2\chi_\eps)+ h\mathbf{R}) + \mc{O}_{a,\eps}(R^{-1}).
\end{split}
\]
(The factor $h\mathbf{R}$ is used in the second inequality to justify the appearance of the sum over the spectrum of $\Delta_{\mathbf{k}}$.) By Lemma \ref{lemma:trace} and \eqref{equation:l1}, the previous inequality yields:
\[
\begin{split}
&R^{-r} \sum_{(\mathbf{k},\lambda) \in C(R)} |\langle \Op^{\mathrm{BW}}_h(a) u_{\mathbf{k},\lambda},u_{\mathbf{k},\lambda}\rangle_{L^2(F,\mathbf{L}^{\otimes \mathbf{k}})}|^2 \\
& \leq R^{-r} \sum_{|\mathbf{k}| \leq R}  \dfrac{d_{\mathbf{k}} R^n}{\vol(G/T)(2\pi)^n} \left(\int_{T^*M} \int_{F_x} |a|^2\chi_\eps(x,w,d\pi^\top\xi)\, \dd w\dd x \dd \xi  + \mc{O}_{a,\eps}(R^{-1})\right)\\
& \qquad + \mc{O}_{a,\eps}(R^{-1})
\end{split}
\]
Using the Weyl dimension formula \eqref{equation:weyl-dim}, we find that:
\[
\sum_{|\mathbf{k}| \leq R}  \dfrac{d_{\mathbf{k}} R^n}{\vol(G/T)(2\pi)^n}  \leq C R^{n+\dim(G/T)/2+\dim(T)}  = \mc{O}_{a,\eps}(R^{r}),
\]
since by definition of $r = n + (\dim(G)+\dim(T))/2$. We thus obtain:
\[
R^{-r} \sum_{(\mathbf{k},\lambda) \in C(R)} |\langle \Op^{\mathrm{BW}}_h(a) u_{\mathbf{k},\lambda},u_{\mathbf{k},\lambda}\rangle_{L^2(F,\mathbf{L}^{\otimes \mathbf{k}})}|^2 \leq C \|a\|^2_{L^2(B_{1+\eps}\HH^*)} + \mc{O}_{a,\eps}(R^{-1}),
\]
which concludes the proof of the lemma.
\end{proof}

\subsection{End of the proof}

We can now prove Theorems \ref{theorem:quantum-ergodicity} and \eqref{theorem:quantum-ergodicity2}.

\begin{proof}[Proof of Theorem \ref{theorem:quantum-ergodicity}]
Set $N(R) := \sharp(\Omega \cap B(0,R))$. By standard measure-theoretic arguments (see \cite[Proof of Theorem 15.5]{Zworski-12}), the proof of Theorem \ref{theorem:quantum-ergodicity} boils down to proving that for all $a \in C^\infty(F)$, one has:
\begin{equation}
\label{equation:bien}
\dfrac{1}{N(R)} \sum_{(\mathbf{k},\lambda) \in \Omega \cap B(0,R)} \left|\langle au_{\mathbf{k},\lambda}, u_{\mathbf{k},\lambda}\rangle - \dfrac{1}{\vol(F)}\int_F a(w)\, \dd w \right|^2 \to_{R\to\infty} 0.
\end{equation}
Up to changing $a$ by $a- \tfrac{1}{\vol(F)}\int_F a(w)\, \dd w$, we can further assume that $\int_F a(w) \dd w=0$ and prove \eqref{equation:bien} in this case.

Notice that
\[
\begin{split}
\dfrac{1}{N(R)} \sum_{(\mathbf{k},\lambda) \in \Omega \cap B(0,R)} \left|\langle au_{\mathbf{k},\lambda}, u_{\mathbf{k},\lambda}\rangle \right|^2 \leq \dfrac{R^r}{N(R)} R^{-r}\sum_{(\mathbf{k},\lambda) \in C(R)} \left|\langle au_{\mathbf{k},\lambda}, u_{\mathbf{k},\lambda}\rangle \right|^2.
\end{split}
\]
Since by \eqref{equation:size-cr} we have $N(R) \geq CR^r$ for some uniform $C > 0$, we obtain that $R^r/N(R)$ is uniformly bounded as $R \to \infty$. Hence, it suffices to show that
\begin{equation}
\label{equation:kkk}
R^{-r}\sum_{(\mathbf{k},\lambda) \in C(R)} \left|\langle au_{\mathbf{k},\lambda}, u_{\mathbf{k},\lambda}\rangle \right|^2 \to_{R\to\infty} 0.
\end{equation}
Let $\pi : \HH^* \to F$ denote the projection. Let $\psi \in C^\infty_{\mathrm{comp}}(\HH^*)$ be an arbitrary cutoff function equal to $1$ on $B_2\HH^*$. As in the proof of Lemma \ref{lemma:local-weyl}, one has $\Op_h^{\mathrm{BW}}((1-\psi) \pi^* a) u_{\mathbf{k},\lambda} = \mc{O}(h^\infty)$ by the support property of $\psi$ and that the microsupport of $u_{\mathbf{k}, \lambda}$ is contained in $B_1\HH^*$, for all $(\mathbf{k},\lambda) \in C(R)$. Furthermore, by symbolic calculus, one has $\Op_h^{\mathrm{BW}}(\pi^*a) = \pi^*a + \mc{O}_{\Psi^{-1}}(h)$, where we recall $h := R^{-1}$. This implies that 
\begin{equation}
\label{equation:vavite2}
\langle a u_{\mathbf{k},\lambda}, u_{\mathbf{k},\lambda}\rangle_{L^2} = \langle \Op_h^{\mathrm{BW}}(\psi~\pi^* a) u_{\mathbf{k},\lambda}, u_{\mathbf{k},\lambda}\rangle_{L^2} + \mc{O}(h),
\end{equation}
and thus
\begin{equation}
\label{equation:to}
R^{-r}\sum_{(\mathbf{k},\lambda) \in C(R)} \left|\langle au_{\mathbf{k},\lambda}, u_{\mathbf{k},\lambda}\rangle \right|^2 = R^{-r}\sum_{(\mathbf{k},\lambda) \in C(R)} \left|\langle \Op^{\mathrm{BW}}_h(\psi~\pi^*a)u_{\mathbf{k},\lambda}, u_{\mathbf{k},\lambda}\rangle \right|^2 + \mc{O}(R^{-1}).
\end{equation}
So it now suffices to prove \eqref{equation:kkk} with $a$ replaced by $\Op_h^{\mathrm{BW}}(\psi ~\pi^* a)$.

Using Egorov's theorem \eqref{equation:egorov-utile}, we have
\[
\begin{split}
 \langle \Op_h^{\mathrm{BW}}(\psi~\pi^* a) u_{\mathbf{k},\lambda}, u_{\mathbf{k},\lambda}\rangle_{L^2} & =  \langle \Op_h^{\mathrm{BW}}(\psi~\pi^* a) e^{-it h\Delta_{\mathbf{k}}}u_{\mathbf{k},\lambda}, e^{-it h\Delta_{\mathbf{k}}} u_{\mathbf{k},\lambda}\rangle_{L^2} \\
 & = \langle e^{\tfrac{it}{h} h^2\Delta_{\mathbf{k}}}\Op_h^{\mathrm{BW}}(\psi~\pi^* a) e^{-\tfrac{it}{h} h^2\Delta_{\mathbf{k}}}u_{\mathbf{k},\lambda},  u_{\mathbf{k},\lambda}\rangle_{L^2}  \\
 & = \langle \Op_{h}^{\mathrm{BW}}((\psi~\pi^*a) \circ \Phi_t)u_{\mathbf{k},\lambda},  u_{\mathbf{k},\lambda}\rangle_{L^2} + \mc{O}(h) \\
 & = \langle \Op_{h}^{\mathrm{BW}}(\langle\psi~\pi^*a\rangle_T) u_{\mathbf{k},\lambda},  u_{\mathbf{k},\lambda}\rangle_{L^2} + \mc{O}_T(h),
 \end{split}
\]
where $T >0$ is arbitrary, we recall that $(\Phi_t)_{t \in \mathbb{R}}$ is the Hamiltonian flow of $|\xi_{\HH}|^2$ with respect to $\omega_0$, and we use the notation
\[
\langle b\rangle_T := \dfrac{1}{T} \int_0^T b \circ \Phi_t ~ \dd t,\quad b \in C^\infty(\HH^*).
\]
By the previous equality and \eqref{equation:local-weyl} (we also use \eqref{equation:size-cr}), we then obtain:
\[
\begin{split}
R^{-r}\sum_{(\mathbf{k},\lambda) \in C(R)}& \left|\langle \Op^{\mathrm{BW}}_h(\psi~\pi^*a)u_{\mathbf{k},\lambda}, u_{\mathbf{k},\lambda}\rangle \right|^2\\
&  = R^{-r}\sum_{(\mathbf{k},\lambda) \in C(R)} \left|\langle \Op^{\mathrm{BW}}_h(\langle\psi~\pi^*a\rangle_T) u_{\mathbf{k},\lambda}, u_{\mathbf{k},\lambda}\rangle \right|^2 + \mc{O}_T(R^{-1}) \\
& \leq C\|\langle\pi^*a\rangle_T\|^2_{L^2(B_{1+\eps}\HH^*)} + \mc{O}_{T, \varepsilon}(R^{-1}),
\end{split}
\]
where in the last line we used that $\psi = 1$ on $B_{1 + \varepsilon}\HH^*$. Using \eqref{equation:to}, taking the $\limsup$ in $R \to \infty$, and then $\varepsilon \to 0$ (we may do so since the constant $C$ does not depend on $\varepsilon > 0$), we therefore obtain:
\begin{equation}
\label{equation:finito}
\limsup_{R \to \infty} R^{-r}\sum_{(\mathbf{k},\lambda) \in C(R)} \left|\langle au_{\mathbf{k},\lambda}, u_{\mathbf{k},\lambda}\rangle \right|^2 \leq C\|\langle \pi^*a\rangle_T\|^2_{L^2(B_{1}\HH^*)}.
\end{equation}
Now, observe that for each energy layer $S_E\HH^* := \{|\xi_{\HH^*}|=E\}$ with $E > 0$, the Hamiltonian flow $(\Phi_t|_{S_E\HH^*})_{t \in \R}$ is ergodic by Lemma \ref{lemma:ergodic-layer}. (For $E=0$, this flow is the identity map on $S_0\HH^* \simeq F$.) This implies that for all $b \in C^\infty(\HH^*)$,
\[
\langle b|_{S_E\HH^*}\rangle_T \to_{T \to \infty} \dfrac{1}{\vol(S_E\HH^*)}\int_{S_E\HH^*} b(w,\xi)\, \dd\mu_E(w,\xi),
\]
where the convergence holds in $L^2(S_E\HH^*)$, and we write $\dd\mu_E$ for the Riemannian measure on $S_E \HH^*$. As a consequence, we find that
\[
\|\langle\pi^*a\rangle_T\|^2_{L^2(B_{1}\HH^*)} = \int_{E=0}^{1} \int_{S_E\HH^*} \langle\pi^*a\rangle_T^2\, \dd\mu_E(w,\xi)\, \dd E \to_{T \to \infty} 0,
\]
by dominated convergence, since
\[
\int_{S_E\HH^*} \pi^*a(w,\xi)\, \dd\mu_E(w,\xi) = \vol(\mathbb{S}^{n - 1}_E) \int_F a(w)\, \dd w = 0,
\]
by assumption, where $\mathbb{S}^{n - 1}_E \subset \mathbb{R}^n$ denotes the sphere of radius $E$ in $\mathbb{R}^n$. Taking the limit as $T \to \infty$ in \eqref{equation:finito}, we finally obtain \eqref{equation:kkk}, which concludes the proof.
\end{proof}

We now prove Theorem \ref{theorem:quantum-ergodicity2}. Since the proof is very similar to the proof of Theorem \ref{theorem:quantum-ergodicity}, instead of going through the same steps again, we emphasize the main changes in the argument.

\begin{proof}[Proof of Theorem \ref{theorem:quantum-ergodicity2}] We divide the proof into two steps.
\medskip

\emph{Step 1.} First, the estimate \eqref{equation:local-weyl} in Lemma \ref{lemma:local-weyl} should be replaced by the following one: there exists a constant $C > 0$ such that for all $a \in C^\infty_{\mathrm{comp}}(\HH^* \times \mathfrak{a})$ and all $R > 1,\delta > 0$,
\begin{equation}
\label{equation:local-weyl2}
\begin{split}
R^{-r} \sum_{(\mathbf{k},\lambda) \in C(R)} & \left|\langle \Op^{\mathrm{BW}}_{h}(a(\bullet,\mathbf{k}/R)) u_{\mathbf{k},\lambda},u_{\mathbf{k},\lambda}\rangle_{L^2(F,\mathbf{L}^{\otimes \mathbf{k}})}\right|^2 \\
& \qquad \leq C \int_{y \in B_1\mathfrak{a}_+} \|a(\bullet,y)\|^2_{L^2(B_{1+\delta}\HH^*)}\, \dd y + \mc{O}_{a,\delta}(R^{-1}),
\end{split}
\end{equation}
where $\dd y$ is the Riemannian volume on the Lie algebra $\mathfrak{a}$ (induced by the choice of metric on $G$), we recall that $h = R^{-1}$, and $B_1\mathfrak{a}_+$ denotes the unit ball in $\mathfrak{a}_+$. The proof is \emph{verbatim} the same as Lemma \ref{lemma:local-weyl} (keeping track of the dependence of $a$ on the new variable), except at the very last step where the left-hand side of \eqref{equation:local-weyl2} is now bounded by
\[
\begin{split}
R^{-r} \sum_{|\mathbf{k}| \leq R}&  \dfrac{d_{\mathbf{k}} R^n}{\vol(G/T)(2\pi)^n} \int_{\HH^*} |a|^2\chi_\delta(w,\xi, \mathbf{k}/R)\, \dd w \dd \xi  + \mc{O}_{a,\delta}(R^{-1}) \\
& \qquad \leq C R^{-\mathrm{rank}(G)} \sum_{|\mathbf{k}| \leq R}  \|a\chi_\delta^{1/2}(\bullet,\mathbf{k}/R)\|^2_{L^2(\HH^*)} + \mc{O}_{a,\delta}(R^{-1}),
\end{split}
\]
where we recall that $\chi_\delta \in C^\infty_{\mathrm{comp}}(\HH^*, [0, 1])$ is a cutoff function equal to $1$ on $B_{1 + \frac{\delta}{2}}\HH^*$ and to $0$ outside $B_{1 + \delta}\HH^*$. Also, we used in the last inequality the Weyl dimension formula, see \eqref{equation:weyl-dim}, and that $r=n+(\dim(G)+\mathrm{rank}(G))/2$. This is now a Riemann sum and can thus be bounded by the term on the right-hand side of \eqref{equation:local-weyl2}.
\medskip

\emph{Step 2.} The proof is then very similar to that of Theorem \ref{theorem:quantum-ergodicity} and boils down, by a mere counting argument, to showing that for $a \in C^\infty(F)$ with $0$ average,
\[
R^{-r}\sum_{(\mathbf{k},\lambda) \in C(R), \lambda \geq \eps R} \left|\langle au_{\mathbf{k},\lambda}, u_{\mathbf{k},\lambda}\rangle \right|^2 \to_{R\to\infty} 0.
\]
Modulo $\mc{O}(R^{-1})$ remainders, the condition $\lambda \geq \eps R$ allows to replace $\pi^*a$ by $a'$ such that $a' \equiv \pi^* a$ on the energy shells $S_E \HH^*$ for $\eps - \delta' \leq E \leq 1 + \delta'$ and $a'$ has compact support in the energy shells $[\varepsilon - 2\delta', 1+2\delta']$ (where $\delta' > 0$ can be taken arbitrarily small) since the $u_{\mathbf{k},\lambda}$'s have semiclassical microsupport in the energy shells $[\eps,1]$. One then finds using \eqref{equation:local-weyl2} that:
\[
\begin{split}
R^{-r}\sum_{(\mathbf{k},\lambda) \in C(R), \lambda \geq \eps R}& \left|\langle au_{\mathbf{k},\lambda}, u_{\mathbf{k},\lambda}\rangle \right|^2 \\
& \leq R^{-r}\sum_{(\mathbf{k},\lambda) \in C(R)} \left|\langle \Op^{\mathrm{BW}}_h(a')u_{\mathbf{k},\lambda}, u_{\mathbf{k},\lambda}\rangle \right|^2 + \mc{O}(R^{-1})\\
&  = R^{-r}\sum_{(\mathbf{k},\lambda) \in C(R)} \left|\langle \Op^{\mathrm{BW}}_h(\langle a'\rangle_T^{\omega_{h,\mathbf{k}}}) u_{\mathbf{k},\lambda}, u_{\mathbf{k},\lambda}\rangle \right|^2 + \mc{O}_T(R^{-1}) \\
& \leq C \int_{y \in B_1\mathfrak{a}_+} \|\langle a' \rangle_T^{\omega(y)}\|^2_{L^2(\HH^*)} \,\dd y + \mc{O}_T(R^{-1}),
\end{split}
\]
where $\omega(y) := \omega_0 + iy\cdot\mathbf{F}_{\overline{\nabla}}$ and $y \in \mathfrak{a}_+$, the positive Weyl chamber. Taking the limsup as $R \to +\infty$ and then the limit as $\delta' \to 0$, one finds:
\[
\begin{split}
\limsup_{R \to \infty} R^{-r}\sum_{(\mathbf{k},\lambda) \in C(R), \lambda \geq \eps R}& \left|\langle au_{\mathbf{k},\lambda}, u_{\mathbf{k},\lambda}\rangle \right|^2 \\
& \leq C \int_{y \in B_1\mathfrak{a}_+} \int_{E=\eps}^1 \|\langle \pi^*a\rangle_T^{\omega(y)}\|^2_{L^2(S_E\HH^*)}\, \dd E \dd y.
\end{split}
\]
By assumption, $a$ has $0$ average in $F$ and so has $\pi^*a$ on $S_E\HH^*$ for any $E > 0$. Since the connection is by assumption $\eps$-admissible, all the flows $(\Phi_t^{\omega(y)})_{t \in \R}$ are ergodic on the energy shells $S_E\HH^*$ for $E \in [\eps,1]$. Hence by dominated convergence, the previous integral converges to $0$ as $T \to \infty$. This completes the proof.
\end{proof}

\newpage

\hspace{1cm}

\bibliographystyle{alpha}
\bibliography{Biblio}

\begin{thebibliography}{GBGHW24}

\bibitem[AN07]{Anantharaman-Nonnenmacher-07}
Nalini Anantharaman and St{\'e}phane Nonnenmacher.
\newblock Half-delocalization of eigenfunctions for the {Laplacian} on an
  {Anosov} manifold.
\newblock {\em Ann. Inst. Fourier}, 57(7):2465--2523, 2007.

\bibitem[Ana08]{Anantharaman-08}
Nalini Anantharaman.
\newblock Entropy and the localization of eigenfunctions.
\newblock {\em Ann. Math. (2)}, 168(2):435--475, 2008.

\bibitem[BBB81]{BerardBergery-Bourguignon-81}
Lionel B{\'e}rard~Bergery and Jean-Pierre Bourguignon.
\newblock Laplacian and {Riemannian} submersions with totally geodesic fibres.
\newblock Global differential geometry and global analysis, {Proc}. {Colloq}.,
  {Berlin} 1979, {Lect}. {Notes} {Math}. 838, 30-35 (1981)., 1981.

\bibitem[BD18]{Bourgain-Dyatlov-18}
Jean Bourgain and Semyon Dyatlov.
\newblock Spectral gaps without the pressure condition.
\newblock {\em Ann. Math. (2)}, 187(3):825--867, 2018.

\bibitem[BdMG81]{BoutetDeMonvel-Guillemin-81}
L.~Boutet~de Monvel and V.~Guillemin.
\newblock {\em The spectral theory of {Toeplitz} operators}, volume~99 of {\em
  Ann. Math. Stud.}
\newblock Princeton University Press, Princeton, NJ, 1981.

\bibitem[BG80]{Brin-Gromov-80}
Michael Brin and Mikhael Gromov.
\newblock On the ergodicity of frame flows.
\newblock {\em Invent. Math.}, 60(1):1--7, 1980.

\bibitem[BK84]{Brin-Karcher-83}
Michael Brin and Hermann Karcher.
\newblock Frame flows on manifolds with pinched negative curvature.
\newblock {\em Compositio Math.}, 52(3):275--297, 1984.

\bibitem[BL23]{Bonthonneau-Lefeuvre-23}
Yannick~Guedes Bonthonneau and Thibault Lefeuvre.
\newblock Radial source estimates in {H{\"o}lder}-{Zygmund} spaces for
  hyperbolic dynamics.
\newblock {\em Ann. Henri Lebesgue}, 6:643--686, 2023.

\bibitem[BMR24]{Ben-Ovadia-Ma-Rogdriguez-Hertz-24}
Snir {Ben Ovadia}, Qiaochu {Ma}, and Federico {Rodriguez-Hertz}.
\newblock {Mixed quantization and partial hyperbolicity}.
\newblock {\em arXiv e-prints}, September 2024.

\bibitem[BMZ17]{Bismut-Ma-Zhang-17}
Jean-Michel Bismut, Xiaonan Ma, and Weiping Zhang.
\newblock Asymptotic torsion and {T}oeplitz operators.
\newblock {\em J. Inst. Math. Jussieu}, 16(2):223--349, 2017.

\bibitem[Bor53]{Borel-53}
Armand Borel.
\newblock Sur la cohomologie des espaces fibr\'{e}s principaux et des espaces
  homog\`enes de groupes de {L}ie compacts.
\newblock {\em Ann. of Math. (2)}, 57:115--207, 1953.

\bibitem[BP02]{Burns-Paternain-02}
Keith Burns and Gabriel~P. Paternain.
\newblock Anosov magnetic flows, critical values and topological entropy.
\newblock {\em Nonlinearity}, 15(2):281--314, 2002.

\bibitem[BP03]{Burns-Pollicott-03}
Keith Burns and Mark Pollicott.
\newblock Stable ergodicity and frame flows.
\newblock {\em Geom. Dedicata}, 98:189--210, 2003.

\bibitem[Bri75a]{Brin-75-2}
Michael Brin.
\newblock Topological transitivity of a certain class of dynamical systems, and
  flows of frames on manifolds of negative curvature.
\newblock {\em Funkcional. Anal. i Prilo\v{z}en.}, 9(1):9--19, 1975.

\bibitem[Bri75b]{Brin-75-1}
Michael Brin.
\newblock The topology of group extensions of {$C$}-systems.
\newblock {\em Mat. Zametki}, 18(3):453--465, 1975.

\bibitem[Bri82]{Brin-82}
Michael Brin.
\newblock Ergodic theory of frame flows.
\newblock In {\em Ergodic theory and dynamical systems, {II} ({C}ollege {P}ark,
  {M}d., 1979/1980)}, volume~21 of {\em Progr. Math.}, pages 163--183.
  Birkh\"{a}user, Boston, Mass., 1982.

\bibitem[BT82]{Bott-Tu-82}
Raoul Bott and Loring~W. Tu.
\newblock {\em Differential forms in algebraic topology}, volume~82 of {\em
  Graduate Texts in Mathematics}.
\newblock Springer-Verlag, New York-Berlin, 1982.

\bibitem[BtD85]{Brocker-Dieck-85}
Theodor Br\"{o}cker and Tammo tom Dieck.
\newblock {\em Representations of compact {L}ie groups}, volume~98 of {\em
  Graduate Texts in Mathematics}.
\newblock Springer-Verlag, New York, 1985.

\bibitem[BW99]{Burns-Wilkinson-99}
Keith Burns and Amie Wilkinson.
\newblock Stable ergodicity of skew products.
\newblock {\em Ann. Sci. {\'E}c. Norm. Sup{\'e}r. (4)}, 32(6):859--889, 1999.

\bibitem[CdV85]{Colindeverdiere-85}
Yves Colin~de Verdi{\`e}re.
\newblock Ergodicit{\'e} et fonctions propres du {Laplacien}.
\newblock S{\'e}min., {\'E}quations {D{\'e}riv}. {Partielles} 1984-1985, {Exp}.
  {No}. 13, 7 p. (1985)., 1985.

\bibitem[CEG23]{Chen-Erchenko-Gogolev-23}
Dong Chen, Alena Erchenko, and Andrey Gogolev.
\newblock Riemannian {Anosov} extension and applications.
\newblock {\em J. {\'E}c. Polytech., Math.}, 10:945--987, 2023.

\bibitem[CG21]{Cekic-Guillarmou-21}
Mihajlo Ceki\'{c} and Colin Guillarmou.
\newblock First band of {R}uelle resonances for contact {A}nosov flows in
  dimension 3.
\newblock {\em Comm. Math. Phys.}, 386(2):1289--1318, 2021.

\bibitem[Cha]{Charles-00}
Laurent Charles.
\newblock {\em Semi-classical aspects of geometric quantization}.

\bibitem[Cha23]{Charles-23}
Laurent Charles.
\newblock {\em Semiclassical analysis on Kähler manifolds}.
\newblock available online, 2023.

\bibitem[CL21]{Cekic-Lefeuvre-21-2}
Mihajlo {Ceki{\'c}} and Thibault {Lefeuvre}.
\newblock {Generic injectivity of the X-ray transform}.
\newblock {\em accepted in Journal of Differential Geometry}, July 2021.

\bibitem[CL22]{Cekic-Lefeuvre-22}
Mihajlo {Ceki{\'c}} and Thibault {Lefeuvre}.
\newblock {Isospectral connections, ergodicity of frame flows, and polynomial
  maps between spheres}.
\newblock {\em arXiv e-prints}, page arXiv:2209.11109, September 2022.

\bibitem[CL24]{Charles-Lefeuvre-24}
L.~Charles and T.~Lefeuvre.
\newblock Semiclassical defect measures of magnetic laplacians on surface.
\newblock {\em In preparation}, 2024.

\bibitem[CLMS21]{Cekic-Lefeuvre-Moroianu-Semmelmann-21}
Mihajlo {Ceki{\'c}}, Thibault {Lefeuvre}, Andrei {Moroianu}, and Uwe
  {Semmelmann}.
\newblock {On the ergodicity of the frame flow on even-dimensional manifolds}.
\newblock {\em arXiv e-prints}, page arXiv:2111.14811, November 2021.

\bibitem[CLMS22]{Cekic-Lefeuvre-Moroianu-Semmelmann-22}
Mihajlo Ceki{\'c}, Thibault Lefeuvre, Andrei Moroianu, and Uwe Semmelmann.
\newblock Towards {Brin}'s conjecture on frame flow ergodicity: new progress
  and perspectives.
\newblock {\em Math. Res. Rep. (Amst.)}, 3:21--34, 2022.

\bibitem[CLMS23]{Cekic-Lefeuvre-Moroianu-Semmelmann-23}
Mihajlo {Ceki{\'c}}, Thibault {Lefeuvre}, Andrei {Moroianu}, and Uwe
  {Semmelmann}.
\newblock {On the ergodicity of unitary frame flows on K{\"a}hler manifolds}.
\newblock {\em arXiv e-prints}, page arXiv:2301.05933, January 2023.

\bibitem[CP]{Crovisier-Potrie-notes}
Sylvain Crovisier and Rafael Potrie.
\newblock Introduction to partially hyperbolic dynamics.

\bibitem[CS98]{Croke-Sharafutdinov-98}
Christopher~B. Croke and Vladimir~A. Sharafutdinov.
\newblock Spectral rigidity of a compact negatively curved manifold.
\newblock {\em Topology}, 37(6):1265--1273, 1998.

\bibitem[DJ18]{Dyatlov-Jin-18}
Semyon Dyatlov and Long Jin.
\newblock Semiclassical measures on hyperbolic surfaces have full support.
\newblock {\em Acta Math.}, 220(2):297--339, 2018.

\bibitem[DJN22]{Dyatlov-Jin-Nonnenmacher-22}
Semyon Dyatlov, Long Jin, and St{\'e}phane Nonnenmacher.
\newblock Control of eigenfunctions on surfaces of variable curvature.
\newblock {\em J. Am. Math. Soc.}, 35(2):361--465, 2022.

\bibitem[DK00]{Duistermaat-Kolk-00}
J.~J. Duistermaat and J.~A.~C. Kolk.
\newblock {\em Lie groups}.
\newblock Universitext. Berlin: Springer, 2000.

\bibitem[Dol98]{Dolgopyat-98}
Dmitry Dolgopyat.
\newblock On decay of correlations in {Anosov} flows.
\newblock {\em Ann. Math. (2)}, 147(2):357--390, 1998.

\bibitem[Dol02]{Dolgopyat-02}
Dmitry Dolgopyat.
\newblock On mixing properties of compact group extensions of hyperbolic
  systems.
\newblock {\em Isr. J. Math.}, 130:157--205, 2002.

\bibitem[Don88]{Donnay-88}
Victor~J. Donnay.
\newblock Geodesic flow on the two-sphere. {II}. {E}rgodicity.
\newblock In {\em Dynamical systems ({C}ollege {P}ark, {MD}, 1986--87)}, volume
  1342 of {\em Lecture Notes in Math.}, pages 112--153. Springer, Berlin, 1988.

\bibitem[Dro17]{Drouot-17}
Alexis Drouot.
\newblock Stochastic stability of {P}ollicott-{R}uelle resonances.
\newblock {\em Comm. Math. Phys.}, 356(2):357--396, 2017.

\bibitem[DS99]{Dimassi-Sjostrand-99}
Mouez Dimassi and Johannes Sj\"{o}strand.
\newblock {\em Spectral asymptotics in the semi-classical limit}, volume 268 of
  {\em London Mathematical Society Lecture Note Series}.
\newblock Cambridge University Press, Cambridge, 1999.

\bibitem[DS03]{Dairbekov-Sharafutdinov-03}
Nurlan~S. Dairbekov and Vladimir~A. Sharafutdinov.
\newblock Some problems of integral geometry on {A}nosov manifolds.
\newblock {\em Ergodic Theory Dynam. Systems}, 23(1):59--74, 2003.

\bibitem[Dya22]{Dyatlov-22}
Semyon Dyatlov.
\newblock Around quantum ergodicity.
\newblock {\em Ann. Math. Qu{\'e}.}, 46(1):11--26, 2022.

\bibitem[DZ15]{Dyatlov-Zworski-15}
Semyon Dyatlov and Maciej Zworski.
\newblock Stochastic stability of {P}ollicott-{R}uelle resonances.
\newblock {\em Nonlinearity}, 28(10):3511--3533, 2015.

\bibitem[DZ16]{Dyatlov-Zworski-16}
Semyon Dyatlov and Maciej Zworski.
\newblock Dynamical zeta functions for {A}nosov flows via microlocal analysis.
\newblock {\em Ann. Sci. \'{E}c. Norm. Sup\'{e}r. (4)}, 49(3):543--577, 2016.

\bibitem[DZ17]{Dyatlov-Zworski-17}
Semyon Dyatlov and Maciej Zworski.
\newblock Ruelle zeta function at zero for surfaces.
\newblock {\em Invent. Math.}, 210(1):211--229, 2017.

\bibitem[DZ19]{Dyatlov-Zworski-19}
Semyon Dyatlov and Maciej Zworski.
\newblock {\em Mathematical theory of scattering resonances}, volume 200 of
  {\em Grad. Stud. Math.}
\newblock Providence, RI: American Mathematical Society (AMS), 2019.

\bibitem[EL24]{Erchenko-Lefeuvre-24}
Alena Erchenko and Thibault Lefeuvre.
\newblock Marked boundary rigidity for surfaces of {Anosov} type.
\newblock {\em Math. Z.}, 306(3):22, 2024.
\newblock Id/No 36.

\bibitem[FMT07]{Field-Melbourne-Torok-07}
Michael Field, Ian Melbourne, and Andrei T\"{o}r\"{o}k.
\newblock Stability of mixing and rapid mixing for hyperbolic flows.
\newblock {\em Ann. of Math. (2)}, 166(1):269--291, 2007.

\bibitem[FRS08]{Faure-Roy-Sjostrand-08}
Fr\'{e}d\'{e}ric Faure, Nicolas Roy, and Johannes Sj\"{o}strand.
\newblock Semi-classical approach for {A}nosov diffeomorphisms and {R}uelle
  resonances.
\newblock {\em Open Math. J.}, 1:35--81, 2008.

\bibitem[FS11]{Faure-Sjostrand-11}
Fr\'{e}d\'{e}ric Faure and Johannes Sj\"{o}strand.
\newblock Upper bound on the density of {R}uelle resonances for {A}nosov flows.
\newblock {\em Comm. Math. Phys.}, 308(2):325--364, 2011.

\bibitem[GBGHW24]{Bonthonneau-Guillarmou-Hilgert-Weich-23}
Yannick Guedes~Bonthonneau, Colin Guillarmou, Joachim Hilgert, and Tobias
  Weich.
\newblock Ruelle-taylor resonances of anosov actions.
\newblock 2024.
\newblock to appear in Journal of the European Mathematical Society.

\bibitem[GBGW24]{Bonthonneau-Guillarmou-Weich-24}
Yannick Guedes~Bonthonneau, Colin Guillarmou, and Tobias Weich.
\newblock Srb measures for anosov actions.
\newblock 2024.
\newblock to appear in Journal of Differential Geometry.

\bibitem[GK80]{Guillemin-Kazhdan-80}
Victor Guillemin and David Kazhdan.
\newblock Some inverse spectral results for negatively curved {$2$}-manifolds.
\newblock {\em Topology}, 19(3):301--312, 1980.

\bibitem[GK21]{Guillarmou-Kuster-21}
Colin Guillarmou and Benjamin K{\"u}ster.
\newblock Spectral theory of the frame flow on hyperbolic 3-manifolds.
\newblock {\em Ann. Henri Poincar{\'e}}, 22(11):3565--3617, 2021.

\bibitem[GKL22]{Guillarmou-Knieper-Lefeuvre-22}
Colin Guillarmou, Gerhard Knieper, and Thibault Lefeuvre.
\newblock Geodesic stretch, pressure metric and marked length spectrum
  rigidity.
\newblock {\em Ergodic Theory Dyn. Syst.}, 42(3):974--1022, 2022.

\bibitem[GL19]{Guillarmou-Lefeuvre-18}
Colin Guillarmou and Thibault Lefeuvre.
\newblock The marked length spectrum of {A}nosov manifolds.
\newblock {\em Ann. of Math. (2)}, 190(1):321--344, 2019.

\bibitem[Gui17]{Guillarmou-17-1}
Colin Guillarmou.
\newblock Invariant distributions and {X}-ray transform for {A}nosov flows.
\newblock {\em J. Differential Geom.}, 105(2):177--208, 2017.

\bibitem[Hat]{Hatcher-notes}
Allen Hatcher.
\newblock Spectral sequences.

\bibitem[Hat02]{Hatcher-02}
Allen Hatcher.
\newblock {\em Algebraic topology}.
\newblock Cambridge University Press, Cambridge, 2002.

\bibitem[HK09]{Helffer-Kordyukov-09}
B.~Helffer and Y.~A. Kordyukov.
\newblock Spectral gaps for periodic {Schr{\"o}dinger} operators with
  hypersurface magnetic wells: analysis near the bottom.
\newblock {\em J. Funct. Anal.}, 257(10):3043--3081, 2009.

\bibitem[H{\"o}r67]{Hormander-67}
Lars H{\"o}rmander.
\newblock Hypoelliptic second order differential equations.
\newblock {\em Acta Math.}, 119:147--171, 1967.

\bibitem[Kan93]{Kanai-93}
Masahiko Kanai.
\newblock Differential-geometric studies on dynamics of geodesic and frame
  flows.
\newblock {\em Jpn. J. Math., New Ser.}, 19(1):1--30, 1993.

\bibitem[KM12]{Kahn-Markovic-12}
Jeremy Kahn and Vladimir Markovic.
\newblock Immersing almost geodesic surfaces in a closed hyperbolic three
  manifold.
\newblock {\em Ann. of Math. (2)}, 175(3):1127--1190, 2012.

\bibitem[KN63]{Kobayashi-Nomizu-63}
Shoshichi Kobayashi and Katsumi Nomizu.
\newblock {\em Foundations of differential geometry. {I}}, volume~15 of {\em
  Intersci. Tracts Pure Appl. Math.}
\newblock Interscience Publishers, New York, NY, 1963.

\bibitem[KN69]{Kobayashi-Nomizu-69}
Shoshichi Kobayashi and Katsumi Nomizu.
\newblock {\em Foundations of differential geometry. {Vol}. {II}}, volume~15 of
  {\em Intersci. Tracts Pure Appl. Math.}
\newblock Interscience Publishers, New York, NY, 1969.

\bibitem[Kna02]{Knapp-02}
Anthony~W. Knapp.
\newblock {\em Lie groups beyond an introduction}, volume 140 of {\em Prog.
  Math.}
\newblock Boston, MA: Birkh{\"a}user, 2nd ed. edition, 2002.

\bibitem[KWL20]{Kolb-Weich-Wolf-20}
Martin {Kolb}, Tobias {Weich}, and Lasse {Lennart Wolf}.
\newblock {Spectral Asymptotics for Kinetic Brownian Motion on Surfaces of
  Constant Curvature}.
\newblock {\em arXiv e-prints}, page arXiv:2011.06434, November 2020.

\bibitem[Lab13]{Labourie-13}
Fran{\c{c}}ois Labourie.
\newblock {\em Lectures on representations of surface groups}.
\newblock Zur. Lect. Adv. Math. Z{\"u}rich: European Mathematical Society
  (EMS), 2013.

\bibitem[{Lef}]{Lefeuvre-book}
Thibault {Lefeuvre}.
\newblock {Microlocal analysis in hyperbolic dynamics and geometry}.
\newblock {\em In preparation, available at \url{thibaultlefeuvre.blog}}.

\bibitem[Lef23]{Lefeuvre-23}
Thibault Lefeuvre.
\newblock Isometric extensions of {Anosov} flows via microlocal analysis.
\newblock {\em Commun. Math. Phys.}, 399(1):453--479, 2023.

\bibitem[Lin06]{Lindenstrauss-06}
Elon Lindenstrauss.
\newblock Invariant measures and arithmetic unique ergodicity. {Appendix} by
  {E}. {Lindenstrauss} and {D}. {Rudolph}.
\newblock {\em Ann. Math. (2)}, 163(1):165--219, 2006.

\bibitem[Ma24]{Ma-24}
Qiaochu Ma.
\newblock Toeplitz operators and the full asymptotic torsion forms.
\newblock {\em J. Funct. Anal.}, 286(3):Paper No. 110210, 74, 2024.

\bibitem[Mar22]{Maret-22}
Arnaud Maret.
\newblock A note on character varieties, 2022.

\bibitem[Mel94]{Melrose-94}
Richard~B. Melrose.
\newblock Spectral and scattering theory for the {L}aplacian on asymptotically
  {E}uclidian spaces.
\newblock In {\em Spectral and scattering theory ({S}anda, 1992)}, volume 161
  of {\em Lecture Notes in Pure and Appl. Math.}, pages 85--130. Dekker, New
  York, 1994.

\bibitem[MM23]{Ma-Ma-23}
Minghui {Ma} and Qiaochu {Ma}.
\newblock {Semiclassical analysis, geometric representation and quantum
  ergodicity}.
\newblock {\em arXiv e-prints}, page arXiv:2302.09497, February 2023.

\bibitem[Moo87]{Moore-87}
Calvin~C. Moore.
\newblock Exponential decay of correlation coefficients for geodesic flows.
\newblock In {\em Group representations, ergodic theory, operator algebras, and
  mathematical physics ({B}erkeley, {C}alif., 1984)}, volume~6 of {\em Math.
  Sci. Res. Inst. Publ.}, pages 163--181. Springer, New York, 1987.

\bibitem[MP11]{Merry-Paternain-11}
Will Merry and Gabriel~P. Paternain.
\newblock {\em Inverse Problems in Geometry and Dynamics}.
\newblock Lecture notes, 2011.
\newblock Available online:
  \url{https://www.dpmms.cam.ac.uk/~gpp24/ipgd(3).pdf}.

\bibitem[MS75]{Melin-Sjostrand-75}
Anders Melin and Johannes Sj{\"o}strand.
\newblock Fourier integral operators with complex-valued phase functions.
\newblock Fourier {Integr}. {Oper}. part. differ. {Equat}., {Colloq}. int.
  {Nice} 1974, {Lect}. {Notes} {Math}. 459, 120-223 (1975)., 1975.

\bibitem[OK99]{Omori-Kobayashi-99}
Hideki Omori and Takao Kobayashi.
\newblock Global hypoellipticity of subelliptic operators on closed manifolds.
\newblock {\em Hokkaido Math. J.}, 28(3):613--633, 1999.

\bibitem[Omo91]{Omori-91}
Hideki Omori.
\newblock On global hypoellipticity of horizontal {Laplacians} on compact
  principal bundles.
\newblock {\em Hokkaido Math. J.}, 20(2):185--194, 1991.

\bibitem[Pla72]{Plante-72}
Joseph~F. Plante.
\newblock Anosov flows.
\newblock {\em Am. J. Math.}, 94:729--754, 1972.

\bibitem[Pol92]{Pollicott-92}
Mark Pollicott.
\newblock Exponential mixing for the geodesic flow on hyperbolic
  three-manifolds.
\newblock {\em J. Statist. Phys.}, 67(3-4):667--673, 1992.

\bibitem[PP96]{Paternain-Paternain-96}
Gabriel~P. Paternain and Miguel Paternain.
\newblock Anosov geodesic flows and twisted symplectic structures.
\newblock In {\em International {C}onference on {D}ynamical {S}ystems
  ({M}ontevideo, 1995)}, volume 362 of {\em Pitman Res. Notes Math. Ser.},
  pages 132--145. Longman, Harlow, 1996.

\bibitem[PS96]{Pugh-Shub-96}
Charles Pugh and Michael Shub.
\newblock Stable ergodicity and partial hyperbolicity.
\newblock In {\em 1st international conference on dynamical systems,
  Montevideo, Uruguay, 1995 - a tribute to Ricardo Ma\~n\'e. Proceedings},
  pages 182--187. Harlow: Longman, 1996.

\bibitem[PSU14]{Paternain-Salo-Uhlmann-14-1}
Gabriel~P. Paternain, Mikko Salo, and Gunther Uhlmann.
\newblock Tensor tomography: progress and challenges.
\newblock {\em Chin. Ann. Math. Ser. B}, 35(3):399--428, 2014.

\bibitem[PSU23]{Paternain-Salo-Uhlmann-23}
Gabriel~P. Paternain, Mikko Salo, and Gunther Uhlmann.
\newblock {\em Geometric inverse problems---with emphasis on two dimensions},
  volume 204 of {\em Cambridge Studies in Advanced Mathematics}.
\newblock Cambridge University Press, Cambridge, 2023.
\newblock With a foreword by Andr\'as Vasy.

\bibitem[Puc23]{Puchol-23}
Martin Puchol.
\newblock The asymptotics of the holomorphic analytic torsion forms.
\newblock {\em J. Lond. Math. Soc. (2)}, 108(1):80--140, 2023.

\bibitem[PZ24]{Pollicott-Zhang-24}
M.~Pollicott and D.~Zhang.
\newblock Rapid mixing for compact group extensions of hyperbolic flows.
\newblock {\em in preparation}, 2024.

\bibitem[Riv10]{Riviere-10}
Gabriel Rivi{\`e}re.
\newblock Entropy of semiclassical measures in dimension 2.
\newblock {\em Duke Math. J.}, 155(2):271--335, 2010.

\bibitem[RS94]{Rudnick-Sarnak-94}
Ze{\'e}v Rudnick and Peter Sarnak.
\newblock The behaviour of eigenstates of arithmetic hyperbolic manifolds.
\newblock {\em Commun. Math. Phys.}, 161(1):195--213, 1994.

\bibitem[RS96]{Runst-Sickel-96}
Thomas Runst and Winfried Sickel.
\newblock {\em Sobolev spaces of fractional order, {Nemytskij} operators and
  nonlinear partial differential equations}, volume~3 of {\em De Gruyter Ser.
  Nonlinear Anal. Appl.}
\newblock Berlin: de Gruyter, 1996.

\bibitem[RT22a]{Ren-Tao-22b}
Qiuyu {Ren} and Zhongkai {Tao}.
\newblock {Spectral asymptotics for kinetic Brownian motion on locally
  symmetric spaces}.
\newblock {\em arXiv e-prints}, page arXiv:2208.13111, August 2022.

\bibitem[RT22b]{Ren-Tao-22a}
Qiuyu {Ren} and Zhongkai {Tao}.
\newblock {Spectral asymptotics for kinetic Brownian motion on Riemannian
  manifolds}.
\newblock {\em arXiv e-prints}, page arXiv:2212.05394, December 2022.

\bibitem[Sep07]{Sepanski-07}
Mark~R. Sepanski.
\newblock {\em Compact {Lie} groups}, volume 235 of {\em Grad. Texts Math.}
\newblock New York, NY: Springer, 2007.

\bibitem[Shn74]{Shnirelman-74-1}
A.~I. Shnirel'man.
\newblock Ergodic properties of eigenfunctions.
\newblock {\em Usp. Mat. Nauk}, 29(6(180)):181--182, 1974.

\bibitem[Shn22]{Shnirelman-74-2}
Alexander Shnirelman.
\newblock Statistical properties of eigenfunctions.
\newblock {\em Ann. Math. Qu{\'e}.}, 46(1):3--9, 2022.

\bibitem[Tay68]{Taylor-68}
Michael~E. Taylor.
\newblock Fourier series on compact {L}ie groups.
\newblock {\em Proc. Amer. Math. Soc.}, 19:1103--1105, 1968.

\bibitem[Vil70]{Vilms-70}
Jaak Vilms.
\newblock Totally geodesic maps.
\newblock {\em J. Differential Geometry}, 4:73--79, 1970.

\bibitem[Wil11]{Wilkinson-11}
Amie Wilkinson.
\newblock Conservative partially hyperbolic dynamics.
\newblock In {\em Proceedings of the international congress of mathematicians
  (ICM 2010), Hyderabad, India, August 19--27, 2010. Vol. III: Invited
  lectures}, pages 1816--1836. Hackensack, NJ: World Scientific; New Delhi:
  Hindustan Book Agency, 2011.

\bibitem[Zel87]{Zelditch-87}
Steven Zelditch.
\newblock Uniform distribution of eigenfunctions on compact hyperbolic
  surfaces.
\newblock {\em Duke Math. J.}, 55:919--941, 1987.

\bibitem[Zwo12]{Zworski-12}
Maciej Zworski.
\newblock {\em Semiclassical analysis}, volume 138 of {\em Graduate Studies in
  Mathematics}.
\newblock American Mathematical Society, Providence, RI, 2012.

\end{thebibliography}

\end{document}